\documentclass[11pt]{book}
\usepackage[T2A]{fontenc}
\usepackage[cp1251]{inputenc}
\usepackage[english, russian, ukrainian]{babel}
\usepackage{amsmath,amssymb,amsthm}
\usepackage{mathrsfs}
\usepackage{titlesec}
\usepackage{indentfirst}
\usepackage{graphicx}
\usepackage{amsmath, amsthm, amscd, amsfonts, amssymb, graphicx, color}
\usepackage[bookmarksnumbered, colorlinks, plainpages]{hyperref}

\newtheorem{theorem}{\indent Теорема}[chapter]

\newtheorem{lemma}{\indent Лема}[chapter]

\newtheorem{corollary}{\indent Наслідок}[chapter]

\newtheorem{condition}{\indent Умова}[chapter]

\theoremstyle{definition}

\newtheorem{remark}{\indent Зауваження}[chapter]

\renewcommand{\chaptermark}[1]{\markboth{#1}{#1}}

\allowdisplaybreaks[3]

\titlelabel{\thetitle.\quad}

\renewcommand{\phi}{\alpha}

\renewcommand{\rho}{\varrho}

\usepackage{cite}

\begin{document}

\thispagestyle{empty}


\begin{center}

\bf

\LARGE

В. М. ЛОСЬ\\ В. А. МИХАЙЛЕЦЬ \\ О. О. МУРАЧ

\vspace{1cm}

\huge

ПАРАБОЛІЧНІ

\medskip

ГРАНИЧНІ ЗАДАЧІ

\medskip

ТА УЗАГАЛЬНЕНІ

\medskip

ПРОСТОРИ СОБОЛЄВА

\vspace{1cm}

\LARGE

V. M. LOS\\ V. A. MIKHAILETS \\ A. A. MURACH

\vspace{1cm}

\huge

PARABOLIC

\medskip

BOUNDARY-VALUE

\medskip

PROBLEMS

\medskip

AND GENERALIZED

\medskip

SOBOLEV SPACES

\end{center}

\normalsize

\rm

\newpage

\thispagestyle{empty}


\smallskip

У монографії викладено засади нової теорії параболічних початково-крайових задач у шкалах узагальнених анізотропних просторів Соболєва. Ці шкали калібровані істотно більш тонко за допомогою функціонального показника регулярності, ніж класи просторів, які раніше використовували в теорії параболічних диференціальних рівнянь. Працю вирізняє систематичне застосування методу квадратичної інтерполяції з функціональним параметром абстрактних і соболєвських гільбертових просторів.

Для науковців, викладачів університетів, аспірантів і студентів старших курсів, які спеціалізуються в галузі диференціальних рівнянь і математичного аналізу.

\medskip

\noindent Рецензенти:

\noindent доктор фізико-математичних наук, \\
професор \hfill\fbox{\emph{С.~Д.~Івасишен}}

\noindent член-кореспондент НАН України,\\
доктор фізико-математичних наук  \hfill\emph{А.~Н.~Кочубей}

\bigskip

The research monograph expounds the foundation of a new theory of parabolic initial-boundary-value problems in scales of generalized anisotropic Sobolev spaces. These scales are calibrated essentially more finely with the help of a function parameter of regularity than the classes of spaces used earlier in the theory of parabolic problems. The monograph is featured by the systematic application of the method of quadratic interpolation (with function parameter) between abstract and Sobolev inner product spaces.

The monograph is intended for the researches, professors, PhD students, and senior students specializing in differential equations and mathematical analysis.

\medskip

\noindent Reviewers:

\medskip

\noindent Professor \hfill\fbox{\emph{S.~D.~Ivasyshen}}

\medskip

\noindent Professor  \hfill\emph{A.~N.~Kochubei}

\vspace{3mm}

\begin{flushright}
© В. М. Лось, В. А. Михайлець, О. О. Мурач, 2021
\end{flushright}

\newpage

 \chapter*{\textbf{Вступ}}
\addcontentsline{toc}{section}{\textbf{Вступ}}

\chaptermark{Вступ}

\noindent Запропонована монографія присвячена новітній теорії розв'язності параболічних початково-крайових задач у функціональних просторах, що утворюють шкали, калібровані більш тонко, ніж шкали анізотропних просторів Соболєва або Гельдера, які зазвичай використовуються у дослідженні параболічних задач. Це досягається за допомогою функціонального параметра, що задає додаткову (позитивну або негативну) регулярність функцій, яку не можна розрізнити за допомогою числових параметрів. Простори, використані у монографії, належать до класу функціональних просторів, уведених Л.~Хермандером~\cite{Hormander63}, і являють собою широке і змістовне узагальнення анізотропних просторів Соболєва.

Класичну теорію розв'язності параболічних задач у просторах Соболєва і Гельдера було створено у працях М.~С.~Аграновича і М.~І.~Вішика \cite{AgranovichVishik64}, С.~Д.~Ейдельмана \cite{Eidelman64}, O.~О.~Ладиженської, В.~О.~Солоннікова і Н.~М.~Уральцевої \cite{LadyzhenskajaSolonnikovUraltzeva67}, A.~Фрідмана \cite{Fridman68}, Ж.-Л.~Ліонса і Е.~Мадженеса \cite{LionsMagenes72ii}, С.~Д.~Івасишена \cite{Ivasyshen90}, М.~В.~Житарашу \cite{EidelmanZhitarashu98} та інших математиків у 60--80-х роках ХХ ст. Її ядро утворюють теореми про ізоморфізми, які стверджують, що параболічні задачі є коректними у сенсі Адамара на відповідних парах анізотропних просторів Соболєва або Гельдера, тобто обмежені оператори, породжені задачами, є ізоморфізмами на цих парах. Вказані теореми відіграють ключову роль у дослідженні регулярності розв'язків параболічних задач, їх функцій Гріна, задач керування для параболічних систем та інших важливих питань. З~точки зору застосувань окремий інтерес викликають теореми про ізоморфізми у випадку гільбертових просторів, структура яких подібна до геометрії скінченновимірних евклідових просторів. У цьому випадку використовуються анізотропні простори Соболєва, побудовані на основі просторів квадратично інтегровних функції.

Звісно, чим тонше калібрована шкала функціональних просторів, тим точніші результати можна отримати за її допомогою. Калібрування числовими параметрами, які служать показниками регулярності для просторів Соболєва і Гельдера, виявилося досить грубим для низки важливих задач математичного аналізу \cite{Lizorkin86, Stepanets87, Stepanets02, Stepanets05, Triebel01, Triebel10}, теорії диференціальних рівнянь з частинними похідними \cite{Hormander63, Hormander86, MikhailetsMurach10, MikhailetsMurach14, NicolaRodino10, Paneah00}, теорії стохастичних процесів \cite{Jacob010205} та інших. Це з'ясувалося ще 60 років тому принаймні стосовно багатовимірних диференціальних рівнянь і спонукало Л.~Хермандера \cite{Hormander63} увести та дослідити нормовані функціональні простори
$$
\mathcal{B}_{p,\mu}:=\bigl\{w\in\mathcal{S}'(\mathbb{R}^{k})
:\mu\widehat{w}\in L_p(\mathbb{R}^{k})\bigr\},
$$
параметризовані за допомогою достатньо загальної функції $\mu:\nobreak\mathbb{R}^{k}\to(0,\infty)$. Тут числовий параметр $p$ задовольняє умову $1\leq p\leq\infty$, а  $\widehat{w}$ позначає перетворення Фур'є повільно зростаючого розподілу $w$ з простору Шварца $\mathcal{S}'(\mathbb{R}^{k})$. Функція $\mu$ частотного аргументу є показником регулярності для простору $\mathcal{B}_{p,\mu}$; вона характеризує регулярність розподілу $w$ у термінах поведінки на нескінченності його перетворення Фур'є.

У праці Л.~Хермандера \cite{Hormander63}, яка стала важливим етапом розвитку сучасної теорії диференціальних рівнянь з частинними похідним, простори $\mathcal{B}_{p,\mu}$ були систематично застосовані до дослідження умов існування та регулярності розв'язків багатовимірних диференціальних рівнянь. Найбільш повні результати отримано для гіпоеліптичних рівнянь, до яких належать і параболічні рівняння. У~випадку $p=2$ простори Хермандера стають гільбертовими і дають широке узагальнення $L_{2}$-просторів Соболєва. Так, якщо $\mu(\xi)=(1+|\xi|^{2})^{s/2}$ для довільного вектора $\xi\in\mathbb{R}^{k}$ і деякого деякого дійсного числа $s$, то $\mathcal{B}_{2,\mu}$~--- гільбертів простір Соболєва порядку~$s$. Цим результатам присвячено і другий том \cite{Hormander86} відомої чотиритомної монографії Л.~Хермандера «The Analysis of Linear Partial Differential Operators», яку називають біблією теорії лінійних багатовимірних диференціальних рівнянь.

Праця Л.~Хермандера \cite{Hormander63} привернула значну увагу до узагальнених соболєвських просторів і стимулювала їх різні дослідження та застосування, переважно у математичному аналізі. Втім, до недавнього часу у теорії багатовимірних крайових задач ці простори практично не використовували. Це було зумовлено відсутністю коректного означення вказаних просторів на гладких многовидах (у тому сенсі, що такі простори не повинні залежати від вибору локальних карт, які покривають многовид). Окрім того, бракувало зручного аналітичного інструментарію для роботи з цими просторами.

Недавно ситуація кардинально змінилася. У 2005--2010 роках зусиллями В.~А.~Михайлеця та О.~О.~Мурача було створено теорію еліптичних крайових задач в узагальнених соболєвських просторах вигляду  $H^{s;\varphi}:=\mathcal{B}_{2,\mu}$, де
$$
\mu(\xi)\equiv(1+|\xi|^{2})^{s/2}\,\varphi((1+|\xi|^{2})^{1/2}),
$$
число $s$ дійсне, а функція $\varphi:[1,\infty)\to(0,\infty)$ повільно змінюється на нескінченності за Караматою. Стандартними прикладами такої функції $\varphi$ є логарифмічна функція, її довільні ітерації та їх довільні степені. Ці простори ізотропні (тобто  інваріантні відносно рухів у евклідовому просторі), оскільки показник регулярності залежить лише від $|\xi|$. Для них функціональний параметр $\varphi$ задає регулярність, підпорядковану основній (степеневій) регулярності, заданої числовим параметром~$s$. Головним аналітичним методом у цій теорії є квадратична інтерполяція з функціональним параметром гільбертових просторів і операторів, що діють на них. За його допомогою можна коректно означити вказані простори на гладких многовидах і суттєво полегшити їх застосування до еліптичних диференціальних операторів та еліптичних крайових задач. Цю теорію викладено у працях \cite{MikhailetsMurach10, MikhailetsMurach14}.

Зовсім недавно автори запропонованої монографії застосували цей метод у теорії параболічних початково-крайових задач. За його допомогою встановлено коректну розв'язність параболічних задач в анізотропних узагальнених просторах Соболєва $H^{s,s/(2b);\varphi}:=\mathcal{B}_{2,\mu}$, де
$$
\mu(\xi',\xi_{k})=\bigl(1+|\xi'|^2+|\xi_{k}|^{1/b}\bigr)^{s/2}
\varphi\bigl((1+|\xi'|^2+|\xi_{k}|^{1/b})^{1/2}\bigr)
$$
для довільних аргументів $\xi'\in\mathbb{R}^{k-1}$ і $\xi_{k}\in\mathbb{R}$. Тут число $s$ додатне, число $b$ натуральне, а функція $\varphi$ є такою, як зазначено вище. Анізотропія цих просторів задається за допомогою параметра $2b$, що характеризує параболічність диференціального рівнян\-ня, для якого ставиться початково-крайова задача. Наприклад, $2b=2$ для рівняння теплопровідності (а для еліптичних рівнянь формально покладають $2b=1$ і отримують ізотропні простори, наведені вище). Якщо $\varphi(\cdot)\equiv1$, то $H^{s,s/(2b);\varphi}$ стає анізотропним соболєвським простором порядку $s$ за просторовими змінними і порядку $s/(2b)$ за часовою змінною.

У цій монографії викладено теорію розв'язності параболічних початково-крайових задач у вказаних узагальнених просторах Соболєва. Теорію було створено авторами монографії в останнє десятиліття. Праця складається з трьох розділів. У розділі~1 розглянуто узагальнені соболєвські простори, потрібні для дослідження параболічних задач. Вивчено властивості цих просторів, насамперед інтерполяційні. З огляду на це обговорено метод квадратичної інтерполяції пар гільбертових просторів, який відіграє ключову роль у монографії. Саму теорію викладено у розділах 2 і 3. Розділ 2 присвячено параболічним задачам із однорідними початковими умовами (такі задачі називаємо напіводнорідними), а розділ~3~--- задачам із неоднорідними початковими умовами. Такий поділ матеріалу зумовлений тим, що досліджен\-ня напіводнорідних параболічних задач та їхні властивості мають свою специфіку, на відміну від неоднорідних задач. Основними результатами цих розділів є теореми про ізоморфізми, породжені параболічними задачами на парах відповідних узагальнених просторів Соболєва, та теореми про локальну регулярність розв'язків параболічних задач (впритул до межі циліндричної області задання параболічного рівняння). У межах кожного з цих розділів розглядаються спочатку простіші за будовою параболічні задачі (основні задачі для рівняння теплопровідності, одновимірні за просторовою змінною задачі), дослідження яких здається простішим. Наприклад, для вивчення одновимірних задач не потрібні простори на многовидах. Остаточні результати отримано для $2b$-параболічних за Петровським багатовимірних диференціальних рівнянь і загальних крайових умов довільного порядку.

\newpage

\chaptermark{}

\chapter{\textbf{Узагальнені простори Соболєва та їх інтерполяція}}\label{ch2}

\chaptermark{\emph Розд. \ref{ch2}. Узагальнені простори Соболєва та їх інтерполяція}

\noindent У цьому розділі розглянуто узагальнені простори Соболєва $H^{s,s\gamma;\varphi}$, для яких числові параметри $s\in\mathbb{R}$ і $s\gamma$, де $\gamma>0$, задають основну регулярність відповідно за просторовими і часовою змінними, а функціональний параметр $\varphi$ визначає додаткову регулярність за цими змінними. У випадку $\gamma=1/(2b)$ ці простори будуть потрібні для дослідження $2b$-параболічних диференціальних рівнянь і відповідних їм початково-крайових задач. Головна властивість цих просторів полягає у тому, що вони отримуються квадратичною інтерполяцією з деяким функціональним параметром анізотропних просторів Соболєва (їх маємо у випадку одиничної функції $\varphi$). Тому розділ~1 починається з обговорення квадратичної інтерполяції (з функціональним параметром) загальних гільбертових просторів та її властивостей, потрібних надалі. Вона є основним методом дослідження у монографії. Окрема увага приділяється  версіям просторів $H^{s,s\gamma;\varphi}$, які використовуються у теорії параболічних задач; а саме: просторам, заданим на циліндричних областях, їх бічних поверхнях, та версіям цих просторів для функцій, які анулюються, якщо часова змінна $t<0$ (останні позначаємо через $H^{s,s\gamma;\varphi}_{+}$). Окрім інтерполяційних властивостей цих просторів, обговорюються теореми про вкладення і сліди. З огляду на подальші застосування важливу роль відіграватимуть теорема про сліди, породжені оператором даних Коші, та теорема про вкладення цих просторів у анізотропні простори неперервно диференційовних функцій. Цим двом теоремам присвячено підрозділи 1.5 і 1.6.


\markright{\emph \ref{sec2.1}. Інтерполяція з функціональним параметром}

\section[Інтерполяція з функціональним параметром]{Інтерполяція з функціональним \\параметром}\label{sec2.1}

\markright{\emph \ref{sec2.1}. Інтерполяція з функціональним параметром}

У цьому підрозділі розглянемо метод квадратичної інтерполяції з функціональним параметром пар гільбертових просторів, який
введено Ч.~Фойашем і Ж.-Л.~Ліонсом \cite[с.~278]{FoiasLions61}. Для наших цілей достатньо обмежитися випадком сепарабельних комплексних гільбертових просторів.

Нехай $X:=[X_{0},X_{1}]$ є впорядкованою парою
сепарабельних комплексних гільбертових просторів таких, що $X_{1}$ ~--- щільний лінійний многовид у просторі $X_{0}$ і вкладення $X_{1}\hookrightarrow X_{0}$ неперервне. Таку пару називають регулярною. Для неї існує самоспряжений додатно визначений оператор $J$ у просторі $X_{0}$ з областю визначення $X_{1}$ такий, що $\|Jv\|_{X_0}=\|v\|_{X_1}$ для кожного $v\in X_{1}$. Оператор $J$  визначається парою $X$ однозначно і називається породжуючим оператором для~$X$ (див., наприклад, \cite[розд.~4, теорема~1.12]{KreinPetuninSemenov82}). Він задає ізометричний ізоморфізм $J:X_{1}\leftrightarrow X_{0}$.

Позначимо через $\mathcal{B}$ множину всіх вимірних
за Борелем функцій $\psi:(0,\infty)\rightarrow(0,\infty)$ таких, що $\psi$ обмежена на кожному відрізку $[a,b]$, де $0<a<b<\infty$, і функція $1/\psi$ обмежена на кожному промені $[a,\infty)$, де $a>0$.

Для заданої функції $\psi\in\mathcal{B}$ розглянемо (взагалі необмежений) оператор $\psi(J)$, означений у гільбертовому просторі $X_{0}$ як борелева функція $\psi$ від $J$. Цей оператор будується за допомогою спектральної теореми, застосованої до самоспряженого оператора $J$ (див., наприклад, \cite[розд.~XIII, \S~6]{BerezanskyUsSheftel90}). Позначимо через $[X_{0},X_{1}]_{\psi}$, або скорочено через
$X_{\psi}$, область визначення оператора $\psi(J)$, наділену скалярним добутком
\begin{equation*}
(v_{1},v_{2})_{X_{\psi}}:=(\psi(J)v_{1},\psi(J)v_{2})_{X_{0}}.
\end{equation*}
Лінійний простір $X_{\psi}$ є гільбертовим і сепарабельним відносно
цього скалярного добутку. Останній породжує норму
\begin{equation*}
\|v\|_{X_{\psi}}:=\|\psi(J)v\|_{X_{0}}.
\end{equation*}

Функцію $\psi\in\mathcal{B}$ називаємо \textit{інтерполяційним параметром}, якщо
для всіх регулярних пар $X=[X_{0},X_{1}]$ та $Y=[Y_{0},Y_{1}]$ гільбертових
просторів і для довільного лінійного відображення $T$, заданого на $X_{0}$,
виконується така властивість: якщо звуження відображення $T$ на $X_{j}$ є обмеженим
оператором $T:X_{j}\rightarrow Y_{j}$ для кожного $j\in\{0,1\}$, то і звуження
відображення $T$ на $X_{\psi}$ є також обмеженим оператором
$T:X_{\psi}\rightarrow Y_{\psi}$.

Якщо $\psi$ є інтерполяційним параметром, то будемо казати, що гільбертів простір $X_{\psi}$ отримано внаслідок квадратичної інтерполяції з функціональним параметром $\psi$ пари $X=\nobreak[X_{0},X_{1}]$ (або, інакше кажучи, між просторами $X_{0}$ і $X_{1}$). Крім того, будемо казати, що обмежений оператор $T:X_{\psi}\rightarrow Y_{\psi}$ є результатом інтерполяції операторів $T:X_{j}\rightarrow Y_{j}$, де $j\in\{0,1\}$. У цьому випадку виконуються неперервні та щільні вкладення
$X_{1}\hookrightarrow X_{\psi}\hookrightarrow X_{0}$.

Клас усіх інтерполяційних параметрів (у сенсі наведеного щойно означення) допускає конструктивний опис.
А саме, функція $\psi\in\nobreak\mathcal{B}$ є інтерполяційним параметром тоді і тільки тоді,
коли вона псевдоугнута в околі нескінченності. Остання властивість означає існування угнутої додатної
функції $\psi_{1}(r)$ аргументу $r\gg1$ такої, що обидві функції
$\psi/\psi_{1}$ і $\psi_{1}/\psi$ обмежені в деякому околі нескінченності. Цей критерій випливає з опису Ж.~Петре \cite{Peetre66, Peetre68} класу всіх інтерполяційних функцій
для вагових просторів типу $L_{p}(\mathbb{R}^{n})$ (див. \cite[п.~1.1.9]{MikhailetsMurach10, MikhailetsMurach14}).

У праці систематично використовується такий наслідок з цього критерію, наведений, наприклад в \cite[теорема 1.11]{MikhailetsMurach10, MikhailetsMurach14}.

\begin{theorem}\label{9prop6.1}
Припустимо, що функція $\psi\in\mathcal{B}$ правильно змінна на нескінченності порядку $\theta$, де  $0<\theta<1$, тобто
$$
\lim_{r\rightarrow\infty}\;\frac{\psi(\lambda r)}{\psi(r)}=
\lambda^{\theta}\quad\mbox{для кожного}\quad\lambda>0.
$$
Тоді $\psi$ є інтерполяційним параметром.
\end{theorem}

Зауважимо, що поняття правильно змінної функції увів Й.~Карамата~\cite{Karamata30a}. Як звичайно, припускається, що така функція є вимірною за Борелем в околі нескінченності. Правильно змінна функція (на нескінченності) порядку $\theta=0$ називається повільно змінною (за Караматою). Звісно, функція $\psi$ правильно змінна порядку $\theta\in\mathbb{R}$ тоді і тільки тоді, коли $\psi(r)\equiv r^{\theta}\psi_{0}(r)$ для деякої повільно змінної функції
$\psi_{0}$.

У важливому випадку степеневих функцій теорема~\ref{9prop6.1} приводить до класичної інтерполяційної теореми Ж.-Л.~Ліонса і С.~Г.~Крейна (див. їх монографії \cite[розд.4, п.~1.10]{KreinPetuninSemenov82} і
\cite[розд.~1, пп. 2 і~5]{LionsMagenes72i}). Згідно з нею функція $\psi(r)\equiv r^{\theta}$ є інтерполяційним параметром, якщо $0<\theta<1$. У цьому випадку показник $\theta$ розглядається як числовий параметр інтерполяції.

Сформулюємо три властивості інтерполяції, які будемо систематично використовувати у доведеннях. Перша з них дає змогу звести інтерполяцію деяких підпросторів або фактор-просторів до інтерполяції вихідних гільбертових просторів (див. \cite[теорема~1.6]{MikhailetsMurach10, MikhailetsMurach14} або \cite[п.~1.17.1, теорема~1]{Triebel80}). Як звичайно, підпростори припускаються замкненими. Проєктори на них вважаємо, взагалі кажучи, не ортогональними.

\begin{theorem}\label{9prop6.2}
Нехай $X=[X_{0},X_{1}]$ є регулярною парою гільбертових просторів, а $Y_{0}$ є підпростором простору $X_{0}$. Тоді $Y_{1}:=X_{1}\cap Y_{0}$ є підпростором простору $X_{1}$. Припустимо, що існує лінійне відображення $P:X_{0}\rightarrow X_{0}$, яке для кожного $j\in\{0,\,1\}$ є проєктором простору $X_{j}$ на його підпростір $Y_{j}$. Тоді пари $[Y_{0},Y_{1}]$ і $[X_{0}/Y_{0},X_{1}/Y_{1}]$ регулярні та для довільного інтерполяційного параметра $\psi\in\mathcal{B}$ виконуються рівності
\begin{align}\label{9f6.1}
[Y_{0},Y_{1}]_{\psi}&=X_{\psi}\cap Y_{0},\\
[X_{0}/Y_{0},X_{1}/Y_{1}]_{\psi}&=X_{\psi}/(X_{\psi}\cap Y_{0}) \label{9f6.2}
\end{align}
з еквівалентністю норм. Тут $X_{\psi}\cap Y_{0}$ є підпростором простору~$X_{\psi}$.
\end{theorem}

Друга властивість дає змогу звести інтерполяцію прямих сум гільбертових просторів до інтерполяції їхніх доданків (див. \cite[теорема~1.8]{MikhailetsMurach10, MikhailetsMurach14}).

\begin{theorem}\label{9prop6.3}
Нехай $[X_{0}^{(j)},X_{1}^{(j)}]$, де $j=1,\ldots,q$, є скінченним набором регулярних пар гільбертових просторів. Тоді
$$
\biggl[\,\bigoplus_{j=1}^{q}X_{0}^{(j)},\,
\bigoplus_{j=1}^{q}X_{1}^{(j)}\biggr]_{\psi}=\,
\bigoplus_{j=1}^{q}\bigl[X_{0}^{(j)},\,X_{1}^{(j)}\bigr]_{\psi}
$$
з рівністю норм. Тут функція $\psi\in\mathcal{B}$ є довільним інтерполяційним параметром.
\end{theorem}

Третя властивість показує, що повторне застосування інтерполяції з функціональним параметром
дає знову інтерполяцію з деяким функціональним параметром
(див. \cite[теорема~1.3]{MikhailetsMurach10, MikhailetsMurach14}).

\begin{theorem}\label{9prop6.4}
Нехай $\chi,\eta,\psi\in\mathcal{B}$ і функція $\chi/\eta$
обмежена в околі нескінченності. Означимо функцію $\omega$ за формулою
$\omega(r):=\chi(r)\psi(\eta(r)/\chi(r))$ для довільного де $r>0$. Тоді $\omega\in\mathcal{B}$ і $[X_{\chi},X_{\eta}]_{\psi}=X_{\omega}$ з рівністю норм для кожної регулярної пари $X$ гільбертових просторів.
Крім того, якщо $\chi$, $\eta$ і $\psi$ є інтерполяційними параметрами, то $\omega$ також є інтерполяційним параметром.
\end{theorem}

\markright{\emph \ref{sec2.3}. Узагальнені простори Соболєва}

\section[Узагальнені простори Соболєва]{Узагальнені простори Соболєва}\label{sec2.3}

\markright{\emph \ref{sec2.3}. Узагальнені простори Соболєва}

У 1963 році Л.~Хермандер \cite[п.~2.2]{Hormander63} запропонував широке і змістовне узагальнення просторів Соболєва на $\mathbb{R}^{k}$
у категорії гільбертових просторів; тут ціле
$k\geq1$. Л.~Хермандер використав для відповідних просторів позначення $\mathcal{B}_{2,\mu}$; ми застосовуємо для них символіку $H^{\mu}(\mathbb{R}^{k})$, успадковану від досить поширеного позначення гільбертових просторів Соболєва (див., наприклад, \cite{LionsMagenes72i}). Показником регулярності функцій або розподілів, що утворюють простір $H^{\mu}(\mathbb{R}^{k})$, є довільна вимірна за Борелем функція
$\mu:\mathbb{R}^{k}\rightarrow(0,\infty)$, яка задовольняє таку умову:
\begin{equation}\label{Hermander-cond}
\frac{\mu(\xi)}{\mu(\eta)}\leq
c\,(1+|\xi-\eta|)^{l}\quad\mbox{для довільних}\quad \xi,\eta\in\mathbb{R}^{k},
\end{equation}
де $c$ і $l$~--- деякі додатні числа, які не залежать від дійсних векторів $\xi$ і $\eta$.

За означенням комплексний лінійний простір $H^{\mu}(\mathbb{R}^{k})$ складається з усіх повільно зростаючих розподілів $w\in\mathcal{S}'(\mathbb{R}^{k})$ таких, що перетворення
Фур'є $\widehat{w}:=\mathcal{F}w$ розподілу $w$ є локально інтегровною за Лебегом функцією, яка задовольняє умову
\begin{equation*}
\int\limits_{\mathbb{R}^{k}}\mu^{2}(\xi)\,|\widehat{w}(\xi)|^{2}\,d\xi
<\infty.
\end{equation*}
Тут $\mathcal{S}'(\mathbb{R}^{k})$ позначає лінійний топологічний простір Л.~Шварца повільно зростаючих розподілів (узагальнених функцій) на $\mathbb{R}^{k}$. Він є дуальним до лінійного топологічного простору  $\mathcal{S}(\mathbb{R}^{n})$ усіх швидко спадних нескінченно диференційовних функцій на $\mathbb{R}^{n}$. Надалі, якщо не зазначено інше, усі функції та розподіли вважаємо комплекснозначними, а  функціональні простори~--- комплексними.

У просторі $H^{\mu}(\mathbb{R}^{k})$ означено скалярний добуток
за формулою
\begin{equation*}
(w_1,w_2)_{H^{\mu}(\mathbb{R}^{k})}=
\int\limits_{\mathbb{R}^{k}}\mu^{2}(\xi)\,\widehat{w_1}(\xi)\,
\overline{\widehat{w_2}(\xi)}\,d\xi,
\end{equation*}
де $w_1,w_2\in H^{\mu}(\mathbb{R}^{k})$. Останній породжує
норму
$$
\|w\|_{H^{\mu}(\mathbb{R}^{k})}:=(w,w)^{1/2}_
{H^{\mu}(\mathbb{R}^{k})}.
$$
Згідно з \cite[п.~2.2]{Hormander63} простір $H^{\mu}(\mathbb{R}^{k})$
є гільбертовим і сепарабельним відносно введеного в ньому скалярного добутку. Крім того, цей простір неперервно вкладений у
$\mathcal{S}'(\mathbb{R}^{k})$, а множина $C^{\infty}_{0}(\mathbb{R}^{k})$ фінітних нескінченно диференційовних функцій на $\mathbb{R}^{k}$ є щільною в ньому. Функціональний параметр $\mu$ називаємо показником регулярності для простору $H^{\mu}(\mathbb{R}^{k})$ та його версій для різних підмножин евклідового простору $\mathbb{R}^{k}$.

Версія простору $H^{\mu}(\mathbb{R}^{k})$ для довільної непорожньої відкритої множини
$V\subset\mathbb{R}^{k}$ вводиться у стандартний спосіб, а саме:
\begin{gather}\notag
H^{\mu}(V):=\bigl\{w\!\upharpoonright\!V:\,
w\in H^{\mu}(\mathbb{R}^{k})\bigr\},\\
\|u\|_{H^{\mu}(V)}:= \inf\bigl\{\|w\|_{H^{\mu}(\mathbb{R}^{k})}:\,w\in
H^{\mu}(\mathbb{R}^{k}),\;u=w\!\upharpoonright\!V\bigr\}, \label{8f40}
\end{gather}
де $u\in H^{\mu}(V)$. Тут, як звичайно, $w\!\upharpoonright\!V$ означає звуження розподілу $w\in H^{\mu}(\mathbb{R}^{k})$ на відкриту множину~$V$. Іншими словами, $H^{\mu}(V)$ є фактор-простором простору $H^{\mu}(\mathbb{R}^{k})$ за його підпростором
\begin{equation}\label{8f41}
H^{\mu}_{Q}(\mathbb{R}^{k}):=\bigl\{w\in
H^{\mu}(\mathbb{R}^{k}):\, \mathrm{supp}\,w\subseteq Q\bigr\}, \quad\mbox{де}\;\;Q:=\mathbb{R}^{k}\backslash V.
\end{equation}
Тому простір $H^{\mu}(V)$ є гільбертовим і сепарабельним. (Нагадаємо, що за означенням носій $\mathrm{supp}\,w$ розподілу $w$ є доповненням найширшої відкритої множини, на якій $w=0$. Якщо функція $w$ неперервна, то її носій збігається із замиканням множини усіх точок $x$ таких, що $w(x)\neq0$. Звісно, при цьому доповнення і замикання беруться в області визначення розподілу чи функції.)

Норма
\eqref{8f40} породжена скалярним добутком
$$
(u_{1},u_{2})_{H^{\mu}(V)}:= (w_{1}-\Upsilon
w_{1},w_{2}-\Upsilon w_{2})_{H^{\mu}(\mathbb{R}^{k})},
$$
де $w_{j}\in H^{\mu}(\mathbb{R}^{k})$, $w_{j}=u_{j}$ у $V$
для кожного номера $j\in\{1,\,2\}$. Тут $\Upsilon$ позначає ортогональний проєктор простору $H^{\mu}(\mathbb{R}^{k})$ на його підпростір \eqref{8f41}. Простори $H^{\mu}(V)$ і $H^{\mu}_{Q}(\mathbb{R}^{k})$ ввели та дослідили Л.~Р.~Волєвіч і Б.~П.~Панеях \cite[розд.~3]{VolevichPaneah65}.

З означення $H^{\mu}(V)$ та властивостей $H^{\mu}(\mathbb{R}^{k})$ випливає, що простір $H^{\mu}(V)$ неперервно вкладений у лінійний топологічний простір $\mathcal{D}'(V)$ усіх розподілів на $V$, а множина
$$
\bigl\{w\!\upharpoonright\!\overline{V}:w\in C^{\infty}_{0}(\mathbb{R}^{k})\bigr\}
$$
щільна в $H^{\mu}(V)$.

З точки зору застосувань узагальнених просторів Соболєва $H^{\mu}$ до  параболічних задач доцільно вибрати клас показників $\mu$ такий, що:
\begin{itemize}
\item[а)] за допомогою просторів $H^{\mu}$ більш тонко характеризується регулярність приналежних ним розподілів, ніж це можливо у класах соболєвських просторів;
\item[б)] простори $H^{\mu}$ допускають коректне означення на бічній поверхні циліндра, у якому задано параболічну задачу;
\item[в)] ці простори отримуються (квадратичною) інтерполяцією з функціональним параметром пар деяких гільбертових анізотропних просторів Соболєва.
\end{itemize}
\noindent Для просторів $H^{\mu}$, які задовольняють ці вимоги, можна очікувати на побудову змістовної теорії розв'язності параболічних задач на основі методу квадратичної інтерполяції гільбертових просторів.


Припустимо, що ціле число $k\geq2$ та дійсне число $\gamma>0$.
(Хоча введені нижче анізотропні простори потрібні
лише у випадку $\gamma=1/(2b)$, де парне число $2b$ є параболічною вагою задачі, їх природно розглянути для довільного $\gamma>0$.)
Використовуємо простори Хермандера $H^{\mu}(\mathbb{R}^{k})$
(та їх версії) з показником регулярності
\begin{equation}\label{8f4}
\mu(\xi',\xi_{k}):=\bigl(1+|\xi'|^2+|\xi_{k}|^{2\gamma}\bigr)^{s/2}
\varphi\bigl((1+|\xi'|^2+|\xi_{k}|^{2\gamma})^{1/2}\bigr),
\end{equation}
де $\xi'\in\mathbb{R}^{k-1}$ і $\xi_{k}\in\mathbb{R}$ є аргументами функції $\mu$. Тут числовий параметр $s$ дійсний, а функціональний параметр $\varphi$ пробігає клас~$\mathcal{M}$, означений нижче. (У кінці цього підрозділу буде показано, що функція \eqref{8f4} задовольняє умову~\eqref{Hermander-cond}). Частотні змінні $\xi'=(\xi_{1},\ldots,\xi_{k-1})$ і $\xi_{k}$ у формулі \eqref{8f4} є дуальними до дійсних змінних $x'=(x_{1},\ldots,x_{k-1})$ і $x_{k}$ відносно перетворення Фур'є. При цьому $x_{1},\ldots,x_{k-1}$ інтерпретуємо як рівноправні просторові змінні, а $x_{k}$~--- як часову змінну~$t$.

За означенням клас $\mathcal{M}$ складається з усіх вимірних за Борелем функцій $\varphi:[1,\infty)\rightarrow(0,\infty)$, які задовольняють такі дві умови:
\begin{itemize}
\item [а)] обидві функції $\varphi$ та $1/\varphi$ обмежені на кожному відрізку $[1,c]$, де
             $1<c<\infty$;
\item [б)] функція $\varphi$ повільно змінюється (за Караматою) на нескінченності, тобто
             $$\lim_{r\rightarrow\infty}\frac{\varphi(\lambda r)}{\varphi(r)}=1\quad\mbox{для кожного}\quad
             \lambda>0.$$
\end{itemize}

Важливим прикладом функції класу $\mathcal{M}$ є додатна неперервна функція $\varphi(r)$ аргументу $r>0$ така, що
\begin{equation*}
\varphi(r):=(\ln r)^{\theta_{1}}\,(\ln\ln r)^{\theta_{2}} \ldots
(\,\underbrace{\ln\ldots\ln}_{k\;\mbox{\small{разів}}}r\,)^{\theta_{k}}
\quad\mbox{при}\quad r\gg1,
\end{equation*}
де параметри $k\in\mathbb{N}$ та
$\theta_{1},\theta_{2},\ldots,\theta_{k}\in\mathbb{R}$ вибрано довільно. Вказані функцію утворюють мультилогарифмічну шкалу.

З теореми про інтегральне зображення повільно змінних функцій
(див., наприклад, \cite[c.~10]{Seneta85}) випливає такий простий інтегральний опис класу $\mathcal{M}$:

Функція $\varphi:\nobreak[1,\infty)\rightarrow(0,\infty)$ належить до класу $\mathcal{M}$ тоді і тільки тоді, коли
$$
\varphi(r)=\exp\Biggl(\delta(r)+
\int\limits_{1}^{\:r}\frac{\varepsilon(\tau)}{\tau}\;d\tau\Biggr)
\quad\mbox{при}\quad r\geq1
$$
для деякої неперервної дійсної функції $\varepsilon(\tau)$ аргументу $\tau\geq1$, яка прямує до нуля при $\tau\rightarrow\infty$, і деякої вимірної за Борелем обмеженої дійсної функції $\delta(r)$ аргументу $r\geq1$, яка має скінченну границю при $r\rightarrow\infty$.

Нехай $s\in\mathbb{R}$ і $\varphi\in\mathcal{M}$. Покладемо $H^{s,s\gamma;\varphi}(\mathbb{R}^{k}):=H^{\mu}(\mathbb{R}^{k})$, де функціональний параметр $\mu$ означено формулою~\eqref{8f4}. У~класичному випадку, коли $\varphi(\cdot)\equiv1$, простір $H^{s,s\gamma;\varphi}(\mathbb{R}^{k})$ стає  гільбертовим анізотропним простором Соболєва $H^{s,s\gamma}(\mathbb{R}^{k})$ порядку $(s,s\gamma)$, де $s$~--- порядок за просторовими змінними, а $s\gamma$~--- за часовою змінною. У загальному випадку, коли функція $\varphi\in\mathcal{M}$ довільна, виконуються щільні неперервні вкладення
\begin{equation}\label{8f5}
\begin{gathered}
H^{s_{1},s_{1}\gamma}(\mathbb{R}^{k})\hookrightarrow
H^{s,s\gamma;\varphi}(\mathbb{R}^{k})\hookrightarrow
H^{s_{0},s_{0}\gamma}(\mathbb{R}^{k}) \\
\mbox{при}\quad s_{0}<s<s_{1}.
\end{gathered}
\end{equation}
Дійсно, нехай $s_{0}<s<s_{1}$; оскільки $\varphi\in\mathcal{M}$, то існують додатні числа $c_0$ і $c_1$ такі, що
$$
c_0\,r^{s_0-s}\leq\varphi(r)\leq c_1\,r^{s_1-s}\quad\mbox{для
довільного}\quad r\geq1
$$
(див., наприклад, \cite[п.1.5, властивість $1^\circ$]{Seneta85}). Тоді
\begin{align*}
c_{0}&\bigl(1+|\xi'|^2+|\xi_{k}|^{2\gamma}\bigr)^{s_{0}/2}\leq\\
\leq&\bigl(1+|\xi'|^2+|\xi_{k}|^{2\gamma}\bigr)^{s/2}
\varphi\bigl((1+|\xi'|^2+|\xi_{k}|)^{1/2}\bigr)\leq\\
\leq& c_{1}\bigl(1+|\xi'|^2+|\xi_{k}|^{2\gamma}\bigr)^{s_{1}/2}
\end{align*}
для довільних $\xi'\in\mathbb{R}^{k-1}$ і $\xi_{k}\in\mathbb{R}$.
Звідси зразу випливають неперервні вкладення~\eqref{8f5}.
Вони щільні, оскільки множина $C^{\infty}_{0}(\mathbb{R}^{k})$
щільна в усіх просторах, наявних у формулі~\eqref{8f5}.

Розглянемо клас гільбертових просторів
\begin{equation}\label{8f6}
\bigl\{H^{s,s\gamma;\varphi}(\mathbb{R}^{k}):\,
s\in\mathbb{R},\,\varphi\in\mathcal{M}\,\bigr\}.
\end{equation}
Вкладення \eqref{8f5} показують, що в класі \eqref{8f6} функціональний параметр $\varphi$ визначає додаткову гладкість відносно основної анізотропної
$(s,s\gamma)$-гладкості. Якщо $\varphi(r)\rightarrow\infty$ (або
$\varphi(r)\rightarrow0$) при $r\rightarrow\infty$, то $\varphi$ визначає
позитивну (або негативну) додаткову гладкість. Інакше кажучи, $\varphi$
уточнює основну гладкість $(s,s\gamma)$. Тут $\gamma>0$ виконує роль параметра анізотропії просторів, що утворюють цей клас.

У випадку $\gamma=1/(2b)$ говоримо, що  $H^{s,s\gamma;\varphi}(\mathbb{R}^{k})$
є $2b$-анізо\-тропним узагальненим простором Соболєва на $\mathbb{R}^{k}$.

Введемо версії цього простору для множин, на яких будемо розглядати параболічні рівняння і пов'язані з ними крайові та початкові умови.
Нехай довільно задано ціле число $n\geq2$, дійсне число $\tau>0$
і обмежену область $G\subset\mathbb{R}^{n}$ з
нескінченно гладкою межею $\Gamma:=\partial G$.
Покладемо $\Omega:=G\times(0,\tau)$ і $S:=\Gamma\times(0,\tau)$. Отже,
$\Omega$~--- відкритий циліндр у $\mathbb{R}^{n+1}$, а $S$~--- його бічна поверхня. Їх замиканнями є множини $\overline{\Omega}:=\overline{G}\times[0,\tau]$ і $\overline{S}:=\Gamma\times[0,\tau]$ відповідно. Параболічні рівняння задаємо у скінченному циліндрі $\Omega$, крайові умови~--- на його бічній поверхні $S$, а початкові умови~--- на його нижній основі, яку ототожнюємо з областю $G$ простору~$\mathbb{R}^{n}$. Якщо $n=1$, то $G$ є інтервалом $(0,l)$ осі, де дійсне число $l>0$, а  $\Omega:=(0,l)\times(0,\tau)$ є відкритим прямокутником на площині~$\mathbb{R}^{2}$ (цей випадок досліджуємо окремо у пп. \ref{sec3.2.1} і \ref{sec4.2.1}).

Розв'язки і праві частини параболічних рівнянь розглядаємо в гільбертових функціональних просторах
$H^{s,s\gamma;\varphi}(\Omega):=H^{\mu}(\Omega)$, де показник $\mu$
означено формулою~\eqref{8f4}, в якій $k:=n+1$. Праві частини крайових
умов належать до аналогічних просторів, заданих на
бічній поверхні $S=\Gamma\times(0,\tau)$ циліндра $\Omega$. Останні потрібні у випадку $s>0$. Означимо їх за допомогою спеціальних локальних карт на~$S$.

Нехай $s>0$ і $\varphi\in\mathcal{M}$.
Попередньо розглянемо гільбертові
простори $H^{s,s\gamma;\varphi}(\Pi):=H^{\mu}(\Pi)$, задані на відкритій смузі $\Pi:=\nobreak\mathbb{R}^{n-1}\times(0,\tau)$; тут показник $\mu$ означено формулою \eqref{8f4}, в якій $k:=n$.

Довільно виберемо скінченний атлас із
$C^{\infty}$-структури на замкненому многовиді $\Gamma$, породженої евклідовим простором $\mathbb{R}^{n}$. Нехай цей атлас утворено локальними картами $\theta_{j}:\mathbb{R}^{n-1}\leftrightarrow\Gamma_{j}$, де $j=1,\ldots,\lambda$. Тут кожне
$\theta_{j}$ є $C^{\infty}$-дифеоморфізмом усього евклідового простору
$\mathbb{R}^{n-1}$ на деяку відкриту підмножину $\Gamma_{j}$ многовиду $\Gamma$. При цьому $\Gamma:=\Gamma_{1}\cup\cdots\cup\Gamma_{\lambda}$, тобто відкриті множини $\Gamma_{1},\ldots,\Gamma_{\lambda}$ утворюють покриття цього многовиду. Крім того, довільно виберемо функції
$\chi_{j}\in C^{\infty}(\Gamma)$, де $j=1,\ldots,\lambda$, такі, що
$\mathrm{supp}\,\chi_{j}\subset\Gamma_{j}$ і $\chi_{1}+\cdots+\chi_{\lambda}=1$ на $\Gamma$.
Ці функції утворюють $C^{\infty}$-розбиття одиниці на $\Gamma$, підпорядковане вказаному покриттю.

Вибраний атлас породжує набір спеціальних локальних карт:
\begin{equation}\label{8f-local}
\theta_{j}^*:\Pi=\mathbb{R}^{n-1}\times(0,\tau)\leftrightarrow
\Gamma_{j}\times(0,\tau),\quad j=1,\ldots,\lambda,
\end{equation}
на $S=\Gamma\times(0,\tau)$, означених формулою $\theta_{j}^{*}(x,t):=(\theta_{j}(x),t)$ для довільних
$x\in\mathbb{R}^{n-1}$ і $t\in(0,\tau)$.
Розглянемо функції
$\chi_{j}^{*}(x,t):=\chi_{j}(x)$ аргументів $x\in\Gamma$ і  $t\in(0,\tau)$, де $j=1,\ldots,\lambda$. Ці функції утворюють $C^{\infty}$-розбиття одиниці на $S$, підпорядковане покриттю $\{\Gamma_{j}\times(0,\tau):j=1,\ldots,\lambda\}$ многовиду $S$.

За означенням лінійний простір $H^{s,s\gamma;\varphi}(S)$ складається з усіх функцій $v\in L_2(S)$ таких, що для кожного
номера $j\in\{1,\ldots,\lambda\}$ функція
$$
v_{j}(x,t):=\chi_{j}(\theta_{j}(x))\,v(\theta_{j}(x),t)
$$
аргументів $x\in\mathbb{R}^{n-1}$ і $t\in(0,\tau)$ належить до $H^{s,s\gamma;\varphi}(\Pi)$. Тут, як звичайно, $L_2(S)$ -- комплексний гільбертів простір усіх функцій, квадратично інтегровних відносно міри Лебега на многовиді~$S$. У просторі $H^{s,s\gamma;\varphi}(S)$ означено скалярний добуток функцій $v$ і $v'$ за формулою
\begin{equation}\label{9f3.8a}
(v,v')_{H^{s,s\gamma;\varphi}(S)}:=
\sum_{j=1}^{\lambda}\,
(v_{j},v_{j}')_{H^{s,s\gamma;\varphi}(\Pi)}.
\end{equation}
Він породжує норму
$$
\|v\|_{H^{s,s\gamma;\varphi}(S)}:=(v,v)^{1/2}_{H^{s,s\gamma;\varphi}(S)}.
$$

Оскільки $v_{j}=(\chi_{j}^{*}v)\circ\theta_{j}^{*}$, то простір $H^{s,s\gamma;\varphi}(S)$ і топологію у ньому означено за допомогою спеціальних локальних карт~\eqref{8f-local}. Звісно, тут символ $\circ$ позначає композицію функцій чи відображень.

\begin{theorem}\label{9lem3.1a}
Нехай $s>0$, $\gamma>0$ і $\varphi\in\mathcal{M}$. Простір $H^{s,s\gamma;\varphi}(S)$ має такі властивості:
\begin{itemize}
\item[$\mathrm{(i)}$] Він є повним (тобто гільбертовим) і сепарабельним.
\item[$\mathrm{(ii)}$] Він не залежить з точністю до еквівалентності норм від вказаного вибору атласу і розбиття одиниці на $\Gamma$.
\item[$\mathrm{(iii)}$] У ньому щільною є множина $C^{\infty}(\overline{S})$.
\end{itemize}
\end{theorem}

Доведення дамо наприкінці підрозділу~\ref{sec2.6}.

У випадку, коли $\gamma=1$, простори $H^{s,s\gamma;\varphi}(\cdot)$ стають ізотропними і позначаються через $H^{s;\varphi}(\cdot)$. Їх розглянуто і застосовано до еліптичних задач у монографіях \cite{MikhailetsMurach10, MikhailetsMurach14}. Деякі такі простори потрібні нам для дослідження параболічних задач. У просторі $H^{s;\varphi}(G)$ розглядаються початкові дані (залежні лише від просторових змінних). Якщо $n=2$, то крайові дані залежать лише від часової змінної і, отже, є елементами простору $H^{s;\varphi}$ на інтервалі дійсної осі. Крім того, для дослідження умов узгодження, яким задовольняють праві частини параболічних задач, знадобиться простір $H^{s;\varphi}(\Gamma)$. Для зручності дамо означення потрібних нам ізотропних просторів.

Як і раніше, $s\in\mathbb{R}$ і $\varphi\in\mathcal{M}$, а $V$~--- непорожня відкрита підмножина простору $\mathbb{R}^{k}$, де ціле $k\geq1$. Гільбертів простір $H^{s;\varphi}(V)$~--- це простір $H^{\mu}(V)$, де функціональний параметр $\mu$ означено за формулою
\begin{equation*}
\mu(\xi)=\bigl(1+|\xi|^2\bigr)^{s/2}\varphi\bigl((1+|\xi|^2)^{1/2}\bigr)
\end{equation*}
для довільного $\xi\in\mathbb{R}^{k}$. Нас цікавлять випадки коли $V=\mathbb{R}^{k}$ або $V=G$, а також випадок, коли $V$~--- інтервал дійсної осі (якщо $k=1$), наприклад, $(0,\tau)$. В останньому випадку простір позначаємо через $H^{s;\varphi}(0,\tau)$, уникаючи подвійних дужок.

Простір $H^{s;\varphi}(\Gamma)$ означається за допомогою вказаних вище набору локальних карт $\{\theta_{j}\}$ і розбиття одиниці $\{\chi_{j}\}$ на $\Gamma$. А саме, він складається з усіх
розподілів $\omega$ на многовиді $\Gamma$ таких, що $\omega_{j}:=(\chi_{j}\omega)\circ\theta_{j}$ належить до $H^{s;\varphi}(\mathbb{R}^{n-1})$ для кожного номера $j\in\{1,\ldots,\lambda\}$. Тут $(\chi_{j}\omega)\circ\theta_{j}$ позначає зображення розподілу $\chi_{j}\omega$ у локальній карті~$\theta_{j}$. У~просторі $H^{s;\varphi}(\Gamma)$ скалярний добуток розподілів $\omega$ і $\omega'$ означено формулою
\begin{equation*}
(\omega,\omega')_{H^{s;\varphi}(\Gamma)}:=
\sum_{j=1}^{\lambda}\,
(\omega_{j},\omega'_{j})_{H^{s;\varphi}(\mathbb{R}^{n-1})}.
\end{equation*}
Останній породжує норму
$$
\|\omega\|_{H^{s;\varphi}(\Gamma)}:=(\omega,\omega)^{1/2}_{H^{s;\varphi}(\Gamma)}.
$$
Простір $H^{s;\varphi}(\Gamma)$ гільбертів і сепарабельний. Він не залежить з точністю до еквівалентності норм від вказаного вибору локальних карт і розбиття одиниці на $\Gamma$ (див. \cite[теорема~2.3]{MikhailetsMurach10, MikhailetsMurach14}).

Якщо $\varphi(r)\equiv1$, то $H^{s,s\gamma;\varphi}(\cdot)$ та $H^{s;\varphi}(\cdot)$ стають просторами Соболєва, анізотропним $H^{s,s\gamma}(\cdot)$ та ізотропним $H^{s}(\cdot)$ відповідно.
Із вкладень \eqref{8f5} випливає, що
\begin{equation}\label{8f5a}
H^{s_{1},s_{1}\gamma}(\cdot)\hookrightarrow
H^{s,s\gamma;\varphi}(\cdot)\hookrightarrow
H^{s_{0},s_{0}\gamma}(\cdot)\quad\mbox{при}\quad s_{0}<s<s_{1}.
\end{equation}
Зокрема, при $\gamma=1$ маємо
\begin{equation}\label{8f5b}
H^{s_{1}}(\cdot)\hookrightarrow
H^{s;\varphi}(\cdot)\hookrightarrow
H^{s_{0}}(\cdot)\quad\mbox{при}\quad s_{0}<s<s_{1}
\end{equation}
(див. також \cite[теореми 2.3(iii), 3.3(iii)]{MikhailetsMurach10, MikhailetsMurach14}). Ці вкладення неперервні та щільні. Звісно, якщо $s=0$, то простір $H^{s}(\cdot)=H^{s,s\gamma}(\cdot)$ збігається з точністю до еквівалентності норм з гільбертовим простором $L_2(\cdot)$ всіх квадратично інтегрованих функцій, заданих на відповідній вимірній множині.

У випадку $\varphi(r)\equiv1$, будемо зазвичай прибирати індекс $\varphi$ у позначеннях функціональних просторів, уведених нижче  на основі просторів $H^{s,s\gamma;\varphi}(\cdot)$ і $H^{s;\varphi}(\cdot)$.

Як і було обіцяно, покажемо на завершення цього підрозділу, що функція
\eqref{8f4} задовольняє умову~\eqref{Hermander-cond}. Задля більшої лаконічності подальших формул уведемо функцію
\begin{equation}\label{r-gamma-new}
r_{\gamma}(\xi',\xi_{k}):=\bigl(1+|\xi'|^2+|\xi_{k}|^{2\gamma}\bigr)^{1/2}
\end{equation}
аргументів $\xi'\in\mathbb{R}^{k-1}$ і $\xi_{k}\in\mathbb{R}$ та подамо функцію \eqref{8f4} у вигляді $\mu(\xi)=r_{\gamma}^{s}(\xi)\varphi(r_{\gamma}(\xi))$, де  $\xi=(\xi',\xi_{k})$. Оскільки $\varphi\in\mathcal{M}$, то існує число $c_{\varphi}\geq1$ таке, що
\begin{equation}\label{9f-app2}
\frac{\varphi(\lambda r)}{\varphi(r)}\leq c_{\varphi}\lambda
\quad\mbox{і}\quad
\frac{\varphi(r)}{\varphi(\lambda r)}\leq c_{\varphi}\lambda
\end{equation}
для довільних $r\geq1$ і $\lambda\geq1$ (див., наприклад, \cite[п.~2.4.1, формула (2.91)]{MikhailetsMurach10, MikhailetsMurach14}). Крім того, згідно з \cite[розд.~I, \S~2, п.~2]{VolevichPaneah65} існують числа $c_{\gamma}\geq1$ і $l_{\gamma}>0$ такі, що
\begin{equation}\label{9f-app3}
\frac{r_{\gamma}(\xi)}{r_{\gamma}(\eta)}\leq
c_{\gamma}(1+|\xi-\eta|)^{l_{\gamma}}
\end{equation}
для довільних аргументів $\xi,\eta\in\mathbb{R}^{k}$.

Якщо у формулі \eqref{9f-app3} чисельник більший за знаменник або дорівнює йому, то
\begin{equation*}
\begin{split}
\frac{\mu(\xi)}{\mu(\eta)}&=
\frac{r_{\gamma}^{s}(\xi)\,\varphi(r_{\gamma}(\xi))}
{r_{\gamma}^{s}(\eta)\,\varphi(r_{\gamma}(\eta))}\leq
c_{\varphi}\,
\frac{r_{\gamma}^{s+1}(\xi)}{r_{\gamma}^{s+1}(\eta)}\leq\\
&\leq c_{\varphi} c_{\gamma}^{\max\{s+1,0\}}
(1+|\xi-\eta|)^{l_{\gamma}\max\{s+1,0\}}
\end{split}
\end{equation*}
на підставі першої нерівності у формулі \eqref{9f-app2}. Інакше

\begin{equation*}
\begin{split}
\frac{\mu(\xi)}{\mu(\eta)}&\leq c_{\varphi}\,
\frac{r_{\gamma}^{s-1}(\xi)}{r_{\gamma}^{s-1}(\eta)}=
c_{\varphi}\,
\frac{r_{\gamma}^{1-s}(\eta)}{r_{\gamma}^{1-s}(\xi)}\leq\\
&\leq c_{\varphi} c_{\gamma}^{\max\{1-s,0\}}
(1+|\xi-\eta|)^{l_{\gamma}\max\{1-s,0\}}
\end{split}
\end{equation*}
на підставі другої нерівності у формулі \eqref{9f-app2}. Отже, у будь-якому випадку правильна нерівність \eqref{Hermander-cond}, якщо узяти в ній
\begin{equation*}
c:=c_{\varphi} c_{\gamma}^{\max\{s+1,1-s,0\}}\quad\mbox{і}\quad
l:=l_{\gamma}\max\{s+1,1-s,0\}.
\end{equation*}

Стосовно умови \eqref{Hermander-cond} на показник регулярності $\mu$ слід зазначити, що Л.~Хермандер \cite[означення~2.1.1]{Hormander63} використовує більш сильну умову
\begin{equation}\label{Hermander-cond-initial}
\frac{\mu(\xi)}{\mu(\eta)}\leq
(1+c|\xi-\eta|)^{l}\quad\mbox{для довільних}\quad \xi,\eta\in\mathbb{R}^{k}
\end{equation}
(з неї випливає неперервність $\mu$). Втім, дві множини функцій $\mu$, які задовольняють відповідно першу або другу умову, задають один і той самий клас просторів $H^{\mu}(\mathbb{R}^{k})$ (з точністю до еквівалентності норм). Це зазначено в  \cite[зауваження наприкінці п.~2.1]{Hormander63} стосовно неперервних функцій. Як вказують Л.~Р.~Волєвіч і Б.~П.~Панеях \cite[розд~I, \S~1, п.~1]{VolevichPaneah65}, умова неперервності функції $\mu$ не є істотною, якщо $\mu$ задовольняє \eqref{Hermander-cond}. А саме: замінивши її на більш слабку умову вимірності за Борелем, отримаємо той самий клас просторів $H^{\mu}(\mathbb{R}^{k})$.

Зокрема, для борелевої функції $\mu$ вигляду \eqref{8f4} існує нескінченно диференційовна функція $\mu_{1}:\mathbb{R}^{k}\to(0,\infty)$, підпорядкована умові \eqref{Hermander-cond} і така, що $H^{\mu}(\mathbb{R}^{k})=H^{\mu_{1}}(\mathbb{R}^{k})$ з еквівалентністю  норм у просторах. Справді, оскільки $\varphi\in\mathcal{M}$, то знайдеться функція $\varphi_{1}\in\mathcal{M}\cap C^{\infty}([1,\infty))$ така, що обидві функції $\varphi/\varphi_{1}$ і  $\varphi_{1}/\varphi$  обмежені на $[1,\infty)$ (це випливає з \cite[п.~1.4, властивість~$1^{\circ}$]{Seneta85}). Залишається покласти
$$
\mu_{1}(\xi',\xi_k):=r_{\gamma}^{s}\bigl(\xi',(1+\xi_{k}^{2})^{1/2}\bigr)\,
\cdot\varphi_{1}\bigl(r_{\gamma}(\xi',(1+\xi_{k}^{2}))^{1/2}\bigr)
$$
для довільних $\xi'\in\mathbb{R}^{k-1}$ і $\xi_{k}\in\mathbb{R}$.


\section[Окремі функціональні простори]{Окремі функціональні простори}\label{sec2.5}

\markright{\emph \ref{sec2.5}. Окремі функціональні простори}

Для дослідження параболічних задач з однорідними початковими умовами (тобто нульовими даними Коші) потрібні окремі простори, утворені функціями, які дорівнюють нулю, якщо часова змінна $t<0$. Введемо такі простори на основі класу гільбертових просторів \eqref{8f6}, заданих на $\mathbb{R}^{k}$, де ціле $k\geq2$. Як і раніше, останню координату $x_{k}$ вектора $x\in\mathbb{R}^{k}$ інтерпретуємо як часову змінну $t$.

Нехай $s\in\mathbb{R}$, $\gamma>0$ і $\varphi\in\mathcal{M}$. Припустимо, що $V$~--- непорожня відкрита підмножина простору $\mathbb{R}^{k}$, де ціле $k\geq2$. (Для нас головним є випадок, коли $V$~--- циліндр $\Omega$, де $k=n+1$.) Покладемо
\begin{equation}\label{9f3.3}
\begin{split}
H^{s,s\gamma;\varphi}_{+}(V):=\bigl\{w\!\upharpoonright\!V:\,
&w\in H^{s,s\gamma;\varphi}(\mathbb{R}^{k}),\;\,  \\
&\mathrm{supp}\,
w\subseteq\mathbb{R}^{k-1}\times[0,\infty)\bigl\}.
\end{split}
\end{equation}
Означимо норму в лінійному просторі \eqref{9f3.3} за формулою
\begin{equation}\label{9f3.4}
\begin{split}
\|u\|_{H^{s,s\gamma;\varphi}_{+}(V)}:=
\inf\bigl\{\,&\|w\|_{H^{s,s\gamma;\varphi}(\mathbb{R}^{k})}:
w\in H^{s,s\gamma;\varphi}(\mathbb{R}^{k}),\;\,\\
&\mathrm{supp}\,w\subseteq\mathbb{R}^{k-1}\times[0,\infty),\;\,
u=w\!\upharpoonright\!V\bigl\},
\end{split}
\end{equation}
де $u\in H^{s,s\gamma;\varphi}_{+}(V)$.

Зокрема, якщо $V=\mathbb{R}^{k}$, то $H^{s,s\gamma;\varphi}_{+}(\mathbb{R}^{k})$ складається з усіх розподілів $w\in H^{s,s\gamma;\varphi}(\mathbb{R}^{k})$ таких, що $\mathrm{supp}\,w\subseteq\mathbb{R}^{k-1}\times[0,\infty)$, і є (замкненим) підпростором гільбертового простору $H^{s,s\gamma;\varphi}(\mathbb{R}^{k})$. Множина
\begin{align*}
&C^{\infty}_{0}(\mathbb{R}^{k-1}\times(0,\infty)):=\\
&:=\bigl\{w\in C^{\infty}_{0}(\mathbb{R}^{k}):\,
\mathrm{supp}\,w\subseteq\mathbb{R}^{k-1}\times(0,\infty)\bigr\}
\end{align*}
є щільною у просторі $H^{s,s\gamma;\varphi}_{+}(\mathbb{R}^{k})$, як показано в \cite[Лема~3.3]{VolevichPaneah65}.

Простір $H^{s,s\gamma;\varphi}_{+}(V)$ гільбертів і сепарабельний, оскільки за означенням він є фактор-простором сепарабельного гільбертового простору $H^{s,s\gamma;\varphi}_{+}(\mathbb{R}^{k})$ за його підпростором
\begin{equation}\label{9f3.5}
H^{s,s\gamma;\varphi}_{Q}(\mathbb{R}^{k}):=
\bigl\{w\in H^{s,s\gamma;\varphi}(\mathbb{R}^{k}):
\mathrm{supp}\,w\subseteq Q\bigr\},
\end{equation}
де $Q:=(\mathbb{R}^{k-1}\times[0,\infty))\setminus V$.
Норма \eqref{9f3.4} породжена скалярним добутком
$$
(u_{1},u_{2})_{H^{s,s\gamma;\varphi}_{+}(V)}:=
(w_{1}-\Upsilon w_{1},w_{2}-\Upsilon w_{2})_{H^{s,s\gamma;\varphi}(\mathbb{R}^{k})},
$$
де $u_{1},u_{2}\in H^{s,s\gamma;\varphi}_{+}(V)$. Тут $w_{j}\in H^{s,s\gamma;\varphi}_{+}(\mathbb{R}^{k})$, $w_{j}=u_{j}$ на $V$ для кожного $j\in\{1,2\}$, а $\Upsilon$ є ортогональним проєктором простору $H^{s,s\gamma;\varphi}_{+}(\mathbb{R}^{k})$ на його підпростір \eqref{9f3.5}.

Виконуються щільні неперервні вкладення
\begin{equation}\label{9f3.6}
\begin{gathered}
H^{s_{1},s_{1}\gamma}_{+}(V)\hookrightarrow
H^{s,s\gamma;\varphi}_{+}(V)\hookrightarrow
H^{s_{0},s_{0}\gamma}_{+}(V) \\
\mbox{при}\quad s_{0}<s<s_{1}.
\end{gathered}
\end{equation}
Вони є наслідком неперервних вкладень \eqref{8f5} та щільності множини
\begin{equation*}
\bigl\{w\!\upharpoonright\!\overline{V}:w\in C^{\infty}_{0}(\mathbb{R}^{k-1}\times(0,\infty))\bigr\}
\end{equation*}
у просторах, наявних у \eqref{9f3.6}.

Для крайових даних потрібен аналог простору $H^{s,s\gamma;\varphi}_{+}(\Omega)$ на бічній поверхні $S$ циліндра $\Omega$. Цей простір означається подібно до $H^{s,s\gamma;\varphi}(S)$ за допомогою спеціальних локальних карт \eqref{8f-local} на $S$ і відповідного розбиття одиниці. Вони породжені атласом і розбиттям одиниці на $\Gamma$, вказаними у підрозділі~\ref{sec2.3}. Припустимо, що $s>0$. Розглянемо гільбертів простір $H^{s,s\gamma;\varphi}_{+}(\Pi)$, де $\Pi:=\mathbb{R}^{n-1}\times(0,\tau)$. Покладемо
\begin{equation}\label{9f3.7}
\begin{aligned}
H^{s,s\gamma;\varphi}_{+}(S):=
\bigl\{&v\in L_2(S):\,
(\chi_{j}^{*}v)\circ\theta_{j}^{*}\in H^{s,s\gamma;\varphi}_{+}(\Pi)\\
\;\,&\mbox{для всіх}\;\,j\in\{1,\ldots,\lambda\}\bigr\}.
\end{aligned}
\end{equation}
Означимо у лінійному просторі \eqref{9f3.7} скалярний добуток розподілів $v_{1}$ і $v_{2}$ за формулою
\begin{equation}\label{9f3.8}
(v_{1},v_{2})_{H^{s,s\gamma;\varphi}_{+}(S)}:=\sum_{j=1}^{\lambda}\,
((\chi_{j}^{*}v_{1})\circ\theta_{j}^{*},
(\chi_{j}^{*}v_{2})\circ
\theta_{j}^{*})_{H^{s,s\gamma;\varphi}_{+}(\Pi)}.
\end{equation}
Останній породжує норму
\begin{equation*}
\|v\|_{H^{s,s\gamma;\varphi}_{+}(S)}:=
(v,v)_{H^{s,s\gamma;\varphi}_{+}(S)}^{1/2}.
\end{equation*}

\begin{theorem}\label{9lem3.1}
Нехай $s>0$, $\gamma>0$, і $\varphi\in\mathcal{M}$. Простір $H^{s,s\gamma;\varphi}_{+}(S)$ має такі властивості:
\begin{itemize}
\item[$\mathrm{(i)}$] Він є повним (тобто гільбертовим) і сепарабельним.
\item[$\mathrm{(ii)}$] Він не залежить з точністю до еквівалентності норм від вказаного вибору атласу і розбиття одиниці на $\Gamma$.
\item[$\mathrm{(iii)}$] У ньому є щільною множина
\begin{equation}\label{th1.6-dense-set}
\bigl\{h\in C^{\infty}(\overline{S}):\mathrm{supp}\,
h\subseteq\Gamma\times(0,\tau]\bigr\}.
\end{equation}
\end{itemize}
\end{theorem}

Щодо висновку (iii) цієї теореми нагадаємо, що $\overline{S}=\Gamma\times[0,\tau]$. Її доведення дамо наприкінці підрозділу~\ref{sec2.6}.

Зауважимо, що неперервні і щільні вкладення \eqref{9f3.6} виконуються також у випадку, коли $V=S$ і $s_{0}>0$. Це випливає з \eqref{9f3.6} (для $V=\Pi$) і висновку (iii) теореми~\ref{9lem3.1}.

В ізотропному випадку, коли $\gamma=1$, простір $H^{s,s\gamma;\varphi}_{+}(V)$ позначаємо через $H^{s;\varphi}_{+}(V)$.
Останній нам знадобиться у ситуації, коли $k=1$, для дослідження параболічних задач у прямокутнику $\Omega=(0,l)\times(0,\tau)$. А саме, потрібні такі сепарабельні гільбертові простори:
\begin{gather*}
H^{s,\varphi}_{+}(\mathbb{R}):=\bigl\{h\in H^{s,\varphi}(\mathbb{R}):\,
\mathrm{supp}\,h\subseteq[0,\infty)\bigr\},\\
H^{s,\varphi}_{+}(0,\tau):=\bigl\{h\!\upharpoonright\!(0,\tau):\,h\in
H^{s,\varphi}_{+}(\mathbb{R})\bigr\}.
\end{gather*}
Перший з них є (замкненим) підпростором простору
$H^{s,\varphi}(\mathbb{R})$, а другий наділено нормою
\begin{equation*}
\|v\|_{H^{s,\varphi}_{+}(0,\tau)}:=
\inf\bigl\{\|h\|_{H^{s,\varphi}(\mathbb{R})}:\,h\in
H^{s,\varphi}_{+}(\mathbb{R}),\;\;v=h\!\upharpoonright\!(0,\tau)\bigr\},
\end{equation*}
де $v\in H^{s,\varphi}_{+}(0,\tau)$.

\markright{\emph \ref{sec2.6}. Інтерполяція просторів Соболєва}

\section[Інтерполяція просторів Соболєва]{Інтерполяція просторів Соболєва}\label{sec2.6}

\markright{\emph \ref{sec2.6}. Інтерполяція просторів Соболєва}

Узагальнені соболєвські простори $H^{s,s\gamma;\varphi}(\cdot)$ і $H^{s,s\gamma;\varphi}_{+}(\cdot)$, розглянуті вище, мають важливу інтерполяційну властивість, яка відіграє ключову роль у їх застосуваннях до параболічних задач. Вона полягає у тому, що ці простори є результатом квадратичної інтерполяції з функціональним параметром пар просторів Соболєва, які фігурують у вкладеннях \eqref{8f5a} і \eqref{9f3.6}. Розглянемо спочатку ізотропний випадок, коли $\gamma=1$, досліджений у \cite{MikhailetsMurach10, MikhailetsMurach14}.

Нехай
\begin{equation}\label{9f7.1}
s,s_{0},s_{1},\gamma\in\mathbb{R},\quad s_{0}<s<s_{1},\quad\gamma>0
\quad\mbox{і}\quad\varphi\in\mathcal{M}.
\end{equation}
Означимо функцію
\begin{equation}\label{9f7.2}
\psi(r):=
\begin{cases}
\;r^{(s-s_{0})/(s_{1}-s_{0})}\,\varphi(r^{1/(s_{1}-s_{0})})&\text{для}\quad r\geq1, \\
\;\varphi(1) & \text{для}\quad0<r<1.
\end{cases}
\end{equation}
Вона є інтерполяційним параметром за теоремою~\ref{9prop6.1}, оскільки
є правильно змінною на нескінченності порядку
$$
\theta:=\frac{s-s_{0}}{s_{1}-s_{0}}\in(0,1).
$$
Використовуємо цю функцію як функціональний параметр квадратичної інтерполяції пар вказаних соболєвських просторів.

\begin{theorem}\label{8prop4}
За припущення \eqref{9f7.1} виконується рівність просторів
\begin{equation}\label{8f49}
H^{s;\varphi}(W)=
\bigl[H^{s_{0}}(W),H^{s_{1}}(W)\bigr]_{\psi}
\end{equation}
з еквівалентістю норм у них; тут $W=G$ або $W=\Gamma$. Якщо $W=\mathbb{R}^{k}$, де ціле $k\geq1$, то рівність \eqref{8f49} виконується разом із рівністю норм у просторах.
\end{theorem}

Цей результат доведено в \cite[теореми 1.14, 2.2 і 3.2]{MikhailetsMurach10, MikhailetsMurach14} для випадків $W=\mathbb{R}^{k}$, $W=\Gamma$ і $W=G$ відповідно.

Встановимо версії теореми~\ref{8prop4} для анізотропних просторів.
Почнемо з базових просторів на $\mathbb{R}^{k}$, де ціле $k\geq2$.

\begin{theorem}\label{9lem7.1}
За припущення \eqref{9f7.1} виконується рівність просторів
\begin{equation}\label{9f7.3}
H^{s,s\gamma;\varphi}(\mathbb{R}^{k})=
\bigl[H^{s_{0},s_{0}\gamma}(\mathbb{R}^{k}),
H^{s_{1},s_{1}\gamma}(\mathbb{R}^{k})\bigr]_{\psi}
\end{equation}
разом із рівністю норм у них.
\end{theorem}

\begin{proof}[\indent Доведення]
Пара гільбертових просторів
\begin{equation*}
X:=\bigl[H^{s_{0},s_{0}\gamma}(\mathbb{R}^{k}),
H^{s_{1},s_{1}\gamma}(\mathbb{R}^{k})\bigr]
\end{equation*}
є регулярною, що випливає із щільних неперервних вкладень \eqref{8f5}. Для неї породжуючим є оператор
$$
J:w\mapsto\mathcal{F}^{-1}[r_{\gamma}^{s_{1}-s_{0}}\,\mathcal{F}w\,]
\quad\mbox{з}\quad w\in H^{s_{1},s_{1}\gamma}(\mathbb{R}^{k}).
$$
Тут $\mathcal{F}$ і $\mathcal{F}^{-1}$ позначають відповідно пряме і обернене перетворення Фур'є, а функцію $r_{\gamma}$ означено формулою
\eqref{r-gamma-new}. Оператор $J$ за допомогою перетворення Фур'є зводиться до оператора множення на функцію $r_{\gamma}^{s_{1}-s_{0}}$, а перетворення Фур'є встановлює ізометричний ізоморфізм
$$
\mathcal{F}:H^{s_{0},s_{0}\gamma}(\mathbb{R}^{k})\leftrightarrow
L_{2}\bigl(\mathbb{R}^{k},r_{\gamma}^{2s_{0}}(\xi',\xi_k)d\xi' d\xi_k\bigr).
$$
Тут, як звичайно, другий гільбертів простір складається з усіх функцій аргументів $\xi'\in\mathbb{R}^{k-1}$ і $\xi_k\in\mathbb{R}$, квадратично інтегровних на $\mathbb{R}^{k}$ відносно міри $r_{\gamma}^{2s_{0}}(\xi',\xi_k)d\xi'd\xi_k$. Отже, перетворення Фур'є зводить оператор $\psi(J)$ до оператора множення на функцію
$$
\psi(r_{\gamma}^{s_{1}-s_{0}}(\xi',\xi_k))\equiv
r_{\gamma}^{s-s_{0}}(\xi',\xi_k)\,\varphi(r_{\gamma}(\xi',\xi_k)).
$$
Тому для кожної функції $w\in C^{\infty}_{0}(\mathbb{R}^{k})$ виконуються такі рівності:
\begin{align*}
\|w\|_{X_{\psi}}^{2}&=
\|\psi(J)w\|_{H^{s_{0},s_{0}\gamma}(\mathbb{R}^{k})}^{2}=\\
&=\int\limits_{\mathbb{R}^{k}}
|\psi(r_{\gamma}^{s_{1}-s_{0}}(\xi',\xi_k))\,
(\mathcal{F}w)(\xi',\xi_k)|^{2}\,r_{\gamma}^{2s_{0}}(\xi',\xi_k)d\xi'
d\xi_k=\\
&=\int\limits_{\mathbb{R}^{k}}
r_{\gamma}^{2s}(\xi',\xi_k)\,\varphi^{2}(r_{\gamma}(\xi',\xi_k))\,
|(\mathcal{F}w)(\xi',\xi_k)|^{2}\,d\xi' d\xi_k=\\
&=\|w\|_{H^{s,s\gamma;\varphi}(\mathbb{R}^{k})}.
\end{align*}
Звідси випливає потрібна рівність просторів \eqref{9f7.3}, оскільки множина $C^{\infty}_{0}(\mathbb{R}^{k})$ є щільною в них. (Ця множина щільна в другому просторі $X_{\psi}$, бо є щільною в просторі $H^{s_{1},s_{1}\gamma}(\mathbb{R}^{k})$, а він неперервно
та щільно вкладається в $X_{\psi}$.)
\end{proof}

\begin{theorem}\label{9lem7.2}
Додатково до \eqref{9f7.1} припустимо, що $s_{0}\geq0$.
Тоді виконується рівність просторів
\begin{equation}\label{9f7.4}
H^{s,s\gamma;\varphi}_{+}(\mathbb{R}^{k})=
\bigl[H^{s_{0},s_{0}\gamma}_{+}(\mathbb{R}^{k}),
H^{s_{1},s_{1}\gamma}_{+}(\mathbb{R}^{k})\bigr]_{\psi}
\end{equation}
з еквівалентністю норм у них.
\end{theorem}

\begin{proof}[\indent Доведення]
Виведемо формулу \eqref{9f7.4} з теореми~\ref{9lem7.1} за допомогою теореми~\ref{9prop6.2}. Для цього треба мати лінійне відображення $P$, задане на $L_{2}(\mathbb{R}^{k})$ і таке, що $P$ є проєктором простору $H^{s_{j},s_{j}\gamma}(\mathbb{R}^{k})$ на підпростір $H^{s_{j},s_{j}\gamma}_{+}(\mathbb{R}^{k})$ для кожного $j\in\{0,1\}$. Побудуємо це відображення.

Згідно з \cite[Лема 2.9.3]{Triebel80}, існує лінійний обмежений оператор
\begin{equation}\label{9f7.5}
T:L_{2}(-\infty,0)\rightarrow L_{2}(\mathbb{R})
\end{equation}
такий, що $Th=h$ на $(-\infty,0)$ для кожної функції $h\in L_{2}(-\infty,0)$ і його звуження є обмеженим оператором на парі соболєвських просторів
\begin{equation}\label{9f7.6}
T:H^{s_{j}\gamma}(-\infty,0)\rightarrow H^{s_{j}\gamma}(\mathbb{R})\quad
\mbox{для кожного}\quad j\in\{0,1\}.
\end{equation}
Скориставшись тензорним добутком обмежених операторів у гільбертових просторах, отримаємо лінійний обмежений оператор
\begin{equation}\label{9f7.7}
I\otimes T:L_{2}(\mathbb{R}^{k-1}\times(-\infty,0))\to
L_{2}(\mathbb{R}^{k})
\end{equation}
такий, що $(I\otimes T)v=v$ на
$\mathbb{R}^{k-1}\times(-\infty,0)$ для кожної функції $v\in L_{2}(\mathbb{R}^{k-1}\times(-\infty,0))$. Тут $I$
є тотожним оператором у $L_{2}(\mathbb{R}^{k-1})$.

Правильні такі рівності просторів з еквівалентністю норм у них:
\begin{equation}\label{9f7.8}
\begin{split}
&H^{s_{j},s_{j}\gamma}(\mathbb{R}^{k-1}\times(-\infty,0))=\\
&=H^{s_{j}}(\mathbb{R}^{k-1})\otimes L_{2}(-\infty,0)\cap
L_{2}(\mathbb{R}^{k-1})\otimes H^{s_{j}\gamma}(-\infty,0)
\end{split}
\end{equation}
та
\begin{equation}\label{9f7.9}
H^{s_{j},s_{j}\gamma}(\mathbb{R}^{k})=
H^{s_{j}}(\mathbb{R}^{k-1})\otimes L_{2}(\mathbb{R})\cap
L_{2}(\mathbb{R}^{k-1})\otimes H^{s_{j}\gamma}(\mathbb{R})
\end{equation}
(див., наприклад, \cite[\S~8, п.~1]{AgranovichVishik64}). (Як звичайно, $E\otimes F$ позначає тензорний добуток довільних гільбертових просторів $E$ і $F$. Крім того, їх перетин  $E\cap F$ розглядається як гільбертів простір, наділений скалярним добутком $(v_{1},v_{2})_{E\cap F}:=
(v_{1},v_{2})_{E}+(v_{1},v_{2})_{F}$ векторів $v_{1},v_{2}\in E\cap F$.)

З формул \eqref{9f7.5}, \eqref{9f7.6}, \eqref{9f7.8} та \eqref{9f7.9}  випливає, що звужен\-ня оператора \eqref{9f7.7} є обмеженим оператором на парі просторів
\begin{equation}\label{9f7.10}
I\otimes T:H^{s_{j},s_{j}\gamma}(\mathbb{R}^{k-1}\times(-\infty,0))\to
H^{s_{j},s_{j}\gamma}(\mathbb{R}^{k}).
\end{equation}

Розглянемо лінійне відображення
\begin{equation*}
P:w\mapsto w-(I\otimes T)\bigl(w\!\upharpoonright\!(\mathbb{R}^{k-1}\times(-\infty,0))\bigr),\quad
\mbox{де}\quad w\in L_{2}(\mathbb{R}^{k}).
\end{equation*}
Легко бачити, що $\mathrm{supp}\,Pw\subseteq\mathbb{R}^{k-1}\times[0,\infty)$ і що включення $\mathrm{supp}\,w\subseteq\mathbb{R}^{k-1}\times[0,\infty)$ тягне за собою рівність $Pw=w$ на $\mathbb{R}^{k}$.
Скориставшись цими властивостями $P$ і обмеженістю оператора \eqref{9f7.10}, робимо висновок, що $P$ є потрібним відображенням.

Отже, на підставі теореми~\ref{9prop6.2} (формула \eqref{9f6.1}) і теореми~\ref{9lem7.1}, маємо такі рівності просторів (з еквівалентністю норм):
\begin{equation*}
\begin{split}
&\bigl[H^{s_{0},s_{0}\gamma}_{+}(\mathbb{R}^{k}),
H^{s_{1},s_{1}\gamma}_{+}(\mathbb{R}^{k})\bigr]_{\psi}=\\
&=\bigl[H^{s_{0},s_{0}\gamma}(\mathbb{R}^{k}),
H^{s_{1},s_{1}\gamma}(\mathbb{R}^{k})\bigr]_{\psi}\cap
H^{s_{0},s_{0}\gamma}_{+}(\mathbb{R}^{k})=\\
&=H^{s,s\gamma;\varphi}(\mathbb{R}^{k})\cap
H^{s_{0},s_{0}\gamma}_{+}(\mathbb{R}^{k})=
H^{s,s\gamma;\varphi}_{+}(\mathbb{R}^{k}).
\end{split}
\end{equation*}
Зауважимо, що перша пара є регулярною завдяки теоремі~\ref{9prop6.2}. Потрібну формулу \eqref{9f7.4} доведено.
\end{proof}

Згодом ми доведемо версії теорем \ref{9lem7.1} і \ref{9lem7.2} для просторів, заданих на циліндрі $\Omega$ і його бічній поверхні $S$. Для цього буде корисною лема про явний опис просторів
$H^{s,s\gamma}_{+}(\Omega)$ і $H^{s,s\gamma}_{+}(S)$ у термінах просторів $H^{s,s\gamma}(\Omega)$ і $H^{s,s\gamma}(S)$. Вона стосується анізотропних соболєвських просторів. Зауважимо, що з їх означення випливають неперервні вкладення
\begin{equation}\label{9f5.3}
H^{s,s\gamma}_{+}(\Omega)\hookrightarrow H^{s,s\gamma}(\Omega)
\quad\mbox{і}\quad
H^{s,s\gamma}_{+}(S)\hookrightarrow H^{s,s\gamma}(S).
\end{equation}

\begin{lemma}\label{9lem5.1}
Нехай $s>0$, $\gamma>0$ і $s\gamma-1/2\notin\mathbb{Z}$. Тоді простір
$H^{s,s\gamma}_{+}(\Omega)$ складається з усіх функцій $u\in H^{s,s\gamma}(\Omega)$ таких, що
\begin{equation}\label{9f5.4}
\begin{gathered}
\partial^k_t u(x,t)\big|_{t=0}=0\quad\mbox{для майже всіх}\;\,x\in G\\
\mbox{при кожному}\;\,k\in\mathbb{Z}\;\,
\mbox{такому, що}\;\,0\leq k<s\gamma-1/2.
\end{gathered}
\end{equation}
Крім того, норми у просторах $H^{s,s\gamma}_{+}(\Omega)$ і $H^{s,s\gamma}(\Omega)$ є еквівалентними. Висновок цієї леми залишається правильним, якщо замінити у ній $\Omega$ на $S$ та $G$ на $\Gamma$.
\end{lemma}

Лема~\ref{9lem5.1} містить у собі результат М.~С.~Аграновіча і М.~І.~Вішика \cite[твердження~8.1]{AgranovichVishik64}, у якому знайдено необхідні і достатні умови, за яких продовження нулем функції
$u\in H^{s,s\gamma}(G\times(0,\theta))$ належить до простору $H^{s,s\gamma}(G\times(-\infty,\theta))$, де $0<\theta\leq\infty$. У цьому результаті припускається, що $s\in\mathbb{Z}$ і $\gamma=1/(2b)$, де ціле $b\geq1$. Знайдені умови рівносильні умові \eqref{9f5.4} і з них випливає еквівалентність норм функції $u$ та її продовження нулем.
М.~С.~Аграновіч і М.~І.~Вішик також розглянули випадок, коли функції задані на $\Gamma\times(0,\theta)$.

\begin{proof}[\indent Доведення леми $\ref{9lem5.1}$.]
Умова \eqref{9f5.4} коректна завдяки теоремі про сліди для анізотропних просторів Соболєва (див., наприклад, \cite[частина~II, теорема~4]{Slobodetskii58}). Позначимо через $\Upsilon^{s,s\gamma}(\Omega)$ лінійний многовид, утворений усіма функціями $u\in H^{s,s\gamma}(\Omega)$, які задовольняють цю умову. Згідно з вказаною теоремою про сліди можемо і будемо розглядати $\Upsilon^{s,s\gamma}(\Omega)$ як (замкнений) підпростір простору $H^{s,s\gamma}(\Omega)$. Із формули \eqref{9f5.3} випливає неперервне вкладення $H^{s,s\gamma}_{+}(\Omega)\hookrightarrow \Upsilon^{s,s\gamma}(\Omega)$. Тому з огляду на теорему Банаха про обернений оператор залишається довести обернене включен\-ня $\Upsilon^{s,s\gamma}(\Omega)\subset H^{s,s\gamma}_{+}(\Omega)$.

Нехай $u\in\Upsilon^{s,s\gamma}(\Omega)$. Потрібно довести, що $u=w$ на $\Omega$ для деякої функції $w\in H^{s,s\gamma}_{+}(\mathbb{R}^{n+1})$. Для цього використаємо три вказані нижче оператори продовження $O$, $T_{\tau}$, і $T_{G}$, які діють у деяких ізотропних просторах Соболєва.

Для функції $v\in L_{2}(0,\infty)$ означимо функцію $Ov\in L_{2}(\mathbb{R})$ за формулами $(Ov)(t):=v(t)$, якщо $t>0$, і $(Ov)(t):=0$, якщо $t\leq0$. Маємо обмежений лінійний оператор
\begin{equation}\label{9f5.5}
O:L_{2}(0,\infty)\to L_{2}(\mathbb{R}).
\end{equation}
Позначимо через $\Upsilon^{s\gamma}(0,\infty)$ лінійний многовид, утворений усіма функціями $v\in H^{s\gamma}(0,\infty)$ такими, що $v^{(k)}(0)=0$, коли ціле $k$ задовольняє умову $0\leq k<s\gamma-1/2$. За теоремою про сліди многовид $\Upsilon^{s\gamma}(0,\infty)$ є підпростором простору $H^{s\gamma}(0,\infty)$. На підставі \cite[теореми 2.9.3(a) і 2.10.3(b)]{Triebel80} та умови
$s\gamma-1/2\notin\mathbb{Z}$ робимо висновок, що звуження оператора \eqref{9f5.5} є ізоморфізмом на парі просторів:
\begin{equation}\label{9f5.6}
O:\Upsilon^{s\gamma}(0,\infty)\leftrightarrow H^{s\gamma}_{+}(\mathbb{R}).
\end{equation}
Нагадаємо, що $H^{s\gamma}_{+}(\mathbb{R})$ складається з усіх функцій $v\in H^{s\gamma}(\mathbb{R})$ таких, що  $\mathrm{supp}\,v\subseteq[0,\infty)$, і розглядається як підпростір простору $H^{s\gamma}(\mathbb{R})$.

Розглянемо лінійний обмежений оператор
\begin{equation}\label{9f5.7}
T_{\tau}:L_{2}(0,\tau)\rightarrow L_{2}(0,\infty)
\end{equation}
такий, що $T_{\tau}v=v$ на $(0,\tau)$ для кожної функції $v\in L_{2}(0,\tau)$ і, крім того, його звуження на $H^{s\gamma}(0,\tau)$ діє неперервно на парі просторів
\begin{equation}\label{9f5.8}
T_{\tau}:H^{s\gamma}(0,\tau)\rightarrow H^{s\gamma}(0,\infty).
\end{equation}

Розглянемо також лінійний обмежений оператор
\begin{equation}\label{9f5.9}
T_{G}:L_{2}(G)\rightarrow L_{2}(\mathbb{R}^{n})
\end{equation}
такий, що $T_{G}h=h$ на $G$ для кожної функції $h\in L_{2}(G)$ і, крім того, його звуження на $H^{s}(G)$ діє неперервно на парі просторів
\begin{equation}\label{9f5.10}
T_{G}:H^{s}(G)\rightarrow H^{s}(\mathbb{R}^{n}).
\end{equation}
Вказані оператори існують \cite[теореми 4.2.2 і 4.2.3]{Triebel80}.

Відомо, що
\begin{gather}\label{9f5.11}
H^{s,s\gamma}(\Omega)=H^{s}(G)\otimes L_{2}(0,\tau)\cap
L_{2}(G)\otimes H^{s\gamma}(0,\tau),\\
H^{s,s\gamma}(\mathbb{R}^{n+1})=
H^{s}(\mathbb{R}^{n})\otimes L_{2}(\mathbb{R})\cap
L_{2}(\mathbb{R}^{n})\otimes H^{s\gamma}(\mathbb{R}) \label{9f5.12}
\end{gather}
з еквівалентністю норм (див., наприклад, \cite[\S~8, п.~1]{AgranovichVishik64}). З формул \eqref{9f5.7}, \eqref{9f5.8}, \eqref{9f5.11} та включення $u\in\Upsilon^{s,s\gamma}(\Omega)$
випливає, що
\begin{equation*}
(I\otimes T_{\tau})u\in H^{s}(G)\otimes L_{2}(0,\infty)\cap
L_{2}(G)\otimes \Upsilon^{s\gamma}(0,\infty).
\end{equation*}
Тут $I$ позначає тотожний оператор на просторі $L_{2}(G)$. Тоді
\begin{equation}\label{9f5.13}
\begin{split}
w&:=(T_{G}\otimes(OT_{\tau}))u=(T_{G}\otimes O)(I\otimes T_{\tau})u\in\\
&\in H^{s}(\mathbb{R}^{n})\otimes L_{2}(\mathbb{R})\cap L_{2}(\mathbb{R}^{n})\otimes H^{s\gamma}_{+}(\mathbb{R})=
H^{s,s\gamma}_{+}(\mathbb{R}^{n+1})
\end{split}
\end{equation}
завдяки формулам \eqref{9f5.5}, \eqref{9f5.6}, \eqref{9f5.9}, \eqref{9f5.10}, і \eqref{9f5.12}. Крім того, $w=u$ в $\Omega$. Таким чином, $u\in H^{s,s\gamma}_{+}(\Omega)$.

Ці міркування доводять також, що висновок леми~\ref{9lem5.1} залишається правильним, якщо замінити у ній $\Omega$ на $\Pi:=\mathbb{R}^{n-1}\times(0,\tau)$ і $G$ на $\mathbb{R}^{n-1}$. (Звісно, слід узяти $\mathbb{R}^{n}$ замість $\mathbb{R}^{n+1}$ і нема потреби в операторі продовження $T_{G}$.) Звідси та з означення просторів $H^{s,s\gamma}_{+}(S)$ і $H^{s,s\gamma}(S)$ випливає, що висновок леми~\ref{9lem5.1} є правильним, якщо замінити $\Omega$ на $S$ і $G$ на~$\Gamma$.
\end{proof}

\begin{theorem}\label{9lem7.3}
Додатково до умови \eqref{9f7.1} припустимо, що $s_{0}\geq0$.
Тоді правильні такі рівності просторів з еквівалентністю норм у них:
\begin{gather}\label{9f7.12}
H^{s,s\gamma;\varphi}_{+}(\Omega)=
\bigl[H^{s_{0},s_{0}\gamma}_{+}(\Omega),
H^{s_{1},s_{1}\gamma}_{+}(\Omega)\bigr]_{\psi},\\
H^{s,s\gamma;\varphi}_{+}(S)=
\bigl[H^{s_{0},s_{0}\gamma}_{+}(S),
H^{s_{1},s_{1}\gamma}_{+}(S)\bigr]_{\psi}. \label{9f7.13}
\end{gather}
\end{theorem}

\begin{proof}[\indent Доведення]
Припустимо, що
\begin{equation}\label{9f7.11}
s_{j}\gamma-1/2\notin\mathbb{Z}\quad\mbox{для кожного}\quad j\in\{0,1\}.
\end{equation}
Це припущення зумовлено використанням леми \ref{9lem5.1} у подальших міркуваннях. Потім ми його позбудемося. Спочатку обґрунтуємо формулу \eqref{9f7.12}. Нагадаємо, що за означенням виконуються такі рівності:
\begin{gather}\label{9f7.14}
H^{s,s\gamma;\varphi}_{+}(\Omega)=
H^{s,s\gamma;\varphi}_{+}(\mathbb{R}^{n+1})/
H^{s,s\gamma;\varphi}_{Q}(\mathbb{R}^{n+1}),\\
H^{s_{j},s_{j}\gamma}_{+}(\Omega)=
H^{s_{j},s_{j}\gamma}_{+}(\mathbb{R}^{n+1})/
H^{s_{j},s_{j}\gamma}_{Q}(\mathbb{R}^{n+1}), \label{9f7.15}
\end{gather}
де $j\in\{0,1\}$. Тут у знаменниках використано позначення~\eqref{9f3.5}, у якому покладаємо $Q:=(\mathbb{R}^{k-1}\times[0,\infty))\setminus\Omega$. Виведемо формулу \eqref{9f7.12} з теореми~\ref{9lem7.2} за допомогою інтерполяції фактор-просторів (на підставі теореми~\ref{9prop6.2}). Для цього потрібно мати лінійне відображення $P$, задане на $H^{s_{0},s_{0}\gamma}_{+}(\mathbb{R}^{n+1})$ і таке, що $P$ є проєктором простору $H^{s_{j},s_{j}\gamma}_{+}(\mathbb{R}^{n+1})$ на підпростір
$H^{s_{j},s_{j}\gamma}_{Q}(\mathbb{R}^{n+1})$ для кожного  $j\in\{0,1\}$. Побудуємо це відображення.

Скористаємося міркуваннями та позначеннями, наведеними в доведенні леми~\ref{9lem5.1}. Подане там обґрунтування формули \eqref{9f5.13} показує, що лінійне відображення $T_{+}:=T_{G}\otimes(OT_{\tau})$ є обмеженим оператором на парі просторів
\begin{equation}\label{9f7.16}
T_{+}:=T_{G}\otimes(OT_{\tau}):\Upsilon^{s_{j},s_{j}\gamma}(\Omega)\to
H^{s_{j},s_{j}\gamma}_{+}(\mathbb{R}^{n+1})
\end{equation}
для кожного $j\in\{0,1\}$. Нагадаємо, що $T_{+}u=u$ на $\Omega$ для кожного $u\in\Upsilon^{s_{j},s_{j}\gamma}(\Omega)$. Крім того,  $\Upsilon^{s_{j},s_{j}\gamma}(\Omega)=
H^{s_{j},s_{j}\gamma}_{+}(\Omega)$ з огляду на лему~\ref{9lem5.1} та припущення~\eqref{9f7.11}.

Розглянемо лінійне відображення $P:w\mapsto w-T_{+}(w\!\upharpoonright\!\Omega)$, де функція $w$ пробігає простір $H^{s_{0},s_{0}\gamma}_{+}(\mathbb{R}^{n+1})$. Помічаємо, що $Pw=0$ на $\Omega$, а з умови $w=0$ на $\Omega$ випливає рівність $Pw=w$ на $\mathbb{R}^{n+1}$. На підставі цих властивостей та обмеженості оператора~\eqref{9f7.16} робимо висновок, що $P$~--- потрібне відображення.

Тепер, послідовно скориставшись формулою \eqref{9f7.15}, теоремою~\ref{9prop6.2} (рівність \eqref{9f6.2}), теоремою~\ref{9lem7.2} і формулою \eqref{9f7.14}, отримаємо такі рівності:
\begin{equation*}
\begin{split}
&\bigl[H^{s_{0},s_{0}\gamma}_{+}(\Omega),
H^{s_{1},s_{1}\gamma}_{+}(\Omega)\bigr]_{\psi}=\\
&=\bigl[H^{s_{0},s_{0}\gamma}_{+}(\mathbb{R}^{n+1})/
H^{s_{0},s_{0}\gamma}_{Q}(\mathbb{R}^{n+1}),\\
&\quad\;\;\, H^{s_{1},s_{1}\gamma}_{+}(\mathbb{R}^{n+1})/
H^{s_{1},s_{1}\gamma}_{Q}(\mathbb{R}^{n+1})\bigr]_{\psi}=\\
&=X_{\psi}/(X_{\psi}\cap H^{s_{0},s_{0}\gamma}_{Q}(\mathbb{R}^{n+1}))=\\
&=H^{s,s\gamma;\varphi}_{+}(\mathbb{R}^{n+1})/
(H^{s,s\gamma;\varphi}_{+}(\mathbb{R}^{n+1})\cap
H^{s_{0},s_{0}\gamma}_{Q}(\mathbb{R}^{n+1}))=\\
&=H^{s,s\gamma;\varphi}_{+}(\mathbb{R}^{n+1})/
H^{s,s\gamma;\varphi}_{Q}(\mathbb{R}^{n+1})
=H^{s,s\gamma,\varphi}_{+}(\Omega).
\end{split}
\end{equation*}
Тут
\begin{equation*}
X_{\psi}:=\bigl[H^{s_{0},s_{0}\gamma}_{+}(\mathbb{R}^{n+1}),
H^{s_{1},s_{1}\gamma}_{+}(\mathbb{R}^{n+1})\bigr]_{\psi}=
H^{s,s\gamma;\varphi}_{+}(\mathbb{R}^{n+1}).
\end{equation*}
Ці рівності просторів виконуються разом з еквівалентністю норм. (Зауважимо також, що перша пара є регулярною завдяки теоремі~\ref{9prop6.2}.) Формулу \eqref{9f7.12} доведено.

Доведемо другу формулу~\eqref{9f7.13}. Пара просторів у правій її частині регулярна за теоремою~\ref{9lem3.1}, розглянутої у випадку, коли  $\varphi(r)\equiv1$ і $s\gamma-1/2\notin\mathbb{Z}$. (Наведене нижче доведення теореми~\ref{9lem3.1} у цьому випадку не використовує теорему~\ref{9lem7.3}.) Виведемо формулу \eqref{9f7.13} з її аналога
\begin{equation}\label{9f7.17}
H^{s,s\gamma;\varphi}_{+}(\Pi)=
\bigl[H^{s_{0},s_{0}\gamma}_{+}(\Pi),
H^{s_{1},s_{1}\gamma}_{+}(\Pi)\bigr]_{\psi}
\end{equation}
для смуги $\Pi:=\mathbb{R}^{n-1}\times(0,\tau)$. Його доведення таке саме, як і обґрунтування формули \eqref{9f7.12}; при цьому у міркуваннях слід   замінити $\Omega$, $\mathbb{R}^{n+1}$ і $T_{G}$ відповідно на $\Pi$,  $\mathbb{R}^{n}$ і тотожний оператор.

Виходячи з означення просторів на $S$, подане у формулах \eqref{9f3.7} і \eqref{9f3.8}, виведемо потрібне співвідношення \eqref{9f7.13} із рівності \eqref{9f7.17} за допомогою деяких операторів розпрямлення та склеювання многовиду $S$. Означимо лінійний оператор розпрямлення за формулою
\begin{equation}\label{L}
L:v\mapsto\bigl((\chi^*_{1}v)\circ\theta_{1}^*,\ldots,
(\chi^*_{\lambda}v)\circ\theta_{\lambda}^*\bigr),
\quad\mbox{де}\quad v\in L_2(S).
\end{equation}
Його звуження є ізометричними операторами на таких парах просторів:
\begin{gather}\label{9f7.18}
L:H^{s,s\gamma;\varphi}_{+}(S)\rightarrow
\bigl(H^{s,s\gamma;\varphi}_{+}(\Pi)\bigr)^{\lambda},\\
L:H^{s_j,s_j\gamma}_{+}(S)\rightarrow
\bigl(H^{s_j,s_j\gamma}_{+}(\Pi)\bigr)^{\lambda},
\label{9f7.19}
\end{gather}
де $j\in\{0,1\}$. Це~--- прямий наслідок означення цих просторів. Застосувавши інтерполяцію з параметром $\psi$ до операторів \eqref{9f7.19}, отримаємо обмежений оператор
\begin{equation}\label{9f7.20}
\begin{split}
L&:\bigl[H^{s_0,s_0\gamma}_{+}(S),H^{s_1,s_1\gamma}_{+}(S)\bigr]_{\psi}\to\\
&\to
\bigl[\bigl(H^{s_0,s_0\gamma}_{+}(\Pi)\bigr)^{\lambda},
\bigl(H^{s_1,s_1\gamma}_{+}(\Pi)\bigr)^{\lambda}\bigr]_{\psi}.
\end{split}
\end{equation}
Останній інтерполяційний простір дорівнює $(H^{s,s\gamma;\varphi}_{+}(\Pi))^{\lambda}$ з точністю до еквівалентності норм з огляду на теорему~\ref{9prop6.3} і формулу~\eqref{9f7.17}. Отже, обмежений оператор \eqref{9f7.20} діє на парі просторів
\begin{equation}\label{9f7.21}
L:\bigl[H^{s_0,s_0\gamma}_{+}(S),H^{s_1,s_1\gamma}_{+}(S)\bigr]_{\psi}\to
\bigl(H^{s,s\gamma;\varphi}_{+}(\Pi)\bigr)^{\lambda}.
\end{equation}

Означимо оператор склеювання за формулою
\begin{equation}\label{K}
\begin{gathered}
K:(h_{1},\ldots,h_{\lambda})\mapsto\sum_{k=1}^{\lambda}\,
O_{k}((\eta_{k}^{*}h_{k})\circ\theta_{k}^{*-1}),\\
\mbox{де}\;\;h_{1},\ldots,h_{\lambda}\in L_2(\Pi).
\end{gathered}
\end{equation}
Тут кожну функцію $\eta_{k}\in C_{0}^{\infty}(\mathbb{R}^{n-1})$ вибрано так, що $\eta_{k}=1$ на множині $\theta^{-1}_{k}(\mathrm{supp}\,\chi_{k})$ та $\eta^{*}_{k}(x,t):=\eta_{k}(x)$ для всіх $x\in\mathbb{R}^{n-1}$ і $t\in (0,\tau)$. Крім того, $O_{k}$ позначає оператор продовження функцій нулем з $\Gamma_k\times(0,\tau)$ на $S$. Таким чином, для будь-яких $y\in\Gamma$ і $t\in (0,\tau)$, виконується співвідношення
\begin{equation*}
\begin{split}
&\bigl(O_{k}((\eta_{k}^{*}h_{k})\circ\theta_{k}^{*-1})\bigr)(y,t)= \\&\phantom{-} \\
&=\left\{
  \begin{array}{ll}
    \eta_{k}(x)h_{k}(x,t),&\mbox{якщо}\; y=\theta_{k}(x)\in\Gamma_{k}\;\mbox{для деякого}\;x\in\mathbb{R}^{n-1}\\
    0&\mbox{у противному разі.}
  \end{array}
\right.
\end{split}
\end{equation*}

Відображення $K$ є лівим оберненим до $L$. Справді,
\begin{equation*}
\begin{split}
KLv&=\sum_{k=1}^{\lambda}\,O_{k}\bigl(\bigl(\eta^*_{k}\,
((\chi_{k}^*v)\circ\theta_{k}^*)\bigr)\circ\theta_{k}^{*-1}\bigr)=\\
&=\sum_{k=1}^{\lambda}\,O_{k}
\bigl((\chi_{k}^*v)\circ\theta_{k}^*\circ\theta_{k}^{*-1}\bigr)=
\sum_{k=1}^{\lambda}\,\chi_{k}^*v=v,
\end{split}
\end{equation*}
тобто
\begin{equation}\label{9f7.22}
KLv=v\quad\mbox{для кожного}\quad v\in L_2(S).
\end{equation}

Покажемо, що $K$ є обмеженим оператором на парі просторів
\begin{equation}\label{9f7.23}
K:\bigl(H^{s,s\gamma;\varphi}_{+}(\Pi)\bigr)^{\lambda}
\rightarrow H^{s,s\gamma;\varphi}_{+}(S).
\end{equation}
Для довільної вектор-функції
\begin{equation*}
h:=(h_{1},\ldots,h_{\lambda})\in
\bigl(H^{s,s\gamma;\varphi}_{+}(\Pi)\bigr)^{\lambda}
\end{equation*}
виконуються такі рівності:
\begin{equation}\label{9f7.24}
\begin{split}
&\bigl\|Kh\bigr\|^{2}_{H^{s,s\gamma;\varphi}_{+}(S)}=\\
&=\sum_{l=1}^{\lambda}\;\bigl\|(\chi_{l}^*\,Kh)
\circ\theta_{l}^*\bigr\|_{H^{s,s\gamma;\varphi}_{+}(\Pi)}^{2}=\\
&=\sum_{l=1}^{\lambda}\,\Bigl\|\Bigl(\chi_{l}^*\,
\sum_{k=1}^{\lambda}\,
O_{k}\bigl((\eta_{k}^{*}h_{k})\circ\theta_{k}^{*-1}\bigr)
\Bigr)
\circ\theta_{l}^*\,\Bigr\|_{H^{s,s\gamma;\varphi}_{+}(\Pi)}^{2}=\\
&=\sum_{l=1}^{\lambda}\;\Bigl\|\,
\sum_{k=1}^{\lambda}Q_{k,l}\,h_{k}\,
\Bigr\|_{H^{s,s\gamma;\varphi}_{+}(\Pi)}^{2}.
\end{split}
\end{equation}
Тут поклали
\begin{equation}\label{9f7.25}
(Q_{k,l}\,w)(x,t):=\eta_{k,l}(\beta_{k,l}(x))\,w(\beta_{k,l}(x),t)
\end{equation}
для всіх $w\in L_{2}(\Pi)$, $x\in\mathbb{R}^{n-1}$ і $t\in(0,\tau)$, де
\begin{equation*}
\eta_{k,l}:=(\chi_{l}\circ\theta_{k})\eta_{k}\in C_{0}^{\infty}(\mathbb{R}^{n-1}).
\end{equation*}
Крім того, $\beta_{k,l}:\mathbb{R}^{n-1}\leftrightarrow\mathbb{R}^{n-1}$ є деяким нескінченно гладким дифеоморфізмом таким, що $\beta_{k,l}=\theta_{k}^{-1}\circ\theta_{l}$ в околі $\mathrm{supp}\,\eta_{k,l}$. Як відомо \cite[теорема~B.1.8]{Hormander85}, оператор $\omega\mapsto(\eta_{k,l}\,\omega)\circ\beta_{k,l}$ обмежений на кожному соболєвському просторі $H^{\sigma}(\mathbb{R}^{n-1})$, де $\sigma\in\mathbb{R}$. Тому оператор $w\mapsto Q_{k,l}\,w$, означений згідно з формулою \eqref{9f7.25}, де $w\in L_{2}(\mathbb{R}^{n})$, $x\in\mathbb{R}^{n-1}$ і $t\in\mathbb{R}$, є обмеженим на кожному просторі
\begin{equation*}
H^{s_{j},s_{j}\gamma}(\mathbb{R}^{n})=H^{s_{j}}(\mathbb{R}^{n-1})\otimes L_{2}(\mathbb{R})\cap
L_{2}(\mathbb{R}^{n-1})\otimes H^{s_{j}\gamma}(\mathbb{R}),
\end{equation*}
де $j\in\{0,1\}$. Отже, звуження відображення $w\mapsto Q_{k,l}w$, де $w\in L_{2}(\Pi)$, на кожний простір $H^{s_{j},s_{j}\gamma}_{+}(\Pi)$, де $j\in\{0,1\}$, є обмеженим оператором на цьому просторі.
Звідси на підставі інтерполяційної  формули \eqref{9f7.17} робимо висновок, що це відображення є обмеженим оператором і на просторі $H^{s,s\gamma;\varphi}_{+}(\Pi)$. Тому згідно з формулою \eqref{9f7.24}, маємо співвідношення
\begin{equation*}
\begin{split}
\bigl\|Kh\bigr\|^{2}_{H^{s,s\gamma;\varphi}_{+}(S)}&=
\sum_{l=1}^{\lambda}\;\Bigl\|\,
\sum_{k=1}^{\lambda}Q_{k,l}\,h_{k}\,
\Bigr\|_{H^{s,s\gamma;\varphi}_{+}(\Pi)}^{2}\leq\\ &\leq
c\sum_{k=1}^{\lambda}\|h_{k}\|_{H^{s,s\gamma;\varphi}_{+}(\Pi)}^{2},
\end{split}
\end{equation*}
де $c$~--- деяке додатне число, яке не залежить від $h$. Таким чином, оператор $K$ обмежений на парі просторів \eqref{9f7.23}.

Ті самі міркування доводять, що $K$ є також обмеженим оператором на парі просторів
\begin{equation}\label{9f7.26}
K:\bigl(H^{s_{j},s_{j}\gamma}_{+}(\Pi)\bigr)^{\lambda}
\rightarrow H^{s_{j},s_{j}\gamma}_{+}(S)
\end{equation}
для кожного $j\in\{0,1\}$. Звідси за допомогою квадратичної інтерполяції з  параметром $\psi$ отримаємо з огляду на теорему~\ref{9prop6.3} і формулу~\eqref{9f7.17} обмежений оператор
\begin{equation}\label{9f7.27}
K:\bigl(H^{s,s\gamma;\varphi}_{+}(\Pi)\bigr)^{\lambda}\to
\bigl[H^{s_0,s_0\gamma}_{+}(S),H^{s_1,s_1\gamma}_{+}(S)\bigr]_{\psi}.
\end{equation}
Із формул \eqref{9f7.18}, \eqref{9f7.27} і \eqref{9f7.22} випливає, що тотожний оператор $KL$ здійснює неперервне вкладення простору $H^{s,s\gamma,\varphi}_{+}(S)$ в інтерполяційний простір
\begin{equation*}
\bigl[H^{s_0,s_0\gamma}_{+}(S),H^{s_1,s_1\gamma}_{+}(S)\bigr]_{\psi}.
\end{equation*}
Крім того, із формул \eqref{9f7.21} і \eqref{9f7.23} випливає, що тотожний оператор $KL$ здійснює обернене неперервне вкладення. Отже, потрібна рівність \eqref{9f7.13} виконується разом еквівалентністю норм у просторах.

Таким чином, теорему \ref{9lem7.3} доведено за додаткового припущення \eqref{9f7.11}. Розглянемо тепер загальну ситуацію, коли це припущення може не виконуватися. Виберемо число $\sigma>s_{1}$ таке, що $\sigma\gamma-1/2\notin\mathbb{Z}$. Нехай $W:=\Omega$ або $W:=S$. За доведеним,
\begin{equation*}
H^{s_j,s_j\gamma}_{+}(W)=
\bigl[H^{0,0}_{+}(W),H^{\sigma,\sigma\gamma}_{+}(W)\bigr]_{\psi_j}
\end{equation*}
для кожного $j\in\{0,1\}$. Тут інтерполяційний параметр $\psi_j$ означено формулами $\psi_j(r)=r^{s_j/\sigma}$, якщо $r\geq1$, і $\psi_j(r)=1$, якщо $0<r<1$, тобто формулою \eqref{9f7.2}, у якій беремо $s_j$, $0$ і $\sigma$ замість $s$, $s_0$ і $s_1$ відповідно та покладаємо $\varphi(\cdot)\equiv1$. Тоді згідно з теоремою \ref{9prop6.4} і вже доведеним маємо такі рівності:
\begin{align*}
&\bigl[H^{s_0,s_0\gamma}_{+}(W),H^{s_1,s_1\gamma}_{+}(W)\bigr]_{\psi}=\\
&=\bigl[[H^{0,0}_{+}(W),H^{\sigma,\sigma\gamma}_{+}(W)]_{\psi_0},
[H^{0,0}_{+}(W),H^{\sigma,\sigma\gamma}_{+}(W)]_{\psi_1}\bigr]_{\psi}=\\
&=\bigl[H^{0,0}_{+}(W),H^{\sigma,\sigma\gamma}_{+}(W)\bigr]_{\omega}=
H^{s,s\gamma;\varphi}_{+}(W).
\end{align*}
Тут
\begin{align*}
\omega(r)&=\psi_0(r)\cdot\psi\biggl(\frac{\psi_1(r)}{\psi_0(r)}\biggr)=\\
&=r^{s_0/\sigma}\psi(r^{(s_1-s_0)/\sigma})
=r^{s/\sigma}\varphi(r^{1/\sigma}),
\end{align*}
якщо $r\geq1$, і $\omega(r)=1$, якщо $0<r<1$. Це випливає з означення \eqref{9f7.2} функції $\psi$. Отже, теорему \ref{9lem7.3} доведено і у загальній ситуації.
\end{proof}

\begin{remark}\label{lem5.4}
Якщо $n=1$, то циліндр $\Omega=G\times(0,\tau)$ стає інтервалом $(0,\tau)$ часової осі. Тоді формула \eqref{9f7.12} набуває такого вигляду:
\begin{equation*}
H^{s;\varphi}_{+}(0,\tau)=
\bigl[H^{s_{0}}_{+}(0,\tau),H^{s_{1}}_{+}(0,\tau)\bigr]_{\psi}
\end{equation*}
з еквівалентністю норм у просторах. Тут, як і в теоремі \ref{9lem7.3}, додатково до умови \eqref{9f7.1} припускаємо, що $s_{0}\geq0$.
\end{remark}

Розглянемо версію теореми~\ref{9lem7.3} для більш широких просторів $H^{s,s\gamma;\varphi}(\Omega)$ і
$H^{s,s\gamma;\varphi}(S)$.

\begin{theorem}\label{9lem7.3a}
Додатково до умови \eqref{9f7.1} припустимо, що $s_{0}\geq0$.
Тоді правильні такі рівності просторів з еквівалентністю норм у них:
\begin{gather}\label{9f7.12a}
H^{s,s\gamma;\varphi}(\Omega)=
\bigl[H^{s_{0},s_{0}\gamma}(\Omega),
H^{s_{1},s_{1}\gamma}(\Omega)\bigr]_{\psi},\\
H^{s,s\gamma;\varphi}(S)=
\bigl[H^{s_{0},s_{0}\gamma}(S),
H^{s_{1},s_{1}\gamma}(S)\bigr]_{\psi}. \label{9f7.13a}
\end{gather}
\end{theorem}

\begin{proof}[\indent Доведення] Воно подібне до обґрунтування теореми~\ref{9lem7.3}. Наведемо коротко відповідні міркування.
Почнемо з доведення формули \eqref{9f7.12a}.

За означенням
\begin{gather}\label{9f7.14a}
H^{s,s\gamma;\varphi}(\Omega)=
H^{s,s\gamma;\varphi}(\mathbb{R}^{n+1})/
H^{s,s\gamma;\varphi}_{Q}(\mathbb{R}^{n+1}),\\
H^{s_{j},s_{j}\gamma}(\Omega)=
H^{s_{j},s_{j}\gamma}(\mathbb{R}^{n+1})/
H^{s_{j},s_{j}\gamma}_{Q}(\mathbb{R}^{n+1}), \label{9f7.15a}
\end{gather}
де $j\in\{0,1\}$. Тут знаменники означено формулою~\eqref{8f41} для $Q:=\mathbb{R}^{n+1}\setminus\Omega$ і відповідних показників $\mu$. Розглянемо лінійне відображення
\begin{equation*}
P:w\mapsto w-(T_{G}\otimes T_{(0,\tau)})
(w\!\upharpoonright\!\Omega),
\end{equation*}
задане на функціях $w\in L_{2}(\mathbb{R}^{n+1})$. Тут $T_{G}$~--- відображення \eqref{9f5.9}, яке продовжує функції аргументу $x\in G$ на увесь простір $\mathbb{R}^{n}$ і є обмеженим оператором на парі соболєвських просторів \eqref{9f5.10} для кожного $s\in[0,s_1]$. Крім того, $T_{(0,\tau)}$~--- це відображення \eqref{9f5.9} у випадку, коли $G$~--- інтервал $(0,\tau)$ часової осі. Отже, відображення $T_{(0,\tau)}$ продовжує функції аргументу $t\in(0,\tau)$ на вісь $\mathbb{R}$ і є обмеженим оператором на парі просторів $H^{\sigma}(0,\tau)$ і $H^{\sigma}(\mathbb{R})$, де $\sigma\in[0,s_1\gamma]$. Звісно, якщо $\mathrm{supp}\,w\subseteq Q$, то $Pw=w$. Тому з формул \eqref{9f5.11} і \eqref{9f5.12} випливає, що $P$ є проєктором простору $H^{s_{j},s_{j}\gamma}(\mathbb{R}^{n+1})$ на його підпростір $H^{s_{j},s_{j}\gamma}_{Q}(\mathbb{R}^{n+1})$ для кожного $j\in\{0,1\}$.

Отже, на підставі теорем \ref{9prop6.2} і \ref{9lem7.1} та формул \eqref{9f7.14a} і \eqref{9f7.15a} маємо такі рівності:
\begin{equation*}
\begin{split}
\bigl[H^{s_{0},s_{0}\gamma}(\Omega),
H^{s_{1},s_{1}\gamma}(\Omega)\bigr]_{\psi}
&=X_{\psi}/(X_{\psi}\cap H^{s_{0},s_{0}\gamma}_{Q}(\mathbb{R}^{n+1}))=\\
&=H^{s,s\gamma;\varphi}(\mathbb{R}^{n+1})/
H^{s,s\gamma;\varphi}_{Q}(\mathbb{R}^{n+1})=\\
&=H^{s,s\gamma,\varphi}(\Omega),
\end{split}
\end{equation*}
де
\begin{equation*}
X_{\psi}:=\bigl[H^{s_{0},s_{0}\gamma}(\mathbb{R}^{n+1}),
H^{s_{1},s_{1}\gamma}(\mathbb{R}^{n+1})\bigr]_{\psi}=
H^{s,s\gamma;\varphi}(\mathbb{R}^{n+1}).
\end{equation*}
Потрібну формулу \eqref{9f7.12a} доведено.

Формула \eqref{9f7.13a} виводиться з рівності
\begin{equation}\label{9f7.17a}
H^{s,s\gamma;\varphi}(\Pi)=
\bigl[H^{s_{0},s_{0}\gamma}(\Pi),
H^{s_{1},s_{1}\gamma}(\Pi)\bigr]_{\psi},
\end{equation}
де $\Pi:=\mathbb{R}^{n-1}\times(0,\tau)$, за допомогою операторів \eqref{L} і \eqref{K} відповідно розпрямлення і склеювання многовиду $S$. Рівність \eqref{9f7.17a} доводиться так само, як і формула \eqref{9f7.12a}; при цьому у міркуван\-нях треба узяти $\Pi$, $\mathbb{R}^{n}$ і тотожний оператор замість $\Omega$, $\mathbb{R}^{n+1}$ і $T_{G}$ відповідно. Замінивши простори $H^{\bullet}_{+}(\cdot)$ на їх аналоги $H^{\bullet}(\cdot)$ у частині доведення теореми \ref{9lem7.3}, яка стосується обґрунтування формули \eqref{9f7.13}, приходимо до висновку, що тотожний оператор $KL$ здійснює неперервне вкладення простору $H^{s,s\gamma,\varphi}(S)$ у простір
\begin{equation*}
\bigl[H^{s_0,s_0\gamma}(S),H^{s_1,s_1\gamma}(S)\bigr]_{\psi}
\end{equation*}
та реалізує обернене неперервне вкладення. Це дає потрібну формулу \eqref{9f7.13a}.
\end{proof}

Як і було обіцяно, завершимо цей підрозділ тим, що доведемо  теореми \ref{9lem3.1a} і~\ref{9lem3.1}.

\begin{proof}[\indent Доведення теореми $\ref{9lem3.1a}$]
У випадку, коли $\varphi(\cdot)\equiv1$, ця теорема відома (див., наприклад, \cite[розд.~I, \S~5]{Slobodetskii58}). Виведемо її для довільного $\varphi\in\mathcal{M}$ з цього випадку за допомогою теореми~\ref{9lem7.3a}. Висновок (i) теореми $\ref{9lem3.1a}$ є наслідком інтерполяційної формули \eqref{9f7.13a}.

Обґрунтуємо висновок (ii). Нехай $\mathcal{A}_{0}$ і $\mathcal{A}_{1}$~--- дві пари, кожна з яких складається з атласу на $\Gamma$ та відповідного розбиття одиниці, використаних в означенні простору $H^{s,s\gamma;\varphi}(S)$ і скалярного добутку в ньому. Використовуємо позначення $H^{s,s\gamma;\varphi}(S,\mathcal{A}_{k})$ і $H^{s_{j},s_{j}\gamma}(S,\mathcal{A}_{k})$ для  просторів $H^{s,s\gamma;\varphi}(S)$ і
$H^{s_{j},s_{j}\gamma}(S)$, які відповідають парі $\mathcal{A}_{k}$; тут $j,k\in\{0,1\}$. Оскільки простір $H^{s_{j},s_{j}\gamma}(S,\mathcal{A}_{k})$ не залежить від $k\in\{0,1\}$ з точністю до еквівалентності норм, то тотожне відображення здійснює ізоморфізм між просторами $H^{s_{j},s_{j}\gamma}(S,\mathcal{A}_{0})$ і $H^{s_{j},s_{j}\gamma}(S,\mathcal{A}_{1})$ для кожного $j\in\{0,1\}$. Тому згідно з \eqref{9f7.13a} тотожне відображення здійснює ізоморфізм між просторами $H^{s,s\gamma;\varphi}(S,\mathcal{A}_{0})$ і $H^{s,s\gamma;\varphi}(S,\mathcal{A}_{1})$, що і обґрунтовує висновок~(ii).

Висновок~(iii) є наслідком щільності множини $C^{\infty}(\overline{S})$ у
просторі $H^{s_{1},s_{1}\gamma}(S)$, який неперервно і щільно вкладений у простір $H^{s,s\gamma;\varphi}(S)$ з огляду на \eqref{9f7.13a}.
\end{proof}

\begin{proof}[\indent Доведення теореми $\ref{9lem3.1}$]
Якщо $\varphi(\cdot)\equiv1$ і $s\gamma-1/2\notin\mathbb{Z}$, то на підставі леми~\ref{9lem5.1} робимо висновок, що $H^{s,s\gamma}_{+}(S)$ є підпростором простору $H^{s,s\gamma}(S)$ з точністю до еквівалентності норм. Отже, у цьому випадку висновки (i) та (ii) теореми $\ref{9lem3.1}$ є наслідками теореми \ref{9lem3.1a}. У загальному випадку ці висновки випливають тепер з теореми~\ref{9lem7.3} за допомогою міркувань, цілком аналогічних до наведених у доведенні теореми \ref{9lem3.1a}.

Обґрунтуємо висновок (iii) за допомогою операторів $L$ і $K$, використаних у доведенні теореми~\ref{9lem7.3}. У просторі $H^{s,s\gamma}_{+}(\Pi)$ є щільною множина
\begin{equation*}
\Upsilon^{\infty}_{0}(\overline{\Pi}):=
\bigl\{w\!\upharpoonright\!\overline{\Pi}:
w\in C^{\infty}_{0}(\mathbb{R}^{n-1}\times(0,\infty))\bigr\},
\end{equation*}
як зазначено в підрозділі~\ref{sec2.3}. Для довільної функції $v\in H^{s,s\gamma}_{+}(S)$, апроксимуємо вектор $Lv\in(H^{s,s\gamma;\varphi}_{+}(\Pi)\bigr)^{\lambda}$ послідовністю векторів $h^{(j)}\in(\Upsilon^{\infty}_{0}(\overline{\Pi}))^{\lambda}$ у нормі в просторі $(H^{s,s\gamma;\varphi}_{+}(\Pi)\bigr)^{\lambda}$. Послідовність функцій $Kh^{(j)}$ лежить у просторі \eqref{9lem3.1} і апроксимує функцію $KLv=v$ у нормі в просторі $H^{s,s\gamma;\varphi}_{+}(S)$. Висновок (iii) обґрунтовано.
\end{proof}

\markright{\emph \ref{sec2.7}. Оператор даних Коші}

\section[Оператор даних Коші]{Оператор даних Коші}\label{sec2.7}

\markright{\emph \ref{sec2.7}. Оператор даних Коші}

Для дослідження просторів правих частин параболічних задач у циліндрі $\Omega=G\times(0,\tau)$ потрібна одна теорема про властивості оператора, який функції, заданій на бічній поверхні $S$ циліндра, ставить у відповідність її дані Коші на $\Gamma=\partial G$.

\begin{theorem}\label{8lem1}
Нехай задано натуральні числа $b$ і $r$. Розглянемо лінійне відображення
\begin{equation}\label{8f25}
R:g\mapsto\bigl(g\!\upharpoonright\!\Gamma,
\partial_{t}g\!\upharpoonright\!\Gamma,\dots,
\partial^{r-1}_{t}g\!\upharpoonright\!\Gamma\bigr),
\quad\mbox{де}\quad g\in C^{\infty}(\overline{S}).
\end{equation}
Воно продовжується єдиним чином (за неперервністю) до обмеженого лінійного  оператора
\begin{equation}\label{8f65}
R:H^{s,s/(2b);\varphi}(S)\rightarrow \bigoplus_{k=0}^{r-1}
H^{s-2bk-b;\varphi}(\Gamma)=:\mathbb{H}^{s;\varphi}(\Gamma)
\end{equation}
для довільних параметрів $s>2br-b$ і $\varphi\in\mathcal{M}$. Він має правий обернений оператор. Більше того, існує лінійний обмежений оператор $T:(L_2(\Gamma))^r\to L_2(S)$ такий, що
для довільних $s>2br-b$ і $\varphi\in\mathcal{M}$ його звуження є
обмеженим оператором на парі просторів
\begin{equation}\label{8f43}
T:\mathbb{H}^{s;\varphi}(\Gamma)\to H^{s,s/(2b);\varphi}(S)
\end{equation}
та задовольняє умову $RTv=v$ для кожної функції $v\in\mathbb{H}^{s;\varphi}(\Gamma)$.
\end{theorem}

Звісно, у формулі \eqref{8f25} трактуємо $g$ як функцію аргументів $x\in\Gamma$ і $t\in[0,\tau]$.

\begin{proof}[\indent Доведення]
Спочатку доведемо аналог цієї теореми для відповідних просторів, заданих на $\mathbb{R}^{n}$ і $\mathbb{R}^{n-1}$ замість $S$ і $\Gamma$.
Потім виведемо її з цього аналога за допомогою спеціальних локальних карт на~$S$.

Розглянемо лінійне відображення
\begin{equation}\label{8f60}
\begin{gathered}
R_{0}:w\mapsto\bigl(w\!\mid_{t=0},\,
\partial_{t}w\!\mid_{t=0},\dots,
\partial^{r-1}_{t}w\!\mid_{t=0}\bigr),\\
\quad\mbox{де}\;\,w\in \mathcal{S}(\mathbb{R}^{n}).
\end{gathered}
\end{equation}
При цьому інтерпретуємо $w$ як функцію $w(x,t)$ аргументів $x\in$ $\in\mathbb{R}^{n-1}$ і $t\in\mathbb{R}$; отож, $R_{0}w\in(\mathcal{S}(\mathbb{R}^{n-1}))^{r}$ у формулі \eqref{8f60}.

Нехай $\nobreak{s>2br-b}$ і $\varphi\in\mathcal{M}$. Доведемо, що відображення \eqref{8f60} продовжується єдиним чином до лінійного обмеженого оператора
\begin{align}\label{8f63}
R_{0}:H^{s,s/(2b);\varphi}(\mathbb{R}^{n})&\rightarrow \bigoplus_{k=0}^{r-1}
H^{s-2bk-b;\varphi}(\mathbb{R}^{n-1})=:\\
&=:\mathbb{H}^{s;\varphi}(\mathbb{R}^{n-1}). \notag
\end{align}
Це відомо у соболєвському випадку, коли  $\varphi(\cdot)\equiv1$ (див. \cite[розд.~II, теорема~7]{Slobodetskii58}).
Для довільного $\varphi\in\mathcal{M}$ виведемо потрібний результат за допомогою квадратичної інтерполяції відповідних просторів Соболєва.

Виберемо числа $s_0$ і $s_1$ такі, що $2br-b<s_0<s<s_1$, та розглянемо
лінійні обмежені оператори
\begin{equation}
R_{0}:H^{s_j,s_j/(2b)}(\mathbb{R}^{n})\to\mathbb{H}^{s_j}(\mathbb{R}^{n-1}),
\end{equation}
де $j\in\{0,1\}$. Нехай $\psi$~--- інтерполяційний параметр \eqref{9f7.2}. Застосувавши інтерполяцію з параметром $\psi$ до цих операторів, отримаємо на підставі теореми \ref{8prop4} обмежений оператор
\begin{equation}\label{8f27}
R_{0}:H^{s,s/(2b);\varphi}(\mathbb{R}^{n})\to
\bigl[\mathbb{H}^{s_0}(\mathbb{R}^{n-1}),
\mathbb{H}^{s_1}(\mathbb{R}^{n-1})\bigr]_{\psi}.
\end{equation}
Він є продовженням за неперервністю відображення \eqref{8f60}, оскільки множина $\mathcal{S}(\mathbb{R}^{n})$ щільна в просторі $H^{s,s/(2b);\varphi}(\mathbb{R}^{n})$. Тут
\begin{equation}\label{8f-intHH}
\begin{aligned}
&\bigl[\mathbb{H}^{s_0}(\mathbb{R}^{n-1}),
\mathbb{H}^{s_1}(\mathbb{R}^{n-1})\bigr]_{\psi}=\\
&=\bigoplus_{k=0}^{r-1}
\bigl[H^{s_0-2bk-b}(\mathbb{R}^{n-1}),
H^{s_1-2bk-b}(\mathbb{R}^{n-1})\bigr]_{\psi}=\\
&=\bigoplus_{k=0}^{r-1}H^{s-2bk-b;\varphi}(\mathbb{R}^{n-1})=
\mathbb{H}^{s;\varphi}(\mathbb{R}^{n-1})
\end{aligned}
\end{equation}
згідно з теоремами \ref{9prop6.3} і \ref{8prop4}.
Отже, \eqref{8f27} є потрібним обмеженим оператором \eqref{8f63}.

Побудуємо тепер лінійне відображення
\begin{equation}\label{8f58}
T_0:\bigl(L_{2}(\mathbb{R}^{n-1})\bigr)^r\to L_{2}(\mathbb{R}^{n})
\end{equation}
таке, що його звуження є обмеженим оператором на парі просторів
$\mathbb{H}^{s;\varphi}(\mathbb{R}^{n-1})$ і $H^{s,s/(2b);\varphi}(\mathbb{R}^{n})$ та цей оператор є правим оберненим до \eqref{8f63}. Подібно до \cite[доведення теореми~2.5.7]{Hormander63} означимо лінійне відображення
\begin{equation}\label{8f58-def}
T_0:v\mapsto \mathcal{F}_{\xi\mapsto x}^{-1}\biggl[
\beta\bigl(\langle\xi\rangle^{2b}t\bigr)\,\sum_{k=0}^{r-1}
\frac{1}{k!}\,\widehat{v_k}(\xi)\times t^k\biggr](x,t)
\end{equation}
на векторах
$$
v:=(v_0,\dots,v_{r-1})\in\bigl(\mathcal{S}'(\mathbb{R}^{n-1})\bigr)^r.
$$
При цьому розглядаємо $T_{0}v$ як розподіл на евклідовому просторі $\mathbb{R}^{n}$ точок
$(x,t)$, де $x=(x_{1},\ldots,x_{n-1})\in\mathbb{R}^{n-1}$ і $t\in\mathbb{R}$. У формулі \eqref{8f58-def} функцію $\beta\in C^{\infty}_{0}(\mathbb{R})$ вибрано так, що $\beta=1$ в околі нуля. Як звичайно, $\mathcal{F}_{\xi\mapsto x}^{-1}$ позначає обернене перетворення Фур'є за векторною зміною $\xi=(\xi_{1},\ldots,\xi_{n-1})\in\mathbb{R}^{n-1}$, а $\langle\xi\rangle:=(1+|\xi|^2)^{1/2}$. Змінна $\xi$ є дуальною до $x$ відносно прямого перетворення Фур'є $\widehat{w}(\xi)=(\mathcal{F}w)(\xi)$ функції $w(x)$.

Звісно, відображення \eqref{8f58-def} означене коректно і діє неперервно на парі лінійних топологічних просторів $(\mathcal{S}'(\mathbb{R}^{n-1})^r$ і $\mathcal{S}'(\mathbb{R}^{n})$.
Очевидно також, що його звуження є обмеженим оператором на парі гільбертових просторів $(L_{2}(\mathbb{R}^{n-1}))^r$ і $L_{2}(\mathbb{R}^{n})$. Покажемо, що
\begin{equation}\label{8f57}
R_{0}T_{0}v=v \quad\mbox{для довільного}\;\,
v\in\bigl(\mathcal{S}(\mathbb{R}^{n-1})\bigr)^r.
\end{equation}
Зауважимо, що ліва частина рівності \eqref{8f57} має сенс, оскільки
\begin{equation*}
v\in(\mathcal{S}(\mathbb{R}^{n-1})^r\;\Longrightarrow\;
T_{0}v\in\mathcal{S}(\mathbb{R}^{n}).
\end{equation*}
Для довільних номера $j\in\{0,\dots,r-1\}$, вектор-функції
\begin{equation}\label{proof-theorem1.12-v}
v=(v_0,\dots,v_{r-1})\in(S(\mathbb{R}^{n-1}))^r
\end{equation}
і вектора $\xi\in\mathbb{R}^{n-1}$ маємо такі рівності:
\begin{align*}
\mathcal{F}\bigl[\partial^j_tT_0v\!\mid_{t=0}\bigr](\xi)&=
\partial^j_{t}\mathcal{F}_{x\mapsto\xi}[T_0v](\xi,t)\big|_{t=0}=\\
&=\partial^j_t\biggl(
\beta\bigl(\langle\xi\rangle^{2b}t\bigr)\,\sum_{k=0}^{r-1}
\frac{1}{k!}\,\widehat{v_k}(\xi)\,t^k\biggr)\bigg|_{t=0}=\\
&=\biggl(\partial^j_t\sum_{k=0}^{r-1}
\frac{1}{k!}\,\widehat{v_k}(\xi)\,t^k\biggr)\!\bigg|_{t=0}=
\widehat{v_j}(\xi).
\end{align*}
У третій з них використано той факт, що $\beta=1$ в околі нуля.
Отже, перетворення Фур'є всіх компонент векторів $R_ {0}T_{0}v$ і $v$ збігаються, що еквівалентно \eqref{8f57}.

Доведемо тепер, що звуження відображення \eqref{8f58-def} є обмеженим оператором на парі просторів
\begin{equation*}
\mathbb{H}^{2bm}(\mathbb{R}^{n-1})=
\bigoplus_{k=0}^{r-1}H^{2bm-2bk-b}(\mathbb{R}^{n-1})\quad\mbox{і}\quad
H^{2bm,m}(\mathbb{R}^{n})
\end{equation*}
для кожного цілого числа $m\geq0$. Зауважимо, що тут показники $2bm-2bk-b$ можуть бути від'ємними. Скористаємося тим фактом, що норма в просторі
$H^{2bm,m}(\mathbb{R}^{n})$ еквівалентна нормі
\begin{equation*}
\|w\|_{2bm,m}:=\biggl(\|w\|^2+
\sum_{j=1}^{n-1}\|\partial_{x_j}^{2bm}w\|^2+
\|\partial_{t}^{m}w\|^2\biggr)^{1/2}
\end{equation*}
(див., наприклад, \cite[п.~9.1]{BesovIlinNikolskii75}). Тут і нижче у цьому доведенні \nobreak{$\|\cdot\|$} позначає норму у гільбертовому просторі $L_2(\mathbb{R}^{n})$. Як звичайно, $\partial_{x_j}$ і $\partial_{t}$ позначають оператори узагальнених частинних похідних відповідно за змінними $x_j$ та $t$. Для довільної вектор-функції \eqref{proof-theorem1.12-v} виконуються такі рівності:
\begin{align*}
\|T_0v\|^2_{2bm,m}&=\|T_0v\|^2+
\sum_{j=1}^{n-1}\|\partial_{x_j}^{2bm}\,T_0v\|^2+
\|\partial_{t}^{m}\,T_0v\|^2=\\
&=\|\widehat{T_0v}\|^2+
\sum_{j=1}^{n-1}\|\xi_j^{2bm}\,\widehat{T_0v}\|^2+
\|\partial_{t}^{m}\,\widehat{T_0v}\|^2\leq\\
&\leq\sum_{k=0}^{r-1}\frac{1}{k!}\,
\int\limits_{\mathbb{R}^{n}}
\bigl|\beta(\langle\xi\rangle^{2b}t)\,\widehat{v_k}(\xi)\,t^k\bigr|^2
d\xi dt+\\
&+\sum_{j=1}^{n-1}\sum_{k=0}^{r-1}\frac{1}{k!}\,
\int\limits_{\mathbb{R}^{n}}
\bigl|\xi_j^{2bm}\,\beta(\langle\xi\rangle^{2b}t)\,\widehat{v_k}(\xi)\,
t^k\bigr|^2d\xi dt+\\
&+\sum_{k=0}^{r-1}\frac{1}{k!}\,
\int\limits_{\mathbb{R}^{n}}
\bigl|\partial_{t}^{m}\bigl(\beta(\langle\xi\rangle^{2b}t)\,t^k\bigr)\,
\widehat{v_k}(\xi)\bigr|^2d\xi dt.
\end{align*}
Друга з них правильна на підставі рівності Парсеваля.

Оцінимо окремо кожний з трьох інтегралів. Почнемо з останнього.
Зробивши заміну змінної  $\tau=\langle\xi\rangle^{2b}t$ у внутрішньому інтегралі за $t$, отримаємо такі рівності:
\begin{align*}
&\int\limits_{\mathbb{R}^{n}}
\bigl|\partial_{t}^{m}\bigl(\beta(\langle\xi\rangle^{2b}t)\,t^k\bigr)\,
\widehat{v_k}(\xi)\bigr|^2d\xi dt=
\\&=\int\limits_{\mathbb{R}^{n-1}}
|\widehat{v_k}(\xi)|^2d\xi \int\limits_{\mathbb{R}}
|\partial_{t}^{m}(\beta(\langle\xi\rangle^{2b}t)t^k)|^2 dt=\\
&=\int\limits_{\mathbb{R}^{n-1}}
\langle\xi\rangle^{4bm-4bk-2b}\,|\widehat{v_k}(\xi)|^2 d\xi \int\limits_{\mathbb{R}}
|\partial_{\tau}^{m}(\beta(\tau)\tau^{k})|^2 d\tau.
\end{align*}
Тому
$$
\int\limits_{\mathbb{R}^{n}}
\bigl|\partial_{t}^{m}\bigl(\beta(\langle\xi\rangle^{2b}t)\,t^k\bigr)\,
\widehat{v_k}(\xi)\bigr|^2d\xi dt=
c_1\,\|v_k\|^2_{H^{2bm-2bk-b}(\mathbb{R}^{n-1})},
$$
де
$$
c_1:=\int\limits_{\mathbb{R}}
|\partial_{\tau}^{m}(\beta(\tau)\tau^{k})|^2 d\tau<\infty.
$$

Зробивши таку саму заміну змінної $t$ у другому інтегралі, отримаємо рівності
\begin{align*}
&\int\limits_{\mathbb{R}^{n}}
\bigl|\xi_j^{2bm}\,\beta(\langle\xi\rangle^{2b}t)\,\widehat{v_k}(\xi)\,
t^k\bigr|^2d\xi dt=
\\&=\int\limits_{\mathbb{R}^{n-1}}
|\xi_j|^{4bm}|\widehat{v_k}(\xi)|^2d\xi \int\limits_{\mathbb{R}}
|t^{k}\,\beta(\langle\xi\rangle^{2b}t)|^2dt=\\
&=\int\limits_{\mathbb{R}^{n-1}}
|\xi_j|^{4bm}\langle\xi\rangle^{-4bk-2b}\,|\widehat{v_k}(\xi)|^2d\xi \int\limits_{\mathbb{R}}|\tau^{k}\beta(\tau)|^2d\tau\leq\\
&\leq\int\limits_{\mathbb{R}^{n-1}}
\langle\xi\rangle^{4bm-4bk-2b}\,|\widehat{v_k}(\xi)|^2d\xi
\int\limits_{\mathbb{R}}|\tau^{k}\beta(\tau)|^2d\tau.
\end{align*}
Отже,
$$
\int\limits_{\mathbb{R}^{n}}
\bigl|\xi_j^{2bm}\,\beta(\langle\xi\rangle^{2b}t)\,\widehat{v_k}(\xi)\,
t^k\bigr|^2d\xi dt
\leq c_2\,\|v_k\|^2_{H^{2bm-2bk-b}(\mathbb{R}^{n-1})},
$$
де
$$
c_2:=\int\limits_{\mathbb{R}}|\tau^{k}\beta(\tau)|^2d\tau<\infty.
$$

Нарешті, замінивши вираз $\xi_j$ на число $1$ у попередніх міркуваннях,
отримаємо таку оцінку для першого інтеграла:
\begin{align*}
&\int\limits_{\mathbb{R}^{n}}
\bigl|\beta(\langle\xi\rangle^{2b}t)\,\widehat{v_k}(\xi)\,
t^k\bigr|^2d\xi dt\leq
\\&\leq c_2\,\|v_k\|^2_{H^{-2bk-b}(\mathbb{R}^{n-1})}
\leq c_2\,\|v_k\|^2_{H^{2bm-2bk-b}(\mathbb{R}^{n-1})}.
\end{align*}

Отже,
\begin{align*}
\|T_0v\|_{H^{2bm,m}(\mathbb{R}^{n})}^{2}\leq& \,c\,\sum_{k=0}^{r-1}
\|v_k\|^2_{H^{2bm-2bk-b}(\mathbb{R}^{n-1})}=\\
&=c\,\|v\|_{\mathbb{H}^{2bm}(\mathbb{R}^{n-1})}^{2}
\end{align*}
для довільного $v\in(\mathcal{S}(\mathbb{R}^{n-1}))^r$, де $c$~--- деяке додатне число, незалежне від $v$. Оскільки множина $\bigl(S(\mathbb{R}^{n-1})\bigr)^r$ є щільною у просторі $\mathbb{H}^{2bm}(\mathbb{R}^{n-1})$, то з отриманої оцінки випливає, що відображення \eqref{8f58-def} є лінійним обмеженим оператором
\begin{equation*}
T_0:\mathbb{H}^{2bm}(\mathbb{R}^{n-1})\to H^{2bm,m}(\mathbb{R}^{n}),
\quad\mbox{якщо}\quad 0\leq m\in\mathbb{Z}.
\end{equation*}

Виведемо звідси, що відображення \eqref{8f58-def} є обмеженим оператором на парі просторів
$\mathbb{H}^{s;\varphi}(\mathbb{R}^{n-1})$ і $H^{s,s/(2b);\varphi}(\mathbb{R}^{n})$. Нагадаємо, що $s>2br-b$ і $\varphi\in\mathcal{M}$. Покладемо $s_0:=0$, виберемо ціле число $s_1>s$ таке, що $m:=s_1/(2b)\in\mathbb{N}$, і розглянемо обмежені оператори
\begin{equation}\label{8f66}
T_{0}:\mathbb{H}^{s_j}(\mathbb{R}^{n-1})\to H^{s_j,s_j/(2b)}(\mathbb{R}^{n}),
\end{equation}
де $j\in\{0,1\}$. Нехай, як і раніше у цьому доведенні, функція $\psi$ є інтерполяційним параметром \eqref{9f7.2}. Застосувавши до цих операторів інтерполяцію з параметром $\psi$, отримаємо обмежений оператор
\begin{equation}\label{8f48}
T_0:\mathbb{H}^{s;\varphi}(\mathbb{R}^{n-1})\to
H^{s,s/(2b);\varphi}(\mathbb{R}^{n})
\end{equation}
на підставі теореми \ref{8prop4} і з огляду на формулу \eqref{8f-intHH}.

Рівність \eqref{8f57} продовжується за неперервністю на кожний елемент $v\in\mathbb{H}^{s;\varphi}(\mathbb{R}^{n-1})$. Отже, оператор \eqref{8f48} є правим оберненим до \eqref{8f63}. Таким чином, побудовано потрібне відображення \eqref{8f58}.

Введемо аналоги операторів \eqref{8f63} і \eqref{8f48} для смуги
$$
\Pi=\bigl\{(x,t):x\in\mathbb{R}^{n-1},0<t<\tau\bigr\}.
$$
Для довільного $u\in H^{s,s/(2b);\varphi}(\Pi)$ покладемо
$R_{1}u:=R_{0}w$, де функція $w\in H^{s,s/(2b);\varphi}(\mathbb{R}^{n})$ задовольняє умову $w\!\upharpoonright\!\Pi=u$. При цьому $R_{1}u$ не залежить від вказаної функції~$w$. Лінійне відображення
$u\mapsto R_{1}u$ є обмеженим оператором
\begin{equation}\label{8f67}
R_{1}:H^{s,s/(2b);\varphi}(\Pi)\to\mathbb{H}^{s;\varphi}(\mathbb{R}^{n-1}).
\end{equation}
Це випливає з обмеженості оператора \eqref{8f63} і означення норми у просторі $H^{s,s/(2b);\varphi}(\Pi)$.

Введемо правий обернений оператор до \eqref{8f67} за допомогою  відображення \eqref{8f58-def}. Покладемо $T_{1}v:=(T_0v)\!\upharpoonright\!\Pi$ для довільної вектор-функції $v\in(L_{2}(\mathbb{R}^{n-1}))^{r}$. Лінійне відображення $v\mapsto T_{1}v$ є обмеженим оператором
\begin{equation}\label{8f51}
T_1:\mathbb{H}^{s;\varphi}(\mathbb{R}^{n-1})\to H^{s,s/(2b);\varphi}(\Pi).
\end{equation}
Це випливає з обмеженості оператора \eqref{8f48}. Помічаємо, що
\begin{equation*}
R_1T_1v=R_1\bigl((T_0v)\!\upharpoonright\!\Pi\bigr)=R_0T_0v=v\\
\end{equation*}
для довільної функції $v\in\mathbb{H}^{s;\varphi}(\mathbb{R}^{n-1})$.  Отже, оператор \eqref{8f51} є правим оберненим до \eqref{8f67}.

Використовуючи оператори \eqref{8f67} і \eqref{8f51}, доведемо теорему~\ref{8lem1} за допомогою спеціальних локальних карт \eqref{8f-local} на поверхні $S$ і відповідного розбиття одиниці. Для довільної функції $g\in C^{\infty}(\overline{S})$ правильні такі співвідношення:
\begin{align*}
\|Rg\|_{\mathbb{H}^{s;\varphi}(\Gamma)}^{2}
&=\sum_{k=0}^{r-1}\|\partial^{k}_{t}g\!\upharpoonright\!\Gamma\|_
{H^{s-2bk-b;\varphi}(\Gamma)}^{2}=\\
&=\sum_{k=0}^{r-1}\sum_{j=1}^{\lambda}
\|(\chi_{j}(\partial^{k}_{t}g\!\upharpoonright\!\Gamma))
\circ\theta_{j}\|_{H^{s-2bk-b;\varphi}(\mathbb{R}^{n-1})}^{2}=\\
&=\sum_{j=1}^{\lambda}\sum_{k=0}^{r-1}
\|\partial^{k}_{t}((\chi_{j}\,g)\circ\theta^{\ast}_{j})
\!\upharpoonright\!\mathbb{R}^{n-1})\|_
{H^{s-2bk-b;\varphi}(\mathbb{R}^{n-1})}^{2}\leq\\
&\leq c^{2}\,\sum_{j=1}^{\lambda}
\|(\chi_{j}\,g)\circ\theta^{\ast}_{j}\|_{H^{s,s/(2b);\varphi}(\Pi)}^{2}=\\
&=c^{2}\,\|g\|_{H^{s,s/(2b);\varphi}(S)}^{2}.
\end{align*}
Тут $c$~--- норма обмеженого оператора \eqref{8f67}, а
$\{\theta_{j}\}$ і $\{\chi_{j}\}$~--- набір локальних карт на $\Gamma$ і відповідне розбиття одиниці, використані в означенні простору $H^{\sigma;\varphi}(\Gamma)$, де $\sigma\in\mathbb{R}$. Таким чином,
\begin{equation*}
\|Rg\|_{\mathbb{H}^{s;\varphi}(\Gamma)}\leq c\,\|g\|_{H^{s,s/(2b);\varphi}(S)}
\quad\mbox{для усіх}\quad g\in C^{\infty}(\overline{S}).
\end{equation*}
Отже, відображення \eqref{8f25} продовжується за неперервністю до лінійного обмеженого оператора \eqref{8f65}.

Побудуємо лінійне відображення $T:(L_2(\Gamma))^r\to L_2(S)$, звуження якого на простір $\mathbb{H}^{s;\varphi}(\Gamma)$ є правим оберненим оператором до \eqref{8f65}. Подібно до \eqref{L} означимо лінійне відображення розпрямлення многовиду $\Gamma$ за формулою
\begin{equation*}
L_0:\omega\mapsto\bigl((\chi_{1}\omega)\circ\theta_{1},\ldots,
(\chi_{\lambda}\omega)\circ\theta_{\lambda}\bigr),
\end{equation*}
де $\omega\in L_2(\Gamma)$. Звуження цього відображення на вказаний простір є ізометричним оператором
\begin{equation}\label{8f52}
L_0:H^{\sigma;\varphi}(\Gamma)\rightarrow
\bigl(H^{\sigma;\varphi}(\mathbb{R}^{n-1})\bigr)^{\lambda}
\end{equation}
для кожного $\sigma>0$. Подібно до \eqref{K} означимо лінійне відображення склеювання многовиду $\Gamma$ за формулою
\begin{equation*}
K_0:(\omega_{1},\ldots,\omega_{\lambda})\mapsto\sum_{k=1}^{\lambda}\,
O_{k}\bigl((\eta_{k}\omega_{k})\circ\theta_{k}^{-1}\bigr),
\end{equation*}
де $\omega_{1},\ldots,\omega_{\lambda}\in L_2(\mathbb{R}^{n-1})$. Тут кожну функцію $\eta_{k}\in
C_{0}^{\infty}(\mathbb{R}^{n-1})$ вибрано так, що $\eta_{k}=1$ на
множині $\theta^{-1}_{k}(\mathrm{supp}\,\chi_{k})$, а $O_{k}$ позначає оператор продовження нулем на $\Gamma$ функції, заданої на $\Gamma_k$. Оператор $K_0$ є обмеженим на парі просторів
\begin{equation*}
K_0:\bigl(H^{\sigma;\varphi}(\mathbb{R}^{n-1})\bigr)^{\lambda}\to
H^{\sigma;\varphi}(\Gamma)
\end{equation*}
для кожного $\sigma>0$. Крім того, він є лівим оберненим до \eqref{8f52}, оскільки $K_0L_0\omega=w$ для довільної функції $\omega\in L_2(\Gamma)$. Ця рівність доводиться так само, як і формула \eqref{9f7.22}.

Відображення $K_0$ і оператор \eqref{K} склеювання многовиду $S=\Gamma\times(0,\tau)$ пов'язані рівністю
\begin{equation*}
\bigl(K(u_1,\dots,u_\lambda)\bigr)(x,t)=
\bigl(K_0(u_1(\cdot,t),\ldots,u_\lambda(\cdot,t))\bigr)(x)
\end{equation*}
для довільних функцій $u_1,\dots,u_\lambda\in L_2(\Pi)$ та майже всіх $x\in\Gamma$ і $t\in(0,\tau)$. Лінійний оператор $K$ обмежений на парі просторів $(L_2(\Pi))^{\lambda}$ і $L_2(S)$. Крім того, він є обмеженим на парі просторів
\begin{equation}\label{8f53}
K:(H^{\sigma,\sigma/(2b);\varphi}(\Pi))^{\lambda}\to
H^{\sigma,\sigma/(2b);\varphi}(S)
\end{equation}
для кожного $\sigma>0$. Це обґрунтовується так само, як і обмеженість оператора \eqref{9f7.23}.

Для довільної вектор-функції
$$
v:=(v_0,v_1,\dots,v_{r-1})\in(L_{2}(\Gamma))^{r}
$$
покладемо
\begin{equation*}
Tv:=K\bigl(T_{1}(v_{0,1},\ldots,v_{r-1,1}),\ldots,
T_{1}(v_{0,\lambda},\ldots,v_{r-1,\lambda})\bigr),
\end{equation*}
де
$$
(v_{k,1},\ldots,v_{k,\lambda}):=
L_{0}v_{k}\in(L_{2}(\mathbb{R}^{n-1}))^{\lambda}
$$
для кожного цілого числа $k\in\{0,\ldots,r-1\}$. Лінійний оператор $v\mapsto Tv$ обмежений на парі просторів $(L_{2}(\Gamma))^{r}$ і $L_{2}(S)$, що  випливає з означень відображень $L_{0}$, $T$ і~$K$. Крім того, він обмежений на парі просторів \eqref{8f43}, що є наслідком обмеженості операторів \eqref{8f51}--\eqref{8f53}.

Щойно введений оператор  \eqref{8f43} є правим оберненим до \eqref{8f65}. Справді, для довільної  вектор-функції
$$
v=(v_0,v_1,\dots,v_{r-1})\in\mathbb{H}^{s;\varphi}(\Gamma)
$$
виконуються такі рівності:
\begin{align*}
(RTv)_k&=\bigl(RK\bigl(T_1(v_{0,1},\ldots,v_{r-1,1}),\ldots,
T_1(v_{0,\lambda},\ldots,v_{r-1,\lambda})\bigr)\bigr)_k=\\
&=K_0\bigl(\bigl(R_1T_1(v_{0,1},\ldots,v_{r-1,1})\bigr)_k,\ldots\\
&\quad\qquad\ldots,\bigl(R_1T_1(v_{0,\lambda},\ldots,
v_{r-1,\lambda})\bigr)_k\bigr)=\\
&=K_0(v_{k,1},\dots,v_{k,\lambda})=K_{0}L_{0}v_k=v_k.
\end{align*}
Тут індекс $k$ пробігає множину $\{0,\dots,r-1\}$ і позначає $k$-ту компоненту вектор-функції. Таким чином, $RTv=v$. Потрібний оператор $T$ побудовано.
\end{proof}

\markright{\emph \ref{sec2.8}. Теорема вкладення}

\section[Теорема вкладення]{Теорема вкладення}\label{sec2.8}

\markright{\emph \ref{sec2.8}. Теорема вкладення}

Для просторів Хермандера важливу роль відіграє теорема \cite[теорема~2.2.7]{Hormander63} про необхідну і достатню умову їх вкладення у простір $k\geq0$ разів неперервно диференційовних функцій. Під час дослідження умов регулярності розв'язків параболічних задач нам знадобиться версія цієї теореми для узагальнених просторів Соболєва $H^{s,s/(2b);\varphi}(\mathbb{R}^{n+1})$, де $b\in\mathbb{N}$.

Використовуємо такі короткі позначення операторів диференціювання за змінними $x=(x_1,\ldots,x_n)$ і~$t$:
$$
\partial^{\alpha}_{x}:=\frac{\partial^{|\alpha|}}
{\partial_{x_1}^{\alpha_1}\ldots\partial_{x_n}^{\alpha_n}}
\quad\mbox{і}\quad
\partial_t^{\beta}:=\frac{\partial^{\beta}}{\partial t^\beta}.
$$
Тут  $\alpha=(\alpha_1,...,\alpha_n)$~--- мультиіндекс, $\nobreak{|\alpha|:=\alpha_1+\cdots+\alpha_n}$, а числа
$\alpha_1,...,\alpha_n$ і $\beta$ цілі й невід'ємні.
Крім того, $\xi^{\alpha}:=\xi_{1}^{\alpha_{1}}\ldots\xi_{n}^{\alpha_{n}}$ для вектора $\xi:=(\xi_{1},\ldots,\xi_{n})\in\mathbb{C}^{n}$.

\begin{theorem}\label{9lem8.1}
Нехай $0\leq p\in\mathbb{Z}$, $b,n\in\mathbb{N}$, $s:=p+b+n/2$ та $\varphi\in\mathcal{M}$. Тоді правильні такі два твердження:
\begin{itemize}
\item[$\mathrm{(i)}$] Якщо $\varphi$ задовольняє умову
    \begin{equation}\label{9f4.7}
    \int\limits_{1}^{\,\infty}\;\frac{dr}{r\,\varphi^2(r)}<\infty,
    \end{equation}
    то будь-яка функція $w(x,t)$ з простору $H^{s,s/(2b);\varphi}(\mathbb{R}^{n+1})$ має таку властивість: вона і всі її узагальнені частинні похідні
    $\partial_{x}^{\alpha}\partial_{t}^{\beta}w$, де $|\alpha|+2b\beta\leq p$, неперервні на $\mathbb{R}^{n+1}$.
\item[$\mathrm{(ii)}$] Нехай $V$~--- непорожня відкрита підмножина простору $\mathbb{R}^{n+1}$ і нехай $k\in\mathbb{N}$ та $k\leq n$. Якщо для довільної функції $w\in H^{s,s/(2b);\varphi}(\mathbb{R}^{n+1})$ такої, що $\mathrm{supp}\,w\subset V$, її узагальнена частинна похідна $\partial^{p}w/\partial x_{k}^{p}$ неперервна на $C(\mathbb{R}^{n+1})$, то $\varphi$ задовольняє умову \eqref{9f4.7}.
\end{itemize}
\end{theorem}

\begin{proof}[\indent Доведення] Обґрунтуємо твердження~(i). Припустимо, що $w\in H^{s,s/(2b);\varphi}(\mathbb{R}^{n+1})$ і $0\leq|\alpha|+2b\beta\leq p$. Для більшої лаконічності формул використовуємо позначення $w_{\alpha,\beta}(x,t):=\partial_{x}^{\alpha}\partial_{t}^{\beta}w(x,t)$, $\gamma:=1/(2b)$ і $r_{\gamma}(\xi,\eta):=\bigl(1+|\xi|^2+|\eta|^{2\gamma}\bigr)^{1/2}$, де $\xi\in\mathbb{R}^{n}$ і $\eta\in\mathbb{R}$. З~умови
\begin{equation}\label{9f8.9}
\int\limits_{\mathbb{R}^{n}}\int\limits_{\mathbb{R}}
\frac{|\xi^{\alpha}|^{2}\,|\eta|^{2\beta}\,d\xi\,d\eta}
{r_{\gamma}^{2s}(\xi,\eta)\,\varphi^{2}(r_{\gamma}(\xi,\eta))}<\infty
\end{equation}
випливає включення $w_{\alpha,\beta}\in C(\mathbb{R}^{n+1})$. Справді, за нерівністю Шварца маємо оцінку
\begin{align*}
&\int\limits_{\mathbb{R}^{n}}\int\limits_{\mathbb{R}}
|\widehat{w_{\alpha,\beta}}(\xi,\eta)|\,d\xi\,d\eta=
\int\limits_{\mathbb{R}^{n}}\int\limits_{\mathbb{R}}
|\xi^{\alpha}\eta^{\beta}|\cdot|\widehat{w}(\xi,\eta)|\,d\xi\,d\eta
\leq\\
&\leq\|w\|_{H^{s,s/(2b);\varphi}(\mathbb{R}^{n+1})}
\int\limits_{\mathbb{R}^{n}}\int\limits_{\mathbb{R}}
\frac{|\xi^{\alpha}|^{2}\,|\eta|^{2\beta}\,d\xi\,d\eta}
{r_{\gamma}^{2s}(\xi,\eta)\,\varphi^{2}(r_{\gamma}(\xi,\eta))}.
\end{align*}
Отже, якщо $\varphi$ задовольняє умову \eqref{9f8.9}, то перетворення Фур'є $\widehat{w_{\alpha,\beta}}$ функції $w_{\alpha,\beta}$  сумовне на $\mathbb{R}^{n+1}$, а тому сама функція $w_{\alpha,\beta}$ є неперервною на  $\mathbb{R}^{n+1}$.

Залишається показати, що $\eqref{9f4.7}\Rightarrow\eqref{9f8.9}$. У кратному інтегралі, присутньому в формулі \eqref{9f8.9}, зробимо заміну змінної $\eta=\eta_1^{2b}$, перейдемо до сферичних координат $(\varrho,\theta_{1},\ldots,\theta_{n})$, де $\varrho=(|\xi|^2+\eta_1^2)^{1/2}$, і зробимо заміну змінної $r=(1+\varrho^2)^{1/2}$. У~результаті отримаємо такі рівності:
\begin{equation*}
\begin{split}
&\int\limits_{\mathbb{R}^{n}}\int\limits_{\mathbb{R}}
\frac{|\xi^{\alpha}|^{2}\,|\eta|^{2\beta}d\xi\,d\eta}
{r_{\gamma}^{2s}(\xi,\eta)\,\varphi^{2}(r_{\gamma}(\xi,\eta))}=\\
&=2^{n+1}\int\limits_{0}^{\infty}\dots
\int\limits_{0}^{\infty}\int\limits_{0}^{\infty}
\frac{|\xi^{\alpha}|^{2}\,\eta^{2\beta}d\xi\,d\eta}
{r_{\gamma}^{2s}(\xi,\eta)\,\varphi^{2}(r_{\gamma}(\xi,\eta))}=\\
&=2^{n+1}\int\limits_{0}^{\infty}\dots
\int\limits_{0}^{\infty}\int\limits_{0}^{\infty}
\frac{2b\,|\xi^{\alpha}|^{2}\,\eta_1^{4b\beta+2b-1}\,d\xi\,d\eta_1}
{(1+|\xi|^2+\eta_1^2)^{s}
\varphi^2\bigl(\sqrt{1+|\xi|^2+\eta_1^2}\,\bigr)}=\\
&=c_{\alpha,\beta}\int\limits_{0}^{\infty}
\frac{\varrho^{2|\alpha|+4b\beta+2b-1}\,\varrho^{n}\,d\varrho}
{(1+\varrho^2)^{s}\varphi^2(\sqrt{1+\varrho^2}\,)}=\\
&=c_{\alpha,\beta}\int\limits_{1}^{\infty}
\frac{(r^2-1)^{|\alpha|+2b\beta+b-1/2+n/2}\,r\,dr}{r^{2s}\,\varphi^2(r)\,
(r^2-1)^{1/2}}=\\
&=c_{\alpha,\beta}\int\limits_{1}^{\infty}\;
\frac{(r^2-1)^{s-1-\delta(\alpha,\beta)}}{r^{2s-1}\,\varphi^2(r)}\,dr.
\end{split}
\end{equation*}
Тут $c_{\alpha,\beta}$~--- деяке додатне число, а
\begin{equation*}
\delta(\alpha,\beta):=p-|\alpha|-2b\beta\in[0,p].
\end{equation*}
Таким чином,
\begin{equation}\label{9f8.10}
\int\limits_{\mathbb{R}^{n}}\int\limits_{\mathbb{R}}
\frac{|\xi^{\alpha}|^{2}\,|\eta|^{2\beta}\,d\xi\,d\eta}
{r_{\gamma}^{2s}(\xi,\eta)\,\varphi^{2}(r_{\gamma}(\xi,\eta))}=
c_{\alpha,\beta}\int\limits_{1}^{\infty}\;
\frac{(r^2-1)^{s-1-\delta(\alpha,\beta)}}{r^{2s-1}\,\varphi^2(r)}\,dr.
\end{equation}
Помітимо, що
\begin{align}\label{9f8.11}
\eqref{9f4.7}\;&\Longleftrightarrow\;\int\limits_{1}^{\infty}\;
\frac{(r^2-1)^{s-1}}{r^{2s-1}\,\varphi^2(r)}\,dr<\infty\;\Longrightarrow\\
&\Longrightarrow\;\int\limits_{1}^{\infty}\;
\frac{(r^2-1)^{s-1-\delta(\alpha,\beta)}}{r^{2s-1}\,\varphi^2(r)}\,dr
<\infty.\notag
\end{align}
Остання імплікація істинна, оскільки $\delta(\alpha,\beta)\geq0$ та
$$
s-1-\delta(\alpha,\beta)\geq p+b+n/2-1-p\geq0.
$$
Тому з умови \eqref{9f4.7} випливає \eqref{9f8.9} з огляду на \eqref{9f8.10}. Тверджен\-ня~(i) обґрунтоване.

Обґрунтуємо твердження (ii). Нехай виконується зроблене у ньому припущення. Оскільки простір $H^{s,s/(2b);\varphi}(\mathbb{R}^{n+1})$ інваріантний відносно зсувів у $\mathbb{R}^{n+1}$, то можна вважати, не обмежуючи загальності, що $0\in V$. Виберемо функцію $\chi\in C^{\infty}(\mathbb{R}^{n+1})$ таку, що $\mathrm{supp}\,\chi\subset V$ і $\chi(x,t)=1$ в околі точки $(0,0)$. За припущенням лінійний оператор
$w\mapsto(\partial^{p}(\chi w)/\partial x_{k}^{p})\!\upharpoonright\!\overline{V}$ діє з усього простору $H^{s,s/(2b);\varphi}(\mathbb{R}^{n+1})$ у простір $C(\overline{V})$. Цей оператор замкнений, оскільки він є звуженням неперервного оператора на парі лінійних топологічних просторів $\mathcal{S}'(\mathbb{R}^{n+1})$ і $\mathcal{D}'(V)$. Тому за теоремою про замкнений графік перший оператор є обмеженим; нехай $c$~--- його норма. Звідси для довільної функції $w\in\mathcal{S}(\mathbb{R}^{n+1})$ випливає оцінка
\begin{align*}
&\biggl|(2\pi)^{-n-1}\int\limits_{\mathbb{R}^{n}}\int\limits_{\mathbb{R}}
\xi_{k}^{p}\,\widehat{w}(\xi,\eta)\,d\xi\,d\eta\biggr|=
\biggl|\frac{\partial^{p}w}{\partial x_{k}^{p}}(0,0)\biggr|\leq\\
&\leq\biggl\|\frac{\partial^{p}(\chi w)}{\partial x_{k}^{p}}\biggr\|
_{C(\overline{V})}\leq c\cdot\|w\|_{H^{s,s/(2b);\varphi}(\mathbb{R}^{n+1})}.
\end{align*}
Отже,
\begin{equation}\label{bound-C(V)-new}
\biggl|\,\int\limits_{\mathbb{R}^{n}}\int\limits_{\mathbb{R}}
\xi_{k}^{p}\,\widehat{w}(\xi,\eta)\,d\xi\,d\eta\biggr|\leq
c_{1}\|\mu\widehat{w}\|,
\end{equation}
де $c_{1}:=(2\pi)^{n+1}c$ і $\mu(\xi,\eta):=r_{\gamma}^{s}(\xi,\eta)\,\varphi(r_{\gamma}(\xi,\eta))$ для довільних $\xi\in\mathbb{R}^{n}$ і $\eta\in\mathbb{R}$. Тут і далі у доведенні $\|\cdot\|$ і $(\cdot,\cdot)$ позначають норму і скалярний добуток у гільбертовому просторі ${L_{2}(\mathbb{R}^{n+1})}$. У ньому є щільною множина $\{\mu\widehat{w}:w\in\mathcal{S}(\mathbb{R}^{n+1})\}$. Це випливає із щільності множини $\mathcal{S}(\mathbb{R}^{n+1})$ у просторі $H^{\mu}(\mathbb{R}^{n+1})=H^{s,s/(2b);\varphi}(\mathbb{R}^{n+1})$. Для зручності, покладемо $\lambda(\xi,\eta):=\xi_{k}^{p}$ для всіх $\xi=(\xi_{1},\ldots,\xi_{n})\in\mathbb{R}^{n}$ і $\eta\in\mathbb{R}$. На підставі \eqref{bound-C(V)-new} маємо оцінку
\begin{equation*}
\|\lambda/\mu\|=
\sup_{w\in\mathcal{S}(\mathbb{R}^{n+1})}
\frac{|(\mu\widehat{w},\lambda/\mu)|}{\|\mu\widehat{w}\|}\leq c_{1},
\end{equation*}
тобто
\begin{equation*}
\int\limits_{\mathbb{R}^{n}}\int\limits_{\mathbb{R}}
\frac{|\xi_{k}^{p}|^{2}\,d\xi\,d\eta}
{r_{\gamma}^{2s}(\xi,\eta)\,\varphi^{2}(r_{\gamma}(\xi,\eta))}\leq c_{1}^{2}<\infty.
\end{equation*}
З цієї властивості та рівності \eqref{9f8.10} випливає, що
\begin{equation}\label{9f8.12}
\int\limits_{1}^{\infty}\;\frac{(r^2-1)^{s-1}}
{r^{2s-1}\,\varphi^2(r)}\,dr<\infty.
\end{equation}
Тут \eqref{9f8.10} застосовано у випадку, коли $\alpha_{k}=p$ і $\alpha_{q}=0$, якщо $q\neq k$, та $\beta=0$ (тоді $\delta(\alpha,\beta)=0$). Тепер \eqref{9f4.7} випливає з \eqref{9f8.11} і \eqref{9f8.12}. Твердження~(ii) обґрунтоване.
\end{proof}

Стосовно теореми \ref{9lem8.1} зауважимо, що використання функціонального показника регулярності $\varphi$ для простору $H^{s,s/(2b);\varphi}(\mathbb{R}^{n+1})$ дало змогу досягти точного значення $p+b+n/2$ числового показника $s$ в умовах тверджень (i) та (ii) цієї теореми. На класі соболєвських просторів $H^{s,s/(2b)}(\mathbb{R}^{n+1})$ це неможливо, оскільки функція $\varphi(\cdot)\equiv1$ не задовольняє інтегральну умову \eqref{9f4.7}. Для таких просторів замість останньої доводиться використовувати більш сильну умову $s>p+b+n/2$.

\markright{\emph \ref{rem-ch1}. Бібліографічні коментарі до розд.~1}

\section[Бібліографічні коментарі до розд.~1]{Бібліографічні коментарі до розд.~1}\label{rem-ch1}

\markright{\emph \ref{rem-ch1}. Бібліографічні коментарі до розд.~1}

\small

Розглянутий у п.~\ref{sec2.1} метод квадратичної інтерполяції з функціональним параметром є природним узагальненням класичного методу інтерполяції пар гільбертових просторів, уведеного незалежно Ж.-Л.~Ліонсом (J.-L.~Lions) \cite{Lions58} і С.~Г.~Крейном \cite{Krein60a} (див. також їх монографії \cite[розд.~1, пп.~2 і~5]{LionsMagenes72i} і \cite[розд.~4, п.~1.10]{KreinPetuninSemenov82}). У~класичному методі інтерполяційним параметром є число $\theta\in(0,1)$, а інтерполяційні простори є областями визначення степеневих функцій $\psi(t)\equiv t^{\theta}$ від породжуючого оператора регулярної пари гільбертових просторів. За допомогою цього методу будується гільбертова шкала просторів, яка з'єднує гільбертові простори, що утворюють регулярну пару (див., наприклад, довідник \cite[розд.~IV, \S~9]{FunctionalAnalysis72}).

Квадратична інтерполяція з (досить загальним) функціональним параметром уперше з'явилася в статті Ч.~Фояша (C. Foia\c{s}) і Ж.-Л.~Ліонса \cite[с.~278]{FoiasLions61}. Будучи застосованою до сумісної пари гільбертових просторів, вона дає також гільбертів простір, чим зумовлена і назва цієї інтерполяції. Квадратична інтерполяція та пов'язані з нею інтерполяційні властивості гільбертових шкал вивчалися у працях С.~Г.~Крейна і Ю.~І.~Пєтуніна \cite[\S~9]{KreinPetunin66}, В.~Ф.~Донохью (W.~F.~Donoghue ) \cite{Donoghue67}, Ж.-Л.~Ліонса і Е.~Мадженеса (E.~Magenes) \cite[розд.~1]{LionsMagenes72i}, Г.~Шлензак \cite{Shlenzak74}, Є.~І.~Пустильніка \cite{Pustylnik82},
В.~І.~Овчиннікова \cite[\S~11]{Ovchinnikov84}, J.~L\"ofstr\"om \cite{Lofstrom91}, Y.~Ameur \cite{Ameur04, Ameur19}, В.~А.~Михайлеця і О.~О.~Мурача \cite{MikhailetsMurach08MFAT1, MikhailetsMurach10, MikhailetsMurach13UMJ3, MikhailetsMurach21arxiv02}, M.~Fan \cite{Fan11}, Б.~Сімона (B.~Simon) \cite[розд.~15 і~30]{Simon19}. Квадратична інтерполяція з функціональним параметром тісно пов'язана \cite[п.~4]{MikhailetsMurach21arxiv02} з методом змінних гільбертових шкал, запропонованим М.~Хегландом (M.~Hegland) \cite{Hegland92, Hegland95}. Цей метод виявився досить корисним у теорії некоректних задач; див., наприклад, статті П.~Мате (P.~Mathe) і С.~В.~Пєрєвєрзєва \cite{MathePereverzev03}, П.~Мате і У.~Тотенхана (U.~Tautenhahn) \cite{MatheTautenhahn06}, М.~Хегланда \cite{Hegland10}, М.~Хегланда і Р.~С.~Андерсена (R.~S.~Anderssen) \cite{HeglandAnderssen11}, М.~Хегланда і Б.~Хофмана (B.~Hofmann) \cite{HeglandHofmann11}, К.~Джін (Q.~Jin) і У.~Тотенхана \cite{JinTautenhahn11}.

Метод Ліонса і Крейна інтерполяції гільбертових просторів є історично першим методом інтерполяції досить загальних пар просторів. Його було поширено на нормовані простори у вигляді різних еквівалентних методів дійсної інтерполяції у працях Ж.-Л.~Ліонса і Я.~Петре (J.~Peetre) \cite{Lions59, LionsPeetre61, Peetre63} та методу комплексної або голоморфної інтерполяції у працях С.~Г.~Крейна \cite{Krein60b}, Ж.-Л.~Ліонса \cite{Lions60} і А.~П.~Кальдерона (A.~P.~Calderon) \cite{Calderon64}. У цих методах скінченні набори чисел є параметрами інтерполяції. Взагалі кажучи, дійсний і комплексний методи інтерполяції дають різні інтерполяційні простори. Принципи побудови загальних інтерполяційних методів розроблені Е.~Гальярдо (E.~Gagliardo) \cite{Gagliardo68}. Теорія інтерполяції нормованих та більш загальних лінійних топологічних просторів викладена у монографіях К.~Бенета (C.~Bennet) і Р.~Шарплі (R.~Sharpley) \cite{BennetSharpley88}, Й.~Берга (J.~Bergh) і Й.~Льофстрьома (J.~L\"ofstr\"om) \cite{BerghLefstrem76}, Ю.~А.~Брудного і Н.~Я.~Кругляка \cite{BrudnyiKrugljak91}, С.~Г.~Крейна, Ю.~І.~Пєтуніна і Є.~М.~Семьонова \cite{KreinPetuninSemenov82}, В.~І.~Овчиннікова \cite{Ovchinnikov84}, Л.~Тартара (L.~Tartar) \cite{Tartar07}, Ґ.~Трібеля (H.~Triebel) \cite{Triebel80}.

Інтерполяцію з функціональним параметром нормованих просторів уперше розглядали в статті Ч.~Фояша і Ж.-Л.~Ліонса \cite{FoiasLions61}. Різні методи дійсної інтерполяції з функціональним параметром уведено та досліджено у працях Т.~Ф.~Калугіної \cite{Kalugina75}, Й.~Густавсона (J.~Gustavsson) \cite{Gustavsson78}, С.~Янсона (S.~Janson) \cite{Janson81}, К.~Меручі (C.~Merucci) \cite{Merucci82}, $\nobreak{\mbox{Л.-Е.}}$~Персона (L.-E.~Persson) \cite{Persson86}, Н.~Я.~Кругляка \cite{Krugljak93}, В.~І.~Овчиннікова \cite{Ovchinnikov05, Obchinnikov14MatSb, Obchinnikov14UspMatNauk}, L.~Loosveldt і S.~Nicolay \cite{LoosveldtNicolay19}, а комплексної інтерполяції з функціональним параметром~--- у працях М.~Шехтера (M.~Schechter) \cite{Schechter67}, M.~J.~Carro і J.~Cerd\`a \cite{CarroCerda90}, M.~Fan і S.~Kaijser \cite{Fan94, FanKaijser94}. Окрема увага приділялася інтерполяції з логарифмічним параметром; див., наприклад, статті Р.~Я.~Докторського \cite{Doktorskii91}, В.~Д.~Еванса (W.~D.~Evans) і Б.~Опіка (B.~Opic) \cite{EvansOpic00, EvansOpic02}, B.~F.~Besoy і F.~Cobos \cite{BesoyCobos18}.

Різні методи інтерполяції нормованих просторів застосовують у теорії функціональних просторів, теорії операторів, багатовимірних крайових задач, теорії апроксимації (див. монографії Г.~Амана (H.~Amann) \cite{Amann95, Amann19}, Й.~Берга і Й.~Льофстрьома \cite{BerghLefstrem76}, П.~Л.~Бутцера (P.~L.~Butzer) і Х.~Беренса (H.~Behrens) \cite{ButzerBehrens67}, Ж.-Л.~Ліонса і Е.~Мадженеса \cite{LionsMagenes72i, LionsMagenes72ii}, А.~Лунарді (A.~Lunardi) \cite{Lunardi95},
В.~А.~Михайлеця і О.~О.~Мурача \cite{MikhailetsMurach10, MikhailetsMurach14}, Л.~Тартара \cite{Tartar07}, Ґ.~Трібеля \cite{Triebel80, Triebel86} й наведену там літературу). Стаття W.~Yuan, W.~Sickel і D.~Yang \cite{YuanSickelYang15} містить сучасний огляд застосувань різних методів інтерполяції до просторів розподілів. Чимало  прикладів застосувань інтерполяції просторів до рівнянь з частинними похідними наведено у статті Л.~Малігранди (L.~Maligranda), Л.-Е.~Персона і J.~Wyller \cite{MaligrandaPerssonWyller94}.

У нашій монографії систематично використовується квадратична інтерполяція з параметром, який є правильно змінною функцією на нескінченності порядку $\theta\in(0,1)$. Поняття правильно змінної функції уведено Й.~Караматою (J.~Karamata) \cite{Karamata30a} (для неперервних функцій). Ним же \cite{Karamata30b, Karamata33} встановлено основні властивості таких функції. Теорія правильно змінних функцій, їх різних узагальнень та застосувань викладена у монографіях Н.~Х.~Бінхема (N.~H.~Bingham), Ч.~М.~Голді (C.~M.~Goldie) і Дж.~Л.~Тайгельза (J.~L.~Teugels) \cite{BinghamGoldieTeugels89}, В.~В.~Булдигіна, К.-Х.~Індлекофера (K.-H.~Indlekofer), О.~І.~Клесова і Й.~Г.~Стейнбаха (J.~G.~Steinebach) \cite{BuldyginIndlekoferKlesovSteinebach18}, J.~L.~Geluk і
L.~de~Haan \cite{GelukHaan87}, V.~Maric \cite{Maric00}, С.~І.~Рєшніка (S.~I.~Reshnick) \cite{Reshnick87}, Е.~Сєнєти (E.~Seneta) \cite{Seneta85}, Л.~де~Хаан (L.~de~Haan) \cite{Haan70}. У цих працях наведено різні приклади правильно змінних функцій. Їх називають повільно змінними (у сенсі Карамати), якщо $\theta=0$. Повільно змінна (на нескінченності) функція $\varphi$ може осцилювати у такий спосіб, що $\varphi([r,\infty))=(0,\infty)$ для довільного числа $r>1$ (див., наприклад, \cite[п.~1.2.1]{MikhailetsMurach10, MikhailetsMurach14}).

Узагальнені простори Соболєва, розглянуті у розділі~1 монографії, належать до так званих просторів узагальненої гладкості. Остання характеризується за допомогою функціонального параметра, залежного від частотних змінних (дуальних до просторових відносно перетворення Фур'є), або в еквівалентній формі за допомогою нескінченної числової послідовності. Для класичних ізотропних або анізотропних просторів Соболєва використовують числові параметри, тобто степеневі функції частотних змінних. Звісно, у теорії параболічних задач потрібні відповідні анізотропні простори. Теорія анізотропних соболєвських просторів, параметризованих за допомогою скінченного набору числових параметрів, започаткована у статті Л.~Н.~Слободецького \cite{Slobodetskii58}, де розглянуто лише гільбертові простори та наведено деякі застосування вказаних просторів до багатовимірних крайових задач, зокрема, для рівняння теплопровідності. Теорія банахових анізотропних просторів Соболєва та пов'язаних з ними анізотропних просторів Нікольського--Бєсова і Лізоркіна--Трібеля викладена у монографіях О.~В.~Бєсова, В.~П.~Ільїна і С.~М.~Нікольського \cite{BesovIlinNikolskii75}, С.~М.~Нікольського \cite{Nikolskii77}, Ґ.~Трібеля \cite[розд.~5]{Triebel06}, Г.~Амана \cite{Amann19}.

Простори узагальненої гладкості вперше з'явилися у статті Б.~Мальгранжа (B.~{Malgrange) \cite{Malgrange57}, де були застосовані до гіпоеліптичних диференціальних рівнянь (до яких належать і параболічні рівняння). Згодом систематичне вивчення широких класів просторів узагальненої гладкості було виконано у монографії Л.~Хермандера (L.~H\"ormander) \cite[розд.~II]{Hormander63}, статті Л.~Р.~Волєвіча і Б.~П.~Панеяха \cite{VolevichPaneah65}, серії статей Ґ.~Трібеля \cite{Triebel77I, Triebel77II, Triebel77III, Triebel77IV, Triebel79V}. Простори, введені у цих працях, збігаються у гільбертовому випадку. У~вказаній статті Л.~Р.~Волєвіча і Б.~П.~Панеяха \cite{VolevichPaneah65} уперше розглянуто простори узагальненої гладкості в евклідових областях. Ґрунтовне дослідження загальних властивостей багатовимірних диференціальних рівнянь було виконано Л.~Хермандером \cite{Hormander63} за допомогою введених ним просторів (див. також його монографію \cite{Hormander86}).

В останні десятиліття простори узагальненої гладкості є предметом низки глибоких досліджень: див. огляди П.~І.~Лізоркіна
\cite{Lizorkin86}, Г.~А.~Калябіна і П.~І.~Лізоркіна \cite{KalyabinLizorkin87}, монографії Ґ.~Трібеля \cite[розд.~III]{Triebel01} і \cite{Triebel10}, Н.~Якоба (N.~Jacob) \cite{Jacob010205}, Ф.~Нікола (F.~Nicola) і Л.~Родіно (L.~Rodino) \cite{NicolaRodino10}, В.~А.~Михайлеця і О.~О.~Мурача \cite{MikhailetsMurach10, MikhailetsMurach14} та недавні статті
A.~Almeida \cite{Almeida05}, A.~M.~Caetano і H.-G.~Leopold \cite{CaetanoLeopold06, CaetanoLeopold13}, F.~Cobos і T.~K\"uhn \cite{CobosKuhn09}, F.~Cobos, \'O.~Dom\'{\i}ngues і H.~Triebel \cite{CobosDominguesTriebel16}, W.~Farkas, N.~Jacob і R.~L.~Schilling \cite{FarkasJacobScilling01b}, W.~Farkas і H.-G.~Leopold \cite{FarkasLeopold06}, D.~D.~Haroske і S.~D.~Moura \cite{HaroskeMoura04, HaroskeMoura08}, L.~Loosveldt і S.~Nicolay \cite{LoosveldtNicolay19}, S.~D.~Moura, J.~S.~Neves і C.~Schneider \cite{MouraNevesSchneider11, MouraNevesSchneider14}, J.~S.~Neves і B.~Opic \cite{NevesOpic20}, H.~Wang і K.~Wang \cite{WangWang14} й наведену там літературу. Побудовано різні версії просторів Нікольського--Бєсова і Лізоркіна--Трібеля, параметризовані функціональними параметрами. Для них вдалося отримати точні теореми вкладення одних класів просторів у інші, теореми про продовження через межу області, де задано простір, теореми про сліди на межі області та інші вагомі результати.

Інтерполяційні властивості різних просторів узагальненої гладкості досліджено в працях М.~Шехтера \cite{Schechter67}, Ґ.~Трібеля \cite{Triebel77III}, С.~Меручі \cite{Merucci84}, F.~Cobos і D.~L.~Fernandez \cite{CobosFernandez88}, M.~J.~Carro і J.~Cerd\`a \cite{CarroCerda90}, серії статей А.~Г.~Багдасаряна, серед яких відмітимо \cite{Bagdasaryan97, Bagdasaryan10}, у працях A.~Almeida \cite{Almeida05, Almeida09}, A.~Almeida і A.~Caetano \cite{AlmeidaCaetano11}, В.~П.~Кноповой \cite{Knopova06}, В.~А.~Михайлеця і О.~О.~Мурача \cite{MikhailetsMurach08MFAT1, MikhailetsMurach10, MikhailetsMurach13UMJ3, MikhailetsMurach14, MikhailetsMurach15ResMath1}, B.~F.~Besoy і F.~Cobos \cite{BesoyCobos18}, L.~Loosveldt і S.~Nicolay \cite{LoosveldtNicolay19}.

Простори узагальненої гладкості мають важливі застосування у різних розділах математики: див. монографії Ґ.~Трібеля \cite{Triebel01, Triebel10}, О.~І.~Степанця \cite{Stepanets87, Stepanets02, Stepanets05} (математичний аналіз), Л.~Хермандера \cite{Hormander63, Hormander86}, Б.~П.~Панеяха \cite{Paneah00}, В.~А.~Михайлеця і О.~О.~Мурача \cite{MikhailetsMurach10, MikhailetsMurach14}, Ф.~Нікола і Л.~Родіно \cite{NicolaRodino10} (диференціальні рівняння з частинними похідними), В.~Г.~Маз'ї і Т.~О.~Шапошнікової \cite[розд.~16]{MazyaShaposhnikova09} (інтегральні рівняння), Н.~Якоба \cite{Jacob010205} (теорія випадкових процесів) та наведену там літературу.

Ізотропні версії просторів, розглянутих у першому розділі нашої монографії, були застосовані недавно до еліптичних диференціальних рівнянь і еліптичних крайових задач у статтях В.~А.~Михайлеця і О.~О.~Мурача \cite{MikhailetsMurach05UMJ5, MikhailetsMurach06UMJ2, MikhailetsMurach06UMJ3, MikhailetsMurach06UMJ11, MikhailetsMurach07UMJ5, MikhailetsMurach08MFAT1, MikhailetsMurach08UMJ4, MikhailetsMurach08BPAS3, Murach07UMJ6, Murach08UMB3, Murach08MFAT2, Murach09UMJ3}. Їх дослідження підсумовані у монографіях \cite{MikhailetsMurach10, MikhailetsMurach14} і оглядах \cite{MikhailetsMurach11VisnChernivUniv, MikhailetsMurach09OperatorTheory191, MikhailetsMurach12BJMA2}, де викладено нову теорію розв'язності еліптичних систем і еліптичних крайових задач в уточнених соболєвських шкалах. В~останні роки цю теорію було доповнено у працях О.~О.~Мурача і його учнів А.~В.~Аноп, Т.~М.~Зінченко, Т.~М.~Касіренко, І.~С.~Чепурухіної \cite{Anop19Dop2, AnopMurach18Dop3, KasirenkoMurachChepurukhina19Dop3, MurachChepurukhina20Dop8, AnopKasirenko16MFAT, AnopKasirenkoMurach18UMJ3, AnopMurach14MFAT2, AnopMurach14UMJ7, ChepurukhinaMurach15MFAT1, ChepurukhinaMurach20MFAT2, KasirenkoMurach18MFAT2, KasirenkoMurach18UMJ11, MurachChepurukhina15UMJ5, MurachZinchenko13MFAT1, Zinchenko17OpenMath, ZinchenkoMurach12UMJ11, ZinchenkoMurach14JMS}, де розглянуто  більш широкі класи узагальнених просторів Соболєва і досліджено більш загальні еліптичні крайові задачі. У новітній статті А.~В.~Аноп, Р.~Денка (R.~Denk) і О.~О.~Мурача \cite{AnopDenkMurach21CPAA2} наведено застосування узагальнених просторів Соболєва до деяких еліптичних задач з білим шумом у крайових умовах.
У цих працях використовуються виключно гільбертові простори (як і в нашій монографії). На відміну від них, стаття  М.~Файермана (M.~Faierman) \cite{Faierman20} присвячена дослідженню деяких еліптичних рівнянь (із спектральним параметром) у досить широкому класі банахових узагальнених просторів Соболєва. У~працях В.~С.~Ільківа, Н.~І.~Страп і І.~І.~Волянської \cite{IlkivStrap15, IlkivStrapVolyanska20} уточнена соболєвська шкала застосована до дослідження деяких нелокальних крайових задач для диференціально-операторних рівнянь, лінійних та слабко нелінійних.

Класи анізотропних узагальнених просторів Соболєва $H^{s,s\gamma;\varphi}$ і $H_{+}^{s,s\gamma;\varphi}$, розглянутих відповідно в пп.~\ref{sec2.3} і~\ref{sec2.5}, введені та досліджені авторами цієї монографії. Простори на $\mathbb{R}^k$ і евклідових областях в $\mathbb{R}^k$ введені у статті \cite[п.~3]{LosMurach13MFAT2} у випадку $k=2$, а у загальному випадку~--- у статті \cite[п.~2]{LosMurach14Dop6}. Інтерполяційні теореми \ref{9lem7.1}--\ref{9lem7.3} для цих просторів доведені в \cite[п.~5]{LosMurach13MFAT2}, а у загальному випадку~--- в~\cite[п.~7]{LosMikhailetsMurach17CPAA1}. Інтерполяційна теорема~\ref{9lem7.3a} для просторів, заданих у евклідових областях, наведена в замітці~\cite{Los18Dop6}. Простір $H^{s,s\gamma;\varphi}(S)$, де $S$~--- бічна поверхня циліндра, досліджений в статті~\cite[п.~1]{Los16JMS4}. Там для нього доведені теореми \ref{9lem3.1a} і \ref{9lem7.3a} у випадку, коли число $\gamma$ є раціональним. Простір $H^{s,s\gamma;\varphi}_{+}(S)$ вивчався в статті \cite[пп.~3 і~7]{LosMikhailetsMurach17CPAA1}, у якій для нього доведені теореми~\ref{9lem3.1} і~\ref{9lem7.3}. Теорема \ref{8lem1} (про оператор даних Коші) доведена в \cite[п.~6]{LosMurach17OM15} у випадку $b=1$, а у загальному випадку~--- в~\cite[п.~6]{LosMikhailetsMurach21arXiv}.
Теорема \ref{9lem8.1} (про умови вкладення простору $H^{s,s\gamma;\varphi}$ в деякі анізотропні простори неперервно диференційовних функцій) доведена в \cite[п.~8]{LosMikhailetsMurach17CPAA1}.

У соболєвському випадку, коли $\varphi(\cdot)\equiv1$, ці простори та їх $L_{p}$-аналоги широко застосовуються у теорії параболічних задач, починаючи з праць О.~О.~Ладиженської \cite{Ladyzhenskaja54, Ladyzhenskaja58}, Л.~Н.~Слободецького \cite{Slobodetskii58, Slobodetskii58DocAN}, В.~О.~Солоннікова \cite{Solonnikov62, Solonnikov64, Solonnikov65, Solonnikov67}, М.~С.~Аграновіча і М.~І.~Вішика \cite{AgranovichVishik64} (див. також монографії О.~О.~Ладиженської, В.~О.~Солоннікова і Н.~М.~Уральцевої \cite{LadyzhenskajaSolonnikovUraltzeva67}, Ж.-Л.~Ліонса і Е.~Мадженеса \cite{LionsMagenes72ii} та огляд С.~Д.~Ейдельмана \cite[\S~2]{Eidelman94}). Теореми про сліди функцій на многовидах для анізотропних просторів Соболєва та аналоги цих просторів на придатних гладких многовидах вивчалися ще в основоположній праці С.~Л.~Слободецького \cite[розд~I, \S~5 і розд~II, \S~5]{Slobodetskii58}, а також у вказаних щойно працях. Істотно більш складний випадок многовидів, геометрія яких не узгоджена з анізотропією функціонального простору, вивчався С.~В.~Успенським \cite{Uspenskii72} та його послідовниками (див. монографію С.~В.~Успенського, Г.~В.~Демиденко і В.~Г.~Перепьолкіна
\cite[розд.~2]{UspenskiiDemidenkoPerepelkin84} та наведену там літературу), Ґ.~Трібелем \cite{Triebel84I, Triebel84II}, М.~Маларскі (M.~Malarski) і Ґ.~Трібелем \cite{MalarskiTriebel91}. Недавній прогрес у дослідженні слідів на многовидах для анізотропних функціональних просторів відображений у статтях П.~Вайдемайера \cite{Weidemaier05}, J.~Jonsen, S.~M.~Hansen і W.~Sickel  \cite{JonsenHansenSickel15}, де розглядаються анізотропні простори Лізоркіна--Трібеля з огляду на їх застосування до параболічних задач.

Інтерполяція з числовими параметрами анізотропних просторів Соболєва та пов'язаних з ними анізотропних просторів Бєсова і Лізоркіна--Трібеля вивчалася у низці праць, серед яких вкажемо монографії  Ж.-Л.~Ліонса і Е.~Мадженеса \cite[розд.~4, п.~14]{LionsMagenes72ii}, Ґ.~Трібеля \cite[п.~2.13.2]{Triebel80}, Г.~Аманна \cite[розд.~VII, пп. 2.7, 4.5 і 5.4, розд.~VIII, п.~2.4]{Amann19}.

Умови вкладення соболєвських просторів у вибраний простір неперервно диференційовних функцій встановлені С.~Л.~Cоболєвим \cite{Sobolev38} для ізотропних просторів і Л.~Н.~Слободецьким \cite[\S~8]{Slobodetsky60} для анізотропних просторів (див. також монографії С.~Л.~Cоболєва \cite{Sobolev50}, О.~В.~Бєсова, В.~П.~Ільїна і С.~М.~Нікольського\cite[п.~10.4]{BesovIlinNikolskii75}). Для узагальнених соболєвських просторів такі умови знайшли Л.~Хермандер \cite[п.~2.2]{Hormander63}, Л.~Р.~Волєвіч і Б.~П.~Панеях \cite[\S~5, п.~1, \S~13, п.~4]{VolevichPaneah65}. В останній праці    \cite[теорема~9.1]{VolevichPaneah65}
досліджено також умови вкладення в анізотропний простір Гельдера. Цей результат не містить у собі теорему~\ref{9lem8.1}, доведену у п~\ref{sec2.8}. Для узагальнених просторів Бєсова і Лізоркіна--Трібеля умови вкладення у простір неперервних функцій містяться у статтях  М.~Л.~Гольдмана \cite{Goldman76, Goldman77} і Г.~А.~Калябіна \cite{Kalyabin77, Kalyabin81} (див. також огляди П.~Л.~Лізоркіна \cite[п.~Д.1.9]{Lizorkin86} і Г.~А.~Калябіна і П.~Л.~Лізоркіна \cite[п.~5]{KalyabinLizorkin87}). Точні умови вкладення цих просторів у узагальнені простори Гельдера отримані в статтях S.~D.~Moura, J.~S.~Neves і C.~Schneider \cite{MouraNevesSchneider11, MouraNevesSchneider14}.

\normalsize


\chapter{\textbf{Напіводнорідні параболічні задачі}}\label{ch3}

\chaptermark{\emph Розд. \ref{ch3}. Напіводнорідні параболічні задачі}

У цьому розділі досліджуємо параболічні задачі з однорідними початковими умовами. Такі задачі називаємо (дещо умовно) напіводнорідними. Головний результат розділу~--- теорема про ізоморфізми, породжені вказаними задачами на парах гільбертових функціональних просторів, уведених у п.~\ref{sec2.5}. Спочатку розглядаємо одновимірну за просторовою змінною задачу у прямокутнику, а потім --- багатовимірну задачу у циліндрі. Цю теорему застосуємо до дослідження локальної регулярності (аж до межі циліндричної області) розв'язків таких задач в анізотропних узагальнених просторах Соболєва та в анізотропних просторах неперервно диференційовних функцій.

\section[Задача у прямокутнику]
{Задача у прямокутнику}\label{sec3.2.1}

\markright{\emph \ref{sec3.2.1}. Задача у прямокутнику}

Нехай довільно задано дійсні числа $l>0$ і $\tau>0$. На декартовій площині $\mathbb{R}^{2}$ маємо відкритий прямокутник
$\Omega:=(0,l)\times(0,\tau)$. Довільну точку на цій площині інтерпретуємо як упорядковану пару $(x,t)$, де $x$~--- просторова, а $t$~--- часова змінні. Розглядаємо у прямокутнику $\Omega$ лінійну параболічну задачу, яка складається з диференціального рівняння
\begin{align}\notag
&A(x,t,D_x,\partial_t)u(x,t)\equiv\\
&\equiv\sum_{\alpha+2b\beta\leq 2m}a^{\alpha,\beta}(x,t)\,D^\alpha_x\partial^\beta_t
u(x,t)=f(x,t),\label{f2.1}\\
&\mbox{якщо}\;\,0<x<l\;\,\mbox{і}\;\,0<t<\tau,\notag
\end{align}
крайових умов
\begin{align}
&B_{j,0}(t,D_x,\partial_t)u(x,t)\big|_{x=0}\equiv\notag\\
&\equiv\sum_{\alpha+2b\beta\leq m_j}
b_{j,0}^{\alpha,\beta}(t)\,D^\alpha_x\partial^\beta_t u(x,t)\big|_{x=0}=g_{j,0}(t)
\quad\mbox{і}\label{f2.2}\\
&B_{j,1}(t,D_x,\partial_t)u(x,t)\big|_{x=l}\equiv\notag\\
&\equiv\sum_{\alpha+2b\beta\leq m_j}
b_{j,1}^{\alpha,\beta}(t)\,D^\alpha_x\partial^\beta_t
u(x,t)\big|_{x=l}=g_{j,1}(t),\label{f2.3}\\
&\mbox{якщо}\;\,0<t<\tau,\;\,\mbox{де}\;\,j=1,\dots,m,\notag
\end{align}
та однорідних початкових умов
\begin{equation}
\partial^k_t u(x,t)
\big|_{t=0}=0,\;\,\mbox{якщо}\;\,0<x<l,\;\,\mbox{де}\;\,
k=0,\ldots,\varkappa-1. \label{f2.4}
\end{equation}
Тут $b$, $m$ і кожне $m_j$ є довільно вибраними цілими числами, такими, що $m\geq b\geq1$, $\varkappa:=m/b\in\mathbb{Z}$ і $m_j\geq0$. Число $2b$ називають параболічною вагою цієї задачі. Усі коефіцієнти лінійних диференціальних виразів
$A:=A(x,t,D_x,\partial_t)$ та $B_{j,k}:=B_{j,k}(t,D_x,\partial_t)$, де
$j\in\{1,\dots,m\}$ і $k\in\{0,\,1\}$, є нескінченно гладкими
комплекснозначними функціями. А саме, $a^{\alpha,\beta}\in
C^{\infty}(\overline{\Omega})$ та $b_{j,k}^{\alpha,\beta}\in C^{\infty}[0,\tau]$, де
$\overline{\Omega}:=[0,l]\times[0,\tau]$. Використовуємо позначення
$D_x:=i\,\partial/\partial x$ та $\partial_t:=\partial/\partial t$ для частинних похідних; тут, як звичайно, $i$~--- уявна одиниця. Підсумовування здійснюємо за цілими індексами $\alpha,\beta\geq0$, які задовольняють нерівності, записані під знаком суми.

Нагадаємо \cite[\S~9, п.~1]{AgranovichVishik64}, що початково-крайову задачу \eqref{f2.1}--\eqref{f2.4} називають параболічною в $\Omega$, якщо вона задовольняє такі три умови:

\begin{condition}\label{9cond2.3} \rm
Для довільних чисел $x\in[0,l]$ і $t\in[0,\tau]$ та чисел
$\xi\in\mathbb{R}$ і $p\in\mathbb{C}$ таких, що $\mathrm{Re}\,p\geq0$ і $|\xi|+|p|\neq0$, виконується нерівність
\begin{equation*}
A^{\circ}(x,t,\xi,p) \equiv\sum_{\alpha+2b\beta=2m} a^{\alpha,\beta}(x,t)\,\xi^\alpha p^{\beta}\neq0.
\end{equation*}
\end{condition}

\begin{condition}\label{9cond2.4} \rm
Нехай довільно вибрано числа $x\in\{0,l\}$ і $t\in\nobreak[0,\tau]$ та число $p\in\mathbb{C}\setminus\{0\}$ таке, що $\mathrm{Re}\,p\geq\nobreak0$. Тоді многочлен $A^{\circ}(x,t,\zeta,p)$ відносно $\zeta\in\mathbb{C}$ має $m$ коренів $\zeta^{+}_{j}(x,t,p)$, де $j=\nobreak1,\ldots,m$, з додатною уявною частиною та $m$ коренів з \nobreak{від'ємною} уявною частиною з урахуванням їх кратності.
\end{condition}

\begin{condition}\label{9cond2.5} \rm
Нехай числа $x$, $t$ і $p$ такі самі, як в умові~$\ref{9cond2.4}$.
Покладемо $k:=0$, якщо $x=0$, або $k:=1$, якщо $x=l$. Тоді многочлени
$$
B_{j,k}^{\circ}(t,\zeta,p)\equiv\sum_{\alpha+2b\beta=m_{j}}
b_{j,k}^{\alpha,\beta}(t)\,\zeta^{\alpha}p^{\beta}
$$
змінної $\zeta\in\mathbb{C}$, де $j=1,\dots,m$, лінійно незалежні за модулем многочлена
$$
\prod_{j=1}^{m}\bigl(\zeta-\zeta^{+}_{j}(x,t,p)\bigr).
$$
\end{condition}

Умова~\ref{9cond2.3} є умовою $2b$-параболічності за І.~Г.~Петровським \cite{Petrovskii38} диференціального рівняння $Au=f$ у замкненому
прямокутнику $\overline{\Omega}$, а умова~\ref{9cond2.5} виражає той факт, що система крайових диференціальних операторів $\{B_{1,k},\ldots,B_{m,k}\}$ накриває диференціальний оператор $A$ на стороні $x=0$, якщо $k=0$, або $x=l$, якщо $k=1$, цього прямокутника. Умова~\ref{9cond2.5} уперше з'явилася у монографії Т.~Я.~Загорського \cite{Zagorskii61}.

Пов'яжемо з параболічною задачею \eqref{f2.1}--\eqref{f2.4} лінійне відображення
\begin{equation}\label{f2.5}
\begin{gathered}
C^{\infty}_{+}(\overline{\Omega})\ni u\mapsto (Au,Bu):=\\
:=\bigl(Au,B_{1,0}u,B_{1,1}u,\ldots,B_{m,0}u,B_{m,1}u\bigr)\in
\bigl(C^{\infty}_{+}[0,\tau]\bigr)^{2m}.
\end{gathered}
\end{equation}
Тут використано такі позначення:
\begin{gather*}
C^{\infty}_{+}(\overline{\Omega}):=\bigl\{w\!\upharpoonright\overline{\Omega}:\,
w\in C^{\infty}(\mathbb{R}^{2}),\;\,
\mathrm{supp}\,w\subseteq\mathbb{R}\times[0,\infty)\bigr\}=\\
=\bigl\{u\in C^{\infty}(\overline{\Omega}):
\partial_{t}^{\beta}u(x,t)|_{t=0}=0,\;\,\mbox{якщо}\;\,
0\leq\beta\in\mathbb{Z},\;x\in[0,l]\bigr\}
\end{gather*}
та
\begin{gather*}
C^{\infty}_{+}[0,\tau]:=\bigl\{h\!\upharpoonright[0,\tau]:\,h\in
C^{\infty}(\mathbb{R}),\;\,\mathrm{supp}\,h\subseteq[0,\infty)\bigr\}=\\
=\bigl\{v\in C^{\infty}[0,\tau]:\,v^{(\beta)}(0)=0,\;\,\mbox{якщо}\;\,
0\leq\beta\in\mathbb{Z}\bigr\}.
\end{gather*}

Відображення \eqref{f2.5} встановлює взаємно однозначну відповідність між лінійними просторами $C^{\infty}_{+}(\overline{\Omega})$ і
$C^{\infty}_{+}(\overline{\Omega})\times
\bigl(C^{\infty}_{+}[0,\tau]\bigr)^{2m}$
(це пояснено після формулювання теореми~\ref{th3.4}).
Воно продовжується єдиним чином до ізоморфізму між деякими узагальненими соболєвськими просторами. Сформулюємо відповідний результат.

Нехай $\sigma_0$~--- найменше ціле число таке, що
\begin{equation}\label{sigma0-Murach}
\sigma_0\geq\max\{2m,m_1+1,\ldots,m_m+1\}
\;\;\mbox{і}\;\;\frac{\sigma_0}{2b}\in\mathbb{Z}.
\end{equation}
Зокрема, якщо $m_j\leq2m-1$ для кожного номера $j\in\{1,\ldots,m\}$, то $\sigma_0=2m$.

\begin{theorem}\label{th3.4}
Для довільних дійсного числа $\sigma>\sigma_0$ і функціонального параметра $\varphi\in\mathcal{M}$ відображення \eqref{f2.5} продовжується єдиним чином (за неперервністю) до ізоморфізму:
\begin{gather}\label{f2.6}
(A,B):H^{\sigma,\sigma/(2b);\varphi}_{+}(\Omega)\;\leftrightarrow\\ \notag
\leftrightarrow H^{\sigma-2m,(\sigma-2m)/(2b);\varphi}_{+}(\Omega)\oplus
\bigoplus_{j=1}^{m}\bigl(H^{(\sigma-m_j-1/2)/(2b);\varphi}_{+}(0,\tau)\bigr)^{2}:=\\
:=\mathcal{H}^{\sigma-2m,(\sigma-2m)/(2b);\varphi}_{+}. \notag
\end{gather}
\end{theorem}

Якщо $\varphi(\cdot)\equiv1$, то оператор \eqref{f2.6} діє на парі соболєвських просторів. У цьому випадку теорема~\ref{th3.4} випливає з результату М.~С.~Аграновіча і М.~І.~Вішика \cite[теорема~12.1]{AgranovichVishik64}, доведеного у припущенні, що $\sigma/(2b)\in\mathbb{Z}$. (Це показано в наступному підрозділі у першій частині доведення теореми~\ref{9th4.1}.) Їх результат охоплює граничний випадок $\sigma=\sigma_0$ та стосується загальних параболічних
задач із, взагалі кажучи, неоднорідними початковими умовами.

Згідно з теоремою вкладення Соболєва, лемою~\ref{9lem5.1}  і згаданим щойно результатом \cite[теорема 12.1]{AgranovichVishik64} маємо такі рівності:
\begin{gather*}
C^{\infty}_{+}(\overline{\Omega})=
\bigcap_{\substack{\sigma>\sigma_0,\\\sigma/(2b)\in\mathbb{Z}}}
H^{\sigma,\sigma/(2b)}_{+}(\Omega),\\
C^{\infty}_{+}(\overline{\Omega})\times\bigl(C^{\infty}_{+}[0,\tau]\bigr)^{2m}
=\bigcap_{\substack{\sigma>\sigma_0,\\\sigma/(2b)\in\mathbb{Z}}}
(A,B)\bigl(H^{\sigma,\sigma/(2b)}_{+}(\Omega)\bigr).
\end{gather*}
З них випливає, що відображення \eqref{f2.5} встановлює взаємно однозначну відповідність між просторами $C^{\infty}_{+}(\overline{\Omega})$ і
$C^{\infty}_{+}(\overline{\Omega})\times$ $\times
\bigl(C^{\infty}_{+}[0,\tau]\bigr)^{2m}$, як було зазначено вище.

\begin{proof}[\indent Доведення теореми $\ref{th3.4}$.]
Нехай $\sigma>\sigma_0$ і $\varphi\in\mathcal{M}$. Виберемо ціле число $\sigma_1>\sigma$ таке, що $\sigma_1/(2b)\in\mathbb{Z}$. Згідно з \cite[теорема 12.1]{AgranovichVishik64} відображення \eqref{f2.5} продовжується єдиним чином (за неперервністю) до ізоморфізмів
\begin{equation}\label{f5.9}
\begin{gathered}
(A,B):\,H^{\sigma_k,\sigma_k/(2b)}_{+}(\Omega)\leftrightarrow
\mathcal{H}_k,\quad\mbox{де}\quad k\in\{0,1\}.
\end{gathered}
\end{equation}
Тут
$$
\mathcal{H}_k:=H^{\sigma_k-2m,(\sigma_k-2m)/(2b)}_{+}(\Omega)\oplus
\bigoplus_{j=1}^{m}\bigl(H^{(\sigma_k-m_j-1/2)/(2b)}_{+}(0,\tau)\bigr)^{2}.
$$
Означимо інтерполяційний параметр $\psi$ формулою \eqref{9f7.2}, у якій покладаємо $s:=\sigma$, $s_{0}:=\sigma_{0}$ і $s_{1}:=\sigma_{1}$. Застосувавши квадратичну інтерполяцію з функціональним параметром $\psi$ до операторів \eqref{f5.9}, отримаємо ізоморфізм
\begin{equation}\label{f5.10}
(A,B):\,\bigl[H^{\sigma_0,\sigma_0/(2b)}_{+}(\Omega),
H^{\sigma_1,\sigma_1/(2b)}_{+}(\Omega)\bigr]_{\psi}\leftrightarrow
[\mathcal{H}_0,\mathcal{H}_1]_{\psi}.
\end{equation}
Він є звуженням оператора \eqref{f5.9}, де $k=0$.

На підставі теореми~\ref{9lem7.3} і зауваження~\ref{lem5.4} маємо такі рівності просторів з еквівалентністю норм у них:
\begin{equation*}
\bigl[H^{\sigma_0,\sigma_0/(2b)}_{+}(\Omega),
H^{\sigma_1,\sigma_1/(2b)}_{+}(\Omega)\bigr]_{\psi}=
H^{\sigma,\sigma/(2b);\varphi}_{+}(\Omega)
\end{equation*}
та
\begin{align*}
&\bigl[\mathcal{H}_0,\mathcal{H}_1\bigr]_{\psi}\\
&=\bigl[H^{\sigma_0-2m,(\sigma_0-2m)/(2b)}_{+}(\Omega),
H^{\sigma_1-2m,(\sigma_1-2m)/(2b)}_{+}(\Omega)\bigr]_{\psi}\oplus\\
&\oplus\bigoplus_{j=1}^{m}
\bigl(\bigl[H^{(\sigma_0-m_j-1/2)/(2b)}_{+}(0,\tau),
H^{(\sigma_1-m_j-1/2)/(2b)}_{+}(0,\tau)\bigr]_{\psi}\bigr)^{2}=\\
&=H^{\sigma-2m,(\sigma-2m)/(2b);\varphi}_{+}(\Omega)\oplus
\bigoplus_{j=1}^{m}
\bigl(H^{(\sigma-m_j-1/2)/(2b);\varphi}_{+}(0,\tau)\bigr)^{2}.
\end{align*}
Тут також використано теорему~\ref{9prop6.3} про квадратичну інтерполяцію прямих сум гільбертових просторів. Отже, ізоморфізм \eqref{f5.10} діє на парі просторів \eqref{f2.6}. Він є продовженням за неперервністю відображення \eqref{f2.5} оскільки множина $C^{\infty}_{+}(\overline{\Omega})$ щільна в просторі $H^{\sigma,\sigma/(2b);\varphi}_{+}(\Omega)$.
\end{proof}

Дослідимо регулярність узагальненого розв'язку параболічної задачі \eqref{f2.1}--\eqref{f2.4} за допомогою просторів $H^{\sigma,\sigma/(2b);\varphi}_{+}(\Omega)$.
Згідно з \cite[теорема 12.1]{AgranovichVishik64} кожний вектор
\begin{equation}\label{f11}
F:=(f,g_{1,0},g_{1,1},...,g_{m,0},g_{m,1})\in
\mathcal{H}^{\sigma_0-2m,(\sigma_0-2m)/(2b)}_{+}
\end{equation}
має єдиний прообраз $u\in H^{\sigma_0,\sigma_0/(2b)}_{+}(\Omega)$ щодо відображення \eqref{f2.6}. Цю функцію $u$ називаємо узагальненим розв'язком параболічної задачі \eqref{f2.1}--\eqref{f2.4}, праві частини якої задовольняють умову~\eqref{f11}.

Негайним наслідком теореми~\ref{th3.4} є така властивість глобального підвищення регулярності цього розв'язку:

\begin{corollary}\label{cor2}
Припустимо, що  функція $u\in
H^{\sigma_0,\sigma_0/(2b)}_{+}(\Omega)$ є узагальненим розв'язком параболічної задачі \eqref{f2.1}--\eqref{f2.4}, праві частини якої задовольняють умову
$$
F:=(f,g_{1,0},g_{1,1},...,g_{m,0},g_{m,1}) \in
\mathcal{H}^{\sigma-2m,(\sigma-2m)/(2b);\varphi}_{+}
$$
для деяких $\sigma>\sigma_0$ і $\varphi\in\mathcal{M}$. Тоді $u\in
H^{\sigma,\sigma/(2b);\varphi}_{+}(\Omega)$.
\end{corollary}

Як бачимо, уточнена регулярність $\varphi$ правих частин параболічної задачі успадковується її розв'язком.

Розглянемо локальну версію цього результату. Нехай $U$~--- відкрита підмножина простору $\mathbb{R}^{2}$ така, що $\Omega_0:=U\cap\Omega\neq\varnothing$. Покладемо
$\Omega':=U\cap\partial\overline{\Omega}$, $S_{0,0}:=U\cap\{(0,t): 0<t<\tau\}$, $S_{0,1}:=U\cap\{(l,t): 0<t<\tau\}$, $S'_{0}:=U\cap \{(0,0),(0,\tau)\}$ і $S'_{1}:=U\cap \{(l,0),(l,\tau)\}$. Введемо потрібні локальні версії просторів
$H^{s,s\gamma;\varphi}_{+}(\Omega)$ і $H^{s;\varphi}_{+}(0,\tau)$, де $s>0$, $\gamma=1/(2b)$ і $\varphi\in\mathcal{M}$.

Позначимо через $H^{s,s\gamma;\varphi}_{+,\mathrm{loc}}(\Omega_0,\Omega')$ лінійний простір усіх розподілів $u$
в області $\Omega$ таких, що $\chi u\in H^{s,s\gamma;\varphi}_{+}(\Omega)$ для кожної функції $\chi\in C^\infty (\overline\Omega)$, яка задовольняє умову $\mathrm{supp}\,\chi\subset\Omega_0\cup\Omega'$. Аналогічно, позначимо через $H^{s;\varphi}_{+,\mathrm{loc}}(S_{0,k},S'_{k})$ лінійний простір усіх розподілів $v$ на інтервалі $(0,\tau)$ таких, що $\chi v\in H^{s;\varphi}_{+}(0,\tau)$ для кожної функції $\chi\in C^\infty [0,\tau]$, яка задовольняє умову $\mathrm{supp}\,\chi\subset S_{0,k}\cup S'_{k}$; тут $k\in\{0,1\}$.

\begin{theorem}\label{th3.5}
Нехай функція $u\in H^{\sigma_0,\sigma_0/(2b)}_+(\Omega)$ є узагальненим розв'язком параболічної задачі \eqref{f2.1}--\eqref{f2.4}, праві частини якої задовольняють умову~\eqref{f11}. Нехай, крім того, $\sigma>\sigma_0$ і $\varphi\in\mathcal{M}$. Припустимо, що
\begin{equation}\label{f12}
f\in H^{\sigma-2m,(\sigma-2m)/(2b);\varphi}_{+,\,\mathrm{loc}}(\Omega_0,\Omega')
\end{equation}
та
\begin{equation}\label{f13}
\begin{gathered}
g_{j,k}\in H^{(\sigma-m_j-1/2)/(2b);\varphi}_{+,\,\mathrm{loc}}(S_{0,k},S'_{k})\\
\mbox{для всіх}\;\;k\in\{0,\,1\}\;\;\mbox{та}\;\;j\in\{1,\dots,m\}.
\end{gathered}
\end{equation}
Тоді $u\in H^{\sigma,\sigma/(2b);\varphi}_{+,\mathrm{loc}}(\Omega_0,\Omega')$.
\end{theorem}

Якщо $\Omega=\Omega_0$ і $\Omega'=\partial\overline{\Omega}$, то за цією теоремою маємо глобальне підвищення регулярності, тобто наслідок~\ref{cor2}. У випадку, коли $\Omega'=\varnothing$,
теорема~\ref{th3.5} стверджує, що регулярність розв'язку підвищується в околах внутрішніх точок замкненого прямокутника~$\overline{\Omega}$. Відмітимо, що ця теорема є новою навіть у випадку анізотропних просторів Соболєва в $\Omega$ (тобто, коли $\varphi(\cdot)\equiv1$).

Наступна теорема дає достатні умови, за яких узагальнений розв'язк $u$ досліджуваної параболічної задачі та його узагальнені частинні похідні заданого порядку є неперервними на множині $\Omega_0\cup\Omega'$.

\begin{theorem}\label{th3.6}
Нехай функція $u\in H^{\sigma_0,\sigma_0/(2b)}_+(\Omega)$ є узагальненим розв'язком параболічної задачі \eqref{f2.1}--\eqref{f2.4}, праві частини якої задовольняють умову~\eqref{f11}. Нехай ціле число $p\geq0$ таке, що $p+b+1/2>\sigma_{0}$. Припустимо, що праві частини цієї задачі задовольняють умови \eqref{f12} і \eqref{f13} для $\sigma:=p+b+1/2$ та деякого функціонального параметра $\varphi\in\mathcal{M}$, підпорядкованого умові \eqref{9f4.7}. Тоді розв'язок $u(x,t)$ та кожна його узагальнена частинна похідна
$D_{x}^{\alpha}\partial_{t}^{\beta}u(x,t)$, де $\alpha+2b\beta\leq p$, неперервні на множині $\Omega_0\cup\Omega'$.
\end{theorem}

Важливо, що умова \eqref{9f4.7} є точною в останній теоремі (див. зауваження~\ref{9rem4.5} у п.~\ref{sec3.1.3}).

Доведення теорем~\ref{th3.5} і~\ref{th3.6} наводити не будемо, оскільки вони ідентичні до доведень аналогів цих теорем для напіводнорідної параболічної задачі у багатовимірному циліндрі. Ці аналоги разом із їх доведеннями подано у п.~\ref{sec3.1.3} (див. теореми~\ref{9th4.3} і~\ref{9th4.4}).

\markright{\emph \ref{sec3.1.1}. Задача у циліндрі}

\section[Задача у циліндрі]{Задача у циліндрі}\label{sec3.1.1}

\markright{\emph \ref{sec3.1.1}. Задача у циліндрі}

Нехай довільно задано ціле число $n\geq2$, дійсне число $\tau>0$
і обмежену область $G\subset\mathbb{R}^{n}$ з
нескінченно гладкою межею $\Gamma$. Як у п.~\ref{sec2.3}, покладемо
$\Omega:=G\times(0,\tau)$ і $S:=\Gamma\times(0,\tau)$. Нагадаємо, що  $\overline{\Omega}=\overline{G}\times[0,\tau]$ і
$\overline{S}=\Gamma\times[0,\tau]$.

Розглянемо у (багатовимірному) циліндрі $\Omega$ параболічну початково-крайову задачу, яка складається з диференціального рівняння
\begin{equation}\label{9f2.1}
\begin{aligned}
&A(x,t,D_x,\partial_t)u(x,t)\equiv \\
&\equiv\sum_{|\alpha|+2b\beta\leq2m}
a^{\alpha,\beta}(x,t)\,D^\alpha_x\partial^\beta_t
u(x,t)=f(x,t),\\
&\mbox{якщо}\;\,x\in G\;\,\mbox{і}\;\,0<t<\tau,
\end{aligned}
\end{equation}
$m$ крайових умов
\begin{equation}\label{9f2.2}
\begin{aligned}
&B_{j}(x,t,D_x,\partial_t)u(x,t)\big|_{S}\equiv \\
&\equiv\sum_{|\alpha|+2b\beta\leq m_j}
b_{j}^{\alpha,\beta}(x,t)\,D^\alpha_x\partial^\beta_t u(x,t)\big|_{S}=g_{j}(x,t),\\
&\mbox{якщо}\;\,x\in\Gamma\;\,\mbox{і}\;\,0<t<\tau,\;\,\mbox{де}\;\,
j=1,\dots,m,
\end{aligned}
\end{equation}
і $\varkappa$ однорідних початкових умов
\begin{equation}\label{9f2.3}
\partial^k_t u(x,t)
\big|_{t=0}=0,\;\;\mbox{якщо}\;\,x\in G,\;\,\mbox{де}\;\,
k=0,\ldots,\varkappa-1.
\end{equation}
Як і раніше, $b$, $m$ і кожне $m_j$ є довільно заданими цілими числами, що задовольняють умови $m\geq b\geq1$, $\varkappa:=m/b\in\mathbb{Z}$ і $m_j\geq0$. Усі коефіцієнти лінійних диференціальних виразів $A:=A(x,t,D_x,\partial_t)$ і $B_{j}:=B_{j}(x,t,D_x,\partial_t)$, де $j\in\{1,\dots,m\}$, є нескінченно гладкими комплекснозначними функціями, заданими на $\overline{\Omega}$ і $\overline{S}$ відповідно. Отже, \begin{gather*}
a^{\alpha,\beta}\in C^{\infty}(\overline{\Omega})=
\bigl\{w\!\upharpoonright\overline{\Omega}\!:\,w\in C^{\infty}(\mathbb{R}^{n+1})\bigr\},\\
b_{j}^{\alpha,\beta}\in C^{\infty}(\overline{S})=
\bigl\{v\!\upharpoonright\overline{S}\!:\,v\in C^{\infty}(\Gamma\times\mathbb{R})\bigr\}
\end{gather*}
для всіх допустимих значень мультиіндексу $\alpha=(\alpha_1,...,\alpha_n)$ та скалярного індексу $\beta$. У формулах \eqref{9f2.1} і \eqref{9f2.2} та їх аналогах підсумовування здійснюємо за цілими невід'ємними індексами
$\alpha_1$,..., $\alpha_n$ і $\beta$, які задовольняють умову, написану під знаком суми. У постановці параболічної задачі використовуємо позначення
$D_x^\alpha:=i^{|\alpha|}\partial_x^\alpha$. При цьому (пряме) перетворення Фур'є за векторною змінною $x=(x_1,\ldots,x_n)$ означене так, що воно переводить диференціальний оператор $D^\alpha_x$ в оператор множення на функцію  $\xi^{\alpha}:=\xi_{1}^{\alpha_{1}}\ldots\xi_{n}^{\alpha_{n}}$ аргументу $\xi:=(\xi_{1},\ldots,\xi_{n})\in\mathbb{C}^{n}$.

Нагадаємо \cite[\S~9, п.~1]{AgranovichVishik64}, що початково-крайову задачу \eqref{9f2.1}--\eqref{9f2.3} називають параболічною у  циліндрі $\Omega$, якщо вона задовольняє умови \ref{9cond2.1} і \ref{9cond2.2}, які тепер сформулюємо.

\begin{condition}\label{9cond2.1}\rm
Для довільних точок $x\in\overline{G}$ і $t\in[0,\tau]$, вектора
$\xi\in\mathbb{R}^{n}$ і числа $p\in\mathbb{C}$ таких, що  $\mathrm{Re}\,p\geq0$ і $|\xi|+|p|\neq0$, виконується нерівність
\begin{equation*}
A^{\circ}(x,t,\xi,p)\equiv\sum_{|\alpha|+2b\beta=2m} a^{\alpha,\beta}(x,t)\,\xi^\alpha p^{\beta}\neq0.
\end{equation*}
\end{condition}

Довільно виберемо точку $x\in\Gamma$, число $t\in[0,\tau]$, дотичний вектор $\xi\in\mathbb{R}^{n}$ до межі $\Gamma$ у точці $x$ та число $p\in\mathbb{C}$ такі, що
$\mathrm{Re}\,p\geq0$ і $|\xi|+|p|\neq0$. Нехай $\nu(x)$~--- орт
внутрішньої нормалі до межі $\Gamma$ у точці $x$. З умови~\ref{9cond2.1} та нерівності $n\geq2$ випливає, що многочлен $A^{\circ}(x,t,\xi+\zeta\nu(x),p)$ змінної $\zeta\in\mathbb{C}$ має
точно $m$ коренів $\zeta^{+}_{j}(x,t,\xi,p)$, де $j=\nobreak1,\ldots,m$, із додатною уявною частиною і $m$ коренів з від'ємною уявною частиною (з урахуванням їхньої кратності).

\begin{condition}\label{9cond2.2}\rm
При кожному такому виборі $x$, $t$, $\xi$ і $p$ многочлени
$$
B_{j}^{\circ}(x,t,\xi+\zeta\nu(x),p)\equiv\sum_{|\alpha|+2b\beta=m_{j}}
b_{j}^{\alpha,\beta}(x,t)\,(\xi+\zeta\nu(x))^{\alpha}\,p^{\beta}
$$
змінної $\zeta\in\mathbb{C}$, де $j=1,\dots,m$, лінійно незалежні за модулем многочлена
$$
\prod_{j=1}^{m}(\zeta-\zeta^{+}_{j}(x,t,\xi,p)).
$$
\end{condition}

Як бачимо, умови \ref{9cond2.1} і \ref{9cond2.2} є багатовимірними аналогами умов \ref{9cond2.3} і \ref{9cond2.5} відповідно. У випадку $n\geq2$, що розглядається, багатовимірний аналог умови \ref{9cond2.4} є наслідком умови \ref{9cond2.1}, на відміну від випадку $n=1$, дослідженому у попередньому підрозділі.

Пов'яжемо з параболічною задачею \eqref{9f2.1}--\eqref{9f2.3} лінійне відображення
\begin{equation}\label{9f2.4}
\begin{gathered}
C^{\infty}_{+}(\overline{\Omega})\ni u\mapsto(Au,Bu):= \\
:=\bigl(Au,B_1u,\ldots,B_mu\bigr)\in
C^{\infty}_{+}(\overline{\Omega})\times
\bigl(C^{\infty}_{+}(\overline{S})\bigr)^{m}.
\end{gathered}
\end{equation}
При цьому покладаємо
\begin{gather*}
C^{\infty}_{+}(\overline{\Omega}):=
\bigl\{w\!\upharpoonright\!\overline{\Omega}:\,
w\in C^{\infty}(\mathbb{R}^{n+1}),\;\,
\mathrm{supp}\,w\subseteq\mathbb{R}\times[0,\infty)\bigr\}=\\
=\bigl\{u\in C^{\infty}(\overline{\Omega}):
\partial_{t}^{\beta}u(x,t)|_{t=0}=0,\;\,\mbox{якщо}\;\,
0\leq\beta\in\mathbb{Z},\;x\in\overline{G}\,\bigr\}
\end{gather*}
та
\begin{gather*}
C^{\infty}_{+}(\overline{S}):=\bigl\{h\!\upharpoonright\!\overline{S}:
h\in C^{\infty}(\Gamma\times\mathbb{R}),\;\,\mathrm{supp}\,
h\subseteq\Gamma\times[0,\infty)\bigr\}=\\
=\bigl\{v\in C^{\infty}(\overline{S}):
\partial_{t}^{\beta}v(x,t)|_{t=0}=0,\;\,\mbox{якщо}\;\,
0\leq\beta\in\mathbb{Z},\;x\in\Gamma\bigr\}.
\end{gather*}
Відображення \eqref{9f2.4} встановлює взаємно однозначну відповідність між лінійними просторами
$C^{\infty}_{+}(\overline{\Omega})$ і $C^{\infty}_{+}(\overline{\Omega})\times(C^{\infty}_{+}(\overline{S}))^{m}$.
Це обґрунтовується так само, як і у випадку $n=1$, дослідженому у попередньому підрозділі.

Відображення \eqref{9f2.4} продовжується єдиним чином до ізоморфізму між деякими узагальненими анізотропними просторами Соболєва. Сформулюємо відповідний результат; він є головним у цьому розділі.

Нехай $\sigma_0$~--- найменше ціле число, яке задовольняє умову \eqref{sigma0-Murach}. Введемо гільбертів простір
\begin{equation}\label{9f4.2}
\begin{gathered}
\mathcal{H}^{\sigma-2m,(\sigma-2m)/(2b);\varphi}_{+}(\Omega,S)
:=H^{\sigma-2m,(\sigma-2m)/(2b);\varphi}_{+}(\Omega)\oplus \\
\oplus\bigoplus_{j=1}^{m}
H^{\sigma-m_j-1/2,(\sigma-m_j-1/2)/(2b);\varphi}_{+}(S),
\end{gathered}
\end{equation}
де $\sigma\geq\sigma_0$ і $\varphi\in \mathcal{M}$.

\begin{theorem}\label{9th4.1}
Для довільних дійсного числа $\sigma>\sigma_0$ і функціонального параметра $\varphi\in\nobreak\mathcal{M}$ відображення \eqref{9f2.4} продовжується єдиним чином (за неперервністю) до ізоморфізму
\begin{equation}\label{9f4.1}
(A,B):H^{\sigma,\sigma/(2b);\varphi}_{+}(\Omega)\leftrightarrow
\mathcal{H}^{\sigma-2m,(\sigma-2m)/(2b);\varphi}_{+}(\Omega,S).
\end{equation}
\end{theorem}

\begin{proof}[\indent Доведення.]
У випадку анізотропних просторів Соболєва, коли $\varphi(\cdot)\equiv1$ і $\sigma/(2b)\in\mathbb{Z}$, ця теорема випливає з результату М.~С.~Аграновіча і М.~І.~Вішика \cite[теорема 12.1]{AgranovichVishik64}, який охоплює граничний випадок $\sigma=\sigma_0$ та стосується параболічних задач із, взагалі кажучи, неоднорідними початковими умовами. Обґрунтуємо це.

Отже, дослідимо спочатку випадок, коли $\sigma\geq\sigma_{0}$, $\sigma/(2b)\in\mathbb{Z}$ і $\varphi(\cdot)\equiv1$.
Нехай праві частини параболічної задачі \eqref{9f2.1}--\eqref{9f2.3} задовольняють умову
\begin{equation}\label{9f5.14}
(f,g_1,\ldots,g_m)\in
\mathcal{H}_{+}^{\sigma-2m,(\sigma-2m)/(2b)}(\Omega,S).
\end{equation}
(Нагадаємо, що у випадку, коли $\varphi(\cdot)\equiv1$, ми прибираємо індекс $\varphi$ у позначеннях просторів.) Вектор \eqref{9f5.14} задовольняє умови узгодження \cite[\S~11, с.~136]{AgranovichVishik64} правих частин параболічної задачі у випадку нульових початкових даних.

Теорема Аграновіча--Вішика \cite[теорема~12.1]{AgranovichVishik64}
стверджує з огляду на вкладення \eqref{9f5.3}, що задача \eqref{9f2.1}--\eqref{9f2.3} має єдиний розв'язок $u\in H^{\sigma,\sigma/(2b)}(\Omega)$ і він задовольняє двобічну нерівність
\begin{equation}\label{9f5.15}
\begin{split}
\|u\|_{H^{\sigma,\sigma/(2b)}(\Omega)}&\leq c_1
\|(f,g_1,\ldots,g_m)\|_
{\mathcal{H}^{\sigma-2m,(\sigma-2m)/(2b)}(\Omega,S)}\\
&\leq c_2\,\|u\|_{H^{\sigma,\sigma/(2b)}(\Omega)},
\end{split}
\end{equation}
де $c_1$ і $c_2$~--- деякі додатні числа, незалежні від вектора \eqref{9f5.14} і функції~$u$. Тут
\begin{equation*}
\begin{split}
&\mathcal{H}^{\sigma-2m,(\sigma-2m)/(2b)}(\Omega,S):=\\
&:=H^{\sigma-2m,(\sigma-2m)/(2b)}(\Omega)\oplus
\bigoplus_{j=1}^{m}H^{\sigma-m_j-1/2,(\sigma-m_j-1/2)/(2b)}(S).
\end{split}
\end{equation*}

Покажемо, що $u\in H^{\sigma,\sigma/(2b)}_{+}(\Omega)$. Для цього скористаємося лемою~\ref{9lem5.1}, у якій візьмемо $s:=\sigma$ і $\gamma:=1/(2b)$. Згідно з початковими умовами \eqref{9f2.3} функція $u$ задовольняє \eqref{9f5.4}, якщо $0\leq k\leq\varkappa-1$. Тут
\begin{equation*}
\varkappa-1=\frac{2m}{2b}-1<\frac{\sigma}{2b}-\frac{1}{2}.
\end{equation*}
Доведемо, що $u$ задовольняє \eqref{9f5.4} для усіх інших значень цілого параметра $k$ таких, що $\varkappa-1<k<\sigma/(2b)-1/2$ (якщо такі значення існують).

Нехай кількість цих значень $l\geq1$; тоді
\begin{equation}\label{9f5.16}
\varkappa+l-1<\frac{\sigma}{2b}-\frac{1}{2}<\varkappa+l.
\end{equation}
Умова \ref{9cond2.1} у випадку $\xi=0$ і $p=1$ означає, що коефіцієнт $a^{(0,\ldots,0),\varkappa}(x,t)\neq0$ для всіх $x\in\overline{G}$ і $t\in[0,\tau]$. Тому параболічне рівняння \eqref{9f2.1} можна розв'язати відносно  $\partial^\varkappa_t u(x,t)$. Отже,
\begin{equation}\label{9f5.17}
\begin{split}
\partial^\varkappa_t u(x,t)=&
\sum_{\substack{|\alpha|+2b\beta\leq 2m,\\ \beta\leq\varkappa-1}}
a_{0}^{\alpha,\beta}(x,t)\,D^\alpha_x\partial^\beta_t
u(x,t)+\\&+(a^{(0,\ldots,0),\varkappa}(x,t))^{-1}f(x,t)
\end{split}
\end{equation}
для деяких функцій $a_{0}^{\alpha,\beta}\in C^{\infty}(\overline{\Omega})$. Якщо $l\geq2$, то диференціюючи рівність \eqref{9f5.17} $l-1$ разів за змінною $t$, отримаємо $l-1$ рівностей
\begin{equation}\label{9f5.18}
\begin{split}
\partial^{\varkappa+j}_{t}u(x,t)=&
\sum_{\substack{|\alpha|+2b\beta\leq 2m+2bj,\\
|\alpha|\leq2m,\;\beta\leq\varkappa+j-1}}
a_{j}^{\alpha,\beta}(x,t)\,D^\alpha_x\partial^\beta_t u(x,t)+ \\
&+\partial^{j}_{t}((a^{(0,\ldots,0),\varkappa}(x,t))^{-1}f(x,t)),\\
&\mbox{де}\quad j=1,\ldots,l-1.
\end{split}
\end{equation}
Тут кожне $a_{j}^{\alpha,\beta}(x,t)$~--- деяка функція класу $C^{\infty}(\overline{\Omega})$. Рівності \eqref{9f5.17} і
\eqref{9f5.18} розглядаються у відкритому циліндрі $\Omega$. Оскільки $f\in H_{+}^{\sigma-2m,(\sigma-2m)/(2b)}(\Omega)$, то
$\partial^{j}_{t}f(x,t)\big|_{t=0}=0$ для майже всіх $x\in G$ і для кожного цілого $j\in\{0,\ldots,l-1\}$ згідно з лемою~\ref{9lem5.1} і двобічною нерівністю \eqref{9f5.16}. Використовуючи ці рівності, виводимо послідовно з формул \eqref{9f5.17} і \eqref{9f5.18}, що $\partial^{\varkappa+j}_{t}u(x,t)\big|_{t=0}=0$ для вказаних $x$ і~$j$.

Таким чином, функція $u\in H^{\sigma,\sigma/(2b)}(\Omega)$ задовольняє умову \eqref{9f5.4}. Тому $u\in H_{+}^{\sigma,\sigma/(2b)}(\Omega)$ за  лемою~\ref{9lem5.1}. Більше того, на підставі цієї леми та формул \eqref{9f5.14} і \eqref{9f5.15} виконуються нерівності
\begin{equation}\label{9f5.19}
\begin{split}
\|u\|_{H_{+}^{\sigma,\sigma/(2b)}(\Omega)}&\leq c_3
\|(f,g_1,\ldots,g_m)\|_
{\mathcal{H}_{+}^{\sigma-2m,(\sigma-2m)/(2b)}(\Omega,S)}\leq\\
&\leq c_4\,\|u\|_{H_{+}^{\sigma,\sigma/(2b)}(\Omega)}.
\end{split}
\end{equation}
Тут $c_3$ і $c_4$~--- деякі додатні числа, які не залежать від вектора \eqref{9f5.14} і функції~$u$.

Таким чином, для довільного вектора \eqref{9f5.14} існує єдиний розв'язок $u\in H_{+}^{\sigma,\sigma/(2b)}(\Omega)$ параболічної задачі \eqref{9f2.1}--\eqref{9f2.3} і цей розв'язок задовольняє двобічну нерівність~\eqref{9f5.19}. Очевидно, цей висновок еквівалентний теоремі~\ref{9th4.1} у випадку, що досліджується. Тому теорема~\ref{9th4.1} у досліджуваному випадку є наслідком результату М.~С.~Аграновіча і М.~І.~Вішика \cite[теорема~12.1]{AgranovichVishik64}.

У загальній ситуації виведемо теорему~\ref{9th4.1} із цього випадку за допомогою квадратичної інтерполяції з функціональним параметром. Отже, виберемо довільні число $\sigma>\sigma_0$ і функцію $\varphi\in\mathcal{M}$. Нехай ціле число $\sigma_1>\sigma$ таке, що $\sigma_1/(2b)\in\mathbb{Z}$. Згідно з теоремою Аграновіча--Вішика \cite[теорема 12.1]{AgranovichVishik64}, відображення \eqref{9f2.4} продовжується єдиним чином (за неперервністю) до ізоморфізмів
\begin{equation}\label{9f8.1}
\begin{gathered}
(A,B):H^{\sigma_k,\sigma_k/(2b)}_{+}(\Omega)\leftrightarrow
\mathcal{H}^{\sigma_k-2m,(\sigma_k-2m)/(2b)}_{+}(\Omega,S),
\\ \quad\mbox{де}\quad k\in\{0,1\}.
\end{gathered}
\end{equation}
Означимо інтерполяційний параметр $\psi$ за формулою \eqref{9f7.2}, в якій покладемо $s:=\sigma$, $s_{0}:=\sigma_{0}$ і $s_{1}:=\sigma_{1}$. Інтерполюючи з функціональним параметром $\psi$ оператори \eqref{9f8.1} отримаємо ізоморфізм
\begin{equation}\label{9f8.2}
\begin{split}
&(A,B):\bigl[H^{\sigma_0,\sigma_0/(2b)}_{+}(\Omega),
H^{\sigma_1,\sigma_1/(2b)}_{+}(\Omega)\bigr]_{\psi}\leftrightarrow\\
&\leftrightarrow
\bigl[\mathcal{H}^{\sigma_0-2m,(\sigma_0-2m)/(2b)}_{+}(\Omega,S),
\mathcal{H}^{\sigma_1-2m,(\sigma_1-2m)/(2b)}_{+}(\Omega,S)\bigr]_{\psi}.
\end{split}
\end{equation}
Він є звуженням оператора \eqref{9f8.1}, де $k=0$.

На підставі теореми~\ref{9lem7.3} маємо такі рівності просторів з  еквівалентністю норм у них:
\begin{equation*}
\bigl[H^{\sigma_0,\sigma_0/(2b)}_{+}(\Omega),
H^{\sigma_1,\sigma_1/(2b)}_{+}(\Omega)\bigr]_{\psi}=
H^{\sigma,\sigma/(2b);\varphi}_{+}(\Omega)
\end{equation*}
та
\begin{equation*}
\begin{aligned}
&[\mathcal{H}^{\sigma_0-2m,(\sigma_0-2m)/(2b)}_{+}(\Omega,S),
\mathcal{H}^{\sigma_1-2m,(\sigma_1-2m)/(2b)}_{+}(\Omega,S)]_{\psi}=\\
&=\bigl[H^{\sigma_0-2m,(\sigma_0-2m)/(2b)}_{+}(\Omega),
H^{\sigma_1-2m,(\sigma_1-2m)/(2b)}_{+}(\Omega)\bigr]_{\psi}\oplus\\
&\phantom{kk}\oplus\bigoplus_{j=1}^{m}
\bigl[H^{\sigma_0-m_j-1/2,(\sigma_0-m_j-1/2)/(2b)}_{+}(S),\\
&\phantom{kkkkkkkk}
H^{\sigma_1-m_j-1/2,(\sigma_1-m_j-1/2)/(2b)}_{+}(S)\bigr]_{\psi}=\\
&=H^{\sigma-2m,(\sigma-2m)/(2b);\varphi}_{+}(\Omega)\oplus\\
&\phantom{kk}\oplus\bigoplus_{j=1}^{m}
H^{\sigma-m_j-1/2,(\sigma-m_j-1/2)/(2b);\varphi}_{+}(S)=\\
&=\mathcal{H}^{\sigma-2m,(\sigma-2m)/(2b);\varphi}_{+}(\Omega,S).
\end{aligned}
\end{equation*}
Тут скористалися також теоремою~\ref{9prop6.3}. Отже, ізоморфізм \eqref{9f8.2} діє на парі просторів \eqref{9f4.1}. Він є продовженням за неперервністю відображення \eqref{9f2.4}, оскільки множина $C^{\infty}_{+}(\overline{\Omega})$ є щільною у просторі $H^{\sigma,\sigma/(2b);\varphi}_{+}(\Omega)$.
\end{proof}

\markright{\emph \ref{sec3.1.3}. Регулярність розв'язків}

\section[Регулярність розв'язків]
{Регулярність розв'язків}\label{sec3.1.3}

\markright{\emph \ref{sec3.1.3}. Регулярність розв'язків}

Як було показано у попередньому підрозділі, із теореми Аграновіча--Вішика \cite[теорема 12.1]{AgranovichVishik64} випливає, що відображення \eqref{9f2.4} продовжується за неперервністю до ізоморфізму
\begin{equation}\label{9f4.3}
(A,B):H^{\sigma_{0},\sigma_{0}/(2b)}_{+}(\Omega)\leftrightarrow
\mathcal{H}^{\sigma_{0}-2m,(\sigma_{0}-2m)/(2b)}_{+}(\Omega,S).
\end{equation}
Звісно, кожний ізоморфізм \eqref{9f4.1}, де $\sigma>\sigma_0$ і $\varphi\in\nobreak\mathcal{M}$, є звуженням оператора \eqref{9f4.3}.
Для довільного вектора
\begin{equation}\label{9f4.4}
(f,g_{1},...,g_{m})\in
\mathcal{H}^{\sigma_{0}-2m,(\sigma_{0}-2m)/(2b)}_{+}(\Omega,S)
\end{equation}
існує єдиний прообраз $u\in H^{\sigma_{0},\sigma_{0}/(2b)}_{+}(\Omega)$  відносно відображення \eqref{9f4.3}. Цю функцію $u$ називаємо узагальненим розв'язком параболічної задачі \eqref{9f2.1}--\eqref{9f2.3}, праві частини якої задовольняють умову~\eqref{9f4.4}.

Розглянемо застосування теореми \ref{9th4.1} до дослідження регулярності узагальненого розв'язку задачі \eqref{9f2.1}--\eqref{9f2.3}. Негайним наслідком цієї теореми є така властивість розв'язку:

\begin{corollary}\label{9cor4.2}
Припустимо, що функція $u\in H^{\sigma_{0},\sigma_{0}/(2b)}_{+}(\Omega)$ є узагальненим розв'язком параболічної задачі
\eqref{9f2.1}--\eqref{9f2.3}, праві частини якої задовольняють умову
\begin{equation*}
(f,g_{1},...,g_{m})\in
\mathcal{H}^{\sigma-2m,(\sigma-2m)/(2b);\varphi}_{+}(\Omega,S)
\end{equation*}
для деяких $\sigma>\sigma_0$ і  $\varphi\in\nobreak\mathcal{M}$. Тоді
$u\in H^{\sigma,\sigma/(2b);\varphi}_{+}(\Omega)$.
\end{corollary}

Сформулюємо локальний аналог цієї властивості. Нехай $U$~--- відкрита підмножина простору $\mathbb{R}^{n+1}$ така, що $\Omega_0:=U\cap\Omega\neq\varnothing$. Покладемо $\Omega':=U\cap\partial\overline{\Omega}$, $S_0:=U\cap S$ і $S':=U\cap \partial S$. Нехай $s>0$, $\gamma=1/(2b)$ і $\varphi\in\mathcal{M}$. Лінійний простір  $H^{s,s\gamma;\varphi}_{+,\mathrm{loc}}(\Omega_0,\Omega')$ означаємо так само, як і у випадку $n=1$, розглянутому у п.~\ref{sec3.2.1}. Позначимо через $H^{s,s\gamma;\varphi}_{+,\mathrm{loc}}(S_0,S')$ лінійний простір усіх розподілів $v$ на многовиді $S$ таких, що $\chi v\in H^{s,s\gamma;\varphi}_{+}(S)$ для кожної функції $\chi\in C^\infty (\overline S)$, яка задовольняє умову $\mathrm{supp}\,\chi\subset S_0\cup S'$.

\begin{theorem}\label{9th4.3}
Нехай функція $u\in H^{\sigma_0,\sigma_0/(2b)}_+(\Omega)$ є узагальненим розв'язком параболічної задачі \eqref{9f2.1}--\eqref{9f2.3}, праві частини якої задовольняють умову \eqref{9f4.4}. Нехай, крім того, $\sigma>\sigma_0$ і $\varphi\in\mathcal{M}$. Припустимо, що
\begin{equation}\label{9f4.5}
f\in H^{\sigma-2m,(\sigma-2m)/(2b);\varphi}_{+,\,\mathrm{loc}}(\Omega_0,\Omega')
\end{equation}
та
\begin{equation}\label{9f4.6}
\begin{gathered}
g_{j}\in H^{\sigma-m_j-1/2,(\sigma-m_j-1/2)/(2b);\varphi}_{+,\mathrm{loc}}
(S_{0},S')\\
\mbox{для кожного}\;\;j\in\{1,\dots,m\}.
\end{gathered}
\end{equation}
Тоді $u\in H^{\sigma,\sigma/(2b);\varphi}_{+,\mathrm{loc}}(\Omega_0,\Omega')$.
\end{theorem}

Якщо $\Omega_0=\Omega$ і $\Omega'=\partial\overline{\Omega}$ (тоді $S_0=S$ і $S'=\partial S$), то теорема~\ref{9th4.3} збігається з  наслідком~\ref{9cor4.2}. Якщо $\Omega'=\varnothing$, то ця теорема
стверджує, що регулярність розв'язку підвищується в околах внутрішніх точок замкненого циліндра $\overline{\Omega}$.

\begin{proof}[\indent Доведення теореми $\ref{9th4.3}$.]
Спочатку покажемо, що з її умов \eqref{9f4.5} і \eqref{9f4.6} випливає, що  імплікація
\begin{equation}\label{9f8.3}
\begin{split}
&u\in H^{\sigma-\lambda,(\sigma-\lambda)/(2b);\varphi}_{+,\mathrm{loc}}
(\Omega_0,\Omega')\Longrightarrow\\
&\Longrightarrow\;u\in H^{\sigma-\lambda+1,(\sigma-\lambda+1)/(2b);\varphi}_{+,\mathrm{loc}}
(\Omega_0,\Omega')
\end{split}
\end{equation}
істинна для кожного цілого числа $\lambda\geq1$, яке задовольняє нерівність $\sigma-\lambda+1>\sigma_{0}$.

Виберемо довільно функцію $\chi\in C^\infty(\overline\Omega)$, яка задовольняє умову $\mbox{supp}\,\chi\subset\Omega_0\cup\Omega'$.
Для $\chi$ існує функція $\eta\in C^\infty(\overline\Omega)$ така, що $\mbox{supp}\,\eta\subset\Omega_0\cup\Omega'$ і
$\eta=1$ в околі $\mbox{supp}\,\chi$. Переставивши диференціальні  оператори $A$ і $B_{j}$ з оператором множення на функцію $\chi$, отримаємо такі рівності:
\begin{equation}\label{9f8.4}
\begin{split}
(A,B)(\chi u)&=(A,B)(\chi\eta u)=\chi\,(A,B)(\eta u)+ (A',B')(\eta u)=\\
&=\chi\,(A,B)u+(A',B')(\eta u)=\\
&=\chi\,(f,g_{1},...,g_{m})+(A',B')(\eta u).
\end{split}
\end{equation}
Тут оператор
$$
(A',B'):=(A',B'_{1},\ldots,B'_{m})
$$
утворений за допомогою лінійного диференціального оператора
\begin{equation*}
A'(x,t,D_x,\partial_t)=\sum_{|\alpha|+2b\beta\leq 2m-1}a^{\alpha,\beta}_{1}(x,t)\,D^\alpha_x\partial^\beta_t
\end{equation*}
і $m$ крайових лінійних диференціальних операторів
\begin{equation*}
B_{j}'(x,t,D_x,\partial_t)=\sum_{|\alpha|+2b\beta\leq m_j-1}
b_{j,1}^{\alpha,\beta}(x,t)\,D^\alpha_x\partial^\beta_t,
\end{equation*}
де $j=1,\ldots,m$. Коефіцієнти цих диференціальних операторів задовольняють умови $a^{\alpha,\beta}_{1}\in C^{\infty}(\overline{\Omega})$ і $b_{j,1}^{\alpha,\beta}\in C^{\infty}(\overline{S})$. Оператор $(A',B')$ обмежений на парі просторів
\begin{equation}\label{9f8.7}
(A',B'):\,H^{s,s/(2b);\varphi}_{+}(\Omega)\rightarrow
\mathcal{H}^{s+1-2m,(s+1-2m)/(2b);\varphi}_{+}(\Omega,S)
\end{equation}
для кожного числа $s>\sigma_{0}-1$. Якщо $\varphi(\cdot)\equiv1$ і другі індекси не напівцілі, це є прямим наслідком відомих властивостей анізотропного простору Соболєва $H^{s,s/(2b)}(\Omega)$ (див., наприклад, \cite[розд.~I, лема~4 і розд.~II, теореми~3 і~7]{Slobodetskii58}). Звідси обмеженість оператора \eqref{9f8.7} у загальній ситуації випливає на підставі інтерполяційної теореми~\ref{9lem7.3}.

Згідно з умовами \eqref{9f4.5} і \eqref{9f4.6} виконується включення
$$
\chi\,(f,g_{1},...,g_{m})
\in\mathcal{H}^{\sigma-2m,(\sigma-2m)/(2b);\varphi}_{+}(\Omega,S).
$$
Крім того, згідно з \eqref{9f8.7}, де $s:=\sigma-\lambda$, істинною є імплікація
\begin{equation*}
\begin{split}
&u\in H^{\sigma-\lambda,(\sigma-\lambda)/(2b);\varphi}_{+,\mathrm{loc}}
(\Omega_0,\Omega')\Longrightarrow\\
&\Longrightarrow\;(A',B')(\eta u)\in
\mathcal{H}^{\sigma-\lambda+1-2m,(\sigma-\lambda+1-2m)/(2b);\varphi}_{+}
(\Omega,S).
\end{split}
\end{equation*}
Тому на підставі рівностей \eqref{9f8.4} і наслідку~\ref{9cor4.2} робимо висновок, що
\begin{equation*}
\begin{split}
&u\in H^{\sigma-\lambda,(\sigma-\lambda)/(2b);\varphi}_{+,\mathrm{loc}}
(\Omega_0,\Omega')\Longrightarrow\\
&\Longrightarrow\;(A,B)(\chi u)\in
\mathcal{H}^{\sigma-\lambda+1-2m,(\sigma-\lambda+1-2m)/(2b);\varphi}_{+}
(\Omega,S)\Longrightarrow\\
&\Longrightarrow\;\chi u\in
H^{\sigma-\lambda+1,(\sigma-\lambda+1)/(2b);\varphi}_{+}(\Omega).
\end{split}
\end{equation*}
Тут наслідок~\ref{9cor4.2} застосовний, оскільки
$\chi u\in H^{\sigma_0,\sigma_0/(2b)}_{+}(\Omega)$
за умовою теореми і з огляду на нерівність $\sigma-\lambda+1>\sigma_{0}$.
Таким чином, потрібну імплікацію \eqref{9f8.3} доведено, якщо зважити на зроблений вибір функції $\chi$.

Використаємо цю імплікацію для доведення включення $u\in\nobreak H^{\sigma,\sigma/(2b);\varphi}_{+,\mathrm{loc}}(\Omega_0,\Omega')$. Розглянемо окремо випадки, коли $\sigma\notin\mathbb{Z}$ і коли $\sigma\in\mathbb{Z}$.

Дослідимо перший з них. Якщо $\sigma\notin\mathbb{Z}$, то існує ціле число $\lambda_{0}\geq1$ таке, що
\begin{equation}\label{9f8.8}
\sigma-\lambda_{0}<\sigma_{0}<\sigma-\lambda_{0}+1.
\end{equation}
Використовуючи імплікацію \eqref{9f8.3} послідовно для значень $\nobreak{\lambda:=\lambda_{0}}$, $\lambda:=\lambda_{0}-1$, ...,$\lambda:=1$, виводимо потрібне включення у такий спосіб:
\begin{equation*}
\begin{split}
&u\in H^{\sigma_0,\sigma_0/(2b)}_{+}(\Omega)\subset
H^{\sigma-\lambda_{0},(\sigma-\lambda_{0})/(2b);\varphi}_
{+,\mathrm{loc}}(\Omega_0,\Omega')\Longrightarrow\\
&\Longrightarrow\;u\in
H^{\sigma-\lambda_{0}+1,(\sigma-\lambda_{0}+1)/(2b);\varphi}_
{+,\mathrm{loc}}(\Omega_0,\Omega')\Longrightarrow\\
&\Longrightarrow\;\ldots\;\Longrightarrow\;u\in H^{\sigma,\sigma/(2b);\varphi}_
{+,\mathrm{loc}}(\Omega_0,\Omega').
\end{split}
\end{equation*}
Зауважимо, що $u\in H^{\sigma_0,\sigma_0/(2b)}_{+}(\Omega)$ за умовою теореми.

Дослідимо тепер випадок, коли $\sigma\in\mathbb{Z}$. У цьому випадку не існує такого цілого числа $\lambda_{0}$, яке задовольняє умову~\eqref{9f8.8}. Але, оскільки $\sigma-1/2\notin\mathbb{Z}$ і $\sigma-1/2>\sigma_{0}$, то, як доведено вище, виконується включення
\begin{equation*}
u\in H^{\sigma-1/2,(\sigma-1/2)/(2b);\varphi}_
{+,\mathrm{loc}}(\Omega_0,\Omega').
\end{equation*}
Звідси, скориставшись імплікацією \eqref{9f8.3}, де  $\lambda:=1$,
виводимо потрібне включення у такий спосіб:
\begin{align*}
u&\in H^{\sigma-1/2,(\sigma-1/2)/(2b);\varphi}_
{+,\mathrm{loc}}(\Omega_0,\Omega')\subset
H^{\sigma-1,(\sigma-1)/(2b);\varphi}_
{+,\mathrm{loc}}(\Omega_0,\Omega')\Longrightarrow\\
&\Longrightarrow\;u\in H^{\sigma,\sigma/(2b);\varphi}_
{+,\mathrm{loc}}(\Omega_0,\Omega').
\end{align*}
\end{proof}

За допомогою узагальнених анізотропних просторів Соболєва можна сформулювати тонкі достатні умови, за яких узагальнений розв'язк $u$ напіводнорідної параболічної задачі та його узагальнені частинні похідні заданого порядку неперервні на множині $\Omega_0\cup\Omega'$.

\begin{theorem}\label{9th4.4}
Нехай функція $u\in H^{\sigma_0,\sigma_0/(2b)}_+(\Omega)$ є узагальненим розв'язком параболічної задачі \eqref{9f2.1}--\eqref{9f2.3}, праві частини якої задовольняють умову~\eqref{9f4.4}. Нехай ціле число $p\geq0$ таке, що $p+b+n/2>\sigma_{0}$. Припустимо, що праві частини цієї задачі задовольняють умови \eqref{9f4.5} і \eqref{9f4.6} для $\sigma:=p+b+n/2$ і деякого функціонального параметра $\varphi\in\mathcal{M}$, підпорядкованого умові \eqref{9f4.7}. Тоді розв'язок $u(x,t)$ і кожна його узагальнена частинна похідна
$D_{x}^{\alpha}\partial_{t}^{\beta}u(x,t)$, де $|\alpha|+2b\beta\leq p$, є неперервними на множині $\Omega_0\cup\Omega'$.
\end{theorem}

\begin{remark}\label{rem-continuous-Murach}\rm
Стосовно висновку цієї теореми та подібних тверджень, домовимося про таке: узагальнену функцію $v\in\mathcal{D}'(\Omega)$ називаємо неперервною на множині $\Omega_0\cup\Omega'$, якщо існує неперервна функція $v_{0}$ на $\Omega_0\cup\Omega'$ така, що
\begin{equation}\label{rem-to-th4.4}
v(\omega)=\int\limits_{\Omega_0}v_{0}(x,t)\,\omega(x,t)\,dxdt
\end{equation}
для довільної функції $\omega\in C^{\infty}(\Omega)$, носій якої задовольняє умову $\mathrm{supp}\,\omega\subset\Omega_0$.
Тут $v(\omega)$~--- значення функціонала $v$ на функції~$\omega$. Звісно, якщо $\Omega'=\varnothing$, то це означення еквівалентно тому, що звуження узагальненої функції $v$ на відкриту множину $\Omega_0$ є неперервним на цій множині.
\end{remark}

\begin{proof}[\indent Доведення теореми $\ref{9th4.4}$.]
Виберемо достатньо мале число $\varepsilon>\nobreak0$ і позначимо
$$
U_{\varepsilon}:=\{x\in U:\mathrm{dist}(x,\partial U)>\varepsilon\},
$$
$\Omega_{\varepsilon}:=U_{\varepsilon}\cap\Omega$ та $\Omega'_{\varepsilon}:=U_{\varepsilon}\cap\partial\overline{\Omega}$. Розглянемо функцію $\chi_{\varepsilon}\in C^\infty(\overline\Omega)$ таку, що $\mbox{supp}\,\chi_{\varepsilon}\subset\Omega_0\cup\Omega'$ і $\chi_{\varepsilon}=1$ на $\Omega_{\varepsilon}\cup\Omega'_{\varepsilon}$. Згідно з теоремою~\ref{9th4.3} виконується включення $\chi_{\varepsilon}u\in H^{\sigma,\sigma/(2b);\varphi}_{+}(\Omega)$, де  $\sigma=p+b+n/2$, а $\varphi$ задовольняє \eqref{9f4.7}. Отже, існує функція $w_{\varepsilon}\in H^{\sigma,\sigma/(2b);\varphi}(\mathbb{R}^{n+1})$ така, що $w_{\varepsilon}=\chi_{\varepsilon}u=u$ на $\Omega_{\varepsilon}$. Нехай мультиіндекс $\alpha=(\alpha_{1},\ldots,\alpha_{n})$ і цілий індекс $\beta\geq0$ задовольняють умову $|\alpha|+2b\beta\leq p$. За  теоремою~\ref{9lem8.1}~(i) узагальнена частинна похідна $D_{x}^{\alpha}\partial_{t}^{\beta}w_{\varepsilon}(x,t)$ є неперервною на множині $\mathbb{R}^{n+1}$. (Зокрема, там неперервна і сама функція $w_{\varepsilon}(x,t)$ як похідна порядку нуль.) Отже, узагальнена функція $v(x,t):=D_{x}^{\alpha}\partial_{t}^{\beta}u(x,t)$ (яка належить до $\mathcal{D}'(\Omega)$) задовольняє умову
\begin{equation*}
v(\omega)=\int\limits_{\Omega_\varepsilon}v_{\varepsilon}(x,t)\,
\omega(x,t)\,dxdt
\end{equation*}
для будь-якої функції $\omega\in C^{\infty}(\Omega)$ такої, що $\mathrm{supp}\,\omega\subset\Omega_\varepsilon$. Тут $v_{\varepsilon}$ позначає неперервну функцію $v_{\varepsilon}(x,t):=
D_{x}^{\alpha}\partial_{t}^{\beta}w_{\varepsilon}(x,t)$ аргументів $(x,t)\in\Omega_{\varepsilon}\cup\Omega'_{\varepsilon}$. Означимо неперервну функцію $v_{0}$ на $\Omega_{0}\cup\Omega'_{0}$ за формулою $v_{0}:=v_{\varepsilon}$ на $\Omega_{\varepsilon}\cup\Omega'_{\varepsilon}$, де $0<\varepsilon\ll1$. Це означення коректне, бо з нерівності $0<\delta<\varepsilon$ випливає, що $v_{\delta}=v_{\varepsilon}$ на $\Omega_{\varepsilon}\cup\Omega'_{\varepsilon}$. Узагальнена функція $v$ задовольняє умову \eqref{rem-to-th4.4} для будь-якої функції $\omega\in C^{\infty}(\Omega)$ такої, що $\mathrm{supp}\,\omega\subset\Omega_0$, оскільки $\mathrm{supp}\,\omega\subset\Omega_\varepsilon$ для достатньо малого числа $\varepsilon>0$ (залежного від $\omega$). Отже, узагальнена функція $v(x,t)=D_{x}^{\alpha}\partial_{t}^{\beta}u(x,t)$ неперервна на множині $\Omega_{0}\cup\Omega'_{0}$.
\end{proof}

\begin{remark}\label{9rem4.5}\rm
Інтегральна умова \eqref{9f4.7} у теоремі~\ref{9th4.4} є точною. А саме: нехай $\sigma:=p+b+n/2$ і $\varphi\in\mathcal{M}$ та припустимо, що для кожної функції $u\in\nobreak H^{\sigma_0,\sigma_0/(2b)}_+(\Omega)$ істинна така імплікація:
\begin{gather*}
\bigl(\,u\;\mbox{є розв'язком задачі \eqref{9f2.1}--\eqref{9f2.3},}\\
\mbox{праві частини якої задовольняють \eqref{9f4.5} і \eqref{9f4.6}}\,\bigr)\Longrightarrow\\
\Longrightarrow\;\bigl(\,u\;\mbox{задовольняє висновок теореми \ref{9th4.4}}\,\bigr);
\end{gather*}
тоді $\varphi$ задовольняє умову \eqref{9f4.7}.
\end{remark}

Справді, нехай $\varphi\in\mathcal{M}$, а ціле число $p\geq0$ таке, що  $$
\sigma:=p+b+n/2>\sigma_{0}.
$$
Припустимо, що кожна функція $u\in\nobreak H^{\sigma_0,\sigma_0/(2b)}_{+}(\Omega)$ задовольняє імплікацію, наведену в зауваженні~\ref{9rem4.5}. Отже, довільна функція $u\in H^{\sigma,\sigma/(2b);\varphi}_{+}(\Omega)$ задовольняє висновок теореми~\ref{9th4.4}, якщо покласти
\begin{equation*}
(f,g_{1},...,g_{m}):=(A,B)u\in
\mathcal{H}^{\sigma-2m,(\sigma-2m)/(2b);\varphi}_{+}(\Omega,S).
\end{equation*}
Тому, якщо функція $u$ належить до $H^{\sigma,\sigma/(2b);\varphi}_{+}(\Omega)$, то, зокрема, її узагальнена похідна $\partial^{p}u/\partial x_{1}^{j}$ є неперервною на $\Omega_0\cup\Omega'$.

Нехай $V$~--- непорожня відкрита підмножина $\mathbb{R}^{n+1}$ така, що $\overline{V}\subset\Omega_0$. Довільно виберемо функцію $w\in H^{\sigma,\sigma/(2b);\varphi}(\mathbb{R}^{n+1})$, яка задовольняє умову $\mathrm{supp}\,w\subset V$. Покладемо
\begin{equation*}
u:=w\!\upharpoonright\!\Omega\in H^{\sigma,\sigma/(2b);\varphi}_{+}(\Omega).
\end{equation*}
Узагальнена похідна $\partial^{p}w/\partial x_{1}^{j}$ неперервна на $\mathbb{R}^{n+1}$ внаслідок властивості $u$, вказаної у попередньому абзаці. Отже, $\varphi$ задовольняє умову \eqref{9f4.7} згідно з теоремою~\ref{9lem8.1}~(ii). Зауваження~\ref{9rem4.5} обґрунтоване.

\begin{remark}\rm
Використання узагальнених просторів Соболєва дозволяє в теоремі~\ref{9th4.4} досягти мінімальних допустимих значень числових показників гладкості в умовах \eqref{9f4.5} і \eqref{9f4.6}. Це неможливо для соболєвських просторів, оскільки функція $\varphi(\cdot)\equiv1$ не задовольняє інтегральну умову \eqref{9f4.7}. Якщо формулювати для них аналог теореми~\ref{9th4.4}, то доведеться замінити цю умову на більш сильну. Потрібно припускати, що для деякого числа $\sigma>p+b+n/2$ праві частини задачі \eqref{9f2.1}--\eqref{9f2.3} задовольняють умови \eqref{9f4.5}, \eqref{9f4.6}, в яких $\varphi(\cdot)\equiv1$. Аналогічне зауваження стосується теорем~\ref{16th4.4} і~\ref{24th4.1}, наведених у підрозділах~\ref{sec4.1.3} і~\ref{4.1.5} відповідно.
\end{remark}

\begin{remark}\rm
В умові \eqref{9f4.5} і в кожній з $m$ умов \eqref{9f4.6}, наявних у теоремі~\ref{9th4.4}, можна використовувати різні функціональні параметри $\varphi$ класу $\mathcal{M}$, які задовольняють~\eqref{9f4.7}. При цьому висновок теореми буде правильним. Справа у тому, що функція, яка є мінімумом значень цих функціональних параметрів у кожній точці півосі $[1,\infty)$, належить до $\mathcal{M}$ і також задовольняє~\eqref{9f4.7}. Крім того, якщо виконуються умови \eqref{9f4.5} і \eqref{9f4.6} для різних  $\varphi$, то вони виконуються і для функції-мінімуму цих $\varphi$. Отже, більш загальна ситуація, коли в цих умовах параметри $\varphi$ різні, зводиться до розглянутого випадку, коли в них параметр $\varphi$ один і той самий. Аналогічне зауваження стосується теорем~\ref{16th4.4} і~\ref{24th4.1}.
\end{remark}

\newpage

\chaptermark{}


\chapter{\textbf{Неоднорідні параболічні задачі}}\label{ch4}

\chaptermark{\emph Розд. \ref{ch4}. Неоднорідні параболічні задачі}

У цьому розділі досліджуємо параболічні задачі з неоднорідними початковими умовами. Спочатку розглядаємо основні крайові задачі для рівняння теплопровідності: задачу Діріхле і задачу Неймана. Потім вивчаємо параболічні задачі, одновимірні за просторовою змінною, і зрештою~--- багатовимірні параболічні задачі у циліндрі. Центральне місце у цьому розділі посідають теореми про ізоморфізми, породжені вказаними задачами на парах анізотропних узагальнених просторів Соболєва, введених у п.~\ref{sec2.3}. Іншими словами, ці теореми стверджують, що вказані задачі є коректно розв'язними (за Адамаром) на парах зазначених просторів. Теореми про ізоморфізми застосуємо до дослідження локальної регулярності (аж до межі циліндра) розв'язків параболічних задач у зазначених просторах, а також в анізотропних просторах неперервно диференційовних функцій. Зокрема, отримаємо тонкі достатні умови класичності узагальнених розв'язків параболічних задач.

\section[Задачі для рівняння теплопровідності]
{Задачі для рівняння теплопровідності}\label{sec4.4.1}

\markright{\emph \ref{sec4.4.1}. Задачі для рівняння теплопровідності}

Як і в підрозділі~\ref{sec3.2.1}, додатні числа $l$ і $\tau$ вибрано довільно, а $\Omega:=(0,l)\times(0,\tau)$~--- відкритий прямокутник на декартовій площині $\mathbb{R}^{2}$ точок $(x,t)$, де $x$~--- просторова, а $t$~--- часова змінні. У прямокутнику $\Omega$ розглядаємо рівняння теплопровідності
\begin{equation}\label{6f1}
\begin{aligned}
Au(x,t):=&u_{t}'(x,t)-u_{xx}''(x,t)=f(x,t)\\
&\mbox{для всіх}\;\;x\in(0,l)\;\;\mbox{і}\;\;t\in(0,\tau).
\end{aligned}
\end{equation}
Досліджуємо початково-крайову параболічну задачу, яка складається з рівняння~\eqref{6f1}, початкової умови
\begin{equation}\label{6f3}
u(x,0)=h(x)\quad\mbox{для всіх}\;\;x\in (0,l)
\end{equation}
і крайової умови Діріхле
\begin{equation}\label{6f2}
u(0,t)=g_0(t)\quad\mbox{і}\quad u(l,t)=g_1(t)
\quad\mbox{для всіх}\;\;t\in(0,\tau),
\end{equation}
або крайової умови Неймана
\begin{equation}\label{6f2n}
u_x'(0,t)=g_0(t)\quad\mbox{і}\quad u_x'(l,t)=g_1(t)
\quad\mbox{для всіх}\;\;t\in(0,\tau).
\end{equation}

Для того, щоб існував достатньо регулярний розв'язок $u(x,t)$ задачі \eqref{6f1}, \eqref{6f3}, \eqref{6f2} (або \eqref{6f2n}),
її праві частини повинні задовольняти природні умови узгодження. Вони полягають у тому, що похідні $u^{(k)}_t(x,0)$, які можна обчислити за допомогою параболічного рівняння \eqref{6f1} і початкової умови \eqref{6f3}, задовольняють крайову умову \eqref{6f2} (або \eqref{6f2n}) та співвідношення, що утворюються внаслідок диференціювання крайової умови за часовою змінною~$t$. Сформулюємо ці умови узгодження, скориставшись деякими анізотропними просторами Соболєва.

Почнемо із задачі \eqref{6f1}--\eqref{6f2}. Пов'яжемо з нею лінійне відображення
\begin{equation}\label{6f7}
u\mapsto\Lambda_{0}u:=
(Au,u(0,\cdot),u(l,\cdot),u(\cdot,0)),\;\,
\mbox{де}\;\,u\in C^{\infty}(\overline{\Omega}).
\end{equation}
Звісно, тут $u(0,\cdot)$ і $u(l,\cdot)$ позначають відповідно функції $u(0,t)$ і $u(l,t)$ аргументу $t\in[0,\tau]$, а $u(\cdot,0)$~--- функцію $u(x,0)$ аргументу $x\in[0,l]$. Отже,
$$
u\in C^{\infty}(\overline{\Omega})\,\Longrightarrow\,
\Lambda_{0}u\in C^{\infty}(\overline{\Omega})\times
(C^{\infty}[0,\tau])^{2}\times C^{\infty}[0,l].
$$

Нехай дійсне число $s\geq2$. Відображення \eqref{6f7} продовжується єдиним чином (за неперервністю) до лінійного обмеженого оператора
\begin{equation}\label{6f7-BoundOper}
\begin{split}
\Lambda_0:\,H^{s,s/2}(\Omega)\rightarrow &H^{s-2,s/2-1}(\Omega)\oplus\\
&\oplus(H^{s/2-1/4}(0,\tau))^2\oplus H^{s-1}(0,l).
\end{split}
\end{equation}
Такий висновок робимо на підставі відомих властивостей анізотропних просторів Соболєва \cite[розд.~II, теореми~3 і~7]{Slobodetskii58}.
Виберемо довільно функцію $u(x,t)$ з простору $H^{s,s/2}(\Omega)$ і означимо праві частини
\begin{equation}\label{6f7-FGH}
f\in H^{s-2,s/2-1}(\Omega),\,\, g_0,g_1\in H^{s/2-1/4}(0,\tau)\,\,
\mbox{і}\,\, h\in H^{s-1}(G)
\end{equation}
задачі за формулою $(f,g_0,g_1,h):=\Lambda_0u$ з використанням  оператора~\eqref{6f7-BoundOper}.

Згідно з \cite[розд.~II, теорема 7]{Slobodetskii58} сліди
$$
u^{(k)}_t(\cdot,0)\in H^{s-1-2k}(0,l)
$$
означені за замиканням для всіх цілих чисел $k$, які задовольняють умову $0\leq k<s/2-1/2$ (і лише для цих $k$). Скориставшись \eqref{6f1} і \eqref{6f3}, виразимо ці сліди через функції $f(x,t)$ і $h(x)$ за рекурентною формулою
\begin{equation}\label{6f9}
\begin{split}
&u^{(0)}_t(x,0)=u(x,0)=h(x),\\
&u^{(k)}_t(x,0)=(u^{(k-1)}_t(x,0))''_{xx}+f^{(k-1)}_t(x,0)\\
&\mbox{для кожного}\;\,k\in\mathbb{Z}\;\,\mbox{такого, що}\;\,
1\leq k<s/2-1/2.
\end{split}
\end{equation}
Окрім того, означено сліди $g^{(k)}_0(0)$ і $g^{(k)}_1(0)$ для всіх  $k\in\mathbb{Z}$ таких, що $0\leq k<s/2-3/4$ (і тільки для цих $k$); це випливає з~\eqref{6f7-FGH}. Тому внаслідок крайової умови Діріхле \eqref{6f2} виконуються рівності
\begin{equation}\label{6fcond}
g^{(k)}_0(0)=u^{(k)}_t(0,0)\quad\mbox{і}\quad
g^{(k)}_1(0)=u^{(k)}_t(l,0)
\end{equation}
для цих цілих $k$. Праві частини рівностей \eqref{6fcond} означені коректно, оскільки функція $u^{(k)}_t(\cdot,0)\in H^{s-1-2k}(0,l)$ має сліди
$u^{(k)}_t(0,0)$ і $u^{(k)}_t(l,0)$ за умови $s-3/2-2k>0$.

Тепер, підставивши співвідношення \eqref{6f9} у \eqref{6fcond}, отримуємо умови узгодження
\begin{equation}\label{6f10}
\begin{gathered}
g^{(k)}_{0}(0)=v_k(0)\;\;\mbox{і}\;\;g^{(k)}_{1}(0)=v_k(l)\\
\mbox{для всіх}\;\,k\in\mathbb{Z}\;\,\mbox{таких, що}\;\,
0\leq k<s/2-3/4.
\end{gathered}
\end{equation}
Тут функції $v_k$ означено за рекурентною формулою
\begin{equation}\label{6f9rec}
\begin{split}
&v_0(x)=h(x),\\
&v_k(x)=(v_{k-1}(x))''_{xx}+f^{(k-1)}_t(x,0)\\
\mbox{для кожного}&\;\,k\in\mathbb{Z}\;\,\mbox{такого, що}\;\,
1\leq k<s/2-1/2.
\end{split}
\end{equation}
З огляду на \eqref{6f7-FGH} виконується включення
\begin{equation}\label{6fv-k}
v_k\in H^{s-1-2k}(0,l)\;\;\mbox{для кожного}\;\;
k\in\mathbb{Z}\cap[0,s/2-1/2).
\end{equation}
Тому сліди $v_k(0)$ і $v_k(l)$ означені, якщо $s-3/2-2k>0$. Отже, умови узгодження \eqref{6f10} сформульовано коректно.

Покладемо
$$
E_{0}:=\{2r+3/2:1\leq r\in\mathbb{Z}\}.
$$
Як бачимо, $E_{0}$~--- множина всіх розривів функції, яка дорівнює числу  умов узгодження \eqref{6f10} залежно від $s\geq2$.

Вкажемо узагальнені соболєвські простори, на яких параболічна задача
\eqref{6f1}--\eqref{6f2} породжує ізоморфізми. Нехай $s>2$ і $\varphi\in\mathcal{M}$. Беремо простір $H^{s,s/2;\varphi}(\Omega)$ як область визначення породженого ізоморфізму. Отже, $H^{s,s/2;\varphi}(\Omega)$ є простором розв'язків цієї задачі. Введемо простір її правих частин, тобто область значень ізоморфізму, яку позначаємо через $\widetilde{\mathcal{Q}}_{0}^{s-2,s/2-1;\varphi}$. Вона є лінійним многовидом у гільбертовому просторі
\begin{align*}
&\widetilde{\mathcal{H}}_{0}^{s-2,s/2-1;\varphi}:=\\
&:=H^{s-2,s/2-1;\varphi}(\Omega)\oplus
(H^{s/2-1/4;\varphi}(0,\tau))^2\oplus H^{s-1;\varphi}(0,l).
\end{align*}
Останній у соболєвському випадку $\varphi(\cdot)\equiv1$ збігається з простором, у який діє оператор \eqref{6f7-BoundOper}.

Гільбертів простір $\widetilde{\mathcal{Q}}_{0}^{s-2,s/2-1;\varphi}$ означаємо по-різному у випадку, коли $s\notin E_{0}$, і у випадку, коли $s\in E_{0}$.

Розглянемо спочатку випадок, коли $s\notin E_{0}$. За означенням лінійний простір $\widetilde{\mathcal{Q}}_{0}^{s-2,s/2-1;\varphi}$ складається з усіх векторів
\begin{equation}\label{(f,g,h)-Dirichlet-Murach}
\bigl(f,g_0,g_1,h\bigr)\in\widetilde{\mathcal{H}}_{0}^{s-2,s/2-1;\varphi},
\end{equation}
які задовольняють умови узгодження \eqref{6f10}. Останні сформульовано коректно для кожного вектора
$$
\bigl(f,g_0,g_1,h\bigr)\in
\widetilde{\mathcal{H}}_{0}^{s-2-\varepsilon,s/2-1-\varepsilon/2},
$$
де число $\varepsilon>0$ достатньо мале. Тому ці умови мають сенс для будь-якого вектора \eqref{(f,g,h)-Dirichlet-Murach} з огляду на неперервне вкладення
\begin{equation}\label{v8f69a-1}
\widetilde{\mathcal{H}}_{0}^{s-2,s/2-1;\varphi}
\hookrightarrow\widetilde{\mathcal{H}}_{0}^
{s-2-\varepsilon,s/2-1-\varepsilon/2},
\end{equation}
яке випливає з формул \eqref{8f5a} і \eqref{8f5b}. Отже, лінійний простір $\widetilde{\mathcal{Q}}_{0}^{s-2,s/2-1;\varphi}$ означено коректно. Наділяємо його скалярним добутком і нормою з гільбертового простору
$\widetilde{\mathcal{H}}_{0}^{s-2,s/2-1;\varphi}$. Простір $\widetilde{\mathcal{Q}}_{0}^{s-2,s/2-1;\varphi}$ є повним, тобто гільбертовим. Це випливає з рівності
$$
\widetilde{\mathcal{Q}}_{0}^{s-2,s/2-1;\varphi}=
\widetilde{\mathcal{H}}_{0}^{s-2,s/2-1;\varphi}\cap
\widetilde{\mathcal{Q}}_{0}^{s-2-\varepsilon,s/2-1-\varepsilon/2}.
$$
Тут простір $\widetilde{\mathcal{Q}}_{0}^{s-2-\varepsilon,s/2-1-\varepsilon/2}$ є повним, оскільки диференціальні оператори та оператори слідів, використані в умовах узгодження, обмежені на відповідних парах просторів Соболєва. Тому простір, записаний у правій частині цієї рівності, є повним відносно суми норм у просторах, що є компонентами перетину. Ця сума еквівалентна нормі у $\widetilde{\mathcal{H}}_{0}^{s-2,s/2-1;\varphi}$ на підставі
\eqref{v8f69a-1}. Таким чином, простір  $\widetilde{\mathcal{Q}}_{0}^{s-2,s/2-1;\varphi}$ є повним (відносно останньої норми).

Якщо $s\in E_{0}$, то гільбертів простір $\widetilde{\mathcal{Q}}_{0}^{s-2,s/2-1;\varphi}$ означаємо
за допомогою квадратичної інтерполяції у такий спосіб:
\begin{equation}\label{6f-interp1}
\widetilde{\mathcal{Q}}_{0}^{s-2,s/2-1;\varphi}:=\bigl[
\widetilde{\mathcal{Q}}_{0}^{s-2-\varepsilon,s/2-1-\varepsilon/2;\varphi},
\widetilde{\mathcal{Q}}_{0}^{s-2+\varepsilon,s/2-1+\varepsilon/2;\varphi}
\bigr]_{1/2}.
\end{equation}
Тут число $\varepsilon\in(0,1/2)$ вибрано довільно, а права частина рівності є результатом квадратичної інтерполяції з числовим параметром $1/2$ (тобто із степеневим параметром $\psi(r)\equiv r^{1/2}$). Гільбертів простір $\widetilde{\mathcal{Q}}_{0}^{s-2,s/2-1;\varphi}$, означений  формулою \eqref{6f-interp1}, не залежить з точністю до еквівалентності норм від вибору числа $\varepsilon$ і неперервно вкладається у простір $\widetilde{\mathcal{H}}_{0}^{s-2,s/2-1;\varphi}$. Це покажемо у зауваженні~\ref{v8rem2-6} наприкінці поточного підрозділу.

Сформулюємо теорему про ізоморфізми, породжені параболічною початково-крайовою задачею \eqref{6f1}--\eqref{6f2}.

\begin{theorem}\label{6th1}
Для довільних дійсного числа $s>2$ і функціонального параметра $\varphi\in\nobreak\mathcal{M}$ відображення \eqref{6f7} продовжується єдиним чином (за неперервністю) до ізоморфізму
\begin{equation}\label{6f8}
\Lambda_{0}:H^{s,s/2;\varphi}(\Omega)\leftrightarrow
\widetilde{\mathcal{Q}}_{0}^{s-2,s/2-1;\varphi}.
\end{equation}
\end{theorem}

Зауважимо, що потреба означити простір $\widetilde{\mathcal{Q}}_{0}^{s-2,s/2-1;\varphi}$ за формулою \eqref{6f-interp1} у випадку, коли $s\in E_{0}$, зумовлена такою обставиною: якщо означити цей простір у цьому випадку в той самий спосіб, що і для $s\notin E_{0}$, то висновок цієї теореми буде хибним, принаймні, коли $\varphi(\cdot)\equiv1$. Це випливає з результату В.~А.~Солоннікова \cite[\S~6, с.~186]{Solonnikov64} (див. також \cite[зауваження~6.4]{LionsMagenes72ii}).

Доведення сформульованої теореми (як і її версії для крайових умов Неймана) дамо наприкінці цього підрозділу.

Розглянемо тепер параболічну задачу \eqref{6f1}, \eqref{6f3}, \eqref{6f2n} (з крайовими умовами Неймана). Запишемо умови узгодження для її правих частин. Пов'яжемо з нею лінійне відображення:
\begin{gather}\label{6f11}
u\mapsto\Lambda_1u:=(Au,u'_x(0,\cdot),u'_x(l,\cdot),u(\cdot,0)),
\;\,\mbox{де}\;\,u\in C^{\infty}(\overline{\Omega}).
\end{gather}
Тут, нагадаємо, $u=u(x,t)$, а $u'_x(0,\cdot)$ і $u'_x(l,\cdot)$
позначають відповідно функції $u'_x(0,t)$ і $u'_x(l,t)$ аргументу $t\in[0,\tau]$. Отже,
$$
u\in C^{\infty}(\overline{\Omega})\,\Longrightarrow\,
\Lambda_{1}u\in C^{\infty}(\overline{\Omega})\times
(C^{\infty}[0,\tau])^{2}\times C^{\infty}[0,l].
$$

Для довільного дійсного числа $s\geq2$ це відображення продовжується єдиним чином (за неперервністю) до лінійного обмеженого оператора
\begin{equation}\label{6f7N-BoundOper}
\begin{split}
\Lambda_1:\,H^{s,s/2}(\Omega)\rightarrow&
H^{s-2,s/2-1}(\Omega)\\ &\oplus
(H^{s/2-3/4}(0,\tau))^2\oplus H^{s-1}(0,l).
\end{split}
\end{equation}
Виберемо довільно функцію $u(x,t)$ з простору $H^{s,s/2}(\Omega)$ і означимо праві частини
\begin{equation*}
f\in H^{s-2,s/2-1}(\Omega),\quad g_0,g_1\in H^{s/2-3/4}(0,\tau)\quad
\mbox{і}\quad h\in H^{s-1}(0,l)
\end{equation*}
розглянутої задачі за формулою $(f,g_0,g_1,h):=\Lambda_1u$, використовуючи оператор~\eqref{6f7N-BoundOper}.

Тут (на відміну від \eqref{6f7-FGH}) з включення $u\in H^{s,s/2}(\Omega)$ випливає, що функції $g_0(\cdot)=u'_x(0,\cdot)$ і $g_1(\cdot)=u'_x(l,\cdot)$ належать до простору $H^{s/2-3/4}(0,\tau)$ згідно з \cite[розд.~II, теорема 7]{Slobodetskii58}. Отже, сліди $g^{(k)}_0(0)$ і $g^{(k)}_1(0)$ означено для всіх $k\in\mathbb{Z}$ таких, що $0\leq k<s/2-5/4$ (і лише для цих~$k$). Ці сліди можна виразити через функцію $u(x,t)$ та її похідні за часом, а саме:
\begin{equation}\label{6fNcond}
g^{(k)}_0(0)=(u^{(k)}_t(x,0))'_x\big|_{x=0}\quad\mbox{і}\quad
g^{(k)}_1(0)=(u^{(k)}_t(x,0))'_x\big|_{x=l}.
\end{equation}
Тут функція $u^{(k)}_t(x,0)$ змінної $x\in(0,l) $ виражена через функції $f(x,t)$ і $h(x)$ за рекурентною формулою \eqref{6f9}.

Підставивши \eqref{6f9} у праві частини рівностей \eqref{6fNcond}, отримаємо умови узгодження:
\begin{equation}\label{6f14}
\begin{gathered}
g^{(k)}_{0}(0)=(v_k)'_x(0)\;\;\mbox{і}\;\; g^{(k)}_{1}(0)=(v_k)'_x(l)\\
\mbox{для всіх}\;\,k\in\mathbb{Z}\;\,\mbox{таких, що}\;\,
0\leq k<s/2-5/4,
\end{gathered}
\end{equation}
де функції $v_k$ означено на інтервалі $(0,l)$ за рекурентною формулою \eqref{6f9rec}. З огляду на включення \eqref{6fv-k} праві частини рівностей~\eqref{6f14} означено, якщо $s-5/2-2k>0$. Отже, умови узгодження~\eqref{6f14} сформульовано коректно. Звісно, якщо $s\leq5/2$, то умови узгодження для розглянутої задачі відсутні.

Покладемо
$$
E_{1}:=\{2r+1/2:1\leq r\in\mathbb{Z}\}.
$$
Помічаємо, що $E_{1}$~--- множина всіх розривів функції, яка дорівнює числу умов узгодження \eqref{6f14} залежно від змінної $s\geq2$.

Перед тим як сформулювати версію теореми \ref{6th1} для параболічної задачі \eqref{6f1}, \eqref{6f3}, \eqref{6f2n}, вкажемо гільбертові простори, на яких діє ізоморфізм, породжений задачею. Нехай $s>2$ і $\varphi\in\mathcal{M}$. Як і у випадку крайових умов Діріхле, беремо $H^{s,s/2;\varphi}(\Omega)$ як область визначення ізоморфізму. Область значень ізоморфізму позначаємо через $\widetilde{\mathcal{Q}}_{1}^{s-2,s/2-1;\varphi}$. Вона є лінійним многовидом у гільбертовому просторі
\begin{align*}
&\widetilde{\mathcal{H}}_{1}^{s-2,s/2-1;\varphi}:=\\
&:=H^{s-2,s/2-1;\varphi}(\Omega)\oplus
(H^{s/2-3/4;\varphi}(0,\tau))^2\oplus H^{s-1;\varphi}(0,l).
\end{align*}
У соболєвському випадку, коли $\varphi(\cdot)\equiv1$, простір $\widetilde{\mathcal{H}}_{1}^{s-2,s/2-1;\varphi}$ збігається з простором, в який діє оператор \eqref{6f7N-BoundOper}.

Якщо $s\notin E_{1}$, то за означенням лінійний простір $\widetilde{\mathcal{Q}}_{1}^{s-2,s/2-1;\varphi}$ складається з усіх векторів
$$
\bigl(f,g_0,g_1,h\bigr)\in\widetilde{\mathcal{H}}_{1}^{s-2,s/2-1;\varphi},
$$
які задовольняють умови узгодження \eqref{6f14}. Це означення коректне, оскільки умови узгодження мають сенс для довільного вектора
$$
\bigl(f,g_0,g_1,h\bigr)\in
\widetilde{\mathcal{H}}_{1}^{s-2-\varepsilon,s/2-1-\varepsilon/2},
$$
де число $\varepsilon>0$ досить мале, та оскільки виконується неперервне вкладення
\begin{equation*}
\widetilde{\mathcal{H}}_{1}^{s-2,s/2-1;\varphi}\hookrightarrow
\widetilde{\mathcal{H}}_{1}^{s-2-\varepsilon,s/2-1-\varepsilon/2}.
\end{equation*}
Останнє є наслідком формул \eqref{8f5a} і \eqref{8f5b}.
Лінійний простір $\widetilde{\mathcal{Q}}_{1}^{s-2,s/2-1;\varphi}$ наділяємо скалярним добутком і нормою з гільбертового простору
$\widetilde{\mathcal{H}}_{1}^{s-2,s/2-1;\varphi}$. Простір $\widetilde{\mathcal{Q}}_{1}^{s-2,s/2-1;\varphi}$ повний, тобто гільбертів, що доводиться так само, як повнота простору $\widetilde{\mathcal{Q}}_{0}^{s-2,s/2-1;\varphi}$.
Якщо $2\leq s<5/2$, то простори $\widetilde{\mathcal{H}}_{1}^{s-2,s/2-1;\varphi}$ і $\widetilde{\mathcal{Q}}_{1}^{s-2,s/2-1;\varphi}$ збігаються, оскільки в цьому випадку умов узгоджен\-ня нема.

Якщо $s\in E_{1}$, то покладаємо
\begin{equation}\label{6f-interp2}
\widetilde{\mathcal{Q}}_{1}^{s-2,s/2-1;\varphi}:=\bigl[
\widetilde{\mathcal{Q}}_{1}^{s-2-\varepsilon,s/2-1-\varepsilon/2;\varphi},
\widetilde{\mathcal{Q}}_{1}^{s-2+\varepsilon,s/2-1+\varepsilon/2;\varphi}
\bigr]_{1/2},
\end{equation}
де число $\varepsilon\in(0,1/2)$ вибираємо довільно. Так означений гільбертів простір $\widetilde{\mathcal{Q}}_{1}^{s-2,s/2-1;\varphi}$ не залежить з точністю до еквівалентності норм від вибору числа $\varepsilon$  і неперервно вкладається у простір $\widetilde{\mathcal{H}}_{1}^{s-2,s/2-1;\varphi}$, як буде показано в зауваженні~\ref{v8rem2-6}.

Сформулюємо теорему про ізоморфізми для параболічної початково-крайової задачі \eqref{6f1}, \eqref{6f3}, \eqref{6f2n}.

\begin{theorem}\label{6th2}
Для довільних дійсного числа $s>2$ і функціонального параметра $\varphi\in\nobreak\mathcal{M}$ відображення \eqref{6f11} продовжується єдиним чином (за неперервністю) до ізоморфізму
\begin{equation}\label{6f12}
\Lambda_{1}:H^{s,s/2;\varphi}(\Omega)\leftrightarrow
\widetilde{\mathcal{Q}}_{1}^{s-2,s/2-1;\varphi}.
\end{equation}
\end{theorem}

Якщо означити простір $\widetilde{\mathcal{Q}}_{1}^{s-2,s/2-1;\varphi}$ для $s\in E_{1}$ у той самий спосіб, що і для $s\notin E_{1}$, то висновок цієї теореми буде хибним для $s\in E_{1}$ принаймні, коли $\varphi(\cdot)\equiv1$. Це випливає з результату В.~А.~Солоннікова \cite[\S~6, c.~186]{Solonnikov64}.

У соболєвському випадку, коли $\varphi(\cdot)\equiv1$, теореми \ref{6th1} і \ref{6th2} відомі. Вони є окремими випадками результату Ж.-Л.~Ліонса і Е.~Мадженеса \cite[теорема 6.2]{LionsMagenes72ii}, який стосується
параболічного диференціального рівняння довільного парного порядку $2m$ за просторовими змінними і першого порядку за часовою змінною та нормальних крайових умов (до яких належать умови Діріхле і Неймана) й доведений за припущення, що $s+1/2\notin\mathbb{Z}$ і $s/(2m)+1/2\notin\mathbb{Z}$
(їх результат охоплює і граничний випадок $s=2$). Аналогічна теорема для загальних $2b$-параболічних початково-крайових задач у анізотропних просторах Соболєва встановлена раніше М.~С.~Аграновічем
та М.~І.~Вішиком \cite[теорема 12.1]{AgranovichVishik64} за припущення, що $s/(2b)\in\mathbb{N}$. Його можна позбутися, як це
випливає з результату М.~В.~Житарашу
\cite[теорема 9.1]{Zhitarashu85}.

Ми виведемо теореми \ref{6th1} і \ref{6th2} зі згаданого результату Ж.-Л.~Ліонса і Е.~Мадженеса за допомогою квадратичної інтерполяції з функціональним параметром. Для цього встановимо деякі інтерполяційні формули для гільбертових
просторів $\widetilde{\mathcal{Q}}_{0}^{s-2,s/2-1;\varphi}$, де $s\in(2,\infty)\setminus E_{0}$, і $\widetilde{\mathcal{Q}}_{1}^{s-2,s/2-1;\varphi}$, де $s\in(2,\infty)\setminus E_{1}$, які є областями значень ізоморфізмів \eqref{6f8} і \eqref{6f12} відповідно. Точки множини $E_{0}$ розбивають промінь $(2,\infty)$ на інтервали
$$
J_{0,1}:=(2,\,7/2),\quad J_{0,r}:=(2r-1/2,\,2r+3/2),
\;\;\mbox{де}\;\; 2\leq r\in\mathbb{Z},
$$
а точки множини $E_{1}$~--- на інтервали
$$
J_{1,0}:=(2,5/2),\quad J_{1,r}:=(2r+1/2,\,2r+5/2),
\;\;\mbox{де}\;\; 1\leq r\in\mathbb{Z}.
$$
Якщо $s$ належить до деякого $J_{0,r}$, то число умов узгодження \eqref{6f10} дорівнює $2r$. Аналогічно, якщо $s$ належить до деякого $J_{1,r}$, то число умов узгодження \eqref{6f14} також дорівнює  $2r$. Отже, промінь $(2,\infty)$ розбито на інтервали сталості числа умов узгодження як функції аргументу~$s$.

\begin{lemma}\label{6lem2}
Нехай $\lambda\in\{0,1\}$ і $r\in\mathbb{N}$. Припустимо, що $s_0,s,s_1\in J_{\lambda,r}$, $s_0<s<s_1$ і $\varphi\in\mathcal{M}$. Нехай інтерполяційний параметр $\psi\in\mathcal{B}$ означено формулою \eqref{9f7.2}. Тоді виконується така рівність просторів з еквівалентністю норм у них:
\begin{equation}\label{6f28}
\widetilde{\mathcal{Q}}_{\lambda}^{s-2,s/2-1;\varphi}=
\bigl[\widetilde{\mathcal{Q}}_{\lambda}^{s_0-2,s_{0}/2-1},
\widetilde{\mathcal{Q}}_{\lambda}^{s_1-2,s_{1}/2-1}\bigr]_{\psi}.
\end{equation}
\end{lemma}

\begin{proof}[\indent Доведення.]
Нагадаємо, що $\widetilde{\mathcal{Q}}_{\lambda}^{s-2,s/2-1;\varphi}$ і $\widetilde{\mathcal{Q}}_{\lambda}^{s_j-2,s_{j}/2-1}$, де $\nobreak{j\in\{0,1\}}$,
є підпросторами гільбертових просторів $\widetilde{\mathcal{H}}_{\lambda}^{s-2,s/2-1;\varphi}$ і $\widetilde{\mathcal{H}}_{\lambda}^{s_j-2,s_{j}/2-1}$  відповідно.
На підставі теорем  \ref{9prop6.3}, \ref{8prop4} і \ref{9lem7.3a}
виконуються такі рівності:
\begin{align*}
&\bigl[\widetilde{\mathcal{H}}_{\lambda}^{s_0-2,s_0/2-1},
\widetilde{\mathcal{H}}_{\lambda}^{s_1-2,s_1/2-1}\bigr]_{\psi}=\\
&=\bigl[H^{s_0-2,s_0/2-1}(\Omega),
H^{s_1-2,s_1/2-1}(\Omega)\bigr]_{\psi}\oplus\\
&\quad\oplus
\bigl(\bigl[H^{s_0-(2\lambda+1)/2,s_0/2-(2\lambda+1)/4}(0,\tau),\\
&\qquad\quad\, H^{s_1-(2\lambda+1)/2,s_1/2-(2\lambda+1)/4}(0,\tau)\bigr]_{\psi}\bigr)^2
\oplus\\
&\quad\oplus\bigl[H^{s_0-1}(0,l),H^{s_1-1}(0,l)\bigr]_{\psi}=\\
&=
H^{s-2,s/2-1;\varphi}(\Omega)\oplus
(H^{s-(2\lambda+1)/2,s/2-(2\lambda+1)/4;\varphi}(0,\tau))^2\oplus\\
&\qquad\qquad\qquad\qquad\;\oplus H^{s-1;\varphi}(0,l)=\\
&=\widetilde{\mathcal{H}}_{\lambda}^{s-2,s/2-1;\varphi}.
\end{align*}
Отже,
\begin{equation}\label{6f30}
\bigl[\widetilde{\mathcal{H}}_\lambda^{s_0-2,s_0/2-1},
\widetilde{\mathcal{H}}_\lambda^{s_1-2,s_1/2-1}\bigr]_{\psi}=
\widetilde{\mathcal{H}}_\lambda^{s-2,s/2-1;\varphi}
\end{equation}
з еквівалентністю норм.

Виведемо потрібну формулу \eqref{6f28} з рівності \eqref{6f30} за допомогою теореми~\ref{9prop6.2} (тієї її частини, яка стосується інтерполяції підпросторів). Для цього побудуємо лінійне відображення $P$, задане на $\widetilde{\mathcal{H}}_\lambda^{s_0-2,s_0/2-1}$ і таке, що $P$ є проєктором простору $\widetilde{\mathcal{H}}_\lambda^{s_j-2,s_j/2-1}$ на його підпростір $\widetilde{\mathcal{Q}}_\lambda^{s_j-2,s_j/2-1}$ для кожного $j\in\{0,\,1\}$. За наявності такого відображення отримаємо потрібну рівність \eqref{6f28}, оскільки
\begin{align*}
\bigl[&\widetilde{\mathcal{Q}}_{\lambda}^{s_0-2,s_{0}/2-1},
\widetilde{\mathcal{Q}}_{\lambda}^{s_1-2,s_{1}/2-1}\bigr]_{\psi}=\\
&=\bigl[\widetilde{\mathcal{H}}_{\lambda}^{s_0-2,s_0/2-1},
\widetilde{\mathcal{H}}_{\lambda}^{s_1-2,s_1/2-1}\bigr]_{\psi}\cap
\widetilde{\mathcal{Q}}_{\lambda}^{s_0-2,s_{0}/2-1}=\\
&=\widetilde{\mathcal{H}}_{\lambda}^{s-2,s/2-1;\varphi}\cap
\widetilde{\mathcal{Q}}_{\lambda}^{s_0-2,s_{0}/2-1}=\\
&=\widetilde{\mathcal{Q}}_{\lambda}^{s-2,s/2-1;\varphi}
\end{align*}
з огляду на теорему~\ref{9prop6.2}, формулу \eqref{6f30} та умови $s_0,s\in J_{\lambda,r}$ і $s_0<s$. Зауважимо, що з цих умов випливає остання рівність, оскільки елементи просторів
$\widetilde{\mathcal{Q}}_{\lambda}^{s_0-2,s_{0}/2-1}$ і $\widetilde{\mathcal{Q}}_{\lambda}^{s-2,s/2-1;\varphi}$ задовольняють одні й ті самі умови узгодження, а простір $\widetilde{\mathcal{H}}_{\lambda}^{s-2,s/2-1;\varphi}$ неперервно вкладений у
$\widetilde{\mathcal{H}}_{\lambda}^{s_{0}-2,s_{0}/2-1}$. Крім того, пара
просторів $\widetilde{\mathcal{Q}}_{\lambda}^{s_0-2,s_{0}/2-1}$ і $\widetilde{\mathcal{Q}}_{\lambda}^{s_1-2,s_{1}/2-1}$ регулярна за теоремою~\ref{9prop6.2} і, отже, права частина рівності \eqref{6f28} має сенс.

Побудуємо вказане відображення $P$. Розглянемо лінійне відображення
\begin{equation}\label{6f32-Murach}
T:\{z_0,\dots,z_{r-1}\}\mapsto
w(t):=\sum_{k=0}^{r-1}\frac{z_k t^k}{k!},
\end{equation}
де $z_0,\dots,z_{r-1}\in\mathbb{C}$. Воно є обмеженим оператором на парі просторів
\begin{equation}\label{6f33a}
T:\mathbb{C}^{r}\rightarrow H^{s;\varphi}(0,\tau)
\end{equation}
для будь-яких $s\in\mathbb{R}$ і $\varphi\in\mathcal{M}$.
Якщо $w=T(z_0,\dots,z_{r-1})$, то $w^{(k)}_{t}(0)=z_k$ для кожного номера $k\in\{0,\dots,r-1\}$.

Дослідимо спочатку випадок, коли $\lambda=0$. Для кожного вектора
$$
(f,g_0,g_1,h)\in\widetilde{\mathcal{H}}_0^{s_{0}-2,s_{0}/2-1}
$$
покладемо
\begin{equation*}
\begin{split}
&g^*_0=g_0+T\bigl(v_0(0)-g_0^{(0)}(0),\dots,v_{r-1}(0)-g_0^{(r-1)}(0)\bigr),
\\
&g^*_1=g_1+T\bigl(v_0(l)-g_1^{(0)}(0),\dots,v_{r-1}(l)-g_1^{(r-1)}(0)\bigr).
\end{split}
\end{equation*}
Тут функції $v_k\in H^{s_{0}-1-2k}(0,l)$, де $k=0,\ldots,r-1$, означено за рекурентною формулою \eqref{6f9rec}. Лінійне відображення
$$
P:(f,g_0,g_1,h)\mapsto (f,g^*_0,g^*_1,h),\;\;\mbox{де}\;\;
(f,g_0,g_1,h)\in\widetilde{\mathcal{H}}_0^{s_{0}-2,s_{0}/2-1},
$$
шукане. Справді, його звуження на кожний простір $\widetilde{\mathcal{H}}_0^{s_j-2,s_j/2-1}$, де $j\in\{0,\,1\}$,
є обмеженим оператором на цьому просторі. Це випливає з обмеженості оператора~\eqref{6f33a}, розглянутого у випадку, коли $s=s_j-1/2$ і $\varphi(\cdot)=1$. Крім того, якщо
$$
(f,g_0,g_1,h)\in\widetilde{\mathcal{Q}}_0^{s_j-2,s_j/2-1},
$$
то з огляду на умови узгодження \eqref{6f10} виконується рівність
$$
P(f,g_0,g_1,h)=(f,g_0,g_1,h).
$$

Дослідимо тепер випадок, коли $\lambda=1$. Для будь-якого вектора
$$
(f,g_0,g_1,h)\in\widetilde{\mathcal{H}}_1^{s_{0}-2,s_{0}/2-1}
$$
покладемо
\begin{equation*}
\begin{split}
&g^*_0=g_0+T\bigl((v_0)'_x(0)-g_0^{(0)}(0),\dots,
(v_{r-1})'_x(0)-g_0^{(r-1)}(0)\bigr),\\
&g^*_1=g_1+T\bigl((v_0)'_x(l)-g_1^{(0)}(0),\dots,
(v_{r-1})'_x(l)-g_1^{(r-1)}(0)\bigr).
\end{split}
\end{equation*}
Тут функції $v_0,\ldots,v_{r-1}$ і відображення $T$ такі самі, як і у випадку $\lambda=0$. Як бачимо, лінійне відображення
$$
P:(f,g_0,g_1,h)\mapsto (f,g^*_0,g^*_1,h),\;\;\mbox{де}\;\;
(f,g_0,g_1,h)\in\widetilde{\mathcal{H}}_1^{s_{0}-2,s_{0}/2-1},
$$
є шуканим.
\end{proof}

\begin{remark}\label{v8rem1-6}
Якщо $\lambda=1$ і $r=0$, то висновок леми~\ref{6lem2} залишається істинним. Справді, у цьому випадку виконуються рівності
$$
\widetilde{\mathcal{Q}}_{1}^{s-2,s/2-1;\varphi}=
\widetilde{\mathcal{H}}_{1}^{s-2,s/2-1;\varphi}
\;\;\,\mbox{і}\;\;\,
\widetilde{\mathcal{Q}}_{1}^{s_j-2,s_{j}/2-1}=
\widetilde{\mathcal{H}}_{1}^{s_j-2,s_{j}/2-1}
$$
для кожного номера $j\in\{0,\,1\}$. Тому \eqref{6f28} збігається з рівністю \eqref{6f30}. Звісно, остання є правильною і в розглянутому випадку.
\end{remark}

\begin{proof}[\indent Доведення теорем $\ref{6th1}$ і $\ref{6th2}$]
Нехай $s>2$, $\varphi\in\mathcal{M}$ і $\lambda\in\{0,\,1\}$. Якщо $\lambda=0$ (або $\lambda=1$), то наші міркування стосуються теореми \ref{6th1} (відповідно, теореми \ref{6th2}). Спочатку дослідимо випадок, коли $s\notin E_{\lambda}$. Тоді $s\in J_{\lambda,r}$ для деякого цілого числа~$r$. Виберемо числа $s_0,s_1\in J_{\lambda,r}$ такі, що жодне з чисел $s_0$, $s_0/2$, $s_1$ чи $s_1/2$ не є напівцілим і виконуються нерівності $s_0<s<s_1$. Згідно з результатом Ж.-Л.~Ліонса і Е.~Мадженеса \cite[теорема~6.2]{LionsMagenes72ii} відображення
\begin{equation}\label{v8fmap-smooth-6}
u\mapsto\Lambda_{\lambda}u,\;\;\mbox{де}\;\; u\in C^{\infty}(\overline{\Omega}),
\end{equation}
продовжується єдиним чином (за неперервністю) до ізоморфізму
\begin{equation}\label{v8f35-6}
\Lambda_\lambda:H^{s_j,s_j/2}(\Omega)\leftrightarrow
\widetilde{\mathcal{Q}}_{\lambda}^{s_j-2,s_j/2-1}
\end{equation}
для кожного номера $j\in\{0,1\}$. Нехай $\psi$~--- інтерполяційний параметр~\eqref{9f7.2}. Тоді звуження оператора \eqref{v8f35-6}, де $j=0$, на простір
\begin{equation*}
\bigl[H^{s_0,s_{0}/2}(\Omega),H^{s_1,s_{1}/2}(\Omega)\bigr]_{\psi}=
H^{s,s/2;\varphi}(\Omega)
\end{equation*}
є ізоморфізмом
\begin{equation}\label{v8f36-6}
\begin{split}
\Lambda_\lambda:H^{s,s/2;\varphi}(\Omega)\leftrightarrow
\bigl[&\widetilde{\mathcal{Q}}_{\lambda}^{s_0-2,s_0/2-1},
\widetilde{\mathcal{Q}}_{\lambda}^{s_1-2,s_1/2-1}\bigr]_{\psi}=\\
&=\widetilde{\mathcal{Q}}_{\lambda}^{s-2,s/2-1;\varphi}.
\end{split}
\end{equation}
Тут рівності просторів виконуються разом з еквівалентністю норм з огляду на теорему~\ref{9lem7.3a} і лему~\ref{6lem2} (див. також зауваження~\ref{v8rem1-6}). Оператор \eqref{v8f36-6} є продовженням за неперервністю відображення \eqref{v8fmap-smooth-6}, оскільки множина $C^{\infty}(\overline{\Omega})$ щільна в $H^{s,s/2;\varphi}(\Omega)$.
Отже, у вказаному випадку теореми \ref{6th1} і \ref{6th2} доведено.

Дослідимо тепер випадок, коли $s\in E_{\lambda}$. Виберемо довільно число $\varepsilon\in(0,1/2)$. Оскільки $s\pm\varepsilon\notin E_{\lambda}$ і $s-\varepsilon>2$, то виконуються ізоморфізми
\begin{equation*}
\Lambda_\lambda:H^{s\pm\varepsilon,(s\pm\varepsilon)/2;\varphi}(\Omega)
\leftrightarrow
\widetilde{\mathcal{Q}}_{\lambda}^
{s\pm\varepsilon-2,(s\pm\varepsilon)/2-1;\varphi}.
\end{equation*}
З них випливає, що відображення \eqref{v8fmap-smooth-6} продовжується єдиним чином (за неперервністю) до ізоморфізму
\begin{equation}\label{v8f38-6}
\begin{aligned}
\Lambda_\lambda\,:\,&
\bigl[H^{s-\varepsilon,(s-\varepsilon)/2;\varphi}(\Omega),
H^{s+\varepsilon,(s+\varepsilon)/2;\varphi}(\Omega)\bigr]_{1/2}
\leftrightarrow\\
&\leftrightarrow
\bigl[\widetilde{\mathcal{Q}}_{\lambda}^{s-\varepsilon-2,(s-\varepsilon)/2-1;\varphi},
\widetilde{\mathcal{Q}}_{\lambda}^{s+\varepsilon-2,
(s+\varepsilon)/2-1;\varphi}\bigr]_{1/2}=\\
&=\widetilde{\mathcal{Q}}_{\lambda}^{s-2,s/2-1;\varphi}.
\end{aligned}
\end{equation}
Нагадаємо, що остання рівність є означенням простору $\widetilde{\mathcal{Q}}_{\lambda}^{s-2,s/2-1;\varphi}$.
Для завершення доведення потрібно до \eqref{v8f38-6} застосувати формулу
\begin{equation*}
H^{s,s/2;\varphi}(\Omega)=
\bigl[H^{s-\varepsilon,(s-\varepsilon)/2;\varphi}(\Omega),
H^{s+\varepsilon,(s+\varepsilon)/2;\varphi}(\Omega)\bigr]_{1/2},
\end{equation*}
яка є окремим випадком формули~\eqref{16f46}.
\end{proof}

\begin{remark}\label{v8rem2-6}
Простори, означені формулами \eqref{6f-interp1} і \eqref{6f-interp2}, не залежать з точністю до еквівалентності норм від вибору числа $\varepsilon\in(0,1/2)$. Справді, нехай $\lambda\in\{0,\,1\}$ і $s\in E_{\lambda}$; тоді за теоремами~\ref{6th1} і \ref{6th2} виконуються ізоморфізми
\begin{equation*}
\Lambda_\lambda:H^{s,s/2;\varphi}(\Omega)\leftrightarrow
\bigl[\widetilde{\mathcal{Q}}_{\lambda}^{s-2-\varepsilon,
s/2-1-\varepsilon/2;\varphi},
\widetilde{\mathcal{Q}}_{\lambda}^{s-2+\varepsilon,
s/2-1+\varepsilon/2;\varphi}\bigr]_{1/2},
\end{equation*}
якщо $0<\varepsilon<1/2$. Це означає вказану незалежність.
Крім того, виконується неперервне вкладення
$$
\widetilde{\mathcal{Q}}_{\lambda}^{s-2,s/2-1;\varphi}\hookrightarrow
\widetilde{\mathcal{H}}_{\lambda}^{s-2,s/2-1;\varphi}.
$$
Воно обґрунтовується так само, як і його аналог у зауваженні~\ref{16rem8.1}, наведеному в наступному підрозділі.
\end{remark}

\markright{\emph \ref{sec4.2.1}. Неоднорідна задача у прямокутнику}

\section[Неоднорідна задача у прямокутнику]
{Неоднорідна задача у прямокутнику}\label{sec4.2.1}

\markright{\emph \ref{sec4.2.1}. Неоднорідна задача у прямокутнику}

У цьому підрозділі досліджуємо загальну неоднорідну по\-чат\-ко\-во-крайову задачу для $2b$-параболічного диференціального рівняння, одновимірного за просторовою змінною. Буде показано, що ця задача породжує ізоморфізми між відповідними узагальненими просторами Соболєва, тобто є коректно розв'язною на парах цих просторів.

Як і в попередньому підрозділі, задаємо числа $l>0$ і $\tau>0$ та покладаємо $\Omega:=(0,l)\times(0,\tau)$. Розглядаємо у прямокутнику $\Omega$ лінійну параболічну задачу, яка складається з диференціального рівняння
\begin{align}\notag
&A(x,t,D_x,\partial_t)u(x,t)\equiv\\
&\equiv\sum_{\alpha+2b\beta\leq 2m}a^{\alpha,\beta}(x,t)\,D^\alpha_x\partial^\beta_t
u(x,t)=f(x,t),\label{f2.1a}\\
&\mbox{якщо}\;\,0<x<l\;\,\mbox{і}\;\,0<t<\tau,\notag
\end{align}
крайових умов
\begin{align}
&B_{j,0}(t,D_x,\partial_t)u(x,t)\big|_{x=0}\equiv\notag\\
&\equiv\sum_{\alpha+2b\beta\leq m_j}
b_{j,0}^{\alpha,\beta}(t)\,D^\alpha_x\partial^\beta_t u(x,t)\big|_{x=0}=g_{j,0}(t)
\quad\mbox{і}\label{f2.2a}\\
&B_{j,1}(t,D_x,\partial_t)u(x,t)\big|_{x=l}\equiv\notag\\
&\equiv\sum_{\alpha+2b\beta\leq m_j}
b_{j,1}^{\alpha,\beta}(t)\,D^\alpha_x\partial^\beta_t
u(x,t)\big|_{x=l}=g_{j,1}(t),\label{f2.3a}\\
&\mbox{якщо}\;\,0<t<\tau,\;\,\mbox{де}\;\,j=1,\dots,m,\notag
\end{align}
та початкових умов
\begin{equation}\label{f2.4a}
\begin{aligned}
&\partial^k_t u(x,t)\big|_{t=0}=h_k(x),\\
&\mbox{якщо}\;\,0<x<l,\;\,\mbox{де}\;\,k=0,\ldots,\varkappa-1.
\end{aligned}
\end{equation}
Припущення щодо цілих чисел $b$, $m$, $m_j$ і $\varkappa$ та коефіцієнтів диференціальних виразів $A:=A(x,t,D_x,\partial_t)$ і $B_{j,k}:=B_{j,k}(t,D_x,\partial_t)$ робимо такі самі, як і в п.~\ref{sec3.2.1}. Отже, $m\geq b\geq1$,
$\varkappa:=m/b\in\mathbb{Z}$ і $m_j\geq0$ та $a^{\alpha,\beta}\in
C^{\infty}(\overline{\Omega})$ і $b_{j,k}^{\alpha,\beta}\in\nobreak C^{\infty}[0,\tau]$, де
$\overline{\Omega}:=[0,l]\times[0,\tau]$.

Нагадаємо, що початково-крайова задача \eqref{f2.1a}--\eqref{f2.4a} називається параболічною в $\Omega$, якщо вона задовольняє
умови \ref{9cond2.3}--\ref{9cond2.5}, сформульовані у п.~\ref{sec3.2.1}.

У третьому розділі покладаємо
\begin{equation}\label{sigma0-Murach3}
\sigma_0:=\max\{2m,m_1+1,\dots,m_m+1\}
\end{equation}
(на відміну від другого розділу додатково не припускаємо, що $\sigma_0/2b\in\mathbb{Z}$). Якщо $m_j\leq2m-1$ для кожного  $j\in\{1,\ldots,m\}$, то $\sigma_0=2m$.

Пов'яжемо з параболічною задачею \eqref{f2.1a}--\eqref{f2.4a} лінійне відображення
\begin{equation}\label{f4.1a}
\begin{aligned}
u\mapsto\Lambda u:=
\bigl(&Au,B_{1,0}u,B_{1,1}u,\ldots,B_{m,0}u,B_{m,1}u,\\
&u(\cdot,0),\ldots,\partial^{\varkappa-1}_{t}u(\cdot,0)\bigr),
\;\;\mbox{де}\;\;u\in C^{\infty}(\overline{\Omega}).
\end{aligned}
\end{equation}
Звісно, тут $u(\cdot,0)$ і $\partial^{\varkappa-1}_{t}u(\cdot,0)$ позначають функції $u(x,0)$ і $\partial^{\varkappa-1}_{t}u(x,0)$ аргументу $x\in[0,l]$, де $u=u(x,t)$. Отже,
$$
u\in C^{\infty}(\overline{\Omega})\,\Longrightarrow\,
\Lambda u\in C^{\infty}(\overline{\Omega})\times
(C^{\infty}[0,\tau])^{2m}\times(C^{\infty}[0,l])^{\varkappa}.
$$

Нехай $s\geq\sigma_0$. Введемо гільбертів простір
\begin{equation*}
\begin{split}
&\widetilde{\mathcal{H}}^{s-2m,(s-2m)/(2b)}:=
H^{s-2m,(s-2m)/(2b)}(\Omega)\oplus\\
&\oplus\bigoplus_{j=1}^{m}\bigl(H^{(s-m_j-1/2)/(2b)}(0,\tau)\bigr)^{2}
\oplus\bigoplus_{k=0}^{\varkappa-1}H^{s-2bk-b}(0,l).
\end{split}
\end{equation*}
Відображення \eqref{f4.1a} продовжується єдиним чином (за неперервністю) до обмеженого лінійного оператора
\begin{equation}\label{16f4a-a}
\Lambda:H^{s,s/(2b)}(\Omega) \rightarrow
\widetilde{\mathcal{H}}^{s-2m,(s-2m)/(2b)}.
\end{equation}
Це випливає з \cite[розд.~I, лема~4 і розд.~II, теореми 3 і 7]{Slobodetskii58}. Виберемо довільно функцію $u(x,t)$ з простору $H^{s,s/(2b)}(\Omega)$ і означимо праві частини досліджуваної задачі за формулою
$$
(f,g_{1,0},g_{1,1},...,g_{m,0},g_{m,1},h_0,...,h_{\varkappa-1}):=
\Lambda u
$$
з використанням оператора \eqref{16f4a-a}. Отже,
\begin{equation}\label{16f4b-a}
\begin{aligned}
&f\in H^{s-2m,(s-2m)/(2b)}(\Omega),\\
&g_{j,\lambda}\in H^{(s-m_j-1/2)/(2b)}(0,\tau)\;\;\mbox{і}\\
&h_k\in H^{s-2bk-b}(0,l)
\end{aligned}
\end{equation}
для будь-яких $j\in\{1,\dots,m\}$, $\lambda\in\{0,1\}$ і $k\in\{0,\dots,\varkappa-1\}$.

Функції \eqref{16f4b-a} задовольняють природні умови узгодження. Запишемо ці умови. Згідно з \cite[розд.~II, теорема 7]{Slobodetskii58} означені за замиканням сліди
$$
\partial^{\,k}_t u(\cdot,0)\in H^{s-2bk-b}(0,l)
$$
для всіх $k\in\mathbb{Z}$ таких, що $0\leq k<s/(2b)-1/2$ (і лише для цих номерів $k$). Використовуючи параболічне диференціальне рівняння \eqref{f2.1a} і початкові умови \eqref{f2.4a}, виразимо ці сліди через функції $f$ і $h_k$.

Умова~\ref{9cond2.3} параболічності у випадку, коли $\xi=0$ і $p=1$, означає, що коефіцієнт $a^{0,\varkappa}(x,t)\neq0$ для всіх $x\in[0,l]$ і $t\in[0,\tau]$. Тому параболічне рівняння \eqref{f2.1a} можна розв'язати відносно старшої похідної $\partial^\varkappa_t u(x,t)$ за часом. Отже,
\begin{equation}\label{16f5a}
\begin{split}
\partial^\varkappa_t u(x,t)&=
\sum_{\substack{\alpha+2b\beta\leq 2m,\\ \beta\leq\varkappa-1}}
a_{0}^{\alpha,\beta}(x,t)\,D^\alpha_x\partial^\beta_t u(x,t)+\\
&+(a^{0,\varkappa}(x,t))^{-1}f(x,t),
\end{split}
\end{equation}
де $a_{0}^{\alpha,\beta}:=-a^{\alpha,\beta}/a^{0,\varkappa}\in C^{\infty}(\overline{\Omega})$.

З умов \eqref{f2.4a}, рівності \eqref{16f5a} та рівностей, одержаних внаслідок диференціювання за змінною $t$ рівності \eqref{16f5a} $k-\varkappa$ разів (при $\varkappa<k<s/(2b)-1/2$, якщо такі цілі $k$ існують) отримуємо для слідів $\partial^{\,k}_t u(x,0)$ таку рекурентну формулу:
\begin{equation}\label{16f6a}
\begin{split}
&\partial^{\,k}_t u(x,0)=h_k(x),\quad\mbox{якщо}
\quad 0\leq k\leq\varkappa-1,\\
&\partial^{\,k}_t u(x,0)=\\
&=\sum_{\substack{\alpha+2b\beta\leq 2m\\ \beta\leq\varkappa-1}}
\sum\limits_{q=0}^{k-\varkappa}
\binom{k-\varkappa}{q}\partial^{\,k-\varkappa-q}_t
a_{0}^{\alpha,\beta}(x,0)D^\alpha_x\partial^{\beta+q}_t
u(x,0)+\\
&\phantom{aaaaa}+
\partial^{\,k-\varkappa}_t\bigl((a^{0,\varkappa}(x,0))^{-1}f(x,0)\bigr),
\quad\mbox{якщо}\quad k\geq\varkappa.
\end{split}
\end{equation}
У ній ціле $k$ задовольняє умову $0\leq k<s/(2b)-1/2$, а  рівності виконуються для майже всіх $x\in (0,l)$. Тут, як звичайно,
$$
\binom{k-\varkappa}{q}:=\frac{(k-\varkappa)!}{q!\,(k-\varkappa-q)!}
$$
є біномним коефіцієнтом.

Крім того, згідно з \eqref{16f4b-a} для будь-яких $j\in\{1,\dots,m\}$ і $\lambda\in\nobreak\{0,1\}$ сліди $\partial^{\,k}_t g_{j,\lambda}(0)\in\mathbb{C}$ означені для усіх $k\in\mathbb{Z}$ таких, що $0\leq k<(s-m_j-1/2-b)/(2b)$ (і тільки для цих $k$). Ці сліди виражаються через функцію $u(x,t)$ та її похідні за часом за формулою
\begin{equation}\label{16f4bb-a}
\begin{aligned}
\partial^{k}_t g_{j,\lambda}(0)&=\bigl(\partial^{k}_{t}B_{j,\lambda}u(x,t)\bigr)|_{t=0}=\\
&=\sum_{\alpha+2b\beta\leq m_j}
\sum_{q=0}^{k}\binom{k}{q}
\partial^{\,k-q}_t b^{\alpha,\beta}_{j,\lambda}(0)
D^\alpha_x\partial^{\,\beta+q}_t u(x,0),
\end{aligned}
\end{equation}
де $x=0$, якщо $\lambda=0$, та $x=l$, якщо $\lambda=1$. Тут функції
$$
u(x,0),\,\partial_{t}u(x,0),...,\,\partial^{\,[m_j/(2b)]+k}_{t}u(x,0)
$$
аргументу $x\in (0,l)$ виражаються за рекурентною формулою \eqref{16f6a} через функції $f(x,t)$ і $h_0(x)$,..., $h_{\varkappa-1}(x)$. Як звичайно, $[m_j/(2b)]$~--- ціла частина числа $m_j/(2b)$.

Підставляючи \eqref{16f6a} у праву частину формули \eqref{16f4bb-a}, отримуємо умови узгодження
\begin{equation}\label{16f8a}
\begin{gathered}
\partial^{\,k}_t g_{j,0}(t)|_{t=0}=
B_{j,k,0}(v_0,\dots,v_{[m_j/(2b)]+k})(x)|_{x=0}, \\
\partial^{\,k}_t g_{j,1}(t)|_{t=0}=
B_{j,k,1}(v_0,\dots,v_{[m_j/(2b)]+k})(x)|_{x=l}, \\
\mbox{де}\;\;j\in\{1,\dots,m\},\;\;k\in\mathbb{Z}\;\;\mbox{і}\;\;
0\leq k<\frac{s-m_j-1/2-b}{2b}.
\end{gathered}
\end{equation}
Тут функції $v_0$, ..., $v_{[m_j/(2b)]+k}$ для вказаних індексів $j$ і $k$  означено на інтервалі $(0,l)$ за рекурентною формулою
\begin{equation}\label{16f9a}
\begin{split}
&v_\mu(x):=h_\mu(x),\quad\mbox{якщо}
\quad 0\leq\mu\leq\varkappa-1,\\
&v_\mu(x):=\\
&:=
\sum_{\substack{\alpha+2b\beta\leq 2m\\ \beta\leq\varkappa-1}}
\sum\limits_{q=0}^{\mu-\varkappa}
\binom{\mu-\varkappa}{q}\partial^{\,\mu-\varkappa-q}_t
a_{0}^{\alpha,\beta}(x,0)\,D^\alpha_x v_{\beta+q}(x)+\\
&+\partial^{\,\mu-\varkappa}_t \bigl((a^{0,\varkappa}(x,0))^{-1}f(x,0)\bigr),
\quad\mbox{якщо}\quad \mu\geq\varkappa.
\end{split}
\end{equation}
Крім того, поклали
\begin{equation*}
\begin{split}
&B_{j,k,\lambda}(v_0,\dots,v_{[m_j/(2b)]+k})(x):=\\
&:=
\sum_{\alpha+2b\beta\leq m_j}
\sum_{q=0}^{k}\binom{k}{q}
\partial^{\,k-q}_t b^{\alpha,\beta}_{j,\lambda}(0)
\bigl(D^\alpha_x v_{\beta+q}(x)\bigr)
\end{split}
\end{equation*}
для майже всіх $x\in(0,l)$ і кожного $\lambda\in\{0,1\}$. Згідно з \eqref{16f4b-a} виконується включення
\begin{equation}\label{8f69aa-a}
v_\mu\in H^{s-b-2b\mu}(0,l),\quad\mbox{якщо}\quad
\mu\in\mathbb{Z}\cap[0,s/(2b)-1/2).
\end{equation}

Праві частини рівностей \eqref{16f8a} означені коректно. Справді, кожна функція $D^\alpha_x v_{\beta+q}(x)$ належить до просторів
$$
H^{s-\alpha-b-2b(\beta+q)}(0,l)\hookrightarrow H^{s-m_j-2bk-b}(0,l)
$$
завдяки \eqref{8f69aa-a}. Тому сліди $\bigl(D^\alpha_x v_{\beta+q}(x)\bigr)\big|_{x=0}$ і $\bigl(D^\alpha_x v_{\beta+q}(x)\bigr)\big|_{x=l}$ означені, якщо
$$
s-m_j-2bk-b-\frac{1}{2}>0.
$$

Число умов узгодження  \eqref{16f8a} є функцією аргументу $s\geq\sigma_0$. Вона розривна в точці $s$ тоді і лише тоді, коли
$$
\frac{s-m_j-1/2-b}{2b}\in\mathbb{Z}.
$$
Отже,
\begin{equation}\label{setE}
\begin{aligned}
E&:=\{(2l+1)b+m_j+1/2:j,l\in\mathbb{Z},\,1\leq j\leq m,\,l\geq0\}\cap\\
&\cap(\sigma_0,\infty)
\end{aligned}
\end{equation}
є множиною всіх точок розриву цієї функції. Якщо
$$
s\leq\min\{m_1,\ldots,m_m\}+b+\frac{1}{2},
$$
то умов узгодження нема.

Наш основний результат для параболічної задачі \eqref{f2.1a}--\eqref{f2.4a} полягає в тому, що лінійне відображення \eqref{f4.1a} продовжується єдиним чином (за неперервністю) до ізоморфізму між відповідними узагальненими соболєвськими просторами. Вкажемо їх. Нехай $s>\sigma_0$ і $\varphi\in\mathcal{M}$. У соболєвському випадку, коли $\varphi(\cdot)\equiv1$, допускаємо також граничне значення $s=\sigma_0$.
Беремо гільбертів простір $H^{s,s/(2b);\varphi}(\Omega)$ як область визначення ізоморфізму. Область значень ізоморфізму позначаємо через $\widetilde{\mathcal{Q}}^{s-2m,(s-2m)/(2b);\varphi}$. Вона є лінійним многовидом у гільбертовому просторі
\begin{equation*}
\begin{split}
&\widetilde{\mathcal{H}}^{s-2m,(s-2m)/(2b);\varphi}:=\\
&:=H^{s-2m,(s-2m)/(2b);\varphi}(\Omega)
\oplus\bigoplus_{j=1}^{m}
\bigl(H^{(s-m_j-1/2)/(2b);\varphi}(0,\tau)\bigr)^{2}\oplus\\
&\oplus\bigoplus_{k=0}^{\varkappa-1}H^{s-2bk-b;\varphi}(0,l).
\end{split}
\end{equation*}
Якщо $\varphi(\cdot)\equiv1$, то він збігається з простором, у який діє оператор \eqref{16f4a-a}. Дамо означення простору $\widetilde{\mathcal{Q}}^{s-2m,(s-2m)/(2b);\varphi}$ окремо у випадках, коли $s\notin E$ і коли $s\in E$.

Розглянемо спочатку випадок, коли $s\notin E_{0}$. За означенням лінійний простір $\widetilde{\mathcal{Q}}^{s-2m,(s-2m)/(2b);\varphi}$ складається з усіх векторів
\begin{align*}
F&=\bigl(f,g_{1,0},g_{1,1},...,g_{m,0},g_{m,1},h_0,...,h_{\varkappa-1}\bigr)
\in\\
&\in\widetilde{\mathcal{H}}^{s-2m,(s-2m)/(2b);\varphi},
\end{align*}
які задовольняють умови узгодження \eqref{16f8a}. Для вказаних векторів ці умови сформульовані коректно. Справді, вони мають сенс для довільного вектора
$$
F\in\widetilde{\mathcal{H}}^{s-2m-\varepsilon,(s-2m-\varepsilon)/(2b)},
$$
якщо число $\varepsilon>0$ достатньо мале. Крім того, виконується неперервне вкладення
\begin{equation}\label{8f69a-a}
\widetilde{\mathcal{H}}^{s-2m,(s-2m)/(2b);\varphi}\hookrightarrow
\widetilde{\mathcal{H}}^{s-2m-\varepsilon,(s-2m-\varepsilon)/(2b)}
\end{equation}
внаслідок формул \eqref{8f5a} і \eqref{8f5b}. (Звісно, якщо $s=\sigma_0$ і $\varphi(\cdot)\equiv1$, то ці міркування не потрібні.)

Наділимо лінійний простір $\widetilde{\mathcal{Q}}^{s-2m,(s-2m)/(2b);\varphi}$ скалярним добутком і нормою з гільбертового простору
$\widetilde{\mathcal{H}}^{s-2m,(s-2m)/(2b);\varphi}$. Простір $\widetilde{\mathcal{Q}}^{s-2m,(s-2m)/(2b);\varphi}$ повний, тобто гільбертів. Справді, якщо число $\varepsilon>0$ достатньо мале, то
\begin{align*}
&\widetilde{\mathcal{Q}}^{s-2m,(s-2m)/(2b);\varphi}=\\
&=\widetilde{\mathcal{H}}^{s-2m,(s-2m)/(2b);\varphi}\cap
\widetilde{\mathcal{Q}}^{s-2m-\varepsilon,(s-2m-\varepsilon)/(2b)}.
\end{align*}
Простір $\widetilde{\mathcal{Q}}^{s-2m-\varepsilon,(s-2m-\varepsilon)/(2b)}$  повний, оскільки диференціальні оператори та оператори слідів, що використовуються в умовах узгодження, є обмеженими на відповідних парах просторів Соболєва. Тому простір, вказаний у правій частині цієї рівності, є повним відносно суми норм у просторах, які є компонентами перетину. Ця сума еквівалентна нормі у просторі  $\widetilde{\mathcal{H}}^{s-2m,(s-2m)/(2b);\varphi}$ на підставі вкладення \eqref{8f69a-a}. Таким чином, простір  $\widetilde{\mathcal{Q}}^{s-2m,(s-2m)/(2b);\varphi}$ є повним (відносно останньої норми).

Якщо $s\in E$, то гільбертів простір  $\widetilde{\mathcal{Q}}^{s-2m,(s-2m)/(2b);\varphi}$ означаємо за інтерполяційною формулою
\begin{equation}\label{16f10a}
\begin{split}
&\widetilde{\mathcal{Q}}^{s-2m,(s-2m)/(2b);\varphi}:=\\
&:=\bigl[
\widetilde{\mathcal{Q}}^{s-2m-\varepsilon,(s-2m-\varepsilon)/(2b);\varphi},
\widetilde{\mathcal{Q}}^{s-2m+\varepsilon,(s-2m+\varepsilon)/(2b);\varphi}
\bigr]_{1/2},
\end{split}
\end{equation}
де число $\varepsilon\in(0,1/2)$ вибрано довільно. Так означений простір не залежить з точністю до еквівалентності норм від вибору числа  $\varepsilon$ і неперервно вкладається у простір $\widetilde{\mathcal{H}}^{s-2m,(s-2m)/(2b);\varphi}$. Це обґрунтовується такими самими міркуванням, які наведено у зауваженні~\ref{16rem8.1} наприкінці наступного підрозділу.

Сформулюємо основний результат цього підрозділу~--- теорему про ізоморфізми, породжені неоднорідною параболічною задачею \eqref{6f1}--\eqref{6f2}.

\begin{theorem}\label{th4.1a}
Для довільних дійсного числа $s>\sigma_0$ і функціонального параметра $\varphi\in\nobreak\mathcal{M}$ відображення \eqref{f4.1a} продовжується єдиним чином (за неперервністю) до ізоморфізму
\begin{equation}\label{16f11a}
\Lambda:H^{s,s/(2b);\varphi}(\Omega)\leftrightarrow
\widetilde{\mathcal{Q}}^{s-2m,(s-2m)/(2b);\varphi}.
\end{equation}
\end{theorem}

Ця теорема відома у соболєвському випадку, коли $\varphi(\cdot)\equiv1$. У цьому випадку вона міститься у результаті М.~С.~Аграновіча і М.~І.~Вішика
\cite[теорема~12.1]{AgranovichVishik64}, доведеного за припущення, що  $s/(2b)\in\mathbb{Z}$, та у пізнішому результаті
М.~В.~Житарашу \cite[теорема~9.1]{Zhitarashu85} (доведеного без цього припущення). Зазначені результати охоплюють граничний випадок, коли $s=\sigma_0$. Виведемо теорему~\ref{th4.1a} для довільного $\varphi\in\nobreak\mathcal{M}$ із соболєвського випадку, використовуючи  вказаний результат М.~В.~Житарашу і квадратичну інтерполяцію з функціональним параметром.

Спочатку покажемо, що простір  $\widetilde{\mathcal{Q}}^{s-2m,(s-2m)/(2b);\varphi}$, де $s\notin E$, є результатом квадратичної інтерполяції його соболєвських аналогів.
Позначимо через $\{J_r:r\in\mathbb{N}\}$ набір усіх зв'язних компонент множини $(\sigma_0,\infty)\setminus E$. Кожна компонента $J_r$ є деяким скінченним відкритим інтервалом на проміні $(\sigma_0,\infty)$.

\begin{lemma}\label{16lem7.4a}
Нехай $r\in\mathbb{N}$. Припустимо, що $s_0,s,s_1\in J_{r}$, $\nobreak{s_0<s<s_1}$ і $\varphi\in\mathcal{M}$. Означимо інтерполяційний параметр $\psi\in\mathcal{B}$ за формулою \eqref{9f7.2}. Тоді виконується така рівність просторів з еквівалентністю норм у них:
\begin{equation}\label{16f42a}
\begin{split}
&\widetilde{\mathcal{Q}}^{s-2m,(s-2m)/(2b);\varphi}=\\
&=\bigl[\widetilde{\mathcal{Q}}^{s_0-2m,(s_0-2m)/(2b)},
\widetilde{\mathcal{Q}}^{s_1-2m,(s_1-2m)/(2b)}\bigr]_{\psi}.
\end{split}
\end{equation}
\end{lemma}

\begin{proof}[\indent Доведення.]
На підставі теорем \ref{9prop6.3}, \ref{8prop4} і \ref{9lem7.3a} маємо такі рівності:
\begin{align*}
&\bigl[\widetilde{\mathcal{H}}^{s_0-2m,(s_0-2m)/(2b)},
\widetilde{\mathcal{H}}^{s_1-2m,(s_1-2m)/(2b)}\bigr]_{\psi}=\\
&=\bigl[H^{s_0-2m,(s_0-2m)/(2b)}(\Omega),H^{s_1-2m,(s_1-2m)/(2b)}(\Omega)
\bigr]_{\psi}\oplus\\
&\qquad
\oplus\bigoplus_{j=1}^{m}
\biggl(\bigl[H^{(s_0-m_j-1/2)/(2b)}(0,\tau),
H^{(s_1-m_j-1/2)/(2b)}(0,\tau)\bigr]_{\psi}\biggr)^{2}\oplus\\
&\qquad\oplus\bigoplus_{k=0}^{\varkappa-1}
\bigl[H^{s_0-2bk-b}(0,l),H^{s_1-2bk-b}(0,l)\bigr]_{\psi}=\\
&=H^{s-2m,(s-2m)/(2b);\varphi}(\Omega)
\oplus\bigoplus_{j=1}^{m}
\bigr(H^{(s-m_j-1/2)/(2b);\varphi}(0,\tau)\bigl)^{2}\oplus\\
&\phantom{a}\qquad\qquad\qquad\qquad\qquad\;\;\,
\oplus\bigoplus_{k=0}^{\varkappa-1}H^{s-2bk-b;\varphi}(0,l)=\\
&=\widetilde{\mathcal{H}}^{s-2m,(s-2m)/(2b);\varphi}.
\end{align*}
Отже,
\begin{equation}\label{16f45a}
\begin{aligned}
&\bigl[\widetilde{\mathcal{H}}^{s_0-2m,(s_0-2m)/(2b)},
\widetilde{\mathcal{H}}^{s_1-2m,(s_1-2m)/(2b)}\bigr]_{\psi}=\\
&=\widetilde{\mathcal{H}}^{s-2m,(s-2m)/(2b);\varphi}
\end{aligned}
\end{equation}
з еквівалентністю норм.

Виведемо потрібну формулу \eqref{16f42a} з рівності \eqref{16f45a} за допомогою теореми~\ref{9prop6.2}. Для цього побудуємо лінійне відображення $P$, задане на $\widetilde{\mathcal{H}}^{s_0-2m,(s_0-2m)/(2b)}$ і таке, що $P$ є проєктором простору $\widetilde{\mathcal{H}}^{s_j-2m,(s_j-2m)/(2b)}$ на його підпростір $\widetilde{\mathcal{Q}}^{s_j-2m,(s_j-2m)/(2b)}$ для кожного $j\in\{0,1\}$. Якщо таке відображення існує, то за теоремою~\ref{9prop6.2} пара
$$
\bigl[\widetilde{\mathcal{Q}}^{s_0-2m,(s_0-2m)/(2b)},
\widetilde{\mathcal{Q}}^{s_1-2m,(s_1-2m)/(2b)}\bigr]
$$
є регулярною; отже, ліва частина рівності \eqref{16f42a} має сенс. Крім того,
\begin{align*}
&\bigl[\widetilde{\mathcal{Q}}^{s_0-2m,(s_0-2m)/(2b)},
\widetilde{\mathcal{Q}}^{s_1-2m,(s_1-2m)/(2b)}\bigr]_{\psi}=\\
&=\bigl[\widetilde{\mathcal{H}}^{s_0-2m,(s_0-2m)/(2b)},
\widetilde{\mathcal{H}}^{s_1-2m,(s_1-2m)/(2b)}\bigr]_{\psi}\cap\\
&\quad\;\cap\widetilde{\mathcal{Q}}^{s_0-2m,(s_0-2m)/(2b)}=\\
&=\widetilde{\mathcal{H}}^{s-2m,(s-2m)/(2b);\varphi}\cap
\widetilde{\mathcal{Q}}^{s_0-2m,(s_0-2m)/(2b)}=\\
&=\widetilde{\mathcal{Q}}^{s-2m,(s-2m)/(2b);\varphi}
\end{align*}
на підставі зазначеної теореми і формули \eqref{16f45a} та з огляду на умови $s_0,s\in J_{r}$ і $s_0<s$. Остання рівність випливає з цих умов, оскільки всі елементи просторів
$$
\widetilde{\mathcal{Q}}^{s_0-2m,(s_0-2m)/(2b)}\quad\mbox{і}\quad
\widetilde{\mathcal{Q}}^{s-2m,(s-2m)/(2b);\varphi}
$$
задовольняють одні й ті самі умови узгодження та виконується неперервне вкладення
$$
\widetilde{\mathcal{H}}^{s-2m,(s-2m)/(2b);\varphi}\hookrightarrow
\widetilde{\mathcal{H}}^{s_0-2m,(s_0-2m)/(2b)}.
$$

Побудуємо вказане відображення $P$. З огляду на умови узгодження \eqref{16f8a} виділимо множину
\begin{equation}\label{set-kj}
\biggl\{k\in\mathbb{Z}:0\leq k<\frac{s-m_j-1/2-b}{2b}\biggr\}
\end{equation}
для кожного $j\in\{1,\dots,m\}$. Вона не залежить від $s\in J_r$. Позначимо через $q_{r,j}^{\star}$ число її елементів і покладемо
$q_{r,j}:=q_{r,j}^{\star}-1$ для зручності.

Для довільного вектора
\begin{align*}
F&:=(f,g_{1,0},g_{1,1},...,g_{m,0},g_{m,1},h_0,...,h_{\varkappa-1})\in\\
&\in\widetilde{\mathcal{H}}^{s_0-2m,(s_0-2m)/(2b)}
\end{align*}
покладемо
\begin{equation}\label{16f43a}
\begin{aligned}
&g^{*}_{j,\lambda}:=g_{j,\lambda},\quad\mbox{якщо}\quad q_{r,j}=-1;\\
&g^{*}_{j,\lambda}:=g_{j,\lambda}+T(z_{j,0,\lambda},\dots,
z_{j,q_{r,j},\lambda}),\quad\mbox{якщо}\quad q_{r,j}\geq0,
\end{aligned}
\end{equation}
де $j\in\{1,\dots,m\}$ і $\lambda\in\{0,1\}$. Тут
\begin{align*}
&z_{j,0,\lambda}:=
B_{j,0,\lambda}(v_0,\dots,v_{[m_j/(2b)]})(p_\lambda)-g_{j,\lambda}(0),\\
&\dots\\
&z_{j,q_{r,j},\lambda}:=
B_{j,q_{r,j},\lambda}(v_{0},\dots,v_{[m_j/(2b)]+q_{r,j}})(p_\lambda)-
\partial_{t}^{q_{r,j}} g_{j,\lambda}(0),
\end{align*}
де $p_\lambda=0$, якщо $\lambda=0$, і $p_\lambda=l$, якщо $\lambda=1$, а функції $v_\mu$ означені за рекурентною формулою \eqref{16f9a} і $T$~--- відображення \eqref{6f32-Murach}, в якому замість $r$ беремо $q_{r,j}^{\star}$. Лінійне відображення
\begin{align*}
P&:(f,g_{1,0},g_{1,1},...,g_{m,0},g_{m,1},h_0,...,h_{\varkappa-1})
\mapsto\\
&\mapsto (f,g^*_{1,0},g^*_{1,1},...,g^*_{m,0},g^*_{m,1},h_0,...,h_{\varkappa-1}),
\end{align*}
означене на всіх векторах $F\in\widetilde{\mathcal{H}}^{s_{0}-2m,(s_{0}-2m)/(2b)}$, є шуканим.
Справді, його звуження на кожний простір $\widetilde{\mathcal{H}}^{s_{j}-2m,(s_{j}-2m)/(2b)}$, де $j\in\{0,1\}$, є обмеженим оператором на ньому, що випливає з \eqref{6f33a}. Крім того, якщо $F\in\widetilde{\mathcal{Q}}^{s_{j}-2m,(s_{j}-2m)/(2b)}$, то
$PF=F$ завдяки умовам узгодження \eqref{16f8a}.
\end{proof}

\begin{proof}[\indent Доведення теореми~$\ref{th4.1a}$.]
Нехай $s>\sigma_0$ і $\varphi\in\mathcal{M}$. Спочатку дослідимо випадок, коли $s\notin E$. Тоді $s\in J_{r}$ для деякого $r\in\mathbb{Z}$.
Виберемо числа $s_0,s_1\in J_{r}$ такі, що $s_0<s<s_1$,
$s_j+1/2\notin\mathbb{Z}$ і $s_j/(2b)+1/2\notin\mathbb{Z}$ для кожного номера $j\in\{0,1\}$. Згідно з результатом М.~В.~Житарашу \cite[теорема~9.1]{Zhitarashu85} відображення \eqref{f4.1a} продовжується єдиним чином (за неперервністю) до ізоморфізмів
\begin{equation}\label{16f49a}
\Lambda:H^{s_j,s_j/(2b)}(\Omega)\leftrightarrow
\widetilde{\mathcal{Q}}^{s_j-2m,(s_j-2m)/(2b)},
\end{equation}
де $j\in\{0,1\}$. Означимо інтерполяційний параметр $\psi$ за формулою~\eqref{9f7.2}. Застосувавши до операторів \eqref{16f49a} квадратичну інтерполяцію з функціональним параметром $\psi$, отримаємо ізоморфізм
\begin{equation}\label{16f50a}
\begin{aligned}
\Lambda&:H^{s,s/(2b);\varphi}(\Omega)\leftrightarrow\\
&\leftrightarrow\bigl[\widetilde{\mathcal{Q}}^{s_0-2m,(s_0-2m)/(2b)},
\widetilde{\mathcal{Q}}^{s_1-2m,(s_1-2m)/(2b)}\bigr]_{\psi}=\\
&\quad\;\;=\widetilde{\mathcal{Q}}^{{s-2m,(s-2m)/(2b)};\varphi}.
\end{aligned}
\end{equation}
При цьому скористалися теоремою~\ref{9lem7.3a} і лемою~\ref{16lem7.4a}. Оператор \eqref{16f50a} є продовженням за неперервністю відображення \eqref{f4.1a}, оскільки множина $C^{\infty}(\overline{\Omega})$ щільна у просторі $H^{s,s/(2b);\varphi}(\Omega)$. У вказаному випадку теорему \ref{th4.1a} доведено.

Розглянемо тепер випадок, коли $s\in E$. Виберемо довільно число $\varepsilon\in(0,1/2)$.
Оскільки $s\pm\varepsilon\notin E$ і $s-\varepsilon>\sigma_0$, то виконуються ізоморфізми
\begin{equation}\label{8f37a}
\Lambda:H^{s\pm\varepsilon,(s\pm\varepsilon)/(2b);\varphi}(\Omega)
\leftrightarrow
\widetilde{\mathcal{Q}}^{s\pm\varepsilon-2m,(s\pm\varepsilon-2m)/(2b);
\varphi}.
\end{equation}
Застосувавши до операторів \eqref{8f37a} квадратичну інтерполяцію з числовим параметром $1/2$, отримаємо ізоморфізм
\begin{equation}\label{8f38a}
\begin{aligned}
&\Lambda:
\bigl[H^{s-\varepsilon,(s-\varepsilon)/(2b);\varphi}(\Omega),
H^{s+\varepsilon,(s+\varepsilon)/(2b);\varphi}(\Omega)\bigr]_{1/2}
\leftrightarrow\\
&\leftrightarrow
\bigl[
\widetilde{\mathcal{Q}}^{s-\varepsilon-2m,(s-\varepsilon-2m)/(2b);\varphi},
\widetilde{\mathcal{Q}}^{s+\varepsilon-2m,(s+\varepsilon-2m)/(2b);\varphi}
\bigr]_{1/2}=\\
&\quad\;\;=\widetilde{\mathcal{Q}}^{s-2m,(s-2m)/(2b);\varphi}.
\end{aligned}
\end{equation}
Нагадаємо, що остання рівність є означенням. Областю визначення оператора \eqref{8f38a} є простір $H^{s,s/(2b);\varphi}(\Omega)$ згідно з лемою~\ref{16lem7.5}, доведеною у наступному підрозділі. (Її доведення не використовує результати цього підрозділу і залишається правильним у випадку, коли $\Omega$~--- прямокутник.) Отже, оператор \eqref{8f38a} є продовженням за неперервністю відображення \eqref{f4.1a}. У розглянутому випадку теорему \ref{th4.1a} також доведено.
\end{proof}

\newpage

\markright{\emph \ref{sec4.1.1}. Неоднорідна задача у циліндрі}

\section[Неоднорідна задача у циліндрі]
{Неоднорідна задача у циліндрі}\label{sec4.1.1}

\markright{\emph \ref{sec4.1.1}. Неоднорідна задача у циліндрі}

У цьому підрозділі досліджуємо загальну багатовимірну неоднорідну початково-крайову задачу для $2b$-параболічного диференціального рівняння, заданого у циліндрі. Головний результат підрозділу~--- теорема про ізоморфізми, породжені досліджуваною задачею на парах узагальнених соболєвських просторів.

Нехай ціле число $n\geq2$. Як і в п.~\ref{sec3.1.1}, розглядаємо відкритий багатовимірний циліндр $\Omega:=G\times(0,\tau)$, де $G$~--- довільна обмежена область в $\mathbb{R}^{n}$ з нескінченно гладкою межею $\Gamma$, а $\tau$~--- довільне додатне число. Нагадаємо, що $S:=\Gamma\times(0,\tau)$,
$\overline{\Omega}=\overline{G}\times[0,\tau]$ і
$\overline{S}=\Gamma\times[0,\tau]$.

У циліндрі $\Omega$ задано параболічну початково-крайову задачу, яка складається з диференціального рівняння
\begin{equation}\label{16f1}
\begin{aligned}
&A(x,t,D_x,\partial_t)u(x,t)\equiv \\
&\equiv\sum_{|\alpha|+2b\beta\leq2m}
a^{\alpha,\beta}(x,t)\,D^\alpha_x\partial^\beta_t
u(x,t)=f(x,t),\\
&\mbox{якщо}\;\,x\in G\;\,\mbox{і}\;\,0<t<\tau,
\end{aligned}
\end{equation}
$m$ крайових умов
\begin{equation}\label{16f2}
\begin{aligned}
&B_{j}(x,t,D_x,\partial_t)u(x,t)\big|_{S}\equiv \\
&\equiv\sum_{|\alpha|+2b\beta\leq m_j}
b_{j}^{\alpha,\beta}(x,t)\,D^\alpha_x\partial^\beta_t u(x,t)\big|_{S}=g_{j}(x,t),\\
&\mbox{якщо}\;\,x\in\Gamma\;\,\mbox{і}\;\,0<t<\tau,\;\,\mbox{де}\;\,
j=1,\dots,m,
\end{aligned}
\end{equation}
і $\varkappa$ початкових умов
\begin{equation}\label{16f3}
\begin{aligned}
&\partial^k_t u(x,t)\big|_{t=0}=h_k(x),\\
&\mbox{якщо}\;\,x\in G,\;\,\mbox{де}\;\,k=0,\ldots,\varkappa-1.
\end{aligned}
\end{equation}
Тут цілі числа $b$, $m$, $m_j$ і $\varkappa$ такі самі, як і в попередньому підрозділі, тобто $m\geq b\geq1$, $\varkappa=m/b\in\mathbb{Z}$ і кожне $m_j\geq0$. Стосовно коефіцієнтів лінійних диференціальних виразів $A:=A(x,t,D_x,\partial_t)$ і $B_{j}:=B_{j}(x,t,D_x,\partial_t)$ припускаємо, що $a^{\alpha,\beta}\in C^{\infty}(\overline{\Omega})$ і $b_{j}^{\alpha,\beta}\in C^{\infty}(\overline{S})$ для усіх допустимих значень мультиіндексу $\alpha=(\alpha_1,\ldots,\alpha_n)$ і скалярного індексу $\beta$.

Нагадаємо, що початково-крайова задача
\eqref{16f1}--\eqref{16f3} називається параболічною у циліндрі $\Omega$, якщо вона задовольняє умови~\ref{9cond2.1} і~\ref{9cond2.2}, сформульовані у п.~\ref{sec3.1.1}.

Пов'яжемо з нею лінійне відображення
\begin{equation}\label{16f4}
\begin{gathered}
u\mapsto\Lambda u:=
\bigl(Au,B_1u,\ldots,B_mu,u\!\upharpoonright\!\overline{G},\ldots,
(\partial^{\varkappa-1}_tu)\!\upharpoonright\!\overline{G}\bigr),\\
\mbox{де}\;\,u\in C^{\infty}(\overline{\Omega}).
\end{gathered}
\end{equation}
Нехай $s\geq\sigma_0$ (нагадаємо, що натуральне число $\sigma_0$ означене формулою \eqref{sigma0-Murach3}). Введемо гільбертів простір
\begin{equation}\label{16rangeH}
\begin{aligned}
&\mathcal{H}^{s-2m,(s-2m)/(2b)}
:= H^{s-2m,(s-2m)/(2b)}(\Omega)\oplus\\
&\oplus\bigoplus_{j=1}^{m}H^{s-m_j-1/2,(s-m_j-1/2)/(2b)}(S)
\oplus\bigoplus_{k=0}^{\varkappa-1}H^{s-2bk-b}(G).
\end{aligned}
\end{equation}
Відображення \eqref{16f4} продовжується єдиним чином (за неперервністю) до обмеженого лінійного оператора
\begin{equation}\label{16f4a}
\Lambda:H^{s,s/(2b)}(\Omega)\rightarrow
\mathcal{H}^{s-2m,(s-2m)/(2b)}.
\end{equation}
Це є наслідком відомих властивостей диференціальних операторів і операторів слідів у анізотропних просторах Соболєва (див. \cite[розд.~I, лема~4 і розд.~II, теореми 3 і~7]{Slobodetskii58}).
Виберемо довільно функцію $u(x,t)$ з простору $H^{s,s/(2b)}(\Omega)$ і означимо праві частини параболічної задачі за формулою
$$
(f,g_1,...,g_m,h_0,...,h_{\varkappa-1}):=\Lambda u,
$$
використовуючи оператор \eqref{16f4a}. Таким чином,
\begin{equation}\label{16f4b}
\begin{split}
&f\in H^{s-2m,(s-2m)/(2b)}(\Omega),\\
&g_j\in H^{s-m_j-1/2,(s-m_j-1/2)/(2b)}(S),\\
&h_k\in H^{s-2bk-b}(G)
\end{split}
\end{equation}
для довільних $j\in\{1,\dots,m\}$ і $k\in\{0,\dots,\varkappa-1\}$.

Функції \eqref{16f4b} задовольняють умови узгодження, які полягають у тому, що похідні $\partial^{\,k}_t u(x,0)=\partial^k_t u(x,t)\big|_{t=0}$, які можна виразити через $f$ і $h_k$ за допомогою параболічного рівняння \eqref{16f1} і початкових умов \eqref{16f3}, задовольняють крайові умови \eqref{16f2} та співвідношення, що утворюються внаслідок  диференціювання крайових умов за змінною $t$ (див., наприклад,  \cite[\S~11]{AgranovichVishik64} або \cite[розд.~4, \S~5]{LadyzhenskajaSolonnikovUraltzeva67}). Запишемо ці умови узгодження.

Згідно з \cite[розд.~II, теорема 7]{Slobodetskii58}, сліди
$\partial^{\,k}_t u(\cdot,0)\in H^{s-2bk-b}(G)$ означені коректно (за замиканням) для всіх $k\in\mathbb{Z}$ таких, що $\nobreak{0\leq k<s/(2b)-1/2}$ (і лише для цих $k$). Умова~\ref{9cond2.1} параболічності, розглянута у випадку, коли $\xi=0$ і $p=1$, означає, що $a^{(0,\ldots,0),\varkappa}(x,t)\neq0$ для всіх $x\in\overline{G}$ і $t\in[0,\tau]$. Тому параболічне рівняння \eqref{9f2.1} можна розв'язати відносно $\partial^\varkappa_t u(x,t)$, тобто записати
\begin{equation}\label{16f5}
\begin{split}
\partial^\varkappa_t u(x,t) &=
\sum_{\substack{|\alpha|+2b\beta\leq 2m,\\ \beta\leq\varkappa-1}}
a_{0}^{\alpha,\beta}(x,t)\,D^\alpha_x\partial^\beta_t
u(x,t)+\\
&+(a^{(0,\ldots,0),\varkappa}(x,t))^{-1}f(x,t),
\end{split}
\end{equation}
де $a_{0}^{\alpha,\beta}:=-a^{\alpha,\beta}/a^{(0,\ldots,0),\varkappa}\in C^{\infty}(\overline{\Omega})$.

Нехай ціле число $k$ таке, що $0\leq k<s/(2b)-1/2$. З початкових умов \eqref{16f3}, рівності \eqref{16f5} та рівностей, одержаних диференціюванням її за часовою змінною $k-\varkappa$ разів (де $\varkappa<k<s/(2b)-1/2$, якщо такі $k$ існують) отримуємо для слідів $\partial^{\,k}_t u(x,0)$ рекурентну формулу
\begin{equation}\label{16f6}
\begin{split}
&\partial^{k}_t u(x,0)=h_k(x),\quad\mbox{якщо}
\quad 0\leq k\leq\varkappa-1,\\
&\partial^{k}_t u(x,0)=\\
&=\!\sum_{\substack{|\alpha|+2b\beta\leq 2m\\ \beta\leq\varkappa-1}}
\sum\limits_{q=0}^{k-\varkappa}\!
\binom{k-\varkappa}{q}\partial^{k-\varkappa-q}_t
a_{0}^{\alpha,\beta}(x,0)D^\alpha_x\partial^{\beta+q}_t
u(x,0)+\\
&\phantom{aaaaa}+\partial^{k-\varkappa}_t
\bigl((a^{(0,\ldots,0),\varkappa}(x,0))^{-1}f(x,0)\bigr),
\quad\mbox{якщо}\quad k\geq\varkappa.
\end{split}
\end{equation}
Ці рівності виконуються для майже всіх $x\in G$, а частинні похідні розуміються у сенсі теорії розподілів.

Для кожного $j\in\{1,\dots,m\}$ згідно з \cite[розд.~II, теорема 7]{Slobodetskii58} означені за замиканням сліди
$$
\partial^{\,k}_t g_j(\cdot,0)\in H^{s-m_j-1/2-2bk-b}(\Gamma)
$$
для всіх цілих чисел $k$ таких, що
$$
0\leq k<\frac{s-m_j-1/2-b}{2b}
$$
(і тільки цих $k$). Ці сліди виражаються через функцію $u(x,t)$ та її похідні за часом у такий спосіб:
\begin{equation}\label{16f4bb}
\begin{aligned}
\partial^{k}_t g_j(x,0)&=\bigl(\partial^{k}_{t}B_ju(x,t)\bigr)\big|_{t=0}=\\
&=\!\sum_{|\alpha|+2b\beta\leq m_j}
\sum_{q=0}^{k}\!\binom{k}{q}\!
\partial^{k-q}_t b^{\alpha,\beta}_j(x,0)
D^\alpha_x\partial^{\beta+q}_t u(x,0)
\end{aligned}
\end{equation}
для майже всіх $x\in\Gamma$. Тут функції
$$
u(x,0),\partial_{t}u(x,0),\ldots,\partial^{\,[m_j/(2b)]+k}_{t}u(x,0)
$$
аргументу $x\in G$ виражаються через функції $f(x,t)$ і $h_k(x)$
за рекурентною формулою \eqref{16f6}.

Підставляючи \eqref{16f6} у праву частину формули \eqref{16f4bb}, отримуємо умови узгодження:
\begin{equation}\label{16f8}
\begin{gathered}
\partial^{\,k}_t g_j\!\upharpoonright\!\Gamma=
B_{j,k}(v_0,\dots,v_{[m_j/(2b)]+k})\!\upharpoonright\!\Gamma, \\
\mbox{де}\;\;j\in\{1,\dots,m\},\;\;k\in\mathbb{Z}\;\;
\mbox{і}\;\;0\leq k<\frac{s-m_j-1/2-b}{2b}.
\end{gathered}
\end{equation}
Тут функції $v_0$,..., $v_{[m_j/(2b)]+k}$ для вказаних $j$ і $k$ означені за рекурентною формулою
\begin{equation}\label{16f9}
\begin{split}
&v_\mu(x):=h_\mu(x),\quad\mbox{якщо}
\quad 0\leq \mu\leq\varkappa-1,\\
&v_\mu(x):=\\
&:=\sum_{\substack{|\alpha|+2b\beta\leq 2m\\ \beta\leq\varkappa-1}}
\sum\limits_{q=0}^{\mu-\varkappa}
\binom{\mu-\varkappa}{q}\partial^{\,\mu-\varkappa-q}_t
a_{0}^{\alpha,\beta}(x,0)\,D^\alpha_x v_{\beta+q}(x)+\\
&+\partial^{\,\mu-\varkappa}_t \bigl((a^{(0,\ldots,0),\varkappa}(x,0))^{-1}f(x,0)\bigr),
\quad\mbox{якщо}\quad \mu\geq\varkappa.
\end{split}
\end{equation}
Крім того, поклали
\begin{equation}\label{16f9B}
\begin{aligned}
&B_{j,k}(v_0,\dots,v_{[m_j/(2b)]+k})(x):=\\
&:=\sum_{|\alpha|+2b\beta\leq m_j}
\sum_{q=0}^{k}\binom{k}{q}
\partial^{\,k-q}_t b^{\alpha,\beta}_j(x,0)
D^\alpha_x v_{\beta+q}(x)
\end{aligned}
\end{equation}
для майже всіх  $x\in G$. Згідно з \eqref{16f4b} виконується включення
\begin{equation*}
v_\mu\in H^{s-b-2bk}(G),\quad\mbox{якщо}\quad
\mu\in\mathbb{Z}\cap[0,s/(2b)-1/2\bigr).
\end{equation*}
Права частина рівності, записаної в \eqref{16f8}, має сенс, оскільки
$$
B_{j,k}(v_0,\dots,v_{[m_j/(2b)]+k})\in H^{s-m_j-2bk-b}(G),
$$
тому слід
\begin{equation}\label{8f69aa}
B_{j,k}(v_0,\dots,v_{[m_j/(2b)]+k})\!\upharpoonright\!\Gamma\in H^{s-m_j-2bk-b-1/2}(\Gamma)
\end{equation}
означений за замиканням, якщо є додатним індекс останнього простору.

Число умов узгодження \eqref{16f8} є функцією аргументу $s\geq\sigma_0$.
Її точки розриву утворюють множину $E$, означену за формулою \eqref{setE}.

Вкажемо узагальнені соболєвські простори, між якими діє ізоморфізм, породжений параболічною задачею \eqref{16f1}--\eqref{16f3}. Нехай $s>\sigma_0$ і $\varphi\in\mathcal{M}$. Якщо $\varphi(\cdot)\equiv1$, допускаємо також значення $s=\sigma_0$. Областю визначення ізоморфізму є гільбертів простір $H^{s,s/(2b);\varphi}(\Omega)$, а областю значень~--- деякий лінійний многовид $\mathcal{Q}^{s-2m,(s-2m)/(2b);\varphi}$ у гільбертовому просторі
\begin{equation*}
\begin{aligned}
\mathcal{H}^{s-2m,(s-2m)/(2b);\varphi}&:=
H^{s-2m,(s-2m)/(2b);\varphi}(\Omega)\oplus\\
&\oplus\bigoplus_{j=1}^{m}H^{s-m_j-1/2,(s-m_j-1/2)/(2b);\varphi}(S)\oplus\\
&\oplus\bigoplus_{k=0}^{\varkappa-1}H^{s-2bk-b;\varphi}(G).
\end{aligned}
\end{equation*}
Останній збігається з простором, у який діє оператор \eqref{16f4a}, якщо $\varphi(\cdot)\equiv1$. Означаючи гільбертів простір $\mathcal{Q}^{s-2m,(s-2m)/(2b);\varphi}$, слід розглянути окремо два випадки $s\notin E$ і $s\in E$ (як і в попередньому підрозділі).

Нехай спочатку $s\notin E$. За означенням лінійний простір $\mathcal{Q}^{s-2m,(s-2m)/(2b);\varphi}$ складається з усіх векторів
\begin{equation}\label{F-sect3-3-Murach}
F=\bigl(f,g_1,\dots,g_m,h_0,\dots,h_{\varkappa-1}\bigr)\in
\mathcal{H}^{s-2m,(s-2m)/(2b);\varphi},
\end{equation}
які задовольняють умови узгодження \eqref{16f8}. Ці умови сформульовані коректно для будь-якого вектора
$$
F\in\mathcal{H}^{s-2m-\varepsilon,(s-2m-\varepsilon)/(2b)},
$$
де число $\varepsilon>0$ достатньо мале. Отже, вони мають сенс і для довільного вектора \eqref{F-sect3-3-Murach} завдяки неперервному вкладенню
\begin{equation}\label{8f69a}
\mathcal{H}^{s-2m,(s-2m)/(2b);\varphi}\hookrightarrow
\mathcal{H}^{s-2m-\varepsilon,(s-2m-\varepsilon)/(2b)}.
\end{equation}
Воно випливає із формул \eqref{8f5a} і \eqref{8f5b}. (Якщо $s=\sigma_0$ і $\varphi(\cdot)\equiv1$, то ці міркування зайві.)

Наділимо лінійний простір $\mathcal{Q}^{s-2m,(s-2m)/(2b);\varphi}$  скалярним добутком і нормою з гільбертового простору
$\mathcal{H}^{s-2m,(s-2m)/(2b);\varphi}$. Простір $\mathcal{Q}^{s-2m,(s-2m)/(2b);\varphi}$
є повним, тобто гільбертовим. Це обґрунтовується така само, як і повнота простору $\widetilde{\mathcal{Q}}^{s-2m,(s-2m)/(2b);\varphi}$ у попередньому підрозділі.

Нехай тепер $s\in E$. У цьому випадку покладаємо
\begin{equation}\label{16f10}
\begin{split}
&\mathcal{Q}^{s-2m,(s-2m)/(2b);\varphi}:= \\
&:=
\bigl[\mathcal{Q}^{s-2m-\varepsilon,(s-2m-\varepsilon)/(2b);\varphi},
\mathcal{Q}^{s-2m+\varepsilon,(s-2m+\varepsilon)/(2b);\varphi}\bigr]_{1/2}.
\end{split}
\end{equation}
Тут число $\varepsilon\in(0,1/2)$ вибрано довільно. Так означений гільбертів простір $\mathcal{Q}^{s-2m,(s-2m)/(2b);\varphi}$ не залежить з точністю до еквівалентності норм від вибору числа $\varepsilon$ і неперервно вкладається у простір $\mathcal{H}^{s-2m,(s-2m)/(2b);\varphi}$. Це буде показано у зауваженні~\ref{16rem8.1} наприкінці цього підрозділу.

Сформулюємо головний результат підрозділу.

\begin{theorem}\label{16th4.1}
Для довільного дійсного числа $s>\sigma_0$ і функціонального параметра $\varphi\in\nobreak\mathcal{M}$ відображення \eqref{16f4} продовжується єдиним чином (за неперервністю) до ізоморфізму
\begin{equation}\label{16f11}
\Lambda:H^{s,s/(2b);\varphi}(\Omega)\leftrightarrow
\mathcal{Q}^{s-2m,(s-2m)/(2b);\varphi}.
\end{equation}
\end{theorem}

Ця теорема відома для анізотропних соболєвських просторів, тобто коли $\varphi(\cdot)\equiv1$. У цьому випадку її довели М.~С.~Аграновіч і М.~І.~Вішик \cite[теорема~12.1]{AgranovichVishik64} за умови, що  $s,s/(2b)\in\mathbb{Z}$. Ця умова була знята у випадку $b=1$ і нормальних крайових умов у монографії Ж.-Л.~Ліонса і Е.~Мадженеса \cite[теорема~6.2]{LionsMagenes72ii}, а в загальному випадку~--- у статті М.~В.~Житарашу \cite[теорема~9.1]{Zhitarashu85}. Ці результати охоплюють  граничний випадок, коли $s=\sigma_0$. Ми доведемо теорему~\ref{16th4.1} для довільного $\varphi\in\nobreak\mathcal{M}$ за допомогою вказаного результату М.~В.~Житарашу і квадратичної інтерполяції. Зауважимо, що у зазначених працях використано іншу форму умов узгодження, еквівалентну \eqref{16f8}. Її обговоримо одразу після доведення теореми~\ref{16th4.1}.

Щодо області значень ізоморфізму \eqref{16f11} зауважимо таке: потреба означати простір $\mathcal{Q}^{s-2m,(s-2m)/(2b);\varphi}$ у різні способи у випадку, коли $s\in E$, і у випадку, коли $s\notin E$, зумовлена такою обставиною \cite[зауваження~6.4]{LionsMagenes72ii}: якщо означити вказаний простір для $s\in E$ у той самий спосіб, що і для $s\notin E$, то висновок теореми~\ref{16th4.1} буде хибним для $s\in E$ принаймні, коли $\varphi(\cdot)\equiv1$. Для рівняння теплопровідності і крайових умов Діріхле і Неймана це показано  В.~А.~Солонніковим \cite[\S~6, c.~186]{Solonnikov64}. Звісно, це зауваження стосується і теореми~\ref{th4.1a}.

Спочатку доведемо версію леми~\ref{16lem7.4a} для гільбертового простору $\mathcal{Q}^{s-2m,(s-2m)/(2b);\varphi}$. Як і в п.~\ref{sec4.2.1},
$\{J_l:l\in\mathbb{N}\}$~--- набір усіх зв'язних компонент множини
$(\sigma_0,\infty)\setminus E$.

\begin{lemma}\label{16lem7.4}
Нехай $l\in\mathbb{N}$. Припустимо, що $s_0,s,s_1\in J_{l}$, $\nobreak{s_0<s<s_1}$ і $\varphi\in\mathcal{M}$. Означимо інтерполяційний параметр $\psi\in\mathcal{B}$ за формулою \eqref{9f7.2}. Тоді виконується така рівність просторів з еквівалентністю норм у них:
\begin{equation}\label{16f42}
\begin{aligned}
&\mathcal{Q}^{s-2m,(s-2m)/(2b);\varphi}=\\
&=\bigl[\mathcal{Q}^{s_0-2m,(s_0-2m)/(2b)},
\mathcal{Q}^{s_1-2m,(s_1-2m)/(2b)}\bigr]_{\psi}.
\end{aligned}
\end{equation}
\end{lemma}

\begin{proof}[\indent Доведення]
Скориставшись теоремами \ref{9prop6.3}, \ref{8prop4} і \ref{9lem7.3a},
отримаємо такі рівності:
\begin{align*}
&\bigl[\mathcal{H}^{s_0-2m,(s_0-2m)/(2b)},
\mathcal{H}^{s_1-2m,(s_1-2m)/(2b)}\bigr]_{\psi}= \\
&=\bigl[H^{s_0-2m,(s_0-2m)/(2b)}(\Omega),H^{s_1-2m,(s_1-2m)/(2b)}(\Omega)
\bigr]_{\psi}\oplus\\
&\quad\;
\oplus\bigoplus_{j=1}^{m}
\bigl[H^{s_0-m_j-1/2,\,(s_0-m_j-1/2)/(2b)}(S),\\
&\qquad\qquad\;\,
H^{s_1-m_j-1/2,\,(s_1-m_j-1/2)/(2b)}(S)\bigr]_{\psi}\oplus\\
&\qquad
\oplus\bigoplus_{k=0}^{\varkappa-1}\bigl[H^{s_0-2bk-b}(G),H^{s_1-2bk-b}(G)
\bigr]_{\psi}=\\
&=H^{s-2m,(s-2m)/(2b);\varphi}(\Omega)
\oplus\bigoplus_{j=1}^{m}H^{s-m_j-1/2,(s-m_j-1/2)/(2b);\varphi}(S)\oplus\\
&\phantom{}\qquad\qquad\qquad\qquad\qquad\quad\;
\oplus\bigoplus_{k=0}^{\varkappa-1}H^{s-2bk-b;\varphi}(G)=\\
&=\mathcal{H}^{s-2m,(s-2m)/(2b);\varphi}.
\end{align*}
Отже,
\begin{equation}\label{16f45}
\begin{aligned}
&\bigl[\mathcal{H}^{s_0-2m,(s_0-2m)/(2b)},
\mathcal{H}^{s_1-2m,(s_1-2m)/(2b)}\bigr]_{\psi}=\\
&=\mathcal{H}^{s-2m,(s-2m)/(2b);\varphi}
\end{aligned}
\end{equation}
з еквівалентністю норм.

Виведемо потрібну формулу \eqref{16f42} з щойно отриманої рівності \eqref{16f45} за допомогою теореми~\ref{9prop6.2}. Для цього побудуємо лінійне відображення $P$, задане на $\mathcal{H}^{s_0-2m,(s_0-2m)/(2b)}$
і таке, що $P$ є проєктором кожного простору
$\mathcal{H}^{s_j-2m,(s_j-2m)/(2b)}$, де $j\in\{0,1\}$, на його підпростір $\mathcal{Q}^{s_j-2m,(s_j-2m)/(2b)}$. Якщо таке відображення існує, то
на підставі теореми~\ref{9prop6.2} і формули \eqref{16f45} отримаємо рівності
\begin{align*}
&\bigl[\mathcal{Q}^{s_0-2m,(s_0-2m)/(2b)},
\mathcal{Q}^{s_1-2m,(s_1-2m)/(2b)}\bigr]_{\psi}=\\
&=\bigl[\mathcal{H}^{s_0-2m,(s_0-2m)/(2b)},
\mathcal{H}^{s_1-2m,(s_1-2m)/(2b)}\bigr]_{\psi}\cap\\
&\quad\;\cap\mathcal{Q}^{s_0-2m,(s_0-2m)/(2b)}=\\
&=\mathcal{H}^{s-2m,(s-2m)/(2b);\varphi}\cap
\mathcal{Q}^{s_0-2m,(s_0-2m)/(2b)}=\\
&=\mathcal{Q}^{s-2m,(s-2m)/(2b);\varphi},
\end{align*}
тобто формулу \eqref{16f42}. Тут остання рівність правильна, оскільки $s,s_0\in J_l$ і тому усі елементи підпросторів
$$
\mathcal{Q}^{s_0-2m,(s_0-2m)/(2b)}\quad\mbox{і}\quad
\mathcal{Q}^{s-2m,(s-2m)/(2b);\varphi}
$$
задовольняють одні й ті самі умови узгодження і до того ж виконується неперервне вкладення
$$
\mathcal{H}^{s-2m,(s-2m)/(2b);\varphi}\hookrightarrow
\mathcal{H}^{s_0-2m,(s_0-2m)/(2b)}.
$$
Зазначимо також, що пара
$$
\bigl[\mathcal{Q}^{s_0-2m,(s_0-2m)/(2b)},
\mathcal{Q}^{s_1-2m,(s_1-2m)/(2b)}\bigr]
$$
є регулярною за теоремою~\ref{9prop6.2}; отже, права частина рівності \eqref{16f42} має сенс.

Побудуємо вказане відображення $P$. Для кожного номера  $\nobreak{j\in\{1,\dots,m\}}$ розглянемо множину \eqref{set-kj}. З  умов узгодження \eqref{16f8} випливає, що вона не залежить від $s\in J_l$. Нехай $\nobreak{q_{l,j}:=q_{l,j}^{\star}-1}$, де $q_{l,j}^{\star}$~--- число її елементів.

Для довільного вектора
\begin{equation}\label{F-cylinder-Murach}
F:=\bigl(f,g_1,\dots,g_m,h_0,\dots,h_{\varkappa-1}\bigr)\in
\mathcal{H}^{s_0-2m,(s_0-2m)/(2b)}
\end{equation}
покладемо
\begin{equation}\label{16f43}
\left\{
  \begin{array}{ll}
    g^{*}_{j}:=g_j, & \hbox{якщо}\quad q_{l,j}=-1,\\
    g^{*}_{j}:=g_j+T(w_{j,0},\dots,w_{j,q_{l,j}}),
    & \hbox{якщо}\quad q_{l,j}\geq0,
  \end{array}
\right.
\end{equation}
де $j\in\{1,\dots,m\}$. Тут
\begin{align*}
w_{j,0}&:=B_{j,0}(v_{0},\dots,v_{[m_j/(2b)]})\!\upharpoonright\!
\Gamma-g_j\!\upharpoonright\!\Gamma,\\
&\dots\\
w_{j,q_{l,j}}&:=B_{j,q_{l,j}}(v_{0},\dots,v_{[m_j/(2b)]+q_{l,j}})
\!\upharpoonright\!\Gamma-
\partial_{t}^{q_{l,j}} g_j\!\upharpoonright\!\Gamma,
\end{align*}
де функції $v_{\mu}$ і диференціальні оператори $B_{j,k}$ означені формулами \eqref{16f9} і \eqref{16f9B}, а $T$~--- лінійний оператор з теореми~\ref{8lem1}, де беремо $r=q_{l,j}^\star$.

Лінійне відображення
\begin{equation*}
P:\bigl(f,g_1,\dots,g_m,h_0,\dots,h_{\varkappa-1}\bigr)\mapsto
\bigl(f,g^*_1,\dots,g^*_m,h_0,\dots,h_{\varkappa-1}\bigr),
\end{equation*}
задане на всіх векторах \eqref{F-cylinder-Murach}, є шуканим. Справді, його звужен\-ня на кожний простір $\mathcal{H}^{s_{j}-2m,(s_{j}-2m)/(2b)}$, де $j\in\{0,1\}$, є обмеженим оператором на ньому, що випливає з \eqref{8f69aa} і \eqref{8f43}. Ми використовуємо обмеженість оператора \eqref{8f43}, узявши $r=q_{l,j}^\star$, $s=s_{j}-m_j-1/2$ і $\varphi(\cdot)\equiv1$ в теоремі~\ref{8lem1}; при цьому умова $s>2br-b$ виконується, оскільки
$$
q_{l,j}^\star<\frac{s_{j}-m_j-1/2-b}{2b}+1.
$$

Крім того, з означення оператора $P$ і умов узгодження \eqref{16f8} випливає, що $PF\in\mathcal{Q}^{s_{j}-2m,(s_{j}-2m)/(2b)}$ для будь-якого вектора $F\in\mathcal{H}^{s_{j}-2m,(s_{j}-2m)/(2b)}$. Справді, ці умови для вектора
$$
PF=(f,g^*_1,\dots,g^*_m,h_0,\dots,h_{\varkappa-1})
$$
набувають вигляду
\begin{equation*}
\partial^{k}_t g_j^*\!\upharpoonright\!\Gamma=
B_{j,k}(v_0,\dots,v_{[m_j/(2b)]+k})\!\upharpoonright\!\Gamma
\end{equation*}
для всіх індексів $j$ і $k$, вказаних у \eqref{16f8}. Але
\begin{align*}
\partial^{k}_t g_j^*\!\upharpoonright\!\Gamma=&
\partial^{k}_t g_j\!\upharpoonright\!\Gamma+
\partial^{k}_t T(w_{j,0},\dots,w_{j,q_{l,j}})\!\upharpoonright\!\Gamma=\\
=& \partial^{k}_t g_j\!\upharpoonright\!\Gamma+w_{j,k}=
B_{j,k}(v_0,\dots,v_{[m_j/(2b)]+k})\!\upharpoonright\!\Gamma,
\end{align*}
якщо $q_{l,j}\geq0$. Якщо $q_{l,j}=-1$, то нема тих умов узгодження, які містять~$g_{j}$. Таким чином, вектор $PF$ задовольняє умови узгодження \eqref{16f8}, тобто належить до
$\mathcal{Q}^{\sigma-2m,(\sigma-2m)/(2b)}$. Нарешті, $PF=F$ для довільного вектора $F\in\mathcal{Q}^{s_{j}-2m,(s_{j}-2m)/(2b)}$, оскільки він задовольняє умови \eqref{16f8}, з яких випливає, що всі $w_{j,k}=0$, тобто $g_1^*=g_1$,..., $g_m^*=g_m$.
\end{proof}

Лема~\ref{16lem7.4} відіграватиме ключову роль у доведенні теореми~\ref{16th4.1} у випадку, коли $s\notin E$. У противному випадку знадобиться такий результат:

\begin{lemma}\label{16lem7.5}
Нехай $s,\varepsilon\in\mathbb{R}$, $s>\varepsilon>0$ і $\varphi\in\mathcal{M}$. Тоді виконується така рівність просторів з еквівалентністю норм у них:
\begin{equation}\label{16f46}
\begin{aligned}
&H^{s,s/(2b);\varphi}(W)=\\
&=\bigl[H^{s-\varepsilon,(s-\varepsilon)/(2b);\varphi}(W),
H^{s+\varepsilon,(s+\varepsilon)/(2b);\varphi}(W)\bigr]_{1/2},
\end{aligned}
\end{equation}
де $W:=\Omega$ або $W:=S$.
\end{lemma}

\begin{proof}[\indent Доведення]
Виберемо число $\delta>0$ таке, що $s-\varepsilon-\delta>0$.
За теоремою~\ref{9lem7.3a} маємо рівності
\begin{align*}
&H^{s-\varepsilon,(s-\varepsilon)/(2b);\varphi}(W)=\\
&=\bigl[H^{s-\varepsilon-\delta,(s-\varepsilon-\delta)/(2b)}(W),
H^{s+\varepsilon+\delta,(s+\varepsilon+\delta)/(2b)}(W)\bigr]_{\chi}
\end{align*}
та
\begin{align*}
&H^{s+\varepsilon,(s+\varepsilon)/(2b);\varphi}(W)=\\
&=\bigl[H^{s-\varepsilon-\delta,(s-\varepsilon-\delta)/(2b)}(W),
H^{s+\varepsilon+\delta,(s+\varepsilon+\delta)/(2b)}(W)\bigr]_{\eta}.
\end{align*}
Тут інтерполяційні параметри  $\chi$ і $\eta$ означено формулами
\begin{equation*}
\begin{split}
&\chi(r):=r^{\delta/(2\varepsilon+2\delta)}
\varphi(r^{1/(2\varepsilon+2\delta)}),\\
&\eta(r):=r^{(2\varepsilon+\delta)/(2\varepsilon+2\delta)}
\varphi(r^{1/(2\varepsilon+2\delta)}),
\end{split}
\end{equation*}
якщо $r\geq1$, і $\chi(r)=\eta(r):=1$, якщо $0<r<1$. Тому на підставі теорем~\ref{9prop6.4} і~\ref{9lem7.3a} маємо рівності
\begin{align*}
&\bigl[H^{s-\varepsilon,(s-\varepsilon)/(2b);\varphi}(W),
H^{s+\varepsilon,(s+\varepsilon)/(2b);\varphi}(W)\bigr]_{1/2}=\\
&=\Bigl[
\bigl[H^{s-\varepsilon-\delta,(s-\varepsilon-\delta)/(2b)}(W),
H^{s+\varepsilon+\delta,(s+\varepsilon+\delta)/(2b)}(W)\bigr]_{\chi},\\
&\quad\;\;\;\bigl[H^{s-\varepsilon-\delta,(s-\varepsilon-\delta)/(2b)}(W),
H^{s+\varepsilon+\delta,(s+\varepsilon+\delta)/(2b)}(W)\bigr]_{\eta}
\Bigr]_{1/2}=\\
&=\bigl[H^{s-\varepsilon-\delta,(s-\varepsilon-\delta)/(2b)}(W),
H^{s+\varepsilon+\delta,(s+\varepsilon+\delta)/(2b)}(W)\bigr]_{\omega}=\\
&=H^{s,s/(2b);\varphi}(W).
\end{align*}
Тут інтерполяційний параметр $\omega$ означено формулами
\begin{equation*}
\omega(r):=\chi(r)(\eta(r)/\chi(r))^{1/2}=
r^{1/2}\varphi(r^{1/(2\varepsilon+2\delta)}),
\;\;\mbox{якщо}\;\;r\geq1,
\end{equation*}
і $\omega(r):=1$, якщо $0<r<1$.
Стосовно застосування теореми~\ref{9lem7.3a} зауважимо, що
інтерполяційний параметр $\omega$ дорівнює правій частині рівності \eqref{9f7.2}, якщо покласти у ній $s_0:=s-\varepsilon-\delta$ і $s_1:=s+\varepsilon+\delta$. Потрібну формулу \eqref{16f46} доведено.
\end{proof}

Буде корисною також версія леми \ref{16lem7.5} для гільбертового простору $\mathcal{H}^{s-2m,(s-2m)/(2b);\varphi}$.

\begin{lemma}\label{16lem7.6}
Нехай $s\in\mathbb{R}$, $\varepsilon>0$ і $\varphi\in\mathcal{M}$. Припустимо, що $s-\varepsilon>\sigma_0$. Тоді виконується така рівність просторів з еквівалентністю норм у них:
\begin{equation}\label{16f67}
\begin{split}
&\mathcal{H}^{s-2m,(s-2m)/(2b);\varphi}=\\
&=\bigl[\mathcal{H}^{s-2m-\varepsilon,(s-2m-\varepsilon)/(2b);\varphi},
\mathcal{H}^{s-2m+\varepsilon,(s-2m+\varepsilon)/(2b);\varphi}\bigr]_{1/2}.
\end{split}
\end{equation}
\end{lemma}

\begin{proof}[\indent Доведення]
Ця рівність випливає з твердження~\ref{9prop6.3}, формули \eqref{16f46} та її аналога для ізотропних просторів $H^{\sigma;\varphi}(G)$, де $\sigma>0$. Останній доводиться так само як і формула \eqref{16f46} (див. також \cite[лема~4.3]{MikhailetsMurach14}). Справді,
\begin{align*}
&\bigl[\mathcal{H}^{s-\varepsilon-2m,(s-\varepsilon-2m)/(2b);\varphi},
\mathcal{H}^{s+\varepsilon-2m,(s+\varepsilon-2m)/(2b);\varphi}\bigr]_{1/2}= \\
&=\bigl[H^{s-2m-\varepsilon,(s-2m-\varepsilon)/(2b);\varphi}(\Omega),
H^{s-2m+\varepsilon,(s-2m+\varepsilon)/(2b);\varphi}(\Omega)\bigr]_{1/2}
\oplus\\
&\quad\;\oplus\bigoplus_{j=1}^{m}
\bigl[H^{s-m_j-1/2-\varepsilon,\,(s-m_j-1/2-\varepsilon)/(2b);\varphi}(S),\\
&\qquad\qquad\;\;
H^{s-m_j-1/2+\varepsilon,\,(s-m_j-1/2+\varepsilon)/(2b);\varphi}(S)\bigr]_
{1/2}\oplus\\
&\quad\;
\oplus\bigoplus_{k=0}^{\varkappa-1}\bigl[H^{s-2bk-b-\varepsilon;\varphi}(G),
H^{s-2bk-b+\varepsilon;\varphi}(G)\bigr]_{1/2}=\\
&=H^{s-2m,(s-2m)/(2b);\varphi}(\Omega)
\oplus\bigoplus_{j=1}^{m}H^{s-m_j-1/2,(s-m_j-1/2)/(2b);\varphi}(S)\oplus\\
&\qquad\qquad\qquad\qquad\qquad\quad\,\,
\oplus\bigoplus_{k=0}^{\varkappa-1}H^{s-2bk-b;\varphi}(G)=\\
&=\mathcal{H}^{s-2m,(s-2m)/(2b);\varphi},
\end{align*}
тобто виконується рівність \eqref{16f67}.
\end{proof}

\begin{proof}[\indent Доведення теореми $\ref{16th4.1}$.]
Нехай $s>\sigma_0$ і $\varphi\in\mathcal{M}$. Дослідимо спочатку випадок, коли $s\notin E$. Тоді $s\in J_{l}$ для деякого $l\in\mathbb{N}$.
Виберемо числа $s_0,s_1\in J_{l}$ такі, що
$s_0<s<s_1$,  $s_j+1/2\notin\mathbb{Z}$ і $s_j/(2b)+1/2\notin\mathbb{Z}$ для будь-якого номера $j\in\{0,1\}$. Згідно з результатом М.~В.~Житарашу \cite[теорема~9.1]{Zhitarashu85}, відображення \eqref{16f4} продовжується єдиним чином (за неперервністю) до ізоморфізмів
\begin{equation*}
\Lambda:H^{s_j,s_j/(2b)}(\Omega)\leftrightarrow
\mathcal{Q}^{s_j-2m,(s_j-2m)/(2b)},
\end{equation*}
де $j\in\{0,1\}$ (див. також монографію \cite[теорема~5.7]{EidelmanZhitarashu98}). Застосувавши до них квадратичну інтерполяцію з функціональним параметром $\psi$, означеним за  формулою~\eqref{9f7.2}, отримаємо ізоморфізм
\begin{equation}\label{16f50}
\begin{aligned}
\Lambda&:H^{s,s/(2b);\varphi}(\Omega)=
\bigl[H^{s_0,s_{0}/(2b)}(\Omega),H^{s_1,s_{1}/(2b)}(\Omega)\bigr]_{\psi}
\leftrightarrow\\
&\leftrightarrow
\bigl[\mathcal{Q}^{s_0-2m,(s_0-2m)/(2b)},
\mathcal{Q}^{s_1-2m,(s_1-2m)/(2b)}\bigr]_{\psi}=\\
&\quad\;\;=\mathcal{Q}^{{s-2m,(s-2m)/(2b)};\varphi}.
\end{aligned}
\end{equation}
Тут рівності просторів виконуються з еквівалентністю норм у них завдяки теоремі~\ref{9lem7.3a} і лемі~\ref{16lem7.4}. Оператор \eqref{16f50} є продовженням за неперервністю відображення \eqref{16f4}, оскільки множина $C^{\infty}(\overline{\Omega})$ щільна у просторі $H^{s,s/(2b);\varphi}(\Omega)$. У розглянутому випадку теорему~\ref{16th4.1} доведено.

Дослідимо тепер випадок, коли $s\in E$. Виберемо довільно число $\varepsilon\in(0,1/2)$. Оскільки $s\pm\varepsilon\notin E$ і $s-\varepsilon>\sigma_0$, то за щойно доведеним виконуються ізоморфізми
\begin{equation*}
\Lambda:H^{s\pm\varepsilon,(s\pm\varepsilon)/(2b);\varphi}(\Omega)
\leftrightarrow
\mathcal{Q}^{s\pm\varepsilon-2m,(s\pm\varepsilon-2m)/(2b);\varphi}.
\end{equation*}
Застосувавши до них квадратичну інтерполяцію з числовим параметром $1/2$, отримаємо потрібний ізоморфізм
\begin{equation}\label{8f38}
\begin{aligned}
&\Lambda:H^{s,s/(2b);\varphi}(\Omega)=\\
&\quad\;=\bigl[H^{s-\varepsilon,(s-\varepsilon)/(2b);\varphi}(\Omega),
H^{s+\varepsilon,(s+\varepsilon)/(2b);\varphi}(\Omega)\bigr]_{1/2}
\leftrightarrow\\
&\leftrightarrow
\bigl[\mathcal{Q}^{s-\varepsilon-2m,(s-\varepsilon-2m)/(2b);\varphi},
\mathcal{Q}^{s+\varepsilon-2m,(s+\varepsilon-2m)/(2b);\varphi}
\bigr]_{1/2}=\\
&\quad\;\;=\mathcal{Q}^{s-2m,(s-2m)/(2b);\varphi}.
\end{aligned}
\end{equation}
Тут перша рівність виконується за лемою~\ref{16lem7.5}, а друга~--- за означенням простору $\mathcal{Q}^{s-2m,(s-2m)/(2b);\varphi}$ у випадку, коли $s\in E$. Отже, теорему доведено і у цьому випадку.
\end{proof}

Повторимо, що це доведення спирається на результат М.~В.~Жи\-та\-рашу \cite[теорема~9.1]{Zhitarashu85}, наведений також в  \cite[Теорема~5.7]{EidelmanZhitarashu98}. У ньому використано іншу (менш явну) форму умов узгодження, введену в \cite[\S\,11]{AgranovichVishik64}.
Нагадаємо її та покажемо, що вона еквівалентна умовам \eqref{16f8}.

Нехай $s\geq\sigma_0$, $s\notin E$ і $s/(2b)+1/2\notin\mathbb{Z}$.
Вектор
\begin{equation}\label{vector-F}
F:=\bigl(f,g_1,...,g_m,h_0,...,h_{\varkappa-1}\bigr)\in
\mathcal{H}^{s-2m,(s-2m)/(2b)}
\end{equation}
задовольняє умови узгодження в сенсі \cite[\S\,11]{AgranovichVishik64}, якщо існує функція $v=v(x,t)$ з простору $H^{s,s/(2b)}(\Omega)$ така, що
\begin{gather}\label{16f73}
f-Av\in H^{s-2m,(s-2m)/(2b)}_{+}(\Omega), \\ \label{16f74}
g_j-B_jv\in H^{s-m_j-1/2,(s-m_j-1/2)/(2b)}_{+}(S)\\ \notag
\mbox{для кожного}\quad j\in\{1,\dots,m\},\\ \label{16f75}
h_k=\partial^{k}_{t}v\big|_{t=0}\quad\mbox{для кожного}\quad
k\in\{0,\dots,\varkappa-1\}.
\end{gather}
(Помітимо, що обмеження, накладені на $s_{j}$ на початку доведення теореми~\ref{16th4.1} і спричинені використанням леми~\ref{16lem7.4} у цьому доведенні, є дещо сильнішими, ніж щойно зроблені припущення стосовно $s$.)

Доведемо, що набір цих умов узгодження еквівалентний \eqref{16f8} для кожного вектора \eqref{vector-F}.
Скориставшись лемою~\ref{9lem5.1}, запишемо умову
\eqref{16f73} в еквівалентній формі:
\begin{equation}\label{16f76}
\begin{gathered}
\partial^l_t(f-Av)(x,0)=0\;\;\mbox{для майже всіх}\;\;x\in G,\\
\mbox{де}\;\;l\in\mathbb{Z}\;\;
\mbox{і}\;\;0\leq l<\frac{s-2m}{2b}-\frac{1}{2},
\end{gathered}
\end{equation}
а умову \eqref{16f74} в еквівалентній формі:
\begin{equation}\label{16f77}
\begin{gathered}
\partial^k_t(g_j-B_jv)(x,0)=0\;\;\mbox{для майже всіх}\;\;x\in\Gamma,\\
\mbox{де}\;\;j\in\{1,\dots,m\},\;\;k\in\mathbb{Z}\;\;\mbox{і}\;\;
0\leq k<\frac{s-m_j-1/2}{2b}-\frac{1}{2}.
\end{gathered}
\end{equation}
Набір умов \eqref{16f75} і \eqref{16f76} еквівалентний співвідношенням \eqref{16f9}, де $0\leq \mu<s/(2b)-1/2$, якщо покладемо
\begin{equation}\label{functions-v}
\begin{gathered}
v_\mu(x)=\partial^{\mu}_{t}v(x,0)\;\;\mbox{для майже всіх}\;\;x\in G,\\
\mbox{де}\;\;\mu\in\mathbb{Z}\;\;\mbox{і}\;\;
0\leq\mu<\frac{s}{2b}-\frac{1}{2}.
\end{gathered}
\end{equation}
Справді, \eqref {16f76} еквівалентно набору таких умов:
\begin{gather*}
\partial^{\mu-\varkappa}_{t}\frac{f-Av}{a^{(0,\ldots,0),\varkappa}}
(x,0)=0\;\;\mbox{для майже всіх}\;\;x\in G,\\
\mbox{де}\;\mu\in\mathbb{Z}\;\;\mbox{і}\;\;
\varkappa\leq\mu<\frac{s}{2b}-\frac{1}{2}.
\end{gather*}
За останньої умови на ціле $\mu$ виконується рівність
\begin{align*}
&-\partial^{\mu-\varkappa}_{t}
\frac{Av}{a^{(0,\ldots,0),\varkappa}}(x,0)=\\
&=\sum_{|\alpha|+2b\beta\leq 2m}\,
\sum\limits_{q=0}^{\mu-\varkappa}
\binom{\mu-\varkappa}{q}(\partial^{\mu-\varkappa-q}_{t}
a_{0}^{\alpha,\beta})(x,0)\,D^\alpha_{x}v_{\beta+q}(x)
\end{align*}
для майже всіх $x\in G$. Тут функція $v_{\mu}$ присутня лише у доданку, для якого $\nobreak{\alpha=(0,\ldots,0)}$, $\beta=\varkappa$ і $q=\mu-\varkappa$, причому він дорівнює $-v_{\mu}$. Отже,
\begin{align*}
v_{\mu}(x)&=\partial^{\mu-\varkappa}_{t}
\frac{Av}{a^{(0,\ldots,0),\varkappa}}(x,0)+\\
&+\sum_{\substack{|\alpha|+2b\beta\leq 2m\\ \beta\leq\varkappa-1}}\,
\sum\limits_{q=0}^{\mu-\varkappa}
\binom{\mu-\varkappa}{q}\partial^{\mu-\varkappa-q}_t
a_{0}^{\alpha,\beta}(x,0)\,D^\alpha_{x}v_{\beta+q}(x).
\end{align*}
Звідси випливає, що набір умов \eqref {16f75} і \eqref{16f76} еквівалентний співвідношенням \eqref{16f9}, де $\varkappa\leq\mu<s/(2b)-1/2$.

Припустимо тепер, що вектор \eqref{vector-F} задовольняє умови узгодження \eqref{16f8}, в яких $v_{\mu}$ і $B_{j,k}$ означені формулами \eqref{16f9} і \eqref{16f9B}. Маємо вектор
$$
(v_0,...,v_r)\in \bigoplus_{k=0}^{r}H^{s-2bk-b}(G),
$$
де $r:=[s/(2b)-1/2]$. Як відомо \cite[розд.~2, теорема~10]{Slobodetskii58}, існує функція $v\in H^{s,s/(2b)}(\Omega)$, яка задовольняє \eqref{functions-v}. З цього факту та з формул \eqref{16f2} і \eqref{16f9B} випливає, що
\begin{equation}\label{last-formula}
\begin{gathered}
\partial^k_{t}B_jv\big|_{t=0}=B_{j,k}(v_0,\dots,v_{[m_j/(2b)]+k})
\;\;\mbox{на}\;\;G,\\
\mbox{де}\;\;j\in\{1,\dots,m\}\;\;k\in\mathbb{Z}\;\;\mbox{і}\;\;
0\leq k<\frac{s-m_j-1/2-b}{2b}.
\end{gathered}
\end{equation}
Отже, з умов \eqref{16f8} випливає \eqref{16f77}. Крім того, як показано вище, з формули \eqref{16f9} випливають співвідношення \eqref{16f75} і \eqref{16f76}. Таким чином, вектор \eqref{vector-F} задовольняє умови \eqref{16f73}--\eqref{16f75}.

Зворотно, припустимо, що вектор \eqref{vector-F} задовольняє умови  \eqref{16f73}--\eqref{16f75} для деякої функції $v\in H^{s,s/(2b)}(\Omega)$. Звідси випливають співвідношення \eqref{16f9} і \eqref {16f77} за умови, що функції $v_\mu$ означені формулою \eqref{functions-v}. Тоді умови узгодження \eqref{16f8} є наслідками  цих співвідношень з огляду на \eqref{last-formula}.

\begin{remark}\label{16rem8.1}
Нехай $s\in E$. Простір $\mathcal{Q}^{s-2m,(s-2m)/(2b);\varphi}$, означений формулою \eqref{16f10}, не залежить з точністю до еквівалентності норм від вибору числа $\varepsilon\in(0,1/2)$.
Справді, за теоремою \ref{16th4.1} виконується ізоморфізм
\begin{equation*}
\begin{split}
\Lambda&:H^{s,s/(2b);\varphi}(\Omega)\leftrightarrow\\
&\leftrightarrow
\bigl[\mathcal{Q}^{s-\varepsilon-2m,(s-\varepsilon-2m)/(2b);\varphi},
\mathcal{Q}^{s+\varepsilon-2m,(s+\varepsilon-2m)/(2b);\varphi}\bigr]_{1/2},
\end{split}
\end{equation*}
якщо $0<\varepsilon<1/2$. Це означає вказану незалежність.
Крім того, простір $\mathcal{Q}^{s-2m,(s-2m)/(2b);\varphi}$ є неперервно вкладеним у простір $\mathcal{H}^{s-2m,(s-2m)/(2b);\varphi}$. Справді, для кожного $\varepsilon\in(0,1/2)$ виконуються неперервні вкладення
\begin{equation*}
\mathcal{Q}^{s\mp\varepsilon-2m,(s\mp\varepsilon-2m)/(2b);\varphi}
\hookrightarrow
\mathcal{H}^{s\mp\varepsilon-2m,(s\mp\varepsilon-2m)/(2b);\varphi}
\end{equation*}
з урахуванням включень $s\mp\varepsilon\in(\sigma_{0},\infty)\setminus E$ і означення просторів, що стоять ліворуч у цих вкладеннях. Звідси на підставі формул \eqref{16f10} і \eqref{16f67} робимо висновок, що оператор вкладення є неперервним на парі просторів
$$
\mathcal{Q}^{s-2m,(s-2m)/(2b);\varphi}\quad\mbox{і}\quad \mathcal{H}^{s-2m,(s-2m)/(2b);\varphi}.
$$
\end{remark}

\markright{\emph \ref{sec4.1.3}. Регулярність узагальнених розв'язків}

\section[Регулярність узагальнених розв'язків]
{Регулярність узагальнених розв'язків}\label{sec4.1.3}

\markright{\emph \ref{sec4.1.3}. Регулярність узагальнених розв'язків}

Розглянемо застосування теореми \ref{16th4.1} про ізоморфізми, породжені параболічною задачею \eqref{16f1}--\eqref{16f3}, до дослідження регулярності її узагальненого розв'язку. Дамо його означення.

Нехай праві частини $f$, $g_{j}$ і $h_k$ вказаної задачі є довільними розподілами на $\Omega$, $S$ і $G$ відповідно. Функцію
$u\in H^{\sigma_0,\sigma_0/(2b)}(\Omega)$ називаємо узагальненим розв'язком цієї задачі, якщо
\begin{equation}\label{Lambdau=f-Murach}
\Lambda u=(f,g_{1},\dots,g_{m},h_0,\dots,h_{\varkappa-1}),
\end{equation}
де $\Lambda$~--- обмежений оператор \eqref{16f4a}, розглянутий у випадку, коли  $s=\sigma_0$. З рівності \eqref{Lambdau=f-Murach} випливає, що
\begin{equation}\label{16f12}
(f,g_{1},\dots,g_{m},h_0,\dots,h_{\varkappa-1})\in
\mathcal{Q}^{\sigma_0-2m,(\sigma_0-2m)/(2b)}.
\end{equation}
Згідно з \cite[теорема~9.1]{Zhitarashu85} параболічна задача \eqref{16f1}--\eqref{16f3} має єдиний узагальнений розв'язок  $u\in H^{\sigma_0,\sigma_0/(2b)}(\Omega)$ для кожного вектора \eqref{16f12} .

Негайним наслідком теореми~\ref{16th4.1} є така властивість:

\begin{corollary}\label{16cor4.2}
Припустимо, що функція $u\in H^{\sigma_0,\sigma_0/(2b)}(\Omega)$
є узагальненим розв'язком параболічної задачі
\eqref{16f1}--\eqref{16f3}, праві частини якої задовольняють умову
$$
(f,g_{1},\dots,g_{m},h_0,\dots,h_{\varkappa-1})\in
\mathcal{Q}^{s-2m,(s-2m)/(2b);\varphi}
$$
для деяких $s>\sigma_0$ і $\varphi\in\mathcal{M}$. Тоді $u\in H^{s,s/(2b);\varphi}(\Omega)$.
\end{corollary}

Сформулюємо локальну версію цього результату. Нехай $U$~--- відкрита множина в $\mathbb{R}^{n+1}$ така, що $\nobreak{\Omega_0:=U\cap\Omega\neq\varnothing}$ і $U\cap\Gamma=\varnothing$.
Покладемо $\Omega':=U\cap\partial\overline{\Omega}$, $S_0:=U\cap S$, $S':=U\cap \{(x,\tau):x\in\Gamma\}$ і $G_0:=U\cap G$.
Відповідний результат зручно формулювати за допомогою локальних версій просторів $H^{\sigma,\sigma/(2b);\varphi}(\Omega)$,  $H^{\sigma,\sigma/(2b);\varphi}(S)$ і $H^{\sigma;\varphi}(G)$, де $\sigma>0$ і $\varphi\in\mathcal{M}$.

Позначимо через $H^{\sigma,\sigma/(2b);\varphi}_{\mathrm{loc}}(\Omega_0,\Omega')$ лінійний простір усіх розподілів $u$ на $\Omega$ таких, що $\chi u\in H^{\sigma,\sigma/(2b);\varphi}(\Omega)$ для кожної функції $\chi\in C^\infty (\overline\Omega)$, яка задовольняє умову $\mathrm{supp}\,\chi\subset\Omega_0\cup\Omega'$.
Аналогічно, позначимо через $H^{\sigma,\sigma/(2b);\varphi}_{\mathrm{loc}}(S_0,S')$ лінійний простір усіх розподілів $v$ на $S$ таких, що $\chi v\in H^{\sigma,\sigma/(2b);\varphi}(S)$ для будь-якої функції $\chi\in C^\infty (\overline S)$, яка задовольняє умову $\mathrm{supp}\,\chi\subset S_0\cup S'$. Нарешті, $H^{\sigma;\varphi}_{\mathrm{loc}}( G_0)$ позначає лінійний простір усіх розподілів $w$ на $G$ таких, що $\chi w\in H^{\sigma;\varphi}(G)$ для кожної функції $\chi\in C^\infty (\overline G)$, яка задовольняє умову $\mathrm{supp}\,\chi\subset G_0$.

\begin{theorem}\label{16th4.3}
Нехай $s>\sigma_0$ і $\varphi\in\mathcal{M}$. Припустимо, що функція
$u\in H^{\sigma_0,\sigma_0/(2b)}(\Omega)$ є узагальненим розв'язком параболічної задачі \eqref{16f1}--\eqref{16f3}, праві частини якої задовольняють такі умови:
\begin{equation}\label{16f13}
f\in H^{s-2m,(s-2m)/(2b);\varphi}_{\mathrm{loc}}(\Omega_0,\Omega'),
\end{equation}
\begin{equation}\label{16f14}
\begin{gathered}
g_{j}\in H^{s-m_j-1/2,(s-m_j-1/2)/(2b);\varphi}_{\mathrm{loc}}(S_0,S')\\
\mbox{для кожного}\;\;j\in\{1,\dots,m\},
\end{gathered}
\end{equation}
\begin{equation}\label{16f15}
\begin{gathered}
h_{k}\in H^{s-2bk-b;\varphi}_{\mathrm{loc}}(G_0)\\
\mbox{для кожного}\;\;k\in\{0,\dots,\varkappa-1\}.
\end{gathered}
\end{equation}
Тоді $u\in H^{s,s/(2b);\varphi}_{\mathrm{loc}}(\Omega_0,\Omega')$.
\end{theorem}

Якщо $\Omega'=\varnothing$, то теорема~\ref{16th4.3} стверджує, що підвищення регулярності узагальненого розв'язку $u$ відбувається в околах внутрішніх точок множини $\overline{\Omega}$, при цьому умови \eqref{16f14} і \eqref{16f15} зникають. Якщо $G_0=\varnothing$, то ця теорема стверджує, що регулярність розв'язку $u(x,t)$ підвищується при $t>0$, при цьому зникає умова \eqref{16f15}. У цих випадках теорема~\ref{16th4.3} випливає з теореми~\ref{9th4.3}, якщо $\sigma_0/(2b)\in\mathbb{Z}$. Наголосимо, що зроблене припущення $U\cap\Gamma=\varnothing$ є істотним, оскільки без нього висновок теореми~\ref{16th4.3} є взагалі хибним. Умови \eqref{16f13}--\eqref{16f15} є необхідними для того, щоб $u\in H^{s,s/(2b);\varphi}_{\mathrm{loc}}(\Omega_0,\Omega')$ у випадку $U\cap\Gamma\neq\varnothing$, але не є достатніми. Для того, щоб висновок цієї теореми був істинним у цьому випадку, треба, щоб праві частини задачі \eqref{16f1}--\eqref{16f3} задовольняли деякі додаткові умови узгодження на множині $U\cap\Gamma$.

\begin{proof}[\indent Доведення теореми $\ref{16th4.3}$.]
Спочатку покажемо, що з умов \eqref{16f13}--\eqref{16f15} цієї теореми випливає, що імплікація
\begin{equation}\label{16f54} 
\begin{aligned}
&u\in H^{s-\lambda,(s-\lambda)/(2b);\varphi}_{\mathrm{loc}}
(\Omega_0,\Omega')\Longrightarrow\\
&\Longrightarrow u\in H^{s-\lambda+1,(s-\lambda+1)/(2b);\varphi}_{\mathrm{loc}}
(\Omega_0,\Omega')
\end{aligned}
\end{equation}
істинна для кожного цілого числа $\lambda\geq1$, яке задовольняє нерівність $s-\lambda+1>\sigma_{0}$.

Виберемо довільно функцію  $\chi\in C^\infty(\overline\Omega)$, яка задовольняє умову $\mbox{supp}\,\chi\subset\Omega_0\cup\Omega'$. Для $\chi$ знайдеться функція $\eta\in C^\infty(\overline\Omega)$ така, що $\mbox{supp}\,\eta\subset\Omega_0\cup\Omega'$ і
$\eta=1$ в околі $\mbox{supp}\,\chi$. Переставивши диференціальні  оператори $A$, $B_{j}$ і $\partial^k_t$ з оператором множення на функцію $\chi$, запишемо
\begin{equation}\label{16f55}
\begin{aligned}
\Lambda(\chi u)&=\Lambda(\chi\eta u)=
\chi\Lambda(\eta u)+ \Lambda'(\eta u)=\\
&=\chi\Lambda u+\Lambda'(\eta u)=\\
&=\chi(f,g_{1},\dots,g_{m},h_{0},\dots,h_{\varkappa-1})+
\Lambda'(\eta u).
\end{aligned}
\end{equation}
Тут
$$
\Lambda':=(A',B'_{1},\ldots,B'_{m},C_{0},\dots C_{\varkappa-1})
$$
є оператором, компоненти якого означені формулами
\begin{align*}
&A'(x,t,D_x,\partial_t)w(x,t)=\\
&=\sum_{|\alpha|+2b\beta\leq 2m-1}a^{\alpha,\beta}_{1}(x,t)\,D^\alpha_x\partial^\beta_t w(x,t),\\
&B_{j}'(x,t,D_x,\partial_t)w(x,t)=\\
&=\sum_{|\alpha|+2b\beta\leq m_j-1}
b_{j,1}^{\alpha,\beta}(x,t)\,D^\alpha_x\partial^\beta_t
w(x,t)\!\upharpoonright\!S,
\end{align*}
де $j\in\{1,\ldots,m\}$, та
\begin{align*}
C_{0}w=0\quad\mbox{і}\quad
C_{k}(x,\partial_t)w=\sum_{l=0}^{k-1}
c_{l,k}(x)\,(\partial^{l}_t w)(x,0),
\end{align*}
де $k\in\{1,\ldots,\varkappa-1\}$. У цих формулах $w=w(x,t)$~--- довільна функція з простору $H^{\sigma_{0}-1,(\sigma_{0}-1)/(2b)}(\Omega)$. Крім того,  $a^{\alpha,\beta}_{1}\in C^{\infty}(\overline{\Omega})$,
$b_{j,1}^{\alpha,\beta}\in C^{\infty}(\overline{S})$ і $c_{l,\,k}\in C^{\infty}(\overline{G})$~--- деякі функції. Коротко кажучи, $A'$, $B_{j}'$ і $C_{k}$~--- деякі диференціальні оператори, порядки яких принаймні на одиницю менші, ніж порядки операторів $A$, $B_{j}$ і $\partial^k_t$ відповідно.

Оператор $\Lambda'$ обмежений на парі просторів
\begin{equation}\label{16f59}
\Lambda':H^{\sigma,\sigma/(2b);\varphi}(\Omega)\rightarrow
\mathcal{H}^{\sigma+1-2m,(\sigma+1-2m)/(2b);\varphi}
\end{equation}
для кожного числа $\sigma>\sigma_{0}-1$. Якщо $\varphi(\cdot)\equiv1$, це випливає з відомих властивостей диференціальних операторів і операторів слідів на парах анізотропних просторів Соболєва (див., наприклад, \cite[розд.~I, лема~4 і розд.~II, теореми~3 і~7]{Slobodetskii58}). При цьому зауважимо, що кожний оператор $C_{k}$, де $1\leq k\leq\varkappa-1$, є обмеженим на парі просторів $$
H^{\sigma,\sigma/(2b)}(\Omega)\quad\mbox{і}\quad
H^{\sigma-2b(k-1)-b}(G)\hookrightarrow H^{\sigma+1-2bk-b}(G).
$$
Звідси обмеженість оператора \eqref{16f59} для довільного $\varphi\in\mathcal{M}$ випливає за допомогою квадратичної інтерполяції на підставі теорем \ref{9lem7.3a} і~\ref{8prop4}.

На підставі умов \eqref{16f13}--\eqref{16f15} маємо включення
$$
\chi\,(f,g_{1},\dots,g_{m},h_0,\dots,h_{\varkappa-1})
\in\mathcal{H}^{s-2m,(s-2m)/(2b);\varphi}.
$$
Крім того, згідно з формулою \eqref{16f59}, де беремо $\sigma:=s-\lambda$, є істинною імплікація
\begin{align*}
&u\in H^{s-\lambda,(s-\lambda)/(2b);\varphi}_{\mathrm{loc}}
(\Omega_0,\Omega')\Longrightarrow\\
&\Longrightarrow\Lambda'(\eta u)\in
\mathcal{H}^{s-\lambda+1-2m,(s-\lambda+1-2m)/(2b);\varphi}.
\end{align*}
Тому, скориставшись рівністю \eqref{16f55}, отримаємо імплікацію
\begin{equation}\label{16f60}
\begin{aligned}
&u\in H^{s-\lambda,(s-\lambda)/(2b);\varphi}_{\mathrm{loc}}
(\Omega_0,\Omega')\Longrightarrow\\
&\Longrightarrow\Lambda(\chi u)\in
\mathcal{H}^{s-\lambda+1-2m,(s-\lambda+1-2m)/(2b);\varphi}.
\end{aligned}
\end{equation}

Для того щоб вивести потрібну властивість \eqref {16f54} з формули \eqref{16f60}, покажемо, що
\begin{equation}\label{16f61}
\begin{aligned}
\Lambda(\chi u)\in
\mathcal{H}^{\sigma-2m,(\sigma-2m)/(2b);\varphi}\Longrightarrow\\
\Longrightarrow
\Lambda(\chi u)\in\mathcal{Q}^{\sigma-2m,(\sigma-2m)/(2b);\varphi}
\end{aligned}
\end{equation}
для будь-якого числа $\sigma>\sigma_0$. Припустимо, що посилка $\nobreak\mbox{цієї}$ імплікації є істинною для деякого числа $\sigma>\sigma_0$. Оскільки $\mathrm{dist}(\mathrm{supp}\,\chi,\Gamma)>0$, то $\Lambda(\chi u)=0$ в деякому околі множини $\Gamma$. Тому вектор $\Lambda(\chi u)$ задовольняє умови узгодження \eqref{16f8}, в яких
\begin{equation}\label{(fgh)-Murach}
(f,g_{1},\dots,g_{m},h_{0},\dots,h_{\varkappa-1})
\end{equation}
позначає вектор $\Lambda(\chi u)$, а $s$ замінено на $\sigma$. Отже, імплікація  \eqref{16f61} істинна, якщо $\sigma\notin E$, зважаючи на означення простору $\mathcal{Q}^{\sigma-2m,(\sigma-2m)/(2b);\varphi}$.

У випадку, коли $\sigma\in E$, цей простір означено за допомогою квадратичної інтерполяції. Розглянемо цей випадок. Виберемо функцію $\chi_1\in C^{\infty}(\overline\Omega)$ таку, що $\chi_1=0$ в околі множини $\Gamma$ і  $\chi_1=1$ в околі множини $\mathrm{supp}\,\chi$.
Оператор $M_{\chi_1}:F\mapsto\chi_1 F$ є обмеженим на парі просторів
\begin{equation}\label{16f63}
\begin{aligned}
M_{\chi_1}&:
\mathcal{H}^{\sigma\pm\varepsilon-2m,(\sigma\pm\varepsilon-2m)/(2b);\varphi}
\to\\
&\to\mathcal{Q}^
{\sigma\pm\varepsilon-2m,(\sigma\pm\varepsilon-2m)/(2b);\varphi},
\end{aligned}
\end{equation}
якщо $0<\varepsilon<1/2$, оскільки вектор $\chi_1 F$ задовольняє умови узгодження \eqref{16f8}, в яких \eqref{(fgh)-Murach} позначає вектор $\chi_1 F$, а замість $s$ узято числа $\sigma\pm\varepsilon\notin E$.
Застосувавши квадратичну інтерполяцію з числовим параметром $1/2$ до операторів \eqref{16f63}, одержимо обмежений оператор
\begin{equation*}
\begin{aligned}
M_{\chi_1}&:
\bigl[
\mathcal{H}^{\sigma-\varepsilon-2m,(\sigma-\varepsilon-2m)/(2b);\varphi},
\mathcal{H}^{\sigma+\varepsilon-2m,(\sigma+\varepsilon-2m)/(2b);\varphi}
\bigr]_{1/2}\to
\\
&\to\bigl[
\mathcal{Q}^{\sigma-\varepsilon-2m,(\sigma-\varepsilon-2m)/(2b);\varphi},
\mathcal{Q}^{\sigma+\varepsilon-2m,(\sigma+\varepsilon-2m)/(2b);\varphi}
\bigr]_{1/2}.
\end{aligned}
\end{equation*}
З огляду на інтерполяційні формули \eqref{16f67} і
\eqref{16f10} він діє на парі просторів
\begin{equation}\label{16f65}
M_{\chi_1}:\mathcal{H}^{\sigma-2m,(\sigma-2m)/(2b);\varphi}\to
\mathcal{Q}^{\sigma-2m,(\sigma-2m)/(2b);\varphi}.
\end{equation}
Зважаючи на вибір функції $\chi_1$, маємо рівність $\chi_1\Lambda(\chi u)=\Lambda(\chi u)$. Оскільки за припущенням
$$
\Lambda(\chi u)\in
\mathcal{H}^{\sigma-2m,(\sigma-2m)/(2b);\varphi},
$$
то згідно з формулою \eqref{16f65} виконується включення
$$
\Lambda(\chi u)=\chi_1\Lambda(\chi u)\in \mathcal{Q}^{\sigma-2m,(\sigma-2m)/(2b);\varphi}.
$$
Отже, імплікацію \eqref{16f61} доведено і у випадку, коли $\sigma\in E$.

Скориставшись формулами \eqref{16f60} і \eqref{16f61}, причому поклавши в останній $\sigma:=s-\lambda+1$, та наслідком~\ref{16cor4.2}, отримуємо такий ланцюжок імплікацій:
\begin{align*}
&u\in H^{s-\lambda,(s-\lambda)/(2b);\varphi}_{\mathrm{loc}}
(\Omega_0,\Omega')\Longrightarrow\\
&\Longrightarrow\Lambda(\chi u)\in
\mathcal{H}^{s-\lambda+1-2m,(s-\lambda+1-2m)/(2b);\varphi}\Longrightarrow\\
&\Longrightarrow \Lambda(\chi u)\in
\mathcal{Q}^{s-\lambda+1-2m,(s-\lambda+1-2m)/(2b);\varphi}\Longrightarrow\\
&\Longrightarrow
\chi u\in H^{s-\lambda+1,(s-\lambda+1)/(2b);\varphi}(\Omega).
\end{align*}
Тут, нагадаємо, $\chi$~--- довільна функція класу $C^\infty(\overline\Omega)$ така, що $\mbox{supp}\,\chi\subset\Omega_0\cup\Omega'$.
Зауважимо, що наслідок~\ref{16cor4.2} застосовний у цих міркуваннях, оскільки $\chi u\in H^{\sigma_0,\sigma_0/(2b)}(\Omega)$ згідно з
умовою теореми і нерівністю $s-\lambda+1>\sigma_{0}$. Таким чином, довели імплікацію \eqref{16f54}.

Використаємо її у доведенні потрібного нам включення
$u\in\nobreak H^{s,s/(2b);\varphi}_{\mathrm{loc}}(\Omega_0,\Omega')$. Нагадаємо, що $s>\sigma_{0}$. Розглянемо спочатку випадок, коли $s\notin\mathbb{Z}$. У цьому випадку існує число $\lambda_{0}\in\mathbb{N}$ таке, що
\begin{equation}\label{16f66}
s-\lambda_{0}<\sigma_{0}<s-\lambda_{0}+1.
\end{equation}
Скориставшись формулою \eqref{16f54} послідовно для значень
$\lambda:=\nobreak\lambda_{0}$, $\lambda:=\lambda_{0}-1$,..., $\lambda:=1$, маємо такий ланцюжок імплікацій:
\begin{align*}
&u\in H^{\sigma_0,\sigma_0/(2b)}(\Omega)\subset
H^{s-\lambda_{0},(s-\lambda_{0})/(2b);\varphi}_
{\mathrm{loc}}(\Omega_0,\Omega')\Longrightarrow\\
&\Longrightarrow u\in
H^{s-\lambda_{0}+1,(s-\lambda_{0}+1)/(2b);\varphi}_
{\mathrm{loc}}(\Omega_0,\Omega')
\Longrightarrow\ldots\\
&\ldots\Longrightarrow
u\in H^{s,s/(2b);\varphi}_{\mathrm{loc}}(\Omega_0,\Omega').
\end{align*}
Оскільки $u\in H^{\sigma_0,\sigma_0/(2b)}(\Omega)$ за умовою теореми, то потрібне включення доведено у розглянутому випадку.

Дослідимо тепер випадок, коли $s\in\mathbb{Z}$. У цьому випадку жодне число $\lambda_{0}\in\mathbb{N}$ не задовольняє умову~\eqref{16f66}.
Але оскільки $\nobreak{s-\varepsilon\notin\mathbb{Z}}$ і $s-\varepsilon>\sigma_{0}$, якщо $0<\varepsilon\ll1$, то за щойно доведеним виконується включення
$$
u\in H^{s-\varepsilon,(s-\varepsilon)/(2b);\varphi}_
{\mathrm{loc}}(\Omega_0,\Omega').
$$
Звідси, скориставшись імплікацією \eqref{16f54}, де $\lambda:=1$, робимо висновок, що
\begin{align*}
&u\in H^{s-\varepsilon,(s-\varepsilon)/(2b);\varphi}_
{\mathrm{loc}}(\Omega_0,\Omega')\subset
H^{s-1,(s-1)/(2b);\varphi}_{\mathrm{loc}}(\Omega_0,\Omega')
\Longrightarrow\\
&\Longrightarrow u\in H^{s,s/(2b);\varphi}_{\mathrm{loc}}(\Omega_0,\Omega').
\end{align*}
Отже, потрібне включення доведено і у цьому випадку.
\end{proof}

У термінах узагальнених просторів Соболєва формулюються тонкі достатні умови, за яких узагальнений розв'язок $u$ та його узагальнені похідні заданого порядку є неперервними на множині $\Omega_0\cup\Omega'$.

\begin{theorem}\label{16th4.4}
Нехай ціле число $p\geq0$ таке, що
\begin{equation}\label{p-condition}
p+b+n/2>\sigma_0.
\end{equation}
Припустимо, що функція $u\in H^{\sigma_0,\sigma_0/(2b)}(\Omega)$ є узагальненим розв'язком параболічної задачі \eqref{16f1}--\eqref{16f3}, праві частини якої задовольняють умови \eqref{16f13}--\eqref{16f15} для $s:=p+b+n/2$ і деякого функціонального параметра $\varphi\in\mathcal{M}$, підпорядкованого умові \eqref{9f4.7}. Тоді розв'язок $u(x,t)$ і кожна його узагальнена частинна похідна
$D_{x}^{\alpha}\partial_{t}^{\beta}u(x,t)$, де $|\alpha|+2b\beta\leq p$,  неперервні на множині $\Omega_0\cup\Omega'$.
\end{theorem}

\begin{proof}[\indent Доведення.] Виберемо достатньо мале число $\varepsilon>\nobreak0$. Нехай множини $\Omega_{\varepsilon}$ і $\Omega'_{\varepsilon}$ та функція $\chi_{\varepsilon}$ такі самі, як у доведенні теореми~\ref{9th4.4}, поданому у п.~\ref{sec3.1.3}. Згідно з теоремою~\ref{16th4.3} виконується включення $\chi_{\varepsilon}u\in H^{s,s/(2b);\varphi}(\Omega)$, де $s=p+b+n/2$, а $\varphi$ задовольняє умову \eqref{9f4.7}. Отже, існує функція $w_{\varepsilon}\in H^{s,s/(2b);\varphi}(\mathbb{R}^{n+1})$ така, що $w_{\varepsilon}=\chi_{\varepsilon}u=u$ на $\Omega_{\varepsilon}$. Звідси виводиться висновок теореми~\ref{16th4.4} за допомогою тих самих міркувань, що використані у доведенні теореми~\ref{9th4.4}. Нагадаємо, що вони спираються на теорему~\ref{9lem8.1}~(i).
\end{proof}

\begin{remark}\label{16rem4.5}\rm
Умова \eqref{9f4.7} є точною в теоремі~\ref{16th4.4}. А~саме: нехай $s:=p+b+n/2$ і $\varphi\in\mathcal{M}$ та припустимо, що для кожної функції $u\in\nobreak H^{\sigma_0,\sigma_0/(2b)}(\Omega)$ істинна така імплікація:
\begin{equation}\label{f-rem4.5}
\begin{aligned}
&\bigl(u\;\mbox{є розв'язком задачі \eqref{16f1}--\eqref{16f3}}\\
&\mbox{ для деяких правих частин \eqref{16f13}--\eqref{16f15}}\bigr)
\Longrightarrow\\
&\Longrightarrow\bigl(u\;\mbox{задовольняє висновок теореми \ref{16th4.4}}\bigr);
\end{aligned}
\end{equation}
тоді $\varphi$ задовольняє умову \eqref{9f4.7}.
\end{remark}

Справді, нехай $\varphi\in\mathcal{M}$, а ціле число $p\geq0$ таке, що
$$
s:=p+b+n/2>\sigma_{0}.
$$
Припустимо, що кожна функція $u\in\nobreak H^{\sigma_0,\sigma_0/(2b)}(\Omega)$ задовольняє імплікацію \eqref{f-rem4.5}, і покажемо, що тоді $\varphi$ задовольняє умову \eqref{9f4.7}. Нехай $V$~--- непорожня відкрита підмножина простору $\mathbb{R}^{n+1}$ така, що $\overline{V}\subset\Omega_0$. Виберемо довільно функцію $w\in\nobreak H^{s,s/(2b);\varphi}(\mathbb{R}^{n+1})$, підпорядковану умові $\mathrm{supp}\,w\subset V$. Покладемо $u:=w\!\upharpoonright\!\Omega\in H^{s,s/(2b);\varphi}(\Omega)$ і
\begin{equation*}
(f,g_{1},...,g_{m},h_0,\dots,h_{\varkappa-1}):=\Lambda u\in
\mathcal{H}^{s-2m,(s-2m)/(2b);\varphi}.
\end{equation*}
Функція $u$ задовольняє посилку імплікації \eqref{f-rem4.5}. Тому $u$ задовольняє висновок теореми~\ref{16th4.4} згідно зі зробленим припущенням. Зокрема, узагальнена похідна $\partial^{p}u/\partial x_{1}^{j}$ неперервна на $V$. Тому узагальнена похідна $\partial^{p}w/\partial x_{1}^{j}$ неперервна на $\mathbb{R}^{n+1}$. Отже, функціональний параметр $\varphi$ задовольняє умову \eqref{9f4.7} завдяки теоремі~\ref{9lem8.1}~(ii).
Зауваження~\ref{16rem4.5} обґрунтоване.

\markright{\emph \ref{4.1.5}. Умови класичності розв'язків}

\section[Умови класичності розв'язків]
{Умови класичності розв'язків}\label{4.1.5}

\markright{\emph \ref{4.1.5}. Умови класичності розв'язків}

Розв'язок $u$ початково-крайової задачі \eqref{16f1}--\eqref{16f3} вважають класичним, якщо він має таку гладкість, що ліві частини задачі обчислюються за допомогою неперервних класичних похідних функції $u$. У цьому випадку сліди функцій у крайових і початкових умовах є звуженнями неперервних функцій на відповідних множинах. Виявляється, неперервність правих частин параболічної задачі не гарантує класичності її розв'язку.

Розглянемо, наприклад, початково-крайову задачу Діріхле для параболічного рівняння другого порядку в циліндрі $\nobreak{\Omega=G\times(0,\tau)}$, тобто параболічну задачу вигляду
\begin{equation}\label{8f1}
\begin{gathered}
\partial_{t}u(x,t)+
\sum_{|\alpha|\leq2}a_{\alpha}(x,t)\,D^\alpha_x u(x,t)=f(x,t),\\
\mbox{якщо}\;\,x\in G\;\,\mbox{і}\;\,0<t<\tau,
\end{gathered}
\end{equation}
\begin{equation}\label{8f2}
u(x,t)=g(x,t),\;\;\mbox{якщо}\;\,x\in\Gamma\;\,\mbox{і}\;\,0<t<\tau,
\end{equation}
\begin{equation}\label{8f3}
u(x,0)=h(x),\;\;\mbox{якщо}\;\,x\in G;
\end{equation}
тут кожне $a_{\alpha}\in C^{\infty}(\overline{\Omega})$. Якщо її праві частини задовольняють умови $f\in C(\overline{\Omega})$,
$g\in C(\overline{S})$ і $h\in C(\overline{G})$, то це ще не забезпечує  класичність її розв'язку~$u$. Це випливає, зокрема, з результату Л.~Хермандера \cite[теорема 7.9.8 і зауваження до неї на с.~298]{Hormander86-1}. Згідно з ним навіть у випадку сталих коефіцієнтів $a_{\alpha}$ існує функція $f\in C(\Omega)$, причому $\mathrm{supp}\,f\subset\Omega$, така, що рівняння \eqref{8f1} має узагальнений розв'язок $u$, який задовольняє умови $u\in C^1(\Omega)\setminus C^{2}_{x}(\Omega)$ і $\mathrm{supp}\,u\subset\Omega$. (Як звичайно, $C^{2}_{x}(\Omega)$~--- простір усіх функцій, двічі неперервно диференційовних на $\Omega$ за просторовими змінними $x_1,\ldots,x_n$.) Отже, розглянута задача може мати некласичний розв'язок $u$ у випадку “гарних” правих частин $f\in C(\overline{\Omega})$, $g(\cdot)\equiv0$ і $h(\cdot)\equiv0$.

Для того, щоб існував класичний розв'язок параболічної задачі, її праві частини повинні задовольняти більш сильні умови, ніж неперервність. Зокрема, задача \eqref{8f1}--\eqref{8f3} має класичний розв'язок, якщо $f\in C^\alpha(\overline\Omega)$ для деякого числа $\alpha>0$, $g\in C(\overline{S})$, $h\in C(\overline{G})$ і $g\!\upharpoonright\!\Gamma=h\!\upharpoonright\!\Gamma$ (див., наприклад, \cite[с.42]{MathEncikl-5}). Тут, як звичайно, $C^\alpha(\overline\Omega)$~--- простір Гельдера порядку $\alpha$ на $\overline\Omega$.

За допомогою узагальнених просторів Соболєва можна отримати інші, причому досить тонкі, достатні умови класичності розв'язків параболічних задач. Для загальної параболічної задачі \eqref{16f1}--\eqref{16f3} природно використовувати означення класичного розв'язку, наведене нижче.

Нехай $m_0:=\max\{m_1,\dots,m_m\}$. Покладемо
\begin{align*}
S_{\varepsilon}:=\{x\in\Omega:\mbox{dist}(x,S)<\varepsilon\},\\
G_{\varepsilon}:=\{x\in\Omega:\mbox{dist}(x,G)<\varepsilon\},
\end{align*}
де число $\varepsilon>0$.

Узагальнений розв'язок $u\in H^{\sigma_0,\sigma_0/(2b)}(\Omega)$ задачі \eqref{16f1}--\eqref{16f3} називаємо класичним, якщо узагальнені частинні похідні функції $u=u(x,t)$ задовольняють такі три умови:
\begin{itemize}
\item [(a)] $D^\alpha_x\partial^\beta_t u$ неперервна на $\Omega$, якщо $0\leq|\alpha|+2b\beta\leq 2m$;
\item [(b)] $D^\alpha_x\partial^\beta_t u$ неперервна на $S_{\varepsilon}\cup S$ для деякого числа $\varepsilon>0$, якщо $0\leq|\alpha|+2b\beta\leq m_0$;
\item [(c)] $\partial^k_t u$ неперервна на $G_{\varepsilon}\cup G$ для деякого числа $\varepsilon>0$, якщо $0\leq k\leq\varkappa-1$.
\end{itemize}

У цьому означенні використано мінімальні умови, за яких ліві частини задачі обчислюються за допомогою неперервних класичних похідних функції $u$, причому в лівих частинах крайових і початкових умов сліди функцій розуміються як звуження неперервних функцій на множини $S$ і $G$ відповідно.

\begin{theorem}\label{24th4.1}
Покладемо $\sigma_1:=2m+b+n/2$, $\sigma_2:=m_0+b+n/2$ і $\sigma_3:=2m-b+n/2$.
Припустимо, що $\sigma_2>\sigma_0$ і $\sigma_3>\sigma_0$.
Припустимо також, що функція $u\in H^{\sigma_0,\sigma_0/(2b)}(\Omega)$ є узагальненим розв'язком параболічної задачі \eqref{16f1}--\eqref{16f3}, праві частини якої задовольняють такі умови:
\begin{equation}\label{24f12}
\begin{aligned}
f\in &H_{\mathrm{loc}}^{\sigma_1-2m,(\sigma_1-2m)/(2b);\varphi}
(\Omega,\varnothing)\cap\\
&\cap
H_{\mathrm{loc}}^{\sigma_2-2m,(\sigma_2-2m)/(2b);\varphi}
(S_{\varepsilon},S)\cap\\
&\cap
H_{\mathrm{loc}}^{\sigma_3-2m,(\sigma_3-2m)/(2b);\varphi}
(G_{\varepsilon},G),
\end{aligned}
\end{equation}
\begin{equation}\label{24f12bis}
\begin{gathered}
g_j\in H_{\mathrm{loc}}^{\sigma_2-m_j-1/2,(\sigma_2-m_j-1/2)/(2b);\varphi}
(S,\varnothing)\\
\mbox{для кожного}\;\;j\in\{1,\dots,m\},
\end{gathered}
\end{equation}
\begin{equation}\label{18f9}
\begin{gathered}
h_k\in H_{\mathrm{loc}}^{\sigma_3-2bk-b;\varphi}(G)\\
\mbox{для кожного}\;\;k\in\{0,\dots,\varkappa-1\},
\end{gathered}
\end{equation}
де $\varphi$~--- деякий функціональний параметр класу $\mathcal{M}$, який задовольняє умову \eqref{9f4.7}. Тоді розв'язок $u$ класичний.
\end{theorem}

\begin{proof}[\indent Доведення.]
Покажемо послідовно, що $u$ задовольняє умови (a), (b) і (c) наведеного вище означення класичного розв'язку.

Стосовно (a) скористаємося включенням
$$
f\in H_{\mathrm{loc}}^{\sigma_1-2m,(\sigma_1-2m)/(2b);\varphi}
(\Omega,\varnothing),
$$
яке є частиною умови~\eqref{24f12} теореми, і застосуємо теорему~\ref{16th4.4} у випадку, коли $\Omega_0=\Omega$, $\Omega'=S_0=S'=G_0=\varnothing$, $p=2m$ і $\varphi=\varphi$.
Тоді $\sigma_1:=2m+b+n/2=s$ і нерівність~\eqref{p-condition} виконується, оскільки $\sigma_3:=2m-b+n/2>\sigma_0$ за припущенням. Тому на підставі теореми~\ref{16th4.4} робимо висновок, що $u$ задовольняє умову~(a).

Щодо (b) скористаємося умовами
$$
f\in H_{\mathrm{loc}}^{\sigma_2-2m,(\sigma_2-2m)/(2b);\varphi}(S_{\varepsilon},S)
$$
і \eqref{24f12bis} теореми~\ref{24th4.1} та застосуємо теорему~\ref{16th4.4} у випадку, коли $\Omega_0=S_\varepsilon$, $\Omega'=S_0=S$, $S'=G_0=\varnothing$, $p=m_0$ і $\varphi=\varphi$.
Тоді $\nobreak{\sigma_2:=m_0+b+n/2=s}$ і нерівність~\eqref{p-condition} виконується, оскільки $\nobreak{\sigma_2>\sigma_0}$ за припущенням. Отже, на підставі теореми~\ref{16th4.4} робимо висновок, що $u$ задовольняє умову~(b).

Нарешті, щодо (c) скористаємося умовами
$$
f\in H_{\mathrm{loc}}^{\sigma_3-2m,(\sigma_3-2m)/(2b);\varphi}(G_{\varepsilon},G)
$$
і \eqref{18f9} теореми~\ref{24th4.1} та застосуємо теорему~\ref{16th4.4} у випадку, коли $\Omega_0=G_{\varepsilon}$, $\Omega'=G_0=G$, $S_0=S'=\varnothing$, $p=2m-2b$ і $\varphi=\varphi$. Тоді $\sigma_3:=2m-b+n/2=s$ і нерівність~\eqref{p-condition} виконується, оскільки $\sigma_3>\sigma_0$ за припущенням. Тому згідно з теоремою~\ref{16th4.4} узагальнена похідна $D^\alpha_x\partial^\beta_t u$ неперервна на $G_{\varepsilon}\cup G$, якщо $|\alpha|+2b\beta\leq 2m-2b$. У випадку $|\alpha|=0$ це означає, що $u$ задовольняє умову~(c).
\end{proof}

\newpage

\markright{\emph \ref{rem-ch2-3}. Бібліографічні коментарі до розд.~2 і 3}

\section[Бібліографічні коментарі до розд.~2 і 3]{Бібліографічні коментарі до розд.~2 і 3}\label{rem-ch2-3}

\markright{\emph \ref{rem-ch2-3}. Бібліографічні коментарі до розд.~2 і 3}

\small

Теорія розв'язності загальних параболічних початково-крайових задач була заснована в 60-х роках XX ст. у фундаментальних працях М.~С.~Аграновіча і М.~І.~Вішика \cite{AgranovichVishik64}, В.~О.~Солоннікова \cite{Solonnikov65, Solonnikov67}, С.~Д.~Ейдельмана \cite{Eidelman64}, А.~Фрідмана \cite{Fridman68}, Ж.-Л.~Ліонса і Е.~Мадженеса \cite{LionsMagenes72ii} (див. також монографії О.~О.~Ладиженської, В.~О.~Солоннікова і Н.~М.~Уральцевої \cite{LadyzhenskajaSolonnikovUraltzeva67}, С.~Д.~Івасишена \cite{Ivasyshen87, Ivasyshen90}, огляд С.~Д.~Ейдельмана \cite[\S~2]{Eidelman94} та наведену там літературу). Для параболічних рівнянь другого порядку ця теорія викладена, наприклад, у книгах Г.~Аманна \cite{Amann95}, М.~В.~Крилова \cite{Krylov96, Krylov08}, Г.~М.~Лібермана (G.~M.~Lieberman) \cite{Lieberman96}, А.~Лунарді \cite{Lunardi95}. У ній важливу роль відіграють анізотропні простори Соболєва і Гельдера, в термінах яких формулюються теореми про коректну розв'язність параболічних задач. У~монографії Ж.-Л.~Ліонса і Е.~Мадженеса \cite{LionsMagenes72ii} ці теореми доведені також для анізотропних просторів Соболєва від'ємного порядку. У~праці С.~Д.~Івасишена \cite{Ivasyshen77}~--- для негативних просторів Гельдера. У~статтях М.~В.~Житарашу \cite{Zhitarashu85, Zhitarashu87, Zhitarashu90} такі теореми встановлені для модифікованих анізотропних просторів Соболєва довільного дійсного порядку (ці результати викладені у монографії С.~Д.~Ейдельмана і М.~В.~Житарашу \cite{EidelmanZhitarashu98} та огляді С.~Д.~Ейдельмана \cite[\S~2]{Eidelman94}). Для параболічних псевдодиференціальних задач теореми про коректну розв'язність доведені у статтях Г.~Груб і В.~О.~Солоннікова \cite{GrubbSolonnikov90, Grubb95}.

В останні роки у теорії параболічних задач активно застосовуються інші класи функціональних просторів: вагові простори, простори з мішаною нормою, простори Бєсова і Лізоркіна--Трібеля. З цього приводу вкажемо праці R.~Denk, M.~Hieber і J.~Pr\"uss \cite{DenkHieberPruss03, DenkHieberPruss07}, R.~Denk і M.~Kaip \cite{DenkKaip13}, H.~Dong і D.~Kim \cite{DongKim15}, M.~Hieber і J.~Pr\"uss \cite{HieberPruss97}, F.~Hummel \cite{Hummel21}, F.~Hummel і N.~Lindemulder \cite{HummelLindemulder19}, K.-H.~Kim \cite{Kim08}, M.~Kohne, J.~Pr\"uss і M.~Wilke \cite{KohnePrussWilke10}, М.~В.~Крилова \cite{Krylov99, Krylov01}, P.~C.~Kunstmann і L.~Weis  \cite{KunstmannWeis04}, J.~LeCrone, J.~Pr\"uss і M.~Wilke \cite{LeCronePrussWilke14}, N.~Lindemulder \cite{Lindemulder18arxiv, Lindemulder20}, N.~Lindemulder і M.~Veraar \cite{LindemulderVeraar20}, П.~Вайдемайера \cite{Weidemaier05} та наведені там посилання.

До вказаного тренду належать і дослідження параболічних задач у просторах узагальненої гладкості, які проводили в останнє десятиліття автори цієї монографії \cite{LosMurach13Collection2, LosMurach14Dop6, LosMurach14Collection2, Los14Dop10, Los15Collection2, Los16Collection1, Los16Collection2, Los17Dop8, Los18Dop6, LosMurach13MFAT2, Los15UMJ5, Los16UMJ6, Los16UMJ9, Los16UMJ11, LosMikhailetsMurach17CPAA1, LosMurach17OM15, Los17UMJ3, Los17MFAT2, Los20MFAT2, LosMikhailetsMurach21arXiv}. Істотну частину їх результатів викладено  у другому і третьому її розділах.

Теорема~\ref{th3.4} (про ізоморфізми, породжені напіводнорідною параболічною задачею в прямокутнику) доведена в \cite[пп.~2 і~5]{LosMurach13MFAT2}, а теореми~\ref{th3.5} і~\ref{th3.6} (про регулярність її розв'язків)~--- у \cite[пп.~4 і~5]{LosMurach13Collection2}. Теореми~\ref{9th4.1}, \ref{9th4.3} і~\ref{9th4.4}  (про аналогічні властивості напіводнорідної параболічної задачі у багатовимірному циліндрі) доведені в \cite[пп.~4 і~8]{LosMikhailetsMurach17CPAA1}. Останні дві є новими і у соболєвському випадку.

Теореми~\ref{6th1} і~\ref{6th2} (про ізоморфізми породжені неоднорідними крайовими задачами Діріхле і Неймана для рівняння теплопровідності)
доведені в \cite[пп.~3 і~5]{Los15UMJ5}. Теорема~\ref{th4.1a} (про ізоморфізми породжені неоднорідною параболічною задачею в прямокутнику) доведена в \cite{Los17MFAT2}. Її аналог для багатовимірного циліндра~--- теорема~\ref{16th4.1}~--- встановлена в \cite[пп.~4 і~6]{LosMikhailetsMurach21arXiv}. Там само доведені теореми \ref{16th4.3} і \ref{16th4.4} (про регулярність розв'язків неоднорідної параболічної задачі), причому перша з них анонсована в \cite{Los17Dop8}. Вона є новою і у соболєвському випадку. Теорема \ref{24th4.1} (про умови класичності узагальненого розв'язку параболічної задачі) встановлена
в~\cite{Los16UMJ11}. Подібні умови отримані у працях А.~М.~Ільїна \cite{Ilin60}, А.~М.~Ільїна, А.~С.~Калашникова і О.~А.~Олейнік \cite{IlinKalashOleinik62}, М.~І.~Матійчука і C.~Д.~Ейдельмана \cite{MatiichukEidelman74}, В.~О.~Солоннікова \cite{Solonnikov65},  А.~Фрідмана \cite[розд.~X, \S~7]{Fridman68} у термінах приналежності правих частин задачі деяким просторам Гельдера або Соболєва. Вони характеризують гладкість функцій більш грубо, ніж простори, використані в теоремі~\ref{24th4.1}.

Параболічні задачі для диференціальних рівнянь другого порядку досліджені окремо в \cite{Los15Collection2, Los16UMJ6, Los16Collection1, Los16Collection2, Los16UMJ9, LosMurach17OM15}, а для диференціальних систем~--- в \cite{Los14Dop10, Los17UMJ3, Los20MFAT2} і, крім того, у  праці О.~Дяченко і В.~Лося \cite{DyachenkoLos21arXiv}.

Зазначимо також, що задача Коші для параболічних інтегро-диференціальних рівнянь вивчалася у деяких просторах узагальненої гладкості у новітніх працях R.~Mikulevi\v{c}ius і C.~Phonsom \cite{MikuleviciusPhonsom18arxiv, MikuleviciusPhonsom19PotAnal}. Регулярність у цих просторах задається за допомогою міри Леві (див. статтю W.~Farkas, N.~Jacob і R.~L.~Schilling \cite{FarkasJacobScilling01b}).

\normalsize

\newpage

\addcontentsline{toc}{chapter}{\textbf{Список літератури}}

\chaptermark{Список літератури}

\pagestyle{myheadings}

\markright{Список літератури}

\renewcommand\bibname{\Large\textbf{Список літератури}}

\chaptermark{\emph Список літератури}

\newpage

\normalsize
\renewcommand{\contentsname}{\centerline{\large\textbf{ЗМІСТ}}\vspace{-5mm}}
\thispagestyle{empty}
\tableofcontents

\newpage

\thispagestyle{empty}

\vspace*{2.5cm}

\begin{center}

\large

{\bf CONTENTS}

\vspace{0.5cm}

\normalsize

\end{center}

{\bf Introduction} \dotfill 3

\medskip

\noindent {\bf 1\;\;\;Generalized Sobolev spaces and their interpolation
\hfill 9}

1.1\;\;\; Interpolation with function parameter\dotfill 10

1.2\;\;\; Generalized Sobolev spaces\dotfill 13

1.3\;\;\; Specific function spaces\dotfill 24

1.4\;\;\; Interpolation of Sobolev spaces\dotfill 27

1.5\;\;\; The Cauchy data operator\dotfill 43

1.6\;\;\; An embedding theorem\dotfill 53

1.7\;\;\; Bibliographical comments to Chapter 1\dotfill 57

\medskip

\noindent {\bf 2\;\;\;Semihomogeneous parabolic problems
\hfill 65}

2.1\;\;\; The problem in a rectangle\dotfill 65

2.2\;\;\; The problem in a cylinder\dotfill 72

2.3\;\;\; Regularity of solutions\dotfill 79

\medskip

\noindent {\bf 3\;\;\;Nonhomogeneous parabolic problems
\hfill 87}

3.1\;\;\; Problems for the heat equation\dotfill 88

3.2\;\;\; The inhomogeneous problem in a rectangle\dotfill 101

3.3\;\;\; The inhomogeneous problem in a cylinder\dotfill 113

3.4\;\;\; Regularity of generalized solutions\dotfill 129

3.5\;\;\; Conditions for the solutions to be classical\dotfill 137

3.6\;\;\; Bibliographical comments to Chapters 2 and 3\dotfill 141

\medskip

\noindent {\bf References\hfill 143}

\normalsize


\footnotesize
\begin{thebibliography}{999}

\bibitem{AgranovichVishik64}
Агранович М. С., Вишик М. И. Эллиптические задачи с параметром и параболические задачи общего
вида. \emph{Успехи математических наук}. 1964. Т.~19, №~3. С.53--161.

\bibitem{Anop19Dop2}	
Аноп А. В. Еліптичні за Лавруком крайові задачі для однорідних диференціальних рівнянь. \emph{Доповіді НАН України.} 2019. №~2. С.~3--11.

\bibitem{AnopMurach18Dop3}	
Аноп А. В., Мурач О. О. Однорідні еліптичні рівняння в розширеній соболєвській шкалі. \emph{Доповіді НАН України.} 2018. №~3. С.~3--11.

\bibitem{BerghLefstrem76} 
Берг Й., Лёфстрём Й. Интерполяционные пространства. Введение. Москва: Мир, 1980. 264~с.

\bibitem{BerezanskyUsSheftel90}
Березанский Ю. М., Ус Г.~Ф., Шефтель З. Г. Функциональный анализ. Киев: Вища школа, 1990. 600~с.

\bibitem{BesovIlinNikolskii75}
Бесов О. В., Ильин В.~П., Никольский  С.~М. Интегральные представления функций и теоремы вложения. Москва: Наука, 1975. 480~с.

\bibitem{VolevichPaneah65}
Волевич Л. Р., Панеях Б. П.Некоторые пространства обобщенных функций и теоремы вложения.
\emph{Успехи математических наук}. 1965. Т.~20, №~1.С.~3--74.

\bibitem{Goldman76}
Гольдман М. Л. Описание пространства следов для функций обобщенного гёльдерова класса. \emph{Доклады АН СССР.} 1976. Т.~231, №~3. С.~525--528.

\bibitem{Goldman77}
Гольдман М. Л. Описание следов анизотропного обобщенного класса Лиувилля. \emph{Доклады АН СССР.} 1977. Т.~233, №~3. С.~273--276.

\bibitem{Doktorskii91}
Докторский Р. Я. Реитерационные соотношения метода вещественной интерполяции. \emph{Доклады АН СССР.} 1991. Т.~321, №~2. C.~241--245.

\bibitem{Zhitarashu85}
Житарашу Н. В. Теоремы о полном наборе изоморфизмов в $L_2$--теории
обобщенных решений граничных задач для одного параболического по И. Г. Петровскому
уравнения. \emph{Математический сборник}. 1985. Т.~128(170), №~4(12). С.~451--473.

\bibitem{Zhitarashu87}
Житарашу Н. В. О корректной разрешимости общих модельных параболических граничных
задач в пространствах $\mathcal{H}^{s}$, $-\infty<s<\infty$. \emph{Известия АН СССР}. 1987. T. 51, №~5. С. 962--993.

\bibitem{Zhitarashu90}
Житарашу Н. В. $L_2$--теория обобщенных решений общих линейных параболических
граничных задач. \emph{Математические исследования}. 1990. T. 112. С.~104--115.

\bibitem{Zagorskii61}
Загорский Т. Я. Смешанные задачи для систем дифференциальных уравнений
с частными производными параболического типа. Львов: Изд-во Львовского ун-та, 1961.

\bibitem{Ivasyshen77}
Ивасишен С. Д. О корректной разрешимости общих параболических граничных задач в негативных
пространствах Гельдера. \emph{Доклады АН УССР. Cерия А.} 1977. №~5. С.~396--400.

\bibitem{Ivasyshen87}
Ивасишен С. Д. Линейные параболические граничные задачи. Киев: Вища школа, 1987. 72~с.

\bibitem{Ivasyshen90}
Ивасишен С. Д. Матрицы Грина параболических граничных задач. Киев: Вища школа, 1990. 200~с.

\bibitem{IlinKalashOleinik62}
Ильин А. М., Калашников А. С., Олейник  О. А. Линейные
уравнения второго порядка параболического типа. \emph{Успехи математических наук}. 1962. T. 17, №~3. С.~3--146.

\bibitem{Ilin60}
Ильин В. А. О разрешимости смешанных задач для гиперболического
и параболического уравнений. \emph{Успехи математических наук}. 1960. T.~15, №~2. С.~97--154.

\bibitem{Kalugina75}
Калугина Т. Ф. Интерполяция банаховых пространств с функциональным параметром. Теорема реитерации. \emph{Вестник Московского ун-та. Серия 1, математика, механика.} 1975. Т.~30, №~6. С.~68--77.

\bibitem{Kalyabin77}
Калябин Г. А. Теоремы вложения для обобщенных пространств Бесова и Лиувилля. \emph{Доклады АН СССР.} 1977. Т.~232, №~6. С.~1245--1248.

\bibitem{Kalyabin81}
Калябин Г. А. Критерии мультипликативности и вложение в $C$ пространств типа Бесова--Лизоркина--Трибеля. \emph{Математические заметки.} 1981. Т.~30, №~4. С.~517--526.

\bibitem{KasirenkoMurachChepurukhina19Dop3}
Касіренко Т. М., Мурач О. О., Чепурухіна І. С. Простори Хермандера на многовидах та їх застосування до еліптичних крайових задач. \emph{Доповіді НАН України.} 2019. №~3. С.~9--16.

\bibitem{Krein60a}
Крейн С.~Г. Об одной интерполяционной теореме в теории операторов.
\emph{Доклады АН СССР}. 1960. Т.~130, №~3. С.~491--494.

\bibitem{Krein60b}
Крейн С.~Г. О понятии нормальной шкалы пространств.
\emph{Доклады АН СССР}. 1960. Т.~132. С.~510--513.

\bibitem{KreinPetunin66}
Крейн С. Г., Петунин Ю. И. Шкалы банаховых пространств. \emph{Успехи математических наук}. 1966. Т.~21, вып.~2 (128). С.~89--168.

\bibitem{KreinPetuninSemenov82} 
Крейн С. Г., Петунин Ю.~И., Семенов Е.~М. Интерполяция линейных операторов. Москва: Наука, 1978. 400~с.

\bibitem{Ladyzhenskaja54}
Ладыженская О. А. О разрешимости основных краевых задач для уравнений
параболического и гиперболического типов. \emph{Доклады АН СССР.} 1954. Т.~97, №~3. С.~395--398.

\bibitem{Ladyzhenskaja58}
Ладыженская О. А. О нестационарных операторных уравнениях и их приложениях к линейным задачам математической физики. \emph{Математический сборник.} 1958. Т.~45, вып.~2. С.~123--158.

\bibitem{LadyzhenskajaSolonnikovUraltzeva67}
Ладыженская О. А., Солонников В. А., Уральцева Н. Н. Линейные и квазилинейные уравнения параболического типа. Москва: Наука, 1967. 736~с.

\bibitem{Lizorkin86}
Лизоркин П. И. Пространства обобщенной гладкости. Добавление в кн. Х.~Трибель. Теория функциональных пространств. Москва: Мир, 1986. С.~381--415.

\bibitem{LionsMagenes72i}
Лионс Ж.-Л., Мадженес. Э. Неоднородные граничные задачи и их приложения. Москва: Мир, 1971. 372~с.
(Переклад видання: Lions J.-L., Magenes E. Probl\`{e}mes aux limites non homog\`{e}nes et applications. Vol.~1. Paris: Dunod, 1968. 372~p.)

\bibitem{Los14Dop10}
Лось В. М. Параболічні мішані задачі для систем Петровського у просторах узагальненої гладкості.
\emph{Доповіді НАН України}. 2014. №~10. С.~24--32.

\bibitem{Los15Collection2}
Лось В.~М. Класичні розв'язки параболічної мішаної задачі і 2b-анізотропні простори Хермандера.
\emph{Збірник праць Інституту математики НАН України}. 2015. T. 12, №~2. С. 276--290.

\bibitem{Los16Collection1}
Лось В.~М. Про достатні умови класичності узагальнених розв'язків деяких мішаних параболічних задач.
\emph{Збірник праць Інституту математики НАН України}. 2016. T. 13, №~1. С. 228--243.

\bibitem{Los16Collection2}
Лось В.~М. Умови класичності розв'язків другої крайової задачі для параболічних рівнянь.
\emph{Збірник праць Інституту математики НАН України}. 2016. T.~13, №~2. С. 175--192.

\bibitem{Los18Dop6}
Лось В. М. 2b-анізотропні простори Хермандера у циліндричних областях.
\emph{Доповіді НАН України}. 2018. №~6. С. 3--8.

\bibitem{Los17Dop8}
Лось В. М., Михайлець В. А., Мурач О. О. Регулярність розв'язків загальних параболічних задач у просторах Хермандера.
\emph{Доповіді НАН України}. 2017. №~8. С. 3--10.

\bibitem{LosMurach13Collection2}
Лось В.~М., Мурач О. О. Про гладкість розв'язків параболічних мішаних задач.
\emph{Збірник праць Інституту математики НАН України}. 2013. T. 10, №~2. С.~219--234.

\bibitem{LosMurach14Dop6}
Лось В. Н., Мурач А. А. Параболические смешанные задачи в пространствах обобщенной гладкости.
\emph{Доповіді НАН України}. 2014. №~6. С.~23--31.

\bibitem{LosMurach14Collection2}
Лось В.~М., Мурач О. О. Неоднорідні параболічні мішані задачі і простори узагальненої гладкості.
\emph{Збірник праць Інституту математики НАН України}. 2014. T. 11, №~2. С.~249--267.

\bibitem{MatiichukEidelman74}
Матийчук М. И., Эйдельман С. Д. О корректности задач Дирихле и Неймана для
параболических уравнений второго порядка с коэффициентами из классов Дини.
\emph{Украинский математический журнал}. 1974. T. 26, №~3. С. 328--337.

\bibitem{MikhailetsMurach10}
Михайлец В. А., Мурач А. А. Пространства Хермандера, интерполяция и эллиптические задачи. Kиев: Институт математики НАН Украины, 2010. 372~с.

\bibitem{MikhailetsMurach11VisnChernivUniv}
Михайлець В. А., Мурач О. О. Простори Хермандера та еліптичні задачі. \emph{Науковий вісник Чернівецького національного університету. Математика.} 2011. Т.~1, № 1--2. С.~129--144.

\bibitem{MathEncikl-5}
Михайлов В. П. Смешанная и краевая задачи для параболических
уравнений и систем. Математическая энциклопедия.~Т.~5.
Москва: Советская энциклопедия, 1985. 1248~с.

\bibitem{MurachChepurukhina20Dop8}
Мурач О. О., Чепурухіна І. С. Еліптичні задачі з некласичними крайовими умовами у роз-ширеній соболєвській шкалі. \emph{Доповіді НАН України.} 2020. №~8. С.~3--10.

\bibitem{Nikolskii77}
Никольский С. М. Приближение функций многих переменных и теоремы вложения. Москва: Наука, 1977. 456~с.

\bibitem{Obchinnikov14MatSb}
Овчинников В. И. Обобщенная интерполяционная конструкция Лионса--Петре и оптимальные теоремы вложения для пространств Соболева. \emph{Математический сборник.} 2014. Т.~205, №~1. С.~87--104.

\bibitem{Obchinnikov14UspMatNauk}
Овчинников В. И. Интерполяционные функции и интерполяционная конструкция Лионса–Петре. \emph{Успехи математических наук.} 2014. Т.~69, вып.~4 (418). С.~103--168.

\bibitem{Petrovskii38}
Петровский И. Г. О проблеме Cauchy для систем линейных уравнений
с частными производными в области неаналитических функций.
\emph{Бюллетень МГУ}. 1938. Cекция А, 1, выпуск 7. С. 1--72.

\bibitem{Pustylnik82}
Пустыльник Е. И. О перестановочно-интерполяционных гильбертовых пространствах. \emph{Известия вузов. Математика.} 1982. №~5 (240). С.~43--46.

\bibitem{Seneta85}
Сенета Е. Правильно меняющиеся функции. Москва: Наука, 1985. 142~с.

\bibitem{Slobodetskii58}
Слободецкий Л. Н. Обобщенные пространства С.~Л.~Соболева и их приложения к краевым задачам для дифференциальных уравнений в частных производных.
\emph{Ученые записки Ленинградского государственного педагогического института}. 1958. Т.~187. С. 54--112.

\bibitem{Slobodetskii58DocAN}
Слободецкий Л. Н. Оценки решений эллиптических и параболических
систем. \emph{Доклады АН СССР.} 1958. Т.~120, №~3. С.~468--471.

\bibitem{Slobodetsky60}
Слободецкий Л. Н. Оценки в $L_{2}$ решений линейных эллиптических и параболических систем. І. Оценки решений эллиптической системы.
\emph{Вестник ЛГУ, серия математика, механика и астрономия}. 1960. №~7. С.~28--47.

\bibitem{Sobolev38}
Соболев С. Л. Об одной теореме функционального анализа. \emph{Математический сборник.} 1938. Т.~4. С.~471--497.

\bibitem{Sobolev50}
Соболев С. Л. Некоторые применения функционального анализа в математической физике. Ленинград: Издательство Ленинградского университета, 1950. 255~с.

\bibitem{Solonnikov62}
Солонников В. А. Об априорных сценках для некоторых краевых задач.
\emph{Доклады АН СССР.} 1962. Т.~138, №~4. С.~781--784.

\bibitem{Solonnikov64}
Солонников В. А. Априорные оценки для уравнений второго порядка параболического типа.
\emph{Труды Математического института АН \nobreak{СССР}}. 1964. Т.~70. С.~133--212.

\bibitem{Solonnikov65}
Солонников В. А. О краевых задачах для линейных параболических систем дифференциальных уравнений общего вида.
\emph{Труды Математического института АН СССР}. 1965. Т.~83. С. 3--163.

\bibitem{Solonnikov67}
Солонников В. А. Об оценках в $L_{p}$ решений эллиптических и параболических систем.
\emph{Труды Математического института АН СССР}. 1967. Т.~102. С. 137--160.

\bibitem{Stepanets87}
Степанец А. И. Классификация и приближение периодических функций. Киев: Наукова думка, 1987. 268~с.

\bibitem{Stepanets02}
Степанец А. И. Методы теории приближений. В 2-х томах. Киев: Институт математики НАН Украины, 2002. Т.~1. 468~с., Т.~2. 427~с.

\bibitem{Triebel80}
Трибель Х. Теория интерполяции, функциональные пространства, дифференциальные операторы. Москва: Мир, 1980. 664~с.

\bibitem{Triebel86}
Трибель Х. Теория функциональных пространств. Москва: Мир, 1986. 447~с.

\bibitem{Uspenskii72}
Успенский С.~В. О следах функций класса $W^{l_1,\ldots,l_n}_{p}$ Соболева на гладких поверхностях. \emph{Сибирский математический журнал.} 1972. Т.~13, №~2. С.~429--451.

\bibitem{UspenskiiDemidenkoPerepelkin84}
Успенский С.~В., Демиденко Г. В., Перепелкин В. Г. Теоремы вложения и приложения к дифференциальным уравнениям. Новосибирск: Наука, 1984. 224~с.

\bibitem{Fridman68}
Фридман А. Уравнения с частными производными параболического
типа. Москва: Мир, 1968. 428~с. (Переклад видання: Friedman A. Partial differential equations of parabolic type. Englewood Cliffs, N.J:
Prentice-Hall Inc, 1964.)

\bibitem{FunctionalAnalysis72}
Функциональный анализ: под общ. ред. С.~Г.~Крейна. Москва: Наука, 1972. 544~с.

\bibitem{Hormander63}
Хермандер Л. Линейные дифференциальные операторы с частными производными.
Москва: Мир, 1965. 380~с.
(Переклад видання: H\"ormander~L. Linear partial differential operators. Berlin: Springer, 1963. vii+287~p.)

\bibitem{Hormander86-1}
Хермандер Л. Анализ линейных дифференциальных операторов с частными производными: В 4-х т. Т. 1. Теория распределений и анализ Фурье.
Москва: Мир, 1986. 464~с.

\bibitem{Hormander86}
Хермандер Л. Анализ линейных дифференциальных операторов с частными производными:
В~4-х~т.~Т.~2. Дифференциальные операторы с постоянными коэффициентами. Москва: Мир, 1986. 456~с.

\bibitem{Hormander85}
Хермандер Л. Анализ линейных дифференциальных операторов с частными
производными: В 4-х т. Т. 3. Псевдодифференциальные операторы. Москва: Мир, 1987. 696~с.

\bibitem{Shlenzak74}
Шлензак Г. Эллиптические задачи в уточненной шкале пространств.
\emph{Вестник Московского ун-та. Серия 1, математика, механика.} 1974. Т.~29, №~4. С.~48--58.

\bibitem{Eidelman64}
Эйдельман С. Д. Параболические системы. Москва: Наука, 1964. (Переклад англійською: Eidel'man S. D. Parabolic systems. North-Holland Publishing Co., Amsterdam-London; Wolters-Noordhoff Publishing, Groningen, 1969.)


\bibitem{Almeida05}
Almeida A. Wavelet bases in generalized Besov spaces. \emph{Journal of Mathematical Analysis and Applications.} 2005. Vol.~304, N.~1. P.~198--211.

\bibitem{Almeida09}
Almeida A. On interpolation properties of generalized Besov spaces. Further Progress in Analysis. Editors: H.~G.~W.~Begehr, A.~O.~\c{C}elebi, R.~P.~Gilbert. London: World Scientific, 2009. P.~601--610.

\bibitem{AlmeidaCaetano11}
Almeida A., Caetano A. Real interpolation of generalized Besov-Hardy spaces and applications. \emph{The Journal of Fourier Analysis and Applications.} 2011. Vol.~17, N.~4. P.~691--719.

\bibitem{Amann95}
Amann H. Linear and quasilinear parabolic problems. Volume I. Abstract linear theory. Basel/Boston/Berlin: Birkhauser Verlag, 1995. xxv+335~p.

\bibitem{Amann19}
Amann H. Linear and quasilinear parabolic problems. Volume II. Function spaces. Cham: Birkhauser, 2019. xiv+464~p.

\bibitem{Ameur04}
Ameur Y. A new proof of Donoghue’s interpolation theorem. \emph{Journal of Function Spaces and Applications.} 2004. Vol.~2. P.~253--265.

\bibitem{Ameur19}
Ameur Y. Interpolation between Hilbert spaces. \emph{Analysis of Operators on Function Spaces, The Serguei Shimorin Memorial Volume, Trends in Mathematics. Editors: A. Aleman etc.} Cham: Birkh\"{a}user/Springer, 2019. P.~63--115.

\bibitem{AnopDenkMurach21CPAA2}
Anop A., Denk R., Murach A.  Elliptic problems with rough boundary data in generalized Sobolev spaces. \emph{Communications on Pure and Applied Analysis.} 2021. Vol.~20, N.~2. P.~697--735.

\bibitem{AnopKasirenko16MFAT}
Anop A. V., Kasirenko T. M. Elliptic boundary-value problems in H\"ormander spaces. \emph{Methods of Functional Analysis and Topology}. 2016. Vol.~22, N.~4. P.~295--310.

\bibitem{AnopKasirenkoMurach18UMJ3}
Anop A. V., Kasirenko T. M., Murach O.~O. Irregular elliptic boundary-value problems and H\"ormander spaces. \emph{Ukrainian Mathematical Journal}. 2018. Vol.~70, N.~3. P.~341--361.

\bibitem{AnopMurach14MFAT2}
Anop A. V., Murach A. A. Parameter-elliptic problems and interpolation with a function parameter. \emph{Methods of Functional Analysis and Topology}. 2014. Vol.~20, N.~2. P. 103--116.

\bibitem{AnopMurach14UMJ7}
Anop A. V., Murach A. A.  Regular elliptic boundary-value problems in the extended Sobolev scale. \emph{Ukrainian Mathematical Journal}. 2014. Vol.~66, N.~7. P. 969--985.

\bibitem{Bagdasaryan97}
Bagdasaryan A. G. On the interpolation of some generalized function spaces of different anisotropies. \emph{Mathematical Notes.} 1997. Vol.~62, N.~5--6. P.~557--561.

\bibitem{Bagdasaryan10}
Bagdasaryan A. G. On interpolation of pairs of generalized spaces of Besov type. \emph{Eurasian Mathematical Journal.} 2010. Vol.~1, N.~4. P.~32--77.

\bibitem{BennetSharpley88}
Bennet K, Sharpley R. Interpolation of operators. Boston: Academic Press,
1988. xiv+469~p.

\bibitem{BesoyCobos18}
Besoy B. F., Cobos F. Duality for logarithmic interpolation spaces when $0<q<1$ and applications. \emph{Journal of Mathematical Analysis and Applications.} 2018. Vol.~466, N.~1. P.~373--399.

\bibitem{BinghamGoldieTeugels89}
Bingham N. H., Goldie C.~M., Teugels J.~L. Regular variation.
Cambridge: Cambridge University Press, 1989. 512~p.

\bibitem{BrudnyiKrugljak91}
Brudny\"{\i} Yu. A., Krugljak N.~Ya. Interpolation functors and interpolation spaces. Amsterdam: North-Holland, 1991. xvi+718~p.

\bibitem{BuldyginIndlekoferKlesovSteinebach18}
Buldygin V. V., Indlekofer K.-H., Klesov O.~I., Steinebach J.~G.
Pseudo-regularly varying functions and generalized renewal Ppocesses, Cham: Springer, 2018. xxii+482~p.

\bibitem{ButzerBehrens67}
Butzer P. L., Behrens H. Semi-groups of operators and approximation. Berlin/Heidelberg/New York: Springer, 1967.

\bibitem{CaetanoLeopold06}
Caetano A.~M., Leopold H.-G. Local growth envelopes of Triebel-Lizorkin spaces of generalized smoothness. \emph{The Journal of Fourier Analysis and Applications.} 2006. Vol.~12, N.~4. P.~427--445.

\bibitem{CaetanoLeopold13}
Caetano A.~M., Leopold H.-G. On generalized Besov and Triebel-Lizorkin spaces of regular distributions. \emph{Journal of Functional Analysis.} 2013. Vol.~264, N.~12. P.~2676--2703.

\bibitem{Calderon64}
Calderon A. P. Intermediate spaces and interpolation, the compex method. \emph{Studia Mathematica.} 1964. Vol.~24. P.~113--190.

\bibitem{CarroCerda90}
Carro M. J., Cerd\`a J. On complex interpolation with an analytic functional. \emph{Mathematica Scandinavica.} 1990. Vol.~66, N.~2. P.~264--274.

\bibitem{ChepurukhinaMurach15MFAT1}
Chepurukhina I. S., Murach A. A. Elliptic problems in the sense of B. Lawruk on two-sided re-fined scales of spaces. \emph{Methods of Functional Analysis and Topology.} 2015. Vol.~21, N.~1. P.~6--21.

\bibitem{ChepurukhinaMurach20MFAT2}
Chepurukhina I., Murach A.  Elliptic problems with unknowns on the boundary and irregular boundary data. \emph{Methods of Functional Analysis and Topology.} 2020. Vol.~26, N.~2. P.~91--102.

\bibitem{CobosKuhn09}
Cobos F., K\"uhn T. Approximation and entropy numbers in Besov spaces of generalized smoothness. \emph{Journal of Approximation Theory.} 2009. Vol.~160, N.~1--2. P.~56--70.

\bibitem{CobosDominguesTriebel16}
Cobos F., Dom\'{\i}ngues \'O., Triebel H. Characterizations of logarithmic Besov spaces in terms of differences, Fourier-analytical decompositions, wavelets and semi-groups. \emph{Journal of Functional Analysis.} 2016. Vol.~270, N.~12. P.~4386--4425.

\bibitem{CobosFernandez88}
Cobos F., Fernandez D.~L. Hardy-Sobolev spaces and Besov spaces with a function para\-meter.
\emph{Function Spaces and Applications. Lecture Notes in Mathematics}. Berlin: Springer, 1988. Vol.~1302. P. 158--170.

\bibitem{DenkHieberPruss03}
Denk R., Hieber M., Pr\"uss J. $\mathcal{R}$-boundedness, Fourier multipliers and problems of elliptic and parabolic type. \emph{Memoirs of the American Mathematical Society}. Vol.~166, N.~788. Providence, RI: American Mathematical Society, 2003. viii+114 pp.

\bibitem{DenkHieberPruss07}
Denk R., Hieber M., Pr\"uss J.  Optimal $L_p-L_q$-estimates for parabolic boundary value problems with inhomogeneous data.
\emph{Mathematische Zeitschrift}. 2007. Vol.~257, N.~1. P. 193--224.

\bibitem{DenkKaip13}
Denk R., Kaip M. General parabolic mixed order systems in $L_p$ and applications. \emph{Operator Theory: Advances and Applications}. Vol.~239. Cham: Birkh\"{a}user/Springer, 2013. viii+250 pp.

\bibitem{DongKim15}
Dong H., Kim D. Elliptic and parabolic equations with measurable coefficients in weighted Sobolev spaces. \emph{Advances in Mathematics}. 2015. Vol.~274. P.~681--735.

\bibitem{Donoghue67}
Donoghue W. F. The interpolation of quadratic norms. \emph{Acta Mathematica}. 1967. Vol.~118, N. 3--4. P. 251--270.

\bibitem{DyachenkoLos21arXiv}
Dyachenko A., Los V. Some problems for Petrovskii parabolic systems in generalized Sobolev spaces. \emph{arXiv preprint}. arXiv:2103.16474. 2021.

\bibitem{Eidelman94}
Eidel'man S. D. Parabolic equations.
\emph{Encyclopaedia of Mathematical Sciences}. Vol. 63. Partial differential equations, VI. Berlin: Springer, 1994. P.~205--316.

\bibitem{EidelmanZhitarashu98}
Eidel'man S. D., Zhitarashu N. V. Parabolic boundary value problems.
Basel: Birkh\"auser, 1998. xii+298 p.

\bibitem{EvansOpic00}
Evans W. D., Opic B. Real interpolation with logarithmic functors and reiteration. \emph{Canadian Journal of Mathematics.} 2000. Vol.~52, N.~5. P.~920--960.

\bibitem{EvansOpic02}
Evans W. D., Opic B. Real interpolation with logarithmic functors. \emph{Journal of Inequalities and Applications.} 2002. Vol.~7, N.~2. P.~187--269.

\bibitem{Faierman20}
Faierman M. Fredholm theory for an elliptic differential operator defined on $\mathbb{R}^{n}$ and acting on generalized Sobolev spaces. \emph{Communications on Pure and Applied Analysis.} 2020. Vol.~19, N.~3. P.~1463--1483.

\bibitem{Fan94}
Fan M. Complex interpolation functors with a family of quasi-power function parameters. \emph{Studia Mathematica.} 1994. Vol.~111, N.~3. P.~283--305.

\bibitem{Fan11}
Fan M. Qudratic interpolation and some operator inequalities. \emph{Journal of Mathematical Inequalities}. 2011. Vol.~5, N. 3. P.~413--427.

\bibitem{FanKaijser94}
Fan M., Kaijser S. Complex interpolation with derivatives of analytic functions. \emph{Journal of Functional Analysis.} 1994. Vol.~120, N.~2.  P.~380--402.

\bibitem{FarkasJacobScilling01b}
Farkas W., Jacob N., Schilling R.~L. Function spaces related to continuous negative definite functions: $\psi$-Bessel potential spaces. \emph{Dissertationes Mathematicae.} 2001. Vol.~393. P.~1--62.

\bibitem{FarkasLeopold06}
Farkas W., Leopold H.-G. Characterisations of function spaces of generalized smoothness. \emph{Annali di Matematica Pura ed Applicata.} 2006. Vol.~185, N.~1. P.~1--62.

\bibitem{FoiasLions61}
Foia\c{s} C., Lions J.-L. Sur certains th\'eor\`emes d'interpolation
\emph{Acta Scientiarum Mathematicarum (Szeged)}. 1961. Vol.~22, N.~3--4. P. 269--282.

\bibitem{Gagliardo68}
Gagliardo E. Caratterizzazione construttiva di tutti gli spazi di interpolazione tra spazi di Banach. Symposia Mathematica. Vol. II, INDAM, Rome (1968). London: Academic Press, 1969. P.~95--106.

\bibitem{GelukHaan87}
Geluk J. L., Haan L.~de. Regular variation, extensions and Tauberian theorems. Amsterdam: Stichting Mathematisch Centrum, 1987. iv+132~p.

\bibitem{Grubb95}
Grubb G. Parameter-elliptic and parabolic pseudodifferential boundary problems in global $L_p$ Sobolev spaces.
\emph{Mathematische Zeitschrift}. 1995. Vol. 218, N.~1. P. 43--90.

\bibitem{GrubbSolonnikov90}
Grubb G., Solonnikov  V. A.  Solution of parabolic pseudo-differential initial-boundary value problems.
\emph{Journal Differential Equations}. 1990. Vol.~87, N.~2. P. 256--304.

\bibitem{Gustavsson78}
Gustavsson J. A function parameter in connection with interpolation of Banach spaces. \emph{Mathematica Scandinavica.} 1978. Vol.~42, N.~2. P.~289--305.

\bibitem{Haan70}
Haan L.~de. On Regular variation and its application to the weak convergence of sample extremes. Amsterdam: Mathematisch Centrum, 1970. v+124~p.

\bibitem{HieberPruss97}
Hieber M., Pr\"uss  J. Heat kernels and maximal $L_p-L_q$ estimates for para\-bo\-lic evolution equations. \emph{Communications in Partial Differential Equations}. 1997. Vol.~22, N.~9--10. P.~1647--1669.

\bibitem{HaroskeMoura04}
Haroske D. D., Moura S.~D. Continuity envelopes of spaces of generalised smoothness, entropy and approximation numbers. \emph{Journal of Approximation Theory.} 2004. Vol.~128. P.~151--174.

\bibitem{HaroskeMoura08}
Haroske, D. D., Moura, S. D. Continuity envelopes and sharp embeddings in spaces of generalized smoothness. \emph{Journal of Functional Analysis.} 2008. Vol.~254, N.~6. P.~1487--1521.

\bibitem{Hegland92}
Hegland M. An optimal order regularization method which does not
use additional smoothness assumptions.  \emph{SIAM Journal of Numerical Analysis.} 1992. Vol.~29. P.~1446-1461,

\bibitem{Hegland95}
Hegland M. Variable Hilbert Scales and their interpolation
inequalities with applications to Tikhonov regularization. \emph{Applicable
Analysis.} 1995. Vol.~59, N. 1--4. P.~207--223.

\bibitem{Hegland10}
Hegland M. Error bounds for spectral enhancement which are based on variable Hilbert scale inequalities. \emph{Journal of Integral Equations and Applications.} 2010. Vol.~22, N.~2. P.~285--312.

\bibitem{HeglandAnderssen11}
Hegland M., Anderssen R. S. Dilational interpolatory inequalities. \emph{Mathematics of Computation.} 2011. Vol.~80, N.~274. P.~1019--1036.

\bibitem{HeglandHofmann11}
Hegland M., Hofmann B. Errors of regularisation under range inclusions using variable Hilbert scales. \emph{Inverse Problems and Imaging.} 2011. Vol.~5, N.~3. P.~619--643.

\bibitem{Hummel21}
Hummel F. Boundary value problems of elliptic and parabolic type with boundary data of negative regularity. \emph{Journal of Evolution Equations.} 2021. https://doi.org/10.1007/s00028-020-00664-0

\bibitem{HummelLindemulder19}
Hummel F., Lindemulder N. Elliptic and parabolic boundary value problems in weighted function spaces. \emph{arXiv preprint.} arXiv:1911.04884. 2019.

\bibitem{IlkivStrap15}
Il'kiv V. S., Strap N. I. Solvability of the nonlocal boundary-value problem for a system of differential-operator equations in the Sobolev scale of spaces and in a refined scale. \emph{Ukrainian Mathematical Journal.} 2015. Vol.~67, N.~5. P.~690--710.

\bibitem{IlkivStrapVolyanska20}
Il'kiv V. S., Strap N. I., Volyanska I. I. Solvability conditions for the nonlocal boundary-value problem for a differential-operator equation with weak nonlinearity in the refined Sobolev scale of spaces of functions of many real variables. \emph{Ukrainian Mathematical Journal.} 2020. Vol.~72, N.~4. P.~515--535.

\bibitem{Jacob010205}
Jacob N. Pseudodifferential operators and Markov processes: In 3 vol\-umes.
London: Imperial College Press, 2001, 2002, 2005. xxii+493~p., xxii+453~p., xxviii+474~p.

\bibitem{Janson81}
Janson S. Minimal and maximal methods of interpolation.
\emph{Journal of Functional Analysis}. 1981. Vol.~44, N.~1. P. 50--73.

\bibitem{JonsenHansenSickel15}
Jonsen J., Hansen S. M., Sickel W. Anisotropic Lizorkin–Triebel spaces with mixed norms~--- traces on smooth boundaries. \emph{Mathematische Nachrichten.} 2015. Vol.~288, N.~11--12. P.~1327--1359.

\bibitem{JinTautenhahn11}
Jin Q., Tautenhahn U. Implicit iteration methods in Hilbert scales under general smoothness conditions. \emph{Inverse Problems.} 2011. Vol.~27, N.~4, article N.~045012. 27~pp.

\bibitem{KalyabinLizorkin87}
Kalyabin G. A., Lizorkin P.~I. Spaces of functions of generalized smoothness. \emph{Mathematische Nachrichten.} 1987. Vol.~133, N.~1. P.~7--32.

\bibitem{Karamata30a}
Karamata J. Sur certains "Tauberian theorems" \;de M.~M.~Hardy et
Litt\-lewood.
\emph{Mathematica (Cluj)}. 1930. Vol.~3. P. 33--48.

\bibitem{Karamata30b}
Karamata J. Sur un mode de croissance r\'eguli\`ere des fonctions.
\emph{Ma\-the\-ma\-ti\-ca (Cluj)}. 1930. Vol.~4. P.~38--53.

\bibitem{Karamata33}
Karamata J. Sur un mode de croissance r\'eguli\`ere. Th\'eor\`ems
foun\-d\-amen\-taux. \emph{Bulletin de la Soci\'et\'e Math\'ematique de France.} 1933. Vol.~61. P.~55--62.

\bibitem{KasirenkoMurach18MFAT2}
Kasirenko T., Murach A. Elliptic problems with boundary operators of higher orders in H\"rmander-Roitberg spaces. \emph{Methods of Functional Analysis and Topology.} 2018. Vol.~24, N.~2. P.~120--142.

\bibitem{KasirenkoMurach18UMJ11}
Kasirenko T. M., Murach O. O. Elliptic problems with boundary conditions of higher orders in H\"ormander spaces. \emph{Ukrainian Mathematical Journal.} 2018. Vol.~69, N.~11. P.~1727--1748.

\bibitem{Kim08}
Kim K.-H.  $L_q(L_p)$-theory of parabolic PDEs with variable coefficients. \emph{Bulletin of the Korean Mathematical Society}. 2008. Vol.~45, N.~1. P.~169--190.

\bibitem{Knopova06}
Knopova V. P. Continuity of certain pseudodifferential operators in spaces of generalized smoothness. \emph{Ukrainian Mathematical Journal.} 2006. Vol.~58, N.~5. P~718--736.

\bibitem{KohnePrussWilke10}
Kohne M., Pr\"uss J., Wilke M. On quasilinear parabolic evolution equations in weighted $L_p$-spaces. \emph{Journal of Evolution Equations}. 2010. Vol.~10, N.~2. P.~443--463.

\bibitem{Krugljak93}
Krugljak N. Ja. On the reiteration property of $\overrightarrow{X}_{\varphi,q}$ spaces. \emph{Mathematica Scandinavica.} 1993. Vol.~73, N.~1. P.~65--80.

\bibitem{Krylov96}
Krylov N. V. Lectures on elliptic and parabolic equations in H\"older spaces. Providence, RI: American Mathematical Society, 1996. xii+164~p.

\bibitem{Krylov99}
Krylov N. V. Weighted Sobolev spaces and Laplace's equation and the heat equations in a half space.
\emph{Communications in Partial Differential Equations}. 1999. Vol. 24, N.~9--10. P. 1611--1653.

\bibitem{Krylov01}
Krylov N. V. The heat equation in $L_q((0, T), L_p)$-spaces with weights.
\emph{SIAM Journal on Mathematical Analysis}. 2001. Vol. 32, N.~5. P. 1117--1141.

\bibitem{Krylov08}
Krylov N. V. Lectures on elliptic and parabolic equations in Sobolev spaces. Providence, RI: American Mathematical Society, 2008. xviii+357~p.

\bibitem{KunstmannWeis04}
Kunstmann P. C., Weis L. Maximal $L_p$-regularity for parabolic equations, Fourier multiplier theorems and $H^\infty$-functional calculus.
\emph{Functional analytic methods for evolution equations. Lecture Notes in Mathematics}. Vol.~1855. Berlin: Springer, 2004. P. 65--311.

\bibitem{LeCronePrussWilke14}
LeCrone J., Pr\"uss J., Wilke M. On quasilinear parabolic evolution equations in weighted $L_p$-spaces II.
\emph{Journal of Evolution Equations}. 2014. Vol.~14, N.~3. P. 509--533.

\bibitem{Lieberman96}
Lieberman G.~M. Second order parabolic differential equations. New Jersey etc: World Scientific, 1996. xii+447~p.

\bibitem{Lindemulder18arxiv}
Lindemulder N. Second order operators subject to Dirichlet boundary conditions in weighted Triebel--Lizorkin spaces: Parabolic problems.  \emph{arXiv preprint}. arXiv:1812.05462. 2018.

\bibitem{Lindemulder20}
Lindemulder N. Maximal regularity with weights for parabolic problems with inhomogeneous boundary conditions. \emph{Journal of Evolution Equations}. 2020. Vol.~20, N.~1. P. 59--108.

\bibitem{LindemulderVeraar20}
Lindemulder N., Veraar M. The heat equation with rough boundary conditions and holomorphic functional calculus. \emph{Journal of Differential Equations}. 2020. Vol.~269, N.~7. P.~5832--5899.

\bibitem{Lions58}
Lions J.-L. Espaces interm\'ediaires entre espaces hilbertiens et applications.
\emph{Bulletin Mathematique de la Societe des Sciences Mathematiques de Roumanie}. 1958. Vol.~50, N.~4. P. 419--432.

\bibitem{Lions59}
Lions J.-L. Th\'eor\`emes de trace et d'interpolation.~I. \emph{Annali della Scuola Normale Superiore di Pisa. Serie III.} 1959. Vol.~13. P.~389--403.

\bibitem{Lions60}
Lions J.-L. Une construction d'espaces d'interpolation. \emph{Comptes Rendus Math\'ematique. Acad\'emie des Sciences. Paris.} 1961. Vol.~251. P.~1853--1855.

\bibitem{LionsMagenes72ii}
Lions J.-L., Magenes E. Non-homogeneous boundary-value problems and applications. Vol.~II. Berlin: Springer, 1972. x+242~p. (Переклад видання: Lions J.-L., Magenes E. Probl\`emes aux limites non homog\`enes et applications. Vol. 2. Paris: Dunod, 1968. xvi+251 pp.)

\bibitem{LionsPeetre61}
Lions J.-L., Peetre J. Propri\'et\'es d'espaces d'interpolation. \emph{Comptes Rendus Math\'ematique. Acad\'emie des Sciences. Paris.} 1961. Vol.~253. P.~1747--1749.

\bibitem{Lofstrom91}
L\"ofstr\"om J. A new approach to interpolation in Hilbert spaces. \emph{Journal of Functional Analysis.} 1991. Vol.~101. P.~177--193.

\bibitem{LoosveldtNicolay19}
Loosveldt L., Nicolay S. Some equivalent definitions of Besov spaces of generalized smoothness. \emph{Mathematische Nachrichten.} 2019. Vol.~292, N.~10. P.~2262--2282.

\bibitem{Los15UMJ5}
Los V. M. Mixed Problems for the Two-Dimensional Heat-Conduction Equation in Anisotropic H\"ormander Spaces. \emph{Ukrainian Mathematical Journal}. 2015. Vol. 67, N.~5. P. 735--747.

\bibitem{Los16JMS4}
Los V. M. Anisotropic Hormander Spaces on the Lateral Surface of a Cylinder. \emph{Journal of Mathematical Sciences (New York)}. 2016. Vol. 217, N.~4. P. 456--467.

\bibitem{Los16UMJ6}
Los V. M. Theorems on Isomorphisms for Some Parabolic Initial-Boundary-Value Problems in H\"ormander Spaces: Limiting Case.
\emph{Ukrainian Mathematical Journal}. 2016. Vol. 68, N.~6. P. 894--909.

\bibitem{Los16UMJ9}
Los V. M. Classical Solutions of Parabolic Initial-Boundary-Value Problems and H\"ormander Spaces.
\emph{Ukrainian Mathematical Journal}. 2017. Vol.~68, N.~9. P. 1412--1423.

\bibitem{Los16UMJ11}
Los V. M. Sufficient Conditions for the solutions of General Parabolic Initial-Boundary-Value Problems to be Classical.
\emph{Ukrainian Mathematical Journal}. 2017. Vol. 68, N.~11. P. 1756--1766.

\bibitem{Los17UMJ3}
Los V. M. Systems Parabolic in Petrovskii’s Sense in H\"ormander Spaces.
\emph{Ukrainian Mathematical Journal}. 2017. Vol. 69, N.~3. P. 426--443.

\bibitem{Los17MFAT2}
Los V. M. Initial-boundary value problems for two-dimensional parabolic equations in H\"ormander spaces.
\emph{Methods of Functional Analysis and Topology}. 2017. Vol. 23, N.~2. P. 177--191.

\bibitem{Los20MFAT2}
Los V. M. A condition for generalized solutions of a parabolic problem for a Petrovskii system to be classical.
\emph{Methods of Functional Analysis and Topology}. 2020. Vol. 26, N.~2. P. 111--118.

\bibitem{LosMikhailetsMurach17CPAA1}
Los V. M., Mikhailets V. A., Murach A. A. An isomorphism theorem for parabolic problems in H\"ormander spaces and its applications.
\emph{Communications on Pure and Applied Analysis}. 2017. Vol. 16, N.~1. P. 69--97.

\bibitem{LosMikhailetsMurach21arXiv}
Los V. M., Mikhailets V. A., Murach A. A. Parabolic problems in generalized Sobolev spaces. \emph{Communications on Pure and Applied Analysis}. 2021. Vol.~20. doi: 10.3934/cpaa.2021123

\bibitem{LosMurach13MFAT2}
Los V., Murach A. A. Parabolic problems and interpolation with a function parameter. \emph{Methods of Functional Analysis and Topology}.  2013. Vol.~19, N.~2. P. 146--160.

\bibitem{LosMurach17OM15}
Los V. M., Murach A. A. Isomorphism theorems for some parabolic initial-boundary value problems in H\"ormander spaces.
\emph{Open Mathematics}. 2017. Vol. 15. P. 57--76.

\bibitem{Lunardi95}
Lunardi A.  Analytic semigroups and optimal regularity in parabolic problems. Birkhauser Verlag, Basel, 1995.

\bibitem{MalarskiTriebel91}
Malarski M., Triebel H. Anisotropic function spaces: Hardy’s inequality and traces on surfaces. \emph{Czechoslovak Mathematical Journal.} 1991. Vol.~41 (116), N.~3. P.~518--537.

\bibitem{Malgrange57}
Malgrange B. Sur une classe d'op\'eratuers diff\'erentiels hypoelliptiques. \emph{Bulletin de la Soci\'et\'e Math\'ematique de France.} 1957. Vol.~85, N.~3. P.~283--306.

\bibitem{MaligrandaPerssonWyller94}
Maligranda L., Persson L. E., Wyller J. Interpolation and partial differential equations. \emph{Journal of Mathematical Physics.} 1994. Vol.~35, N.~9. P.~5035--5046.

\bibitem{Maric00}
Maric V. Regular variation and differential equations. New York: Springer
Verlag, 2000. 127~p.

\bibitem{MathePereverzev03}
Math\'e P., Pereverzev S. V. Geometry of linear ill-posed
problems in variable Hilbert scales. \emph{Inverse Problems.} 2003. Vol.~19, N.~3. P.~789--803.

\bibitem{MatheTautenhahn06}
Math\'e P., Tautenhahn U. Interpolation in variable Hilbert scales with application to innverse problems. \emph{Inverse Problems.} 2006. Vol.~22, N.~6. P.~2271--2297.

\bibitem{MazyaShaposhnikova09}
Maz'ya V.~G., Shaposhnikova T.~O. Theory of Sobolev multipliers. With applications to differential and integral operators.
Berlin: Springer, 2009. xiii+609~p.

\bibitem{Merucci82}
Merucci C. Interpolation r\'eelle avec fonction param\`etre: r\'eit\'eration et
applications aux espaces $\Lambda^{\varrho}(\varphi)$.
\emph{Comptes Rendus de l'Academie des Sciences - Series I - Mathematics}. 1982. Vol. 295, N.~6. P. 427--430.

\bibitem{Merucci84}
Merucci C. Application of interpolation with a function parameter to
Lorentz, Sobolev and Besov spaces.
\emph{Interpolation Spaces and Allied Topics in Analysis. Lecture Notes in Mathematics}. Berlin: Springer, 1984. Vol.~1070. P. 183--201.

\bibitem{MikhailetsMurach05UMJ5}
Mikhailets V. A., Murach A. A. Elliptic operators in a refined scale of function spaces.
\emph{Ukrainian Mathematical Journal}. 2005. Vol. 57, N.~5. P.~817--825.

\bibitem{MikhailetsMurach06UMJ2}
Mikhailets V. A., Murach A. A. Refined scales of spaces and elliptic
boundary-value problems. I.
\emph{Ukrainian Mathematical Journal}. 2006. Vol.~58, N.~2. P. 244--262.

\bibitem{MikhailetsMurach06UMJ3}
Mikhailets V. A., Murach A. A. Refined scale of spaces, and elliptic boundary-value problems. II.
\emph{Ukrainian Mathematical Journal}. 2006. Vol.~58, N.~3. P. 398--417.

\bibitem{MikhailetsMurach06UMJ11}
Mikhailets V. A., Murach A. A. A regular elliptic boundary-value problem for a homogeneous equation in a two-sided refined scale of spaces.
\emph{Ukrainian Mathematical Journal}. 2006. Vol. 58, N.~11. P. 1748--1767.

\bibitem{MikhailetsMurach07UMJ5}
Mikhailets V. A., Murach A. A. Refined scale of spaces, and elliptic
boundary-value problems. III.
\emph{Ukrainian Mathematical Journal}. 2007. Vol.~59, N.~5. P. 744--765.

\bibitem{MikhailetsMurach08MFAT1}
Mikhailets V. A., Murach A. A.
Interpolation with a function parameter and refined scale of spaces.
\emph{Methods of Functional Analysis and Topology}. 2008. Vol. 14, N.~1. P. 81--100.

\bibitem{MikhailetsMurach08UMJ4}
Mikhailets V. A., Murach A. A. An elliptic boundary-value problem in a
two-sided refined scale of spaces.
\emph{Ukrainian Mathematical Journal}. 2008. Vol. 60, N.~4. P. 574--597.

\bibitem{MikhailetsMurach08BPAS3}
Mikhailets V. A., Murach A. A. Elliptic systems of pseudodifferential equations in a refined
scale on a closed manifold.
\emph{Bulletin of the Polish Academy of Sciences Mathematics}. 2008. Vol.~56, N.~3--4. P. 213--224.

\bibitem{MikhailetsMurach09OperatorTheory191}
Mikhailets V. A., Murach A. A.
Elliptic problems and H\"ormander spaces.
\emph{Operator Theory: Advances and Applications}. Basel: Birkh\"aser, 2009. Vol.~191. P. 447--470.

\bibitem{MikhailetsMurach12BJMA2}
Mikhailets V. A., Murach A. A.
The refined Sobolev scale, inter\-po\-la\-tion, and elliptic problems.
\emph{Banach Journal of Mathematical Analysis}. 2012. Vol. 6, N.~2. P. 211--281.

\bibitem{MikhailetsMurach13UMJ3}
Mikhailets V. A., Murach A. A.
Extended Sobolev scale and elliptic operators.
\emph{Ukrainian Mathematical Journal}. 2013. Vol. 65, N.~3. P. 435--447.

\bibitem{MikhailetsMurach14}
Mikhailets V. A., Murach A. A. Ho\"rmander spaces, interpolation, and elliptic problems.
Berlin: De Gruyter, 2014. xiv+297~p.

\bibitem{MikhailetsMurach15ResMath1}
Mikhailets V. A., Murach A. A. Interpolation Hilbert spaces between Sobolev spaces.
\emph{Results in Mathematics}. 2015. Vol. 67, N.~1--2. P. 135--152.

\bibitem{MikhailetsMurach21arxiv02}
Mikhailets V., Murach A., Zinchenko T. An extended Hilbert scale and its applications. \emph{arXiv preprint.} arXiv:2102.08089. 2021.

\bibitem{MikuleviciusPhonsom18arxiv}
Mikulevi\v{c}ius R., Phonsom C. On the Cauchy problem for stochastic integrodifferential parabolic equations in the scale of $L^p$-spaces of generalized smoothness. \emph{arXiv preprint.} arXiv:1805.03232.

\bibitem{MikuleviciusPhonsom19PotAnal}
Mikulevi\v{c}ius R., Phonsom C. On the Cauchy problem for integro-differential equations in the scale of spaces of generalized smoothness. \emph{Potential Analysis.} 2019. Vol.~50, N.~3. P.~467--519.

\bibitem{MouraNevesSchneider11}
Moura S. D., Neves J. S., Schneider C. Optimal embeddings of spaces of generalized smoothness in the critical case. \emph{The Journal of Fourier Analysis and Applications.} 2011. Vol.~17, N.~5. P.~777--800.

\bibitem{MouraNevesSchneider14}
Moura S. D., Neves J. S., Schneider C. Spaces of generalized smoothness in the critical case: optimal embeddings, continuity envelopes and approximation numbers. \emph{Journal of Approximayion Theory.} 2014. Vol.~187. P.~82--117.

\bibitem{Murach07UMJ6}
Murach A. A. Elliptic pseudo-differential operators in a refined scale of
spaces on a closed manifold. \emph{Ukrainian Mathematical Journal}. 2007. Vol.~59, N.~6. P. 874--893.

\bibitem{Murach08UMB3}
Murach A. A. Systems elliptic in the Douglis--Nirenberg sense in spaces of generalized smoothness. \emph{Ukrainian Mathematical Bulletin.} 2008. Vol.~5, N.~3. P.~345--359.

\bibitem{Murach08MFAT2}
Murach A. A. Douglis--Nirenberg elliptic systems in the refined scale of spaces on a closed manifold. \emph{Methods of Functional Analysis and Topology.} 2008. Vol.~14, N.~2. P.~142--158.

\bibitem{Murach09UMJ3}
Murach A. A. On elliptic systems in H\"{o}rmander spaces. \emph{Ukrainian Mathematica Journal.} 2009. Vol.~61, N.~3. P.~467--477.

\bibitem{MurachChepurukhina15UMJ5}
Murach A. A., Chepurukhina I. S. Elliptic boundary-value problems in the sense of Lawruk on Sobolev and Hormander spaces. \emph{Ukrainian Mathematical Journal.} 2015. Vol.~67, N.~5. P.~764--784.

\bibitem{MurachZinchenko13MFAT1}
Murach A. A., Zinchenko T. Parameter--elliptic operators on the extended Sobolev scale. \emph{Methods of Functional Analysis and Topology}. 2013. Vol. 19, N.~1. P. 29--39.

\bibitem{NevesOpic20}
Neves J. S., Opic B. Optimal local embeddings of Besov spaces involving only slowly varying smoothness. \emph{Journal of Approximation Theory.} 2020. Vol.~254. Article N.~105393. 25~p.

\bibitem{NicolaRodino10}
Nicola F., Rodino L. Global pseudodifferential calculas on Euclidean spaces. Basel: Birkh\"aser, 2010. x+306~p.

\bibitem{Ovchinnikov84}
Ovchinnikov V. I. The methods of orbits in interpolation theory.
\emph{Mathematical Reports}. 1984. Vol.~1, N.~2. P. 349--515.

\bibitem{Ovchinnikov05}
Ovchinnikov V. I. Interpolation orbits in couples of Lebesgue spaces.
\emph{Functional Analysis and Its Applications}. 2005. Vol.~39, N.~1. P. 46--56.

\bibitem{Paneah00}
Paneah B. The oblique derivative problem. The Poincar\'e problem.
Berlin: Wiley--VCH, 2000. 348~p.

\bibitem{Peetre63}
Peetre J. Sur le nombre de param\`etres dans la d\'efinition de certain espaces d'interpolation. \emph{Ricerche di Matematica.} 1963. Vol.~15. P.~248--261.

\bibitem{Peetre66}
Peetre J. On interpolation functions.
\emph{Acta Scientiarum Mathematicarum (Szeged)}. 1966. Vol.~27, N.~3--4. P. 167--171.

\bibitem{Peetre68}
Peetre J. On interpolation functions. II.
\emph{Acta Scientiarum Mathematicarum (Szeged)}. 1968. Vol.~29, N.~1--2. P. 91--92.

\bibitem{Persson86}
Persson L.-E. Interpolation with a function parameter.
\emph{Mathematica Scandinavica}. 1986. Vol.~59, N.~2. P. 199--222.

\bibitem{Reshnick87}
Reshnick S. I. Extreme values, regular variation and point processes. New York: Springer Verlag, 1987. 320~p.

\bibitem{Schechter67}
Schechter M. Complex interpolation. \emph{Compositio Mathematica.} 1967. Vol.~18, N. 1--2. P.~117--147.

\bibitem{Simon19}
Simon B. Loewner's theorem on monotone matrix functions. Cham: Springer, 2019. xi+459~p.

\bibitem{Stepanets05}
Stepanets A. I. Methods of Approximation Theory. Utrecht: VSP, 2005.

\bibitem{Tartar07}
Tartar L. An introduction to Sobolev spaces and interpolation spaces. Berlin: Springer, 2007. xxv+218~p.

\bibitem{Triebel77I}
Triebel H. General function spaces. I. Decomposition methods. \emph{Mathematische Nachrichten.} 1977. Vol.~79. P.~167--179.

\bibitem{Triebel77II}
Triebel H. General function spaces. II. Inequalities of Plancherel--Polya--Nikol'skii-type, $L_p$-spaces of analytic
functions, $0<p\leq\infty$. \emph{Journal of Approximation Theory.} 1977. Vol.~19. P.~154--175.

\bibitem{Triebel77III}
Triebel H. General function spaces. III. Spaces $B^{g(x)}_{p,q}$
and $F^{g(x)}_{p,q}$, $\nobreak{1<p<\infty}$: basic properties. \emph{Analysis Mathematica.} 1977. Vol.~3, N.~3. P.~221--249.

\bibitem{Triebel77IV}
Triebel H. General function spaces. IV. Spaces $B^{g(x)}_{p,q}$
and $F^{g(x)}_{p,q}$, $1<p<\infty$: special properties. \emph{Analysis Mathematica.} 1977. Vol.~3, N.~4. P.~299--315.

\bibitem{Triebel79V}
Triebel H. General function spaces. V. The spaces $B^{g(x)}_{p,q}$: the case $\nobreak{0<p\leq\infty}$. \emph{Mathematische Nachrichten.} 1979. Vol.~87. P.~129--152.

\bibitem{Triebel84I}
Triebel H. Anisotropic function spaces. I. Hardy’s inequality, decompositions. \emph{Analysis Mathematica.} 1984. Vol.~10, N.~1. P.~53--77.

\bibitem{Triebel84II}
Triebel H. Anisotropic function spaces. II. Traces. \emph{Analysis Mathematica.} 1984. Vol.~10, N.~1. P.~79--96.

\bibitem{Triebel01}
Triebel H. The structure of functions. Basel: Birkh\"aser, 2001. xii+425~p.

\bibitem{Triebel06}
Triebel H. Theory of function spaces. III. Basel: Birkh\"aser, 2006.
xii+426~p.

\bibitem{Triebel10}
Triebel H. Bases in function spaces, sampling, discrepancy, numerical integration.
Z\"urich: European Mathematical Society, 2010. x+296 pp.

\bibitem{WangWang14}
Wang H., Wang K. Optimal recovery of Besov classes of generalized smoothness and Sobolev classes on the sphere. \emph{Journal of Complexity.} 2016. Vol.~32, N.~1. P.~40--52.

\bibitem{Weidemaier05}
Weidemaier P.  Lizorkin–Triebel spaces of vector-valued functions and sharp trace theory for functions in Sobolev spaces with a mixed Lp-norm in parabolic problems. \emph{Sbornik: Mathematics}. 2005. Vol.~196, N.~6. P.~3--16.

\bibitem{YuanSickelYang15}
Yuan W., Sickel W., Yang D. Interpolation of Morrey--Companato and related smoothness spaces. \emph{Science in China. Series A. Mathematics.} 2015. Vol.~58, N.~1. P.~1835--1908.

\bibitem{Zinchenko17OpenMath}
Zinchenko T. Elliptic operators on refined Sobolev scales on vector bundles. \emph{Open Mathematics.} 2017. Vol.~15. P.~907--925.

\bibitem{ZinchenkoMurach12UMJ11}
Zinchenko T. N., Murach A. A. Douglis--Nirenberg elliptic systems in H\"ormander spaces. \emph{Ukrainian Mathematical Journal}.  2013. Vol. 64, N.~11. P. 1672--1687.

\bibitem{ZinchenkoMurach14JMS}
Zinchenko T. N., Murach A. A. Petrovskii elliptic systems in the extended Sobolev scale. \emph{Journal of Mathematical Sciences (New York).} 2014. Vol.~196, N.~5. P.~721--732.


\end{thebibliography}
\end{document}